\documentclass[11pt, reqno, dvipsnames]{amsart}

\usepackage[english]{babel}
\usepackage[utf8]{inputenc}

\usepackage[pdftex]{graphicx}
\usepackage{amsmath}
\usepackage{amssymb}
\usepackage{amsfonts}
\usepackage{amsthm}
\usepackage[all,cmtip]{xy}
\usepackage{tikz}
\usepackage{tikz-cd}
\usepackage[new]{old-arrows}
\usepackage{blindtext}
\usepackage{accents}
\usepackage{scalerel}
\usepackage{stmaryrd}
\usepackage{mathrsfs}
\usepackage{dsfont}
\usepackage{relsize}
\usepackage{url}
\usepackage{contour}
\usepackage{ulem}
\usepackage{footmisc}
\usepackage{enumerate}
\usepackage[T1]{fontenc}
\usepackage{fancyhdr}

\usetikzlibrary{decorations.markings}
\tikzset{double line with arrow/.style args={#1,#2}{decorate,decoration={markings, mark=at position 0 with {\coordinate (ta-base-1) at (0,1pt);
\coordinate (ta-base-2) at (0,-1pt);},
mark=at position 1 with {\draw[#1] (ta-base-1) -- (0,1pt);
\draw[#2] (ta-base-2) -- (0,-1pt);
}}}}

\usepackage{mathtools}

\usetikzlibrary{trees}

\tikzstyle{level 1}=[level distance=2.4cm, sibling distance=6.5cm]
\tikzstyle{level 2}=[level distance=2.4cm, sibling distance=2.5cm]
\tikzstyle{level 3}=[level distance=3cm, sibling distance=0.8cm]

\usepackage[linktocpage=true]{hyperref}
\hypersetup{
    colorlinks = true,
    linkcolor = {ForestGreen},
    citecolor = {RoyalBlue},
    urlcolor = {Black}
}

\newcommand{\listref}[1]{{\color{ForestGreen}(}\ref{#1}{\color{ForestGreen})}}

\newcommand{\Cc}{\mathbb{C}}
\newcommand{\Zz}{\mathbb{Z}}
\newcommand{\Ff}{\mathbb{F}}
\newcommand{\Nn}{\mathbb{N}}
\newcommand{\Qq}{\mathbb{Q}}

\newcommand{\Pp}{\mathbb{P}}

\newcommand{\Aa}{\mathbb{A}}

\newcommand{\Bb}{\mathbb{B}}

\newcommand{\Ll}{\mathbb{L}}
\newcommand{\Mm}{\mathbb{M}}
\newcommand{\Dd}{\mathbb{D}}

\newcommand{\solid}{\otimes^{\scalebox{0.5}{$\blacksquare$}}}
\newcommand{\dsolid}{\otimes^{\LL \scalebox{0.5}{$\blacksquare$}}}

\newcommand{\solidif}{\scalebox{0.5}{$\blacksquare$}}

\newcommand{\cl}[1]{\mathcal{#1}}
\newcommand{\fr}[1]{\mathfrak{#1}}

\contourlength{0.8pt}

\DeclareMathOperator*{\colim}{colim}
\DeclareMathOperator{\ett}{\acute{e}t}
\DeclareMathOperator{\pet}{pro\acute{e}t}
\DeclareMathOperator{\qpet}{qpro\acute{e}t}
\DeclareMathOperator{\an}{an}

\DeclareMathOperator{\rk}{rk}

\DeclareMathOperator{\Gal}{Gal}
\DeclareMathOperator{\Hom}{Hom}

\DeclareMathOperator{\coker}{coker}

\DeclareMathOperator{\Spec}{Spec}
\DeclareMathOperator{\Spa}{Spa}

\DeclareMathOperator{\dR}{dR}

\DeclareMathOperator{\Fil}{Fil}
\DeclareMathOperator{\gr}{gr}

\DeclareMathOperator{\Rig}{Rig}
\DeclareMathOperator{\RigSm}{RigSm}
\DeclareMathOperator{\Mod}{Mod}

\DeclareMathOperator{\nonint}{n-int}
\DeclareMathOperator{\Kos}{Kos}
\DeclareMathOperator{\Vect}{Mod}

\DeclareMathOperator{\im}{im}

\DeclareMathOperator{\cond}{cond}
\DeclareMathOperator{\Cond}{Cond}
\DeclareMathOperator{\qcqs}{qcqs}
\DeclareMathOperator{\Set}{Set}

\DeclareMathOperator{\Shv}{Shv}
\DeclareMathOperator{\Ab}{Ab}
\DeclareMathOperator{\Solid}{Solid}
\DeclareMathOperator{\ssolid}{solid}
\DeclareMathOperator{\LL}{L}

\DeclareMathOperator{\FF}{FF}

\DeclareMathOperator{\Nuc}{Nuc}

\DeclareMathOperator{\Perf}{Perf}

\DeclareMathOperator{\CondAb}{CondAb}

\DeclareMathOperator{\ExtrDisc}{ExtrDisc}

\DeclareMathOperator{\trace}{tr}

\DeclareMathOperator{\Bun}{Bun}

\DeclareMathOperator{\st}{st}

\DeclareMathOperator{\Sym}{Sym}
\DeclareMathOperator{\cris}{cris}
\DeclareMathOperator{\HK}{HK}
\DeclareMathOperator{\syn}{syn}
\DeclareMathOperator{\fib}{fib}

\DeclareMathOperator{\Spf}{Spf}
\DeclareMathOperator{\Tor}{Tor}
\DeclareMathOperator{\gp}{gp}
\DeclareMathOperator{\id}{id}
\DeclareMathOperator{\Fun}{Fun}
\DeclareMathOperator{\op}{op}
\DeclareMathOperator{\Beil}{Beil}
\DeclareMathOperator{\Ch}{Ch}
\DeclareMathOperator{\Dim}{Dim}
\DeclareMathOperator{\eh}{\acute{e}h}
\DeclareMathOperator{\CRIS}{CRIS}
\DeclareMathOperator{\INF}{INF}
\DeclareMathOperator{\hyp}{hyp}
\DeclareMathOperator{\Bl}{Bl}
\DeclareMathOperator{\redd}{red}
\DeclareMathOperator{\sm}{sm}
\DeclareMathOperator{\sss}{ss}
\DeclareMathOperator{\nilp}{nilp}

\DeclareMathOperator{\BMS}{BMS}

\DeclareMathOperator{\htt}{ht}
\DeclareMathOperator{\QCoh}{QCoh}
\DeclareMathOperator{\CoAdm}{CoAd}
\DeclareMathOperator{\eq}{eq}
\DeclareMathOperator{\rig}{rig}
\DeclareMathOperator{\PD}{PD}
\DeclareMathOperator{\Isoc}{Isoc}
\DeclareMathOperator{\Proj}{Proj}
\DeclareMathOperator{\ord}{ord}
\DeclareMathOperator{\Ani}{Ani}

\DeclareMathOperator{\CN}{CN}
\DeclareMathOperator{\ccl}{cl}
\DeclareMathOperator{\psh}{psh}
\DeclareMathOperator{\sht}{sht}

\DeclareMathOperator{\Hdg}{Hdg}
\DeclareMathOperator{\Lie}{Lie}

\numberwithin{equation}{section}
\newtheorem{theorem}{Theorem}
\numberwithin{theorem}{section}

\newtheorem{lemma}[theorem]{Lemma}
\newtheorem{cor}[theorem]{Corollary}
\newtheorem{prop}[theorem]{Proposition}
\newtheorem{conj}[theorem]{Conjecture}

\theoremstyle{definition}
\newtheorem{df}[theorem]{Definition}

\newtheorem{convnot}[theorem]{Notation and conventions}
\newtheorem{notation}[theorem]{Notation}
\newtheorem{construction}[theorem]{Construction}

\newtheorem*{Acknowledgments}{Acknowledgments}

\theoremstyle{remark}
\newtheorem{rem}[theorem]{Remark}
\newtheorem{example}[theorem]{Example}

 \AtBeginDocument{%
 \def\MR#1{}
}

\date{\today}

\makeatletter
\renewcommand{\address}[1]{\gdef\@address{#1}}
\renewcommand{\email}[1]{\gdef\@email{\url{#1}}}
\newcommand{\@endstuff}{\par\vspace{\baselineskip}\noindent\small
\begin{tabular}{@{}l}\scshape\@address\\\textit{E-mail address:} \@email\end{tabular}}
\AtEndDocument{\@endstuff}
\makeatother

\title[Rational $p$-adic Hodge theory for rigid-analytic varieties]{Rational $p$-adic Hodge theory for rigid-analytic varieties}
\author{Guido Bosco}
\address{Max-Planck-Institut f\"ur Mathematik, Vivatsgasse 7, 53111 Bonn, Germany}
\email{bosco@mpim-bonn.mpg.de}

\begin{document}
 
\begin{abstract}
 We study a cohomology theory for rigid-analytic varieties over $\Cc_p$,  without properness or smoothness assumptions, taking values in filtered quasi-coherent complexes over the Fargues--Fontaine curve, which compares to other rational $p$-adic cohomology theories for rigid-analytic varieties --- namely, the rational $p$-adic pro-étale cohomology, the Hyodo--Kato cohomology, and the infinitesimal cohomology over the positive de Rham period ring. In particular, this proves a conjecture of Le Bras. Such comparison results are made possible thanks to the systematic use of the condensed and solid formalisms developed by Clausen--Scholze. As applications, we deduce some general comparison theorems that describe the rational $p$-adic pro-étale cohomology in terms of de Rham data, thereby recovering and extending results of Colmez--Nizioł.
\end{abstract}

\maketitle

\setcounter{tocdepth}{2}

\tableofcontents

\section{\textbf{Introduction}}
 \sectionmark{}
 In this introduction, we fix a prime number $p$. We denote by $K$ a complete discretely valued non-archimedean extension of $\Qq_p$, with perfect residue field $k$, and ring of integers $\cl O_K$.  We fix an algebraic closure $\overline K$ of $K$. We denote by $C:=\widehat{\overline K}$ the completion of $\overline K$, and by $\cl O_C$ its ring of integers. We let $\mathscr{G}_K:=\Gal(\overline K/K)$ denote the absolute Galois group of $K$.  We fix a compatible system $(1, \varepsilon_p, \varepsilon_{p^2}, \ldots)$ of $p$-th power roots of unity in $\cl O_C$, which defines an element $\varepsilon\in \cl O_C^{\flat}$ with Teichm\"uller lift $[\varepsilon]\in A_{\inf}=W(\cl O_C^\flat)$.

 \subsection{Background and motivation}
 In the last decade, the field of $p$-adic Hodge theory has witnessed dramatic advances, starting with Scholze's development of perfectoid geometry, and Fargues--Fontaine's discovery of \textit{the fundamental curve}. In particular, Scholze initiated the study of the $p$-adic Hodge theory for rigid-analytic varieties in \cite{Scholze}, proving the finiteness of the geometric $p$-adic étale cohomology of proper smooth rigid-analytic varieties, as well as the de Rham comparison theorem for such varieties. The latter was known before only for algebraic varieties, and we refer the reader to \cite{NizSurvey} for a historical account on the de Rham, crystalline, and semistable conjectures for algebraic varieties. After Scholze's work, there were efforts by a number of people to prove a version of the crystalline/semistable conjecture for proper rigid-analytic varieties having good/semistable reduction, culminating in the following theorem.

 \begin{theorem}[\cite{CN1}, \cite{BMS1}, \cite{CK}]\label{tostart}
  Let $\fr X$ be a proper $p$-adic formal scheme over $\cl  O_K$ with semistable reduction. We write $\fr X_C$ for the geometric rigid-analytic generic fiber of $\fr X$. Let $i\ge 0$. There is a natural isomorphism
  \begin{equation}\label{initiall}
   H_{\ett}^i(\fr X_C, \Qq_p)\otimes_{\Qq_p}B_{\st}\cong H_{\cris}^i(\fr X_k/W(k)^0)\otimes_{W(k)}B_{\st}
  \end{equation}
  compatible with the Galois $\mathscr{G}_K$, Frobenius $\varphi$ and monodromy $N$ actions, and filtrations.\footnote{Here, we write $H^i_{\cris}$ for the log-crystalline cohomology, $W(k)^0$ denotes the log structure on $W(k)$ associated to $(\Nn\to W(k), 1\mapsto 0)$, and $\fr X_k$ is endowed with the pullback of the canonical log structure on $\fr X$.}
  
  In particular, there is natural $\mathscr{G}_K$-equivariant isomorphism
  \begin{equation}\label{dRdata}
   H_{\ett}^i(\fr X_C, \Qq_p)\cong (H_{\cris}^i(\fr X_k/W(k)^0)\otimes_{W(k)}B_{\st})^{N=0, \varphi=1}\cap \Fil^0(H_{\dR}^i(\fr X_K)\otimes_K B_{\dR}).
  \end{equation}
 \end{theorem}

 We remark that Colmez--Nizioł's strategy in \cite{CN1} relies on a generalization of the syntomic method initiated by Fontaine--Messing, and later refined by Hyodo, Kato and Tsuji. Instead, Bhatt--Morrow--Scholze's strategy, in their epoch-making work \cite{BMS1}, is based on the construction of a cohomology theory for smooth $p$-adic formal schemes $\fr Z$ over $\cl O_C$ (generalized to the semistable case by Česnavičius--Koshikawa in \cite{CK}), called \textit{$A_{\inf}$-cohomology} $R\Gamma_{A_{\inf}}(\fr Z)$, which in the proper case specializes to the (log-)crystalline cohomology of the special fiber and the étale cohomology of the generic fiber, thus allowing to compare the latter two as in Theorem \ref{tostart}. \medskip
 
 More recently, in a series of papers, Colmez--Nizioł, partially in joint work with Dospinescu, further generalized the syntomic method to study the rational $p$-adic Hodge theory of smooth rigid-analytic varieties, which are neither assumed to be proper, nor having semistable reduction, \cite{CDN1}, \cite{CN}, \cite{CN4}, \cite{CN5}. These works are motivated in part by the desire of finding a geometric incarnation of the $p$-adic Langlands correspondence in the $p$-adic cohomology of local Shimura varieties, as partially indicated by \cite{CDN0}. \medskip
 
 In order to state the goals of this paper, we will denote by 
 $$Y_{\FF}:=\Spa(A_{\inf}, A_{\inf})\setminus V(p[p^\flat])$$
 the mixed characteristic punctured open unit disk, $$\FF:=Y_{\FF}/\varphi^{\Zz}$$ the adic Fargues--Fontaine curve (relative to $C^\flat$ and $\Qq_p$), and we fix $\infty$ the $(C, \cl O_C)$-point of $\FF$ corresponding to Fontaine's map $\theta: A_{\inf}\to \cl O_C$. \medskip

  As observed by Fargues, Theorem \ref{tostart} can be reformulated as a natural isomorphism of $\mathscr{G}_K$-equivariant vector bundles on $\FF$
  \begin{equation}\label{fain}
   H^i_{\ett}(\fr X_C, \Qq_p)\otimes_{\Qq_p}\cl O_{\FF}\cong \cl E(H_{\cris}^i(\fr X_k/W(k)^0)_{\Qq_p}, \varphi, N, \Fil)
  \end{equation}
  where the right-hand side of (\ref{fain}) denotes the vector bundle on $\FF$ associated to the filtered $(\varphi, N)$-module $H_{\cris}^i(\fr X_k/W(k)^0)_{\Qq_p}$.
  Since the left-hand side of (\ref{fain}) depends only on the geometric generic fiber of $\fr X$, it is natural to ask whether one can give a more direct cohomological construction of the right-hand side that also depends only on the generic fiber, that interpolates between $H^i_{\ett}(\fr X_C, \Qq_p)$ and the filtered $(\varphi, N)$-module $H_{\cris}^i(\fr X_k/W(k)^0)_{\Qq_p}$, and that allows to prove extensions of the comparison (\ref{fain}) to any rigid-analytic variety over $C$. \medskip
  
 Our first goal in this article will be to give a positive answer to the latter question by extending Bhatt--Morrow--Scholze's strategy and  building crucially upon Le Bras' work \cite{LeBras2}. We will study a cohomology theory for rigid-analytic varieties $X$ over $C$, taking values in filtered quasi-coherent complexes over the Fargues--Fontaine curve $\FF$, which compares to other rational $p$-adic cohomology theories for rigid-analytic varieties over $C$, without properness or smoothness assumptions --- namely, the rational $p$-adic pro-étale cohomology, the Hyodo--Kato cohomology (\cite[\S 4]{CN4}),\footnote{As we will see, the \textit{Hyodo--Kato cohomology} is a cohomology theory for rigid-analytic varieties of $C$ which refines the de Rham cohomology and, in the case of Theorem \ref{tostart}, compares to the rational log-crystalline cohomology of the special fiber.} and the infinitesimal cohomology over $B_{\dR}^+$ (\cite[\S 13]{BMS1}, \cite{Guo2}). 
 In particular, this will allow us to obtain a general comparison theorem for rigid-analytic varieties defined over a $p$-adic field, describing the geometric rational $p$-adic pro-étale cohomology in terms of de Rham data, extending (\ref{dRdata}), and recovering and generalizing the above-mentioned results of Colmez--Nizioł.

 \subsection{$B$-cohomology} In the following, we denote by $B$ the ring of analytic functions on $Y_{\FF}$. \medskip
 
 To pursue the goals stated in the previous section, we begin by defining the \textit{$B$-cohomology theory} for rigid-analytic varieties over $C$. Then, we shall explain how this cohomology theory interpolates several other rational $p$-adic cohomology theories, and how to interpret our main comparison theorems in terms of the Fargues--Fontaine curve $\FF$. \medskip
 
 For the reader willing to assume that $X$ is smooth in Definition \ref{dfintro} below, we note that, in this case, the $\eh$-site of $X$ (\cite[\S 2]{Guo1}, \S \ref{recalleh}) can be replaced by the étale site of $X$ (Proposition \ref{3.11}). Our main results on the $B$-cohomology theory are already new in the smooth case.

 \begin{df}[cf. Definition \ref{defdr}]\label{dfintro}
   Let $X$ be a rigid-analytic variety over $C$.  We denote by $\alpha:X_{v}\to X_{\eh}$ the natural morphism from the $v$-site to the $\eh$-site of $X$.
   \begin{enumerate}[(i)]
    \item We define the \textit{$B$-cohomology} of $X$ as 
   $$R\Gamma_{B}(X):=R\Gamma_{\eh}(X, L\eta_tR\alpha_*\Bb)$$
   where $\Bb$ denotes the $v$-site sheaf theoretic version of the ring $B$, and we write $L\eta_t(-)$ for the d\'ecalage functor with respect to $t=\log([\varepsilon])\in B$, i.e. Fontaine's $2\pi i$.
   \item  We define  the \textit{$B_{\dR}^+$-cohomology} of $X$ as 
   $$R\Gamma_{B_{\dR}^+}(X):=R\Gamma_{\eh}(X, L\eta_tR\alpha_*\Bb_{\dR}^+)$$
   where $\Bb_{\dR}^+$ is the $v$-site sheaf theoretic version of the ring $B_{\dR}^+$.
   \end{enumerate}   
   We endow both $R\Gamma_{B}(X)$ and $R\Gamma_{B_{\dR}^+}(X)$ with the \textit{filtration d\'ecal\'ee}, coming from Bhatt--Morrow--Scholze's interpretation of the décalage functor in terms of the connective cover functor for the Beilinson $t$-structure (Definition \ref{beilifildef}). 
   The Frobenius automorphism of $\Bb$ induces a $\varphi_B$-semilinear automorphism
  $$\varphi: R\Gamma_B(X)\to R\Gamma_B(X)$$ which preserves the filtration décalée.
 \end{df}

 \begin{rem}[Le Bras’ work]\label{remcris}
  We recall that, in the paper \cite{LeBras2}, Le Bras introduced and studied (an overconvergent version of) the $B$-cohomology theory for smooth rigid-analytic varieties over $C$. In particular, building upon results of \cite{BMS1}, for $\fr Z$ a smooth proper $p$-adic formal scheme over $\cl O_C$, he compared the $B$-cohomology of the rigid-analytic generic fiber of $\fr Z$ with the crystalline cohomology of the special fiber of $\fr Z$, \cite[Proposition 6.5]{LeBras2}.
 \end{rem}

  In the following, we denote by $F$ the fraction field of the ring of Witt vectors $W(k)$, we write $\breve F$ for the completion of the maximal unramified extension of $F$ in $\overline K$ and we denote by $\cl O_{\breve F}$ its ring of integers. \medskip
  
  To motivate our first main result on the $B$-cohomology theory, we recall that in \cite[\S 4]{CN4}, for smooth rigid-analytic varieties $X$ over $C$, Colmez--Nizioł (adapting a construction of Beilinson in the case of algebraic varieties \cite{Beili}), via the alterations of Hartl and Temkin, defined a \textit{Hyodo--Kato cohomology theory} $$R\Gamma_{\HK}(X)$$  taking values in the derived category of $(\varphi, N)$-modules over $\breve F$, which refines the de Rham cohomology $R\Gamma_{\dR}(X)$, and in the case $X$ has a semistable formal model $\fr X$ over $\cl O_C$ it is given by the rational log-crystalline cohomology $R\Gamma_{\cris}(\fr X_{\cl O_C/p}^0/\cl O_{\breve F}^0)_{\Qq_p}$ (see also \S \ref{sHK}, and in particular Theorem \ref{mainHK}).

  At this point, based on Le Bras' work (Remark \ref{remcris}), it was natural to ask how the $B$-cohomology compares to the Hyodo--Kato cohomology, and whether, at least in the proper case, the latter cohomology theory (which is defined using log-geometry) can be recovered from the former (which is defined directly in terms of the generic fiber). To answer this question, the difficulty is twofold: the first issue comes from the very definition of the Hyodo--Kato cohomology, which forces us to construct a comparison morphism with the $B$-cohomology locally, and in a functorial way, using log-geometry;\footnote{Likewise, as explained to us by Česnavičius and Le Bras, it is a priori not clear whether the absolute crystalline comparison isomorphism for the $A_{\inf}$-cohomology in the semistable case, constructed in \cite[Theorem 5.4]{CK}, is functorial.} the second issue is of topological nature, since, locally, $R\Gamma_{\HK}(X)$ is in general not a perfect complex over $\breve F$. To avoid the topological issues, one could instead study an overconvergent version of the desired comparison (cf. \cite[Conjecture 6.3]{LeBras2}), however this makes the first mentioned difficulty even more challenging. \medskip
  
  As in our previous work \cite{Bosco}, we overcome the topological issues via the condensed mathematics recently developed by Clausen--Scholze, and we refer the reader to the introduction of \textit{loc. cit.} for a more exhaustive explanation of the relevance of the condensed and solid formalism in the study of the $p$-adic Hodge theory for rigid-analytic varieties. \medskip
  
  Thus, given a condensed ring $A$,\footnote{All condensed rings will be assumed to be commutative and unital. Moreover, we refer the reader to \ref{convent} for the set-theoretic conventions we adopt.} we denote by $\Mod_A^{\cond}$ the category of $A$-modules in condensed abelian groups, and, for $A$ a solid ring, we denote by $\Mod_A^{\ssolid}$ the symmetric monoidal subcategory of $A$-modules in solid abelian groups, endowed with the solid tensor product $\solid_A$. We denote by $D(\Mod_A^{\cond})$ and $D(\Mod_A^{\ssolid})$ the respective derived $\infty$-categories. \medskip

  Our first main result is the following.
 
 \begin{theorem}[cf. Theorem \ref{B=HK}, Theorem \ref{B_dR=dR}, Theorem \ref{secondstep1}, Theorem \ref{compatib2}]\label{mainn1}
   Let $X$ be a connected, paracompact, rigid-analytic variety defined over $C$. 
   \begin{enumerate}[(i)]
    \item We have a natural isomorphism in $D(\Mod^{\ssolid}_{B})$
   \begin{equation}\label{decaBB}
    R\Gamma_{B}(X)\simeq (R\Gamma_{\HK}(X)\dsolid_{\breve F}B_{\log})^{N=0}
   \end{equation}
   compatible with the action of Frobenius $\varphi$, and the action of Galois $\mathscr{G}_K$ in the case when $X$ is the base change to $C$ of a rigid-analytic variety defined over $K$. 
   
   Here, $B_{\log}$ denotes the log-crystalline condensed period ring (see \S\ref{petsh}), and $R\Gamma_{\HK}(X)$ denotes the Hyodo--Kato cohomology of $X$ (Definition \ref{defhk}).\footnote{Forgetting the condensed structure, the Hyodo--Kato cohomology of $X$ agrees with the one defined by Colmez--Nizioł (\cite[\S 4]{CN4}) in the case when $X$ is smooth.}
   \item We have natural isomorphisms in $D(\Mod^{\ssolid}_{B_{\dR}^+})$
    \begin{equation*}
      R\Gamma_{B}(X)\dsolid_{B}B_{\dR}^+\simeq R\Gamma_{B_{\dR}^+}(X)\simeq R\Gamma_{\inf}(X/B_{\dR}^+)
    \end{equation*}
   compatible with the isomorphism (\ref{decaBB}). Here, $R\Gamma_{\inf}(X/B_{\dR}^+)$ denotes the infinitesimal cohomology over $B_{\dR}^+$ (\cite{Guo2}, \S \ref{guoinf}). 
   
   If $X$ is the base change to $C$ of a rigid-analytic variety $X_0$ defined over $K$, then we have a natural isomorphism in $D(\Mod^{\ssolid}_{B_{\dR}^+})$
   \begin{equation*}
     R\Gamma_{B_{\dR}^+}(X)\simeq R\Gamma_{\dR}(X_0)\dsolid_{K} B_{\dR}^+
   \end{equation*}
   compatible with the action of $\mathscr{G}_K$, and with filtrations. Here, $R\Gamma_{\dR}(X_0)$ denotes the de Rham cohomology of $X_0$ (Definition \ref{dRsingg}).
   \end{enumerate}
 \end{theorem}
 
 The proof of Theorem \ref{mainn1} proceeds by constructing functorial local isomorphisms, which are then globalized using some magical properties of the solid tensor product proved by Clausen--Scholze, which rely on the theory of nuclear modules.
 
 \begin{rem}[Overconvergent Fargues--Fontaine cohomology]
  Thanks to the properties of the solid tensor product, one can also easily deduce a version of Theorem \ref{mainn1} for $X$ a dagger variety over $C$. In particular, reinterpreting the latter result in terms of the Fargues--Fontaine curve (see \S \ref{FFperf}), one can deduce a generalization of \cite[Conjecture 6.3]{LeBras2}, as shown in Theorem \ref{lb}: for $i\ge 0$, given $X$ a qcqs dagger variety over $C$, the cohomology group $H^i_B(X)$ is a finite projective $\varphi$-module over $B$; then, denoting by $\cl H^i_{\FF}(X)$ the associated vector bundle on $\FF$, we have a natural isomorphism
  \begin{equation}\label{rhss}
     \cl H_{\FF}^i(X)\cong \cl E(H_{\HK}^i(X))
  \end{equation}
  where $H_{\HK}^i(X)$ is a finite $(\varphi, N)$-module over $\breve F$, and the right-hand side denotes the associated vector bundle on $\FF$; moreover, the completion at $\infty$ of (\ref{rhss}) gives a natural isomorphism
  \begin{equation*}
    \cl H_{\FF}^i(X)^\wedge_{\infty}\cong H_{\inf}^i(X/B_{\dR}^+).
  \end{equation*}
  We note that (\ref{rhss}) implies in particular that the vector bundle $\cl H_{\FF}^i(X)$ determines, up to isomorphisms, the $\varphi$-module structure on $H_{\HK}^i(X)$, and, while the latter is defined via log-geometry, the former is defined directly on the generic fiber. In addition, one can also recover from $\cl H_{\FF}^i(X)$ the ($\varphi$, $N$)-module structure on $H_{\HK}^i(X)$ (see Remark \ref{f-mod}).\footnote{
  We also remark that, recently, Binda--Kato--Vezzani, via a motivic approach, proposed a definition of the overconvergent Hyodo--Kato cohomology theory without using log-geometry, \cite[Appendix A]{BKV}.}
 \end{rem}

 As applications, using Theorem \ref{mainn1}, via the relative fundamental exact sequence of $p$-adic Hodge theory, we show the following result.

 \begin{theorem}[Theorem \ref{fund1}]\label{derr}
  Let $X$ be a qcqs rigid-analytic variety defined over $K$. We have a $\mathscr{G}_K$-equivariant pullback square in $D(\Vect_{\Qq_p}^{\ssolid})$
 \begin{center}
   \begin{tikzcd}
    R\Gamma_{\pet}(X_C, \Qq_p) \arrow[r] \arrow[d] & (R\Gamma_{\HK}(X_C)\dsolid_{\breve F}B_{\log}[1/t])^{N=0, \varphi=1} \arrow[d]  \\
  \Fil^0(R\Gamma_{\dR}(X)\dsolid_K B_{\dR}) \arrow[r] & R\Gamma_{\dR}(X)\dsolid_K B_{\dR}.
   \end{tikzcd}
 \end{center}
 \end{theorem}

 We note that Theorem \ref{derr} can be regarded as a derived generalization of (\ref{dRdata}): it tells us that the rational $p$-adic (pro-)étale cohomology of $X_C$ can be recovered from the Hyodo--Kato cohomology of $X_C$ and the de Rham cohomology of $X$ together with its Hodge filtration.

 \subsection{Syntomic Fargues--Fontaine cohomology}
 
 The search for a theorem comparing the rational $p$-adic pro-étale cohomology of any rigid-analytic variety over $C$ in terms of the $B$-cohomology and its filtration led us to define the following cohomology theory. 
 
 \begin{df}
  Let $X$ be a rigid-analytic variety over $C$.  Let $i\ge 0$ be an integer. We define the \textit{syntomic Fargues--Fontaine cohomology of $X$ with coefficients in $\Qq_p(i)$} as the complex of $D(\Vect_{\Qq_p}^{\cond})$
  $$
   R\Gamma_{\syn, \FF}(X, \Qq_p(i)):=\Fil^iR\Gamma_B(X)^{\varphi=p^i}
  $$
  where $R\Gamma_{B}(X)$ is endowed with the filtration décalée.
 \end{df}

 The first main result on the syntomic Fargues--Fontaine cohomology is the following.
 
 \begin{theorem}[Theorem \ref{BK=pet}]\label{synfan}
  Let $X$ be a rigid-analytic variety over $C$. Let $i\ge 0$. 
  \begin{enumerate}[(i)]
   \item We have a natural isomorphism in $D(\Vect_{\Qq_p}^{\cond})$
   \begin{equation*}\label{glob}
   \tau^{\le i}R\Gamma_{\syn, \FF}(X, \Qq_p(i))\overset{\sim}{\longrightarrow} \tau^{\le i}R\Gamma_{\pet}(X, \Qq_p(i)).
   \end{equation*}
   \item We have a natural isomorphism in $D(\Vect_{\Qq_p}^{\cond})$
   $$ R\Gamma_{\syn, \FF}(X, \Qq_p(i))\simeq \fib(R\Gamma_B(X)^{\varphi=p^i}\to R\Gamma_{B_{\dR}^+}(X)/\Fil^i).$$
  \end{enumerate}
 \end{theorem}
 
 We remark that the construction of the comparison morphisms in Theorem \ref{synfan} is global in nature and it can be extended to coefficients (see \S \ref{coeffi}). \medskip
 
 Combining Theorem \ref{mainn1} with Theorem \ref{synfan}, we obtain the following result. Cf. \cite[\S 6]{LeBras2} and \cite[Theorem 7.13]{AMMN} for some related results in the proper good reduction case, and \cite[Theorem 1.1]{CN4} for smooth rigid-analytic varieties.
 \begin{theorem}[Theorem \ref{tog}]\label{corcor}
  Let $X$ be a connected, paracompact, rigid-analytic variety defined over $K$. For any $i\ge 0$, we have a  $\mathscr{G}_K$-equivariant isomorphism in $D(\Vect_{\Qq_p}^{\ssolid})$
  $$\tau^{\le i}R\Gamma_{\pet}(X_C, \Qq_p(i))\simeq \tau^{\le i}\fib((R\Gamma_{\HK}(X_C)\dsolid_{\breve F}B_{\log})^{N=0, \varphi=p^i}\to (R\Gamma_{\dR}(X)\dsolid_K B_{\dR}^+)/\Fil^i).$$
 \end{theorem}

 Another interesting fact about the syntomic Fargues--Fontaine cohomology concerns its close relationship with the curve $\FF$, as it can be guessed from its very definition. 
 As explained in \S \ref{FFperf}, for any $X$ qcqs rigid-analytic variety over $C$, the $\varphi$-equivariant filtered complex $\Fil^\star R\Gamma_B(X)$ descends to a filtered object $$\Fil^\star\cl H_{\FF}(X)$$ of the $\infty$-category of quasi-coherent complexes $\QCoh(\FF)$, in the sense of Clausen--Scholze (see \S \ref{andr}). Then, relying in particular on results of Andreychev on the analytic descent for nuclear complexes on analytic adic spaces, \cite{Andr}, we show the following theorem.

 \begin{theorem}[Theorem \ref{synlift}]\label{fino}
 Let $X$ be a qcqs rigid-analytic variety over $C$. Let $i\ge 0$.  Consider the quasi-coherent complex on $\FF$ defined by $$\cl H_{\syn}(X)(i):=\Fil^i\cl H_{\FF}(X)\otimes \cl O(i).$$ 
   We have
   \begin{equation*}
    R\Gamma(\FF, \cl H_{\syn}(X)(i))=R\Gamma_{\syn, \FF}(X, \Qq_p(i)).
   \end{equation*}
   If $X$ is proper, the complex $\cl H_{\syn}(X)(i)$ is perfect, in particular the complex $R\Gamma_{\syn, \FF}(X, \Qq_p(i))$ identifies with the $C$-points of a bounded complex of Banach--Colmez spaces.
 \end{theorem}

 \begin{rem}[Comparison with Fontaine--Messing/Colmez--Nizioł's syntomic cohomology]
  Fix $i\ge 0$. In \cite{CN}, via the alterations of Hartl and Temkin, Colmez--Nizioł, starting from the syntomic cohomology of Fontaine--Messing,  defined a syntomic cohomology theory for any smooth rigid-analytic variety $X$ over $C$, that here we will denote by $R\Gamma_{\syn, \CN}(X, \Qq_p(i))$. 
  We observe that $H^j_{\syn, \FF}(X, \Qq_p(i))$ is isomorphic to $H^j_{\syn, \CN}(X, \Qq_p(i))$ for any integer $j\le i$, as in this case the two cohomology groups are both isomorphic to $H^j_{\pet}(X, \Qq_p(i))$; however, in general, for $j>i$, the two cohomology groups are not isomorphic (see Example \ref{FFvsCN}). This difference is reflected in the fact that, for $X$ proper, the complex $R\Gamma_{\syn, \FF}(X, \Qq_p(i))$ canonically lifts to a complex of vector bundles on $\FF$, as shown by Theorem \ref{fino}, while the complex $R\Gamma_{\syn, \CN}(X, \Qq_p(i))$ canonically lifts to a complex of \textit{$\varphi$-modules jaugés over $B^+$}, in the sense of Fargues (\cite[Définition 4.15]{AuDela}), cf. \cite[Theorem 1.1]{NizFF}.\footnote{We recall that the category of $\varphi$-modules jaugés over $B^+$ is equivalent to the category of modifications of vector bundles on $\FF$, \cite[\S 4.2]{AuDela}. However, this is not an equivalence of exact categories, in the sense of Quillen.}
 \end{rem}

 \subsection{Semistable conjectures} 
 
 From the general derived comparison results we have stated above, in particular from Theorem \ref{derr} and Theorem \ref{corcor}, one can deduce in some special cases a refined description of the single rational $p$-adic (pro-)étale cohomology groups in terms of de Rham data. \medskip
 
 For $X$ a proper (possibly singular) rigid-analytic variety over $C$, we prove in Theorem \ref{propsing} a version of the semistable conjecture for $X$, generalizing Theorem \ref{tostart}. In the case when $X$ is the base change to $C$ of a rigid-analytic variety $X_0$ defined over $K$, this result relies on the degeneration of the Hodge-de Rham spectral sequence associated to $X_0$ (\cite[Corollary 1.8]{Scholze}, \cite[Proposition 8.0.8]{Guo1}). In general, we reduce to the previous case via a combination of Conrad--Gabber's spreading out for proper rigid-analytic varieties and a generic smoothness result recently proved by Bhatt--Hansen, \cite{Bhatt-Hansen}. \medskip
 
 Another case in which the Hodge-de Rham spectral sequence simplifies is for smooth Stein spaces, thanks to Kiehl's acyclicity theorem. In this case, we show the following theorem which reproves results of Colmez--Dospinescu--Nizioł \cite{CDN1} (in the semistable reduction case) and Colmez--Nizioł \cite{CN5}.

 \begin{theorem}[cf. Theorem \ref{steindiagram}]\label{mainstein}
 Let $X$ be a smooth Stein space over $C$. For any $i\ge 0$, we have a short exact sequence in $\Vect_{\Qq_p}^{\ssolid}$ 
 \begin{equation*}\label{fin}
  0\to \Omega^{i-1}(X)/\ker d \to H^i_{\pet}(X, \Qq_p(i))\to  (H^i_{\HK}(X)\solid_{\breve F}B_{\log})^{N=0, \varphi=p^i} \to 0.
 \end{equation*}
 \end{theorem}

 \begin{rem}
  A recent conjecture of Hansen, \cite[Conjecture 1.10]{Hans_supercusp}, suggests that any local Shimura variety is a Stein space, therefore Theorem \ref{mainstein} potentially applies to any such variety.
 \end{rem}
 
 \begin{rem}[A $p$-adic Artin vanishing for smooth Stein spaces]\label{stv}
  Let $X$ be a smooth Stein space over $C$. As a corollary of Theorem \ref{mainstein}, we have that $$H^i_{\pet}(X, \Qq_p)=0$$ for all $i>\dim X$ (Corollary \ref{Steinvanish}).
 \end{rem}

 For smooth affinoid rigid spaces, the Hodge-de Rham spectral sequence simplifies similarly to smooth Stein spaces, thanks to Tate's acyclicity theorem (see also Proposition \ref{expic}). Therefore, in view of Theorem \ref{mainstein}, we are led to state the following conjecture.
 
 \begin{conj}[cf. Conjecture \ref{conju}]\label{conjuintro}
  Let $X$ be a smooth affinoid rigid space over $C$. For any $i\ge 0$, we have a short exact sequence in $\Vect_{\Qq_p}^{\ssolid}$ 
  $$0\to \Omega^{i-1}(X)/\ker d \to H^i_{\pet}(X, \Qq_p(i))\to  (H^i_{\HK}(X)\solid_{\breve F}B_{\log})^{N=0, \varphi=p^i} \to 0.$$
 \end{conj}

 We remark that before the advent of condensed mathematics one couldn't even dare to formulate a conjecture in the spirit of the one above; in fact, for a smooth affinoid rigid space $X$ over $C$, the de Rham $H_{\dR}^i(X)$ and Hyodo--Kato $H_{\HK}^i(X)$ cohomology groups are in general pathological if regarded as topological vector spaces (see \cite[Remark 5.14]{Bosco}).\footnote{There were previously some partial and ad hoc solutions to this issue. For example, in \cite{CDNcourbes}, for $X$ a smooth affinoid over $C$ of dimension 1, Colmez--Dospinescu--Nizioł considered the maximal Hausdorff quotient of the topological $\breve F$-vector space $H_{\HK}^1(X)$.} Instead, regarded as condensed vector spaces, these objects are perfectly well-behaved, even though they are non-quasi-separated. Then, exploiting the possibility (provided by the solid formalism) of doing functional analysis with such new objects, we prove the following result.
 
 \begin{theorem}[cf. Theorem \ref{affcurves}]\label{2cases}
  Conjecture \ref{conju} holds true for $X$ a smooth affinoid rigid space over $C$ of dimension 1.
 \end{theorem}

 We discuss in \S \ref{smlabel} the obstruction to proving Conjecture \ref{conju} in dimension higher than 1. We note however that an analogue of Remark \ref{stv} for smooth affinoid spaces over $C$ is known in any dimension, thanks to a result of Bhatt--Mathew (see Lemma \ref{Artinvanish}).

 \subsection{Link to prismatic cohomology: toward an integral theory}\label{linkfor} We conclude this introduction by conjecturing the existence of an integral variant of the $B$-cohomology theory for rigid-analytic varieties, which in particular better explains the relation between the results of this paper and the work of Bhatt--Morrow--Scholze and Česnavičius--Koshikawa. \medskip
 
 In the following, we write $y_C$ for the $(C, \cl O_C)$-point of $Y_{\FF}$ corresponding to Fontaine's map $\theta: A_{\inf}\to \cl O_C$ (and projecting to the point $\infty$ of $\FF$). 
 Let us recall the following result of Fargues. 
 
 \begin{prop}[Fargues, cf. {\cite[Proposition 4.45]{AuDela}},  {\cite[Theorem 14.1.1]{SW}}]\label{farguesequi} The following two categories are equivalent:
   \begin{enumerate}
  \item  Shtukas (of vector bundles) over $\Spa C^{\flat}$ relative to $\Spa \Qq_p$ with one leg at $\varphi^{-1}(y_C)$, i.e. vector bundles $\cl E$ on $\Spa C^{\flat}\overset{.}{\times}\Spa \Qq_p:=Y_{\FF}$\footnote{To motivate the notation and the terminology, we recall that $Y_{\FF}^\diamondsuit\cong(\Spa C)^\diamondsuit\times (\Spa \Qq_p)^\diamondsuit$ as diamonds.} together with an isomorphism
  $$\varphi_{\cl E}: (\varphi^* \cl E)|_{Y_{\FF}\setminus {\varphi^{-1}(y_C)}}\cong \cl E|_{Y_{\FF}\setminus {\varphi^{-1}(y_C)}}$$
  which is meromorphic at $\varphi^{-1}(y_C)$.
  \item  Admissible modifications of vector bundles on $\FF$ at $\infty$, i.e. triples $(\cl F_1, \cl F_2, \alpha)$ where $\cl F_1$ and $\cl F_2$ are vector bundles on $\FF$ with $\cl F_1$ semistable of slope 0, and $\alpha: \cl F_1|_{\FF\setminus\{\infty\}}\cong \cl F_2|_{\FF\setminus\{\infty\}}$ is an isomorphism.
 \end{enumerate}
 \end{prop}

 Now, given $X$ a proper rigid-analytic variety over $C$, for $i\ge 0$ we consider the modification of the vector bundle $\cl H^i_{\FF}(X)$ at $\infty$ given by the $B_{\dR}^+$-lattice $$\Fil^0(H_{B_{\dR}^+}^i(X)\otimes_{B_{\dR}^+}B_{\dR})\subset \cl H^i_{\FF}(X)^{\wedge}_{\infty}\otimes_{B_{\dR}^+}B_{\dR}\cong H^i_{\ett}(X, \Qq_p)\otimes_{\Qq_p}B_{\dR}$$ which gives an admissible modification of vector bundles on $\FF$ at $\infty$
 \begin{equation}\label{modiff}
  (H^i_{\ett}(X, \Qq_p)\otimes_{\Qq_p}\cl O_{\FF},\; \cl H^i_{\FF}(X),\; \alpha)
 \end{equation}
 (see Theorem \ref{propsing}). Then, inspired by Bhatt--Morrow--Scholze's work, it is natural to wonder whether one can give a direct geometric cohomological construction of the shtuka corresponding to (\ref{modiff}) via the above recalled Fargues' equivalence, and how the latter compares to the $A_{\inf}$-cohomology theory in the semistable reduction case. More precisely, we formulate the following conjecture (for simplicity, we restrict ourselves to proper rigid-analytic varieties). \medskip
 
 We will consider the analytic adic space
 $$\cl Y_{\FF}:=\Spa(A_{\inf}, A_{\inf})\setminus V([p^{\flat}])$$
 and denote by $A$ the ring of analytic functions on $\cl Y_{\FF}$. We note that $Y_{\FF}\subset \cl Y_{\FF}$ is the open subset defined by the locus where $p\neq 0$.

 \begin{conj}\label{conjfin} There exists a cohomology theory $$R\Gamma_{\sht}(X/\cl Y_{\FF})$$ for proper rigid-analytic varieties $X$ over $C$, taking values in shtukas of perfect complexes over $\Spa C^{\flat}$ relative to $\Spa \Zz_p$ with one leg at $\varphi^{-1}(y_C)$ (i.e. perfect complexes  $\cl E$ on $\Spa C^{\flat}\overset{.}{\times}\Spa \Zz_p:=\cl Y_{\FF}$ together with an isomorphism $\varphi_{\cl E}: (\varphi^* \cl E)|_{\cl Y_{\FF}\setminus {\varphi^{-1}(y_C)}}\cong \cl E|_{\cl Y_{\FF}\setminus {\varphi^{-1}(y_C)}}$
  which is meromorphic at $\varphi^{-1}(y_C)$), and satisfying the following properties.
 \begin{enumerate}[(i)]
 \item\label{conjfin:2} If $X$ is the generic fiber of a proper $p$-adic formal scheme $\fr X$ over $\cl O_C$ with semistable reduction, there is a natural isomorphism between $R\Gamma_{\sht}(X/\cl Y_{\FF})$ and the shtuka of perfect complexes over $\Spa C^{\flat}$ relative to $\Spa \Zz_p$ with one leg at $\varphi^{-1}(y_C)$ associated to $R\Gamma_{A_{\inf}}(\fr X)\otimes_{A_{\inf}} A$.
  \item\label{conjfin:1} Denoting by $R\Gamma_{\sht}(X/Y_{\FF})$ the restriction of $R\Gamma_{\sht}(X/\cl Y_{\FF})$ to $Y_{\FF}$, the cohomology groups $H_{\sht}^i(X/Y_{\FF})$ are shtukas of vector bundles over $\Spa C^{\flat}$ relative to $\Spa \Qq_p$ with one leg at $\varphi^{-1}(y_C)$, and the admissible modification of vector bundles on $\FF$ at $\infty$ defined in (\ref{modiff}) corresponds to $H^i_{\sht}(X/Y_{\FF})$ via Fargues' equivalence (Proposition \ref{farguesequi}).
  
 \end{enumerate}
 \end{conj}
 
  In the notation of Definition \ref{dfintro}, a natural candidate for the cohomology theory conjectured above is given by a lift of the complex $R\Gamma_{\eh}(X, L\eta_\mu R\alpha_*\Aa)$ to $\cl Y_{\FF}$, where $\Aa$ is the $v$-site sheaf theoretic version of the ring $A$, and $\mu=[\varepsilon]-1\in A_{\inf}$. However, it would be more interesting (especially for questions related to cohomological coefficients) to give a definition of $R\Gamma_{\sht}(X/\cl Y_{\FF})$ in the spirit of prismatic cohomology (\cite{Prisms}, \cite{drinfprism}, \cite{BL2}).\footnote{More precisely, such definition would give a Frobenius descent of $R\Gamma_{\sht}(X/\cl Y_{\FF})$.}  We hope to come back on these questions in a future work.

 \subsection{Leitfaden of the paper}
  We have organized the paper as follows. We begin by defining the $B$-cohomology and the $B_{\dR}^+$-cohomology, together with their filtration décalée, in \S \ref{prelim}. In \S \ref{sHK}, we revisit, in the condensed setting, the Hyodo--Kato cohomology of Colmez--Nizioł, and we extend it to singular rigid-analytic varieties over $C$. In \S \ref{Bleit} and \S \ref{Bleit2}, we prove the first main result of the paper: Theorem \ref{mainn1}. We then proceed in \S \ref{synleit} by introducing the syntomic Fargues--Fontaine cohomology and proving Theorem \ref{synfan} and Theorem \ref{fino}; on the way, we also study nuclear complexes on the Fargues--Fontaine curve. In \S \ref{applleit}, we give applications of the main results proven in the previous sections, showing in particular Theorem \ref{derr} and Theorem \ref{corcor}. We end with Appendix \ref{complem} in which we collect some complements on condensed mathematics used in the main body of the paper.

 \subsection{Notation and conventions}\label{convent}

  \begin{description}
   \item[Ground fields] Fix a prime number $p$. We denote by $K$ a  complete discretely valued non-archimedean extension of $\Qq_p$, with perfect residue field $k$, and ring of integers $\cl O_K$. We choose a uniformizer $\varpi$ of $\cl O_K$.

  We fix an algebraic closure $\overline K$ of $K$. We denote by $C:=\widehat{\overline K}$ the completion of $\overline K$, and by $\cl O_C$ its ring of integers. We denote by $F$ the fraction field of the ring of Witt vectors $W(k)$, we write $\breve F$ for the completion of the maximal unramified extension of $F$ in $\overline K$, and we denote by $\cl O_{\breve F}$ its ring of integers.
  
  Moreover, we let $\mathscr{G}_K:=\Gal(\overline K/K)$ denote the absolute Galois group of $K$. \medskip
  
  \item[$\infty$-categories] We will adopt the term \textit{$\infty$-category} to indicate a $(\infty, 1)$-category, i.e. a higher category in which all $n$-morphisms for $n>1$ are invertible. We will use the language of $\infty$-categories, \cite{HTT}, and higher algebra,  \cite{HA}.  \medskip
  
  We denote by $\Delta$ the simplicial category and, for every integer $m\ge 0$, we write $\Delta_{\le m}$ for the full subcategory of $\Delta$ having as objects $[n]$ for $0\le n\le m$. \medskip
  
  We denote by $\Ani:=\Ani(\Set)$ the $\infty$-category of \textit{anima}, i.e. the $\infty$-category of animated sets, \cite[\S 5.1.4]{CesSch}.
  \medskip
  
  \item[Condensed mathematics]  We fix an uncountable cardinal $\kappa$ as in \cite[Lemma 4.1]{Scholze3}.  
  Unless explicitly stated otherwise, all condensed sets will be $\kappa$-condensed sets (and often the prefix ``$\kappa$'' is tacit). We will denote by $\CondAb$ the category of $\kappa$-condensed abelian groups, and by $\Solid\subset \CondAb$ the full subcategory of $\kappa$-solid abelian groups.  \medskip
  
   All condensed rings will be $\kappa$-condensed commutative unital rings. Given a ($\kappa$-)condensed ring $A$, we denote by $\Mod_A^{\cond}$ the category of $A$-modules in $\CondAb$, and, for $A$ a solid ring, we denote by $\Mod_A^{\ssolid}$ the symmetric monoidal subcategory of $A$-modules in $\Solid$, endowed with the solid tensor product $\solid_A$. We denote by $D(\Mod_A^{\cond})$ and $D(\Mod_A^{\ssolid})$ the respective derived $\infty$-categories; sometimes we abbreviate $D(A)=D(\Mod_A^{\cond})$. Moreover, we write $\underline{\Hom}_A(-, -)$ for the internal Hom in the category $\Mod_A^{\cond}$ (and in the case $A=\Zz$, we often omit the subscript $\Zz$). \medskip
  
  Throughout the paper, we use Clausen--Scholze's non-archimedean condensed function analysis, for which we refer the reader to \cite[Appendix A]{Bosco}. \medskip

  \item[Condensed group cohomology] Given a condensed group $G$, and a $G$-module $M$ in $\CondAb$, the \textit{condensed group cohomology of $G$ with coefficients in $M$} will be denoted by  
   $$R\Gamma_{\cond}(G, M):=R\underline{\Hom}_{\Zz[G]}(\Zz, M)\in D(\CondAb)$$
  where $\Zz$ is endowed with the trivial $G$-action (see e.g. \cite[Appendix B]{Bosco}).
 
\medskip
 
  \item[Adic spaces]  We say that an analytic adic space $X$ is \textit{$\kappa$-small} if the cardinality of the underlying topological space $|X|$ is less than $\kappa$, and for all open affinoid subspaces $\Spa(R, R^+)\subset X$, the ring $R$ has cardinality less than $\kappa$.  In this paper, all the analytic adic spaces will be assumed to be $\kappa$-small.
  \medskip
  
  Throughout the article, all Huber rings will be assumed to be complete, and will be regarded as condensed rings. \medskip

  \item[Pro-étale topology] 
   We recall that there is a natural functor $X\mapsto X^\diamondsuit$ from the category of analytic adic spaces defined over $\Spa(\Zz_p, \Zz_p)$ to the category of locally spatial diamonds, satisfying $|X|=|X^\diamondsuit|$ and $X_{\ett}\cong X_{\ett}^\diamondsuit$,  \cite[Definition 15.5, Lemma 15.6]{Scholze3}. \medskip
   
  For $X$ an analytic adic space defined over $\Spa(\Zz_p, \Zz_p)$, we denote by 
  $$X_{\pet}:=X^{\diamondsuit}_{\qpet}$$ its ($\kappa$-bounded) \textit{pro-\'etale site}, \cite[Definition 14.1]{Scholze3}. \medskip

   Given $f: X\to \Spa(C, \cl O_C)$ an analytic adic space over $C$, and $\cl F$ a sheaf of abelian groups on $X_{\pet}$, we define the complex of $D(\CondAb)$
   $$R\Gamma_{\pet}(X, \cl F):=Rf_{\pet *} \cl F$$
   (see also \cite[Definition 2.7, Remark 2.9]{Bosco}). \medskip

  \item[Fargues--Fontaine curves] For $S=\Spa(R, R^+)$ an affinoid perfectoid space over $\Ff_p$, we let
  $$Y_{\FF, S}:=\Spa(W(R^+), W(R^+))\setminus V(p[p^{\flat}]).$$
  We recall that $Y_{\FF, S}$ defines an analytic adic space over $\Qq_p$, \cite[Proposition II.1.1]{FS}. The $p$-th power Frobenius on $R^+$ induces an automorphism $\varphi$ of $Y_{\FF, S}$ whose action is free and totally discontinuous, \cite[Proposition II.1.16]{FS}. The \textit{Fargues--Fontaine curve relative to $S$} (and $\Qq_p$) will be denoted by
  \begin{equation}\label{FFdef2}
   \FF_S:=Y_{\FF, S}/\varphi^{\Zz}.
  \end{equation}
  For $I=[s, r]\subset (0, \infty)$ an interval with rational endpoints, we define the open subset
 \begin{equation}\label{ratI}
  Y_{\FF, S, I}:=\{|p|^r\le |[p^\flat]|\le |p|^s\}\subset Y_{\FF, S}.
 \end{equation}
 We note that $Y_{\FF, S, I}$ is an affinoid space, as it is a rational open subset of $\Spa(W(R^+), W(R^+))$. 
 
  We denote by $\Bun(\FF_S)$ the category of vector bundles on $\FF_S$, and by $\Isoc_{\overline{\Ff}_p}$ the category of isocrystals over $\overline{\Ff}_p$ (also called finite $\varphi$-modules over $\breve{\Qq}_p$), i.e. the category of pairs $(V, \varphi)$ with $V$ a finite-dimensional $\breve{\Qq}_p$-vector space and $\varphi$ a $\sigma$-semilinear automorphism of $V$, where $\sigma$ is the automorphism of $\breve{\Qq}_p=W(\overline{\Ff}_p)[1/p]$ induced by the $p$-th power Frobenius on $\overline{\Ff}_p$.
  
  Recall that we have a natural exact $\otimes$-functor
  $$\Isoc_{\overline{\Ff}_p}\to \Bun(\FF_S), \;\;\;\; (V, \varphi)\mapsto \cl E(V, \varphi).$$ 
 For $\lambda\in \Qq$, we denote by $(D_{\lambda}, \varphi_{\lambda})$ the simple isocrystal over $\overline{\Ff}_p$ of slope $\lambda$ in the Dieudonné--Manin classification, and we let
 $$\cl O_{\FF_S}(-\lambda):=\cl E(D_{\lambda}, \varphi_{\lambda}).$$
 In particular, for $n\in \Zz$, we have $$\cl O_{\FF_S}(n)=\cl E(\breve{\Qq}_p, p^{-n}\sigma).$$  
  
  In the case $S=\Spa(C^{\flat}, \cl O_{C^{\flat}})$, we omit the subscript $S$ from (\ref{FFdef2}) and (\ref{ratI}). \medskip
  
  We will often use the classification of vector bundles on $\FF$ (see \cite[\S 8]{FF}, \cite[Theorem II.0.3]{FS}): the functor $\Isoc_{\overline{\Ff}_p}\to \Bun(\FF)$ induces a bijection on isomorphism classes; in particular, any vector bundle on $\FF$ is isomorphic to a direct sum of vector bundles of the form $\cl O_{\FF}(\lambda)$ with $\lambda\in \Qq$. \medskip
  
  We will denote by $\infty$ the $(C, \cl O_C)$-point of the curve $\FF$ corresponding to Fontaine's map $\theta:W(\cl O_{C^{\flat}})\to \cl O_C$, and
  $$\iota_{\infty}: \Spa(C, \cl O_C)\to \FF$$
  the inclusion map. \medskip
  
  \item[Rigid-analytic varieties] All rigid-analytic varieties, and all dagger varieties (\cite{GK1}), occurring in this work will be assumed to be quasi-separated, and of finite dimension. \medskip
  
  We say that a rigid-analytic/dagger variety $X$  is \textit{paracompact} if it admits an admissible locally finite affinoid covering, i.e. there exists an admissible covering $\{U_i\}_{i\in I}$ of $X$ by affinoid subspaces such that for each index $i\in I$ the intersection $U_i\cap U_j$ is non-empty for at most finitely many indices $j\in I$.

  We recall that a paracompact rigid-analytic variety is taut, \cite[Definition 5.1.2, Lemma 5.1.3]{Huberbook}, and it is the admissible disjoint union of connected paracompact rigid-analytic varieties of countable type, i.e. having a countable admissible affinoid covering, \cite[Lemma 2.5.7]{DeJong}. We refer the reader to \cite[\S 5.2]{Bosco} for further recollections on paracompact rigid-analytic varieties. \medskip

  \item[Formal schemes] Unless explicitly stated otherwise, all formal schemes will be assumed to be $p$-adic and locally of finite type. \medskip

  \end{description}

  \medskip
  
\begin{Acknowledgments}
    
  I am very grateful to Grigory Andreychev, Ko Aoki, K\c{e}stutis Česnavičius, Dustin Clausen, Pierre Colmez, Gabriel Dospinescu, Haoyang Guo, David Hansen, Hiroki Kato, Teruhisa Koshikawa, Arthur-César Le Bras, Lucas Mann, Matthew Morrow, Wiesława Nizioł, Peter Scholze, and Alberto Vezzani  for helpful conversations, or for their comments on a draft version of this paper. Special thanks go to Teruhisa Koshikawa for a decisive suggestion, to Arthur-César Le Bras for what he thought me about the Fargues--Fontaine curve and for introducing me to \cite[Conjecture 6.3]{LeBras2}, to Wiesława Nizioł for reading several drafts of this manuscript, and to Peter Scholze for his crucial remarks and corrections as well as for discussions related to Conjecture \ref{conjfin}.
  
  I would like to thank the organizers of the RAMpAGE seminar for their invitation in June 2021, on which occasion the main results of this paper were first announced.
  
  This project was carried out while the author was a  Ph.D. student at Sorbonne Université, within the Institut de Mathématiques de Jussieu--Paris Rive Gauche. Moreover, parts of this manuscript were written while visiting the Banach Center, at the invitation of Piotr Achinger, and the Max-Planck-Institut f\"ur Mathematik. I thank all these institutions for their hospitality and support.
   
  \end{Acknowledgments}

 \clearpage

 \section{\textbf{Preliminaries}}\label{prelim}
  \sectionmark{}

 In this first section, our goal is to define the \textit{$B$-cohomology} together with its \textit{filtration décalée}. Along the way, we will establish several preliminary technical results, which will be used in the rest of the paper.
  
  \subsection{D\'ecalage functors and Beilinson $t$-structure}
  
   In this subsection,  we will recall an interpretation of the d\'ecalage functor in terms of the connective cover functor for the Beilinson $t$-structure.

  \subsubsection{D\'ecalage functors}
   
   We shall use the following notation.
   
  \begin{notation}
  Let $(T, \cl O_{T})$ be a ringed topos and let $(f)\subset \cl O_T$ be an invertible ideal sheaf.  We will write $D(\cl O_T)$ for the derived category of $\cl O_T$-modules.
  \end{notation}

  The following slight generalization of \cite[Definition 6.2]{BMS1}, which goes back to Berthelot--Ogus, will be used in particular in \S \ref{nygaa}.

 \begin{df}[cf. {\cite[Definition 8.6]{BO}}]\label{twistdec}
  Let $\delta: \Zz\to \Zz$ be a function. Let $M^\bullet$ be an $f$-torsion-free complex of $\cl O_T$-modules. We denote by $\eta_{\delta, f}(M^\bullet)$ the subcomplex of $M^\bullet[1/f]$ defined by
  $$\eta_{\delta, f}(M^\bullet)^i:=\{x\in f^{\delta(i)}M^i: dx \in f^{\delta(i+1)}M^{i+1}\}.$$
  In the case $\delta=\id$, we put $\eta_{f}(-):=\eta_{\id, f}(-)$.
 \end{df}

 \begin{rem}
  We note that the definition of $\eta_{\delta, f}(-)$ depends on the ideal sheaf $(f)\subset \cl O_T$ and it is independent on the chosen generator of the latter.
 \end{rem}

 \begin{prop}[cf. {\cite[Proposition 8.19]{BO}}]\label{nondec}
  Let $\delta: \Zz\to \Zz$ be a non-decreasing function. The functor $\eta_{\delta,f}$ from $f$-torsion-free complexes of $\cl O_T$-modules to $D(\cl O_T)$ factors canonically over the \textit{d\'ecalage functor} (relative to $(f)$ and $\delta$)
 $$L\eta_{\delta, f}:D(\cl O_T)\to D(\cl O_T).$$ 
 \end{prop}
 \begin{proof}
  First, recall that every complex of $\cl O_T$-modules is quasi-isomorphic to an $f$-torsion-free complex of $\cl O_T$-modules, \cite[Lemma 6.1]{BMS1}. We want to show that the endo-functor $\eta_{\delta,f}$ on the category of $f$-torsion-free complexes of $\cl O_T$-modules preserves quasi-isomorphisms. The latter assertion is implied by the following claim (cf. \cite[Lemma 6.4]{BMS1}): given $M^\bullet$ an $f$-torsion-free complex of $\cl O_T$-modules, for all $i\in \Zz$, the multiplication by $f^{\delta(i)}$ map, i.e. tensoring by $-\otimes_{\cl O_T}(f^{\delta(i)})$, induces an isomorphism
  $$H^i(M^\bullet)/H^i(M^\bullet)[f^{\delta(i)-\delta(i-1)}]\overset{\sim}{\to} H^i(\eta_{\delta, f}M^\bullet)$$
  (note that $\delta(i)-\delta(i-1)\ge 0$ by assumption on the function $\delta: \Zz\to \Zz$). For this, let $Z^i(M^\bullet)\subset M^i$ and $Z^i(\eta_{\delta, f}M^\bullet)\subset (\eta_{\delta, f}M^\bullet)^i$ denote the cocycles. By $f$-torsion-freeness of the terms of the complex $M^\bullet$, the multiplication by $f^{\delta(i)}$ map induces an isomorphism
  $$Z^i(M^\bullet)\cong Z^i(\eta_{\delta, f}M^\bullet)$$
  which in turn induces a surjection
  $$H^i(M^\bullet)\twoheadrightarrow H^i(\eta_{\delta, f}M^\bullet).$$
  Moreover, given a cocycle $z\in Z^i(M^\bullet)$ mapping to the zero class of $H^i(\eta_{\delta, f}M^\bullet)$ via multiplication by $f^{\delta(i)}$, we have that $f^{\delta(i)}z=d(f^{\delta(i-1)}y)$ for some $y\in M^{i-1}$, i.e. $f^{\delta(i)-\delta(i-1)}z=dy$, which means that the image of $z$ in $H^i(M^\bullet)$ is $f^{\delta(i)-\delta(i-1)}$-torsion, as claimed.
 \end{proof}

 \subsubsection{Beilinson $t$-structure}

 Next, as promised, we want to recall the $\infty$-categorical interpretation of the d\'ecalage functor $L\eta_f(-)$ as the connective cover functor for the Beilinson $t$-structure.

 \begin{notation}
  Let $A$ be a commutative unital ring. Let $D(A)$ denote the derived $\infty$-category of $A$-modules. We write
 $$DF(A):=\Fun(\Zz^{\op}, D(A))$$
 for the \textit{filtered derived $\infty$-category of $A$-modules}. Given $F\in DF(A)$, for $i\in \Zz$, we define the $i$-th \textit{graded piece of $F$} as the cofiber 
 $$\gr^i(F):=F(i)/F(i+1).$$
 \end{notation}
 
 We refer the reader to \cite[\S 5.1]{BMS2} for recollections on filtered derived $\infty$-categories.
  
 \begin{df}[{\cite[Definition 5.3]{BMS2}}]\label{notadef}
  Let $DF^{\le 0}(A)\subset DF(A)$ be the full $\infty$-subcategory spanned by those $F$ such that $\gr^i(F)\in D^{\le i}(A)$ for all $i\in \Zz$,  and let $DF^{\ge 0}(A)\subset DF(A)$ be the full $\infty$-subcategory spanned by those $F$ such that $F(i)\in D^{\ge i}(A)$ for all $i\in \Zz$. The pair
 $$(DF^{\le 0}(A), DF^{\ge 0}(A))$$
 is called the \textit{Beilinson $t$-structure} on $DF(A)$.
 \end{df}

 Note that the $t$-structure depends only on the triangulated category underlying the derived $\infty$-category $D(A)$. The definition above is justified by the following result.
 
 \begin{prop}[{\cite[Theorem 5.4]{BMS2}}]\label{bbe} Fix notation as in Definition \ref{notadef}.
 \begin{enumerate}[(i)]
  \item\label{bbe:1} The Beilinson $t$-structure $(DF^{\le 0}(A), DF^{\ge 0}(A))$ is a $t$-structure on  $DF(A)$.
  \item\label{bbe:2} Denoting by 
  $$\tau_{\Beil}^{\le 0}: DF(A)\to DF^{\le 0}(A)$$
  the connective cover functor for the Beilinson $t$-structure on $DF(A)$, there is a natural isomorphism
  $$\gr^i\circ \tau_{\Beil}^{\le 0}(-)\simeq \tau^{\le i}\circ \gr^i(-).$$
  \item\label{bbe:3} Denote by 
  $$H^0_{\Beil}:DF(A)\to DF(A)^{\heartsuit}:=DF^{\le 0}(A)\cap DF^{\ge 0}(A)$$
  the $0$-th cohomology functor for the Beilinson $t$-structure. The heart $DF(A)^{\heartsuit}$ is equivalent to the abelian category $\Ch(A)$ of chain complexes of $A$-modules in abelian groups, via sending, for varying $F\in DF(A)$, the $0$-th cohomology $H^0_{\Beil}(F)\in DF(A)^{\heartsuit}$ to the chain complex $(H^\bullet(\gr^\bullet(F)), d)$ with differential $d$ induced by the boundary map for the exact triangle
  $$\gr^{i+1}(F)\to F(i)/F(i+2)\to \gr^i(F).$$
 \end{enumerate}
 \end{prop}

 \medskip
 
 In the following, let $f$ be a non-zero-divisor in $A$. 
 
 \begin{prop}\label{beilifil}
  Let $M\in D(A)$. Define $\Fil^\star L\eta_f M\in  DF(A)$ the filtration on $L\eta_f M$ whose $i$-th level is given by $L\eta_{\varepsilon_i, f}M$, where $\varepsilon_i:\Zz\to \Zz,j\mapsto \max(i, j)$. Denote by $f^{\star}\otimes M\in DF(A)$ the filtration on $M$ whose $i$-th level is given by $f^{i}\otimes_A M$. Then, $\Fil^\star L\eta_f M$ identifies with $\tau_{\Beil}^{\le 0}(f^\star\otimes M)$ in $DF(A)$.
 \end{prop}
 \begin{proof}
  First, note that the function $\varepsilon_i$ is non-decreasing, hence it satisfies the assumptions of Proposition \ref{nondec}. Then, the statement is contained in the proof of \cite[Proposition 5.8]{BMS2}.
 \end{proof}

 \begin{df}\label{beilifildef}
  Given $M\in D(A)$, we call the filtration $\Fil^\star L\eta_f M$ defined in Proposition \ref{beilifil} the \textit{filtration décalée} on $L\eta_f M$.
 \end{df}

 The definitions and the results above extend to any ringed topos (or site). In particular, they extend to the case of the ringed site $(*_{\kappa{\text -}\pet}, A)$, for $\kappa$ a cut-off cardinal as in \S \ref{convent}, where $*_{\kappa{\text -}\pet}$ is the site of $\kappa$-small profinite sets, with coverings given by finite families of jointly surjective maps, and $A$ is a $\kappa$-condensed ring.
 
 

 \subsection{The \'{e}h-topology}\label{recalleh}
  
  In this subsection, we recall the definition of the $\eh$-site for rigid-analytic varieties, introduced by Guo, and we consider its variant for dagger varieties. This site will be used crucially in the definition of the $B$-cohomology theory for arbitrary (possibly singular) rigid-analytic/dagger varieties.
    
  \subsubsection{Definition of the \'{e}h-site}

  We will use the following notation and conventions.
  
  \begin{convnot}\label{othernot}
   We denote by $L$ a characteristic $0$ complete valued field with a non-archimedean valuation of rank $1$ and residue characteristic $p$. We write $\Rig_{L}$ (resp. $\Rig_{L}^\dagger$) for the category of rigid-analytic (resp. dagger) varieties over $L$, and we denote by $\RigSm_{L}$ (resp. $\RigSm_{L}^\dagger$) the category of smooth rigid-analytic (resp. dagger) varieties over $L$. \medskip
   
   We refer the reader to \cite{GK1} for the foundations of \textit{dagger varieties} (also called \textit{overconvergent rigid varieties}), and to \cite[\S 2]{Vezzani1} for a quick recollection of the definitions and the main results on the subject, which we will freely use in the following. \medskip
 
  Given a dagger variety $X=(\widehat X, \cl O^{\dagger})$ over $L$ with underlying rigid-analytic variety $\widehat X$ and \textit{overconvergent structure sheaf} $\cl O^\dagger$, we say that $\widehat X$ is the \textit{limit} of $X$, and vice versa that $X$ is a \textit{dagger structure} on $\widehat X$, \cite[Definition 2.22]{Vezzani1}; moreover, we regard $\cl O^\dagger$ as a sheaf with values in $\Vect_{L}^{\cond}$. 
 
  \end{convnot}

 Before defining the $\eh$-site for rigid-analytic and dagger varieties, we need to introduce the notion of \textit{blowing-up}. The construction of the blow-up of a rigid-analytic variety along a closed analytic subset, as well as the verification of its universal property, is due to Conrad, \cite[\S 4.1]{Conrad-ampl}. In turn, such construction relies on the definition of the \textit{relative analytified $\Proj$}, \cite[\S 2.3]{Conrad-ampl}, denoted $\Proj^{\an}$. We note that the latter definition translates \textit{verbatim} to dagger varieties (replacing the structure sheaf with the overconvergent structure sheaf). We can then give the following definition (see \cite[Definition 4.1.1]{Conrad-ampl}).
 
 \begin{df}\label{blowups}
  Let $X$ be a rigid-analytic (resp. dagger) variety over $L$, and let $Z=V(\cl I)$ be the Zariski closed subset defined by a coherent ideal sheaf $\cl I$ over $X$.  The \textit{blow-up of $X$ along $Z$} is the rigid-analytic (resp. dagger) variety over $X$ defined by
  $$\Bl_Z(X):=\Proj^{\an}(\oplus_{n\ge 0}\,\cl I^n).$$
 \end{df}

 \begin{rem}\label{univers}
  Keeping the notation above, the blow-up of $X$ along $Z$ has the following universal property (see the discussion after \cite[Definition 4.1.1]{Conrad-ampl}): $\Bl_Z(X)\to X$ is the final object in the category of morphisms $f:Y\to X$ in $\Rig_L$ (resp. $\Rig_L^\dagger$) such that the coherent pullback $f^* \cl I$ is invertible.
 \end{rem}

 \begin{df}[{\cite[Definition 2.4.1]{Guo1}}]
   The \textit{big $\eh$-site} $\Rig_{L, \eh}$ (resp. $\Rig_{L, \eh}^\dagger$) is the Grothendieck topology on the category $\Rig_{L}$ (resp. $\Rig_{L}^\dagger$), such that the covering families are generated by étale coverings, universal homeomorphisms, and morphisms 
  $$\Bl_Z(Y)\sqcup Z\to Y$$
  with $Z$ a closed analytic subset of $Y$.
  
  Given $X$ a rigid-analytic (resp. dagger) variety over $L$, we define the \textit{small $\eh$-site} $X_{\eh}$ as the localization of the site $\Rig_{L, \eh}$ (resp. $\Rig_{L, \eh}^\dagger$) at the object $X$.
   \end{df}

 The definition above is designed to make the following result hold true.
 
 \begin{prop}[cf. {\cite[Corollary 2.4.8]{Guo1}}]\label{baseh}
  Let $X$ be a  quasi-compact, reduced, rigid-analytic (resp. dagger) variety over $L$. Then, there exists a proper $\eh$-covering $f:Y\to X$ with $Y$ a smooth rigid-analytic (resp. dagger) variety over $L$.
 \end{prop}
 \begin{proof}
   We will check that the proof of \cite[Corollary 2.4.8]{Guo1} also holds for dagger varieties. 
   
   Since $X$ is quasi-compact and reduced, by Temkin's non-embedded desingularization theorem, \cite[Theorem 1.2.1, Theorem 5.2.2]{Temk1}, there exists a finite sequence of blowups $$X_n\to X_{n-1}\to\cdots \to X_0=X$$
  such that $X_n$ smooth, with $X_j=\Bl_{Z_{j-1}}(X_{j-1})$ the blowup of $X_{j-1}$ along a smooth Zariski closed subset $Z_{j-1}$ of $X_{j-1}$.\footnote{We observe that in \textit{loc. cit.} the blow-ups considered are analytification of scheme-theoretic blow-ups. However, by the universal property in Remark \ref{univers}, we have natural comparison morphisms between the blow-up in the sense of Definition \ref{blowups} and the analytification of the scheme-theoretic blow-up, which are isomorphisms.} In fact, we note that \textit{loc. cit.} also applies in the case when $X$ is a dagger variety, as any dagger $L$-algebra is an excellent ring: this follows from a criterion of Matsumura \cite[Theorem 102]{Matsumura}, using that Washnitzer algebras are regular, \cite[Proposition 1.5]{GK1}, and $L$ has characteristic 0.
  
  In conclusion, the morphism
  $$Y:=X_n\bigsqcup\left(\sqcup_{i=0}^{n-1}Z_i\right)\to X$$
  is a proper $\eh$-covering with $Y$ smooth.
 \end{proof}
 
 \begin{rem}\label{oyt} 
 Let $X$ be a rigid-analytic (resp. dagger) variety over $L$. By Proposition \ref{baseh}, the $Y\in X_{\eh}$, with $Y$ a smooth rigid-analytic (resp. dagger) variety over $L$, form a basis of $X_{\eh}$. In fact, for any rigid-analytic (resp. dagger) variety $Z$ over $L$, denoting by $Z_{\redd}$ the reduced subspace of $Z$, the natural map $Z_{\redd}\to Z$ is a universal homeomorphism, hence it is an $\eh$-covering.
 \end{rem}

 \subsubsection{\normalfont{\textbf{Differential forms and de Rham cohomology of singular varieties}}}
 
 Next, we want to state a condensed version of Guo's descent result for the $\eh$-differentials (Proposition \ref{ehdescent}), which will be useful in the following sections. For this, we refer the reader to \cite[\S 5.1]{Bosco} for a discussion on how to translate classical results on coherent cohomology of rigid-analytic varieties into the condensed setting. \medskip
 
 The following definition is based on Proposition \ref{baseh} (and Remark \ref{oyt}).

 \begin{df}\label{dRsingg}
  Let $X$ be a rigid-analytic variety over $L$. Denote by $\cl B^{\sm}_{\eh}$ the basis of the site $X_{\eh}$ consisting of all smooth $Y\in X_{\eh}$. For $i\ge 0$, we define $\Omega^i_{X_{\eh}}$ as the sheaf on $X_{\eh}$, with values in $\Vect_{L}^{\cond}$, associated to the presheaf
  $$(\cl B^{\sm}_{\eh})^{\op}\to \Vect_L^{\cond}: Y\mapsto \Omega^i_Y(Y).$$
  Denote by $\Omega_{X_{\eh}}^\bullet$ the de Rham complex of $X$, given by
  $$\Omega_{X_{\eh}}^\bullet:=[\cl O_{X_{\eh}}\overset{d}{\to}\Omega^1_{X_{\eh}}\overset{d}{\to}\Omega^2_{X_{\eh}} \overset{d}{\to}\cdots].$$
  We define the \textit{de Rham cohomology} of $X$ (over $L$) as
  $$R\Gamma_{\dR}(X):=R\Gamma(X, \Omega_{X_{\eh}}^{\bullet})\in D(\Vect_L^{\cond})$$
  and endow it with the $\Nn^{\op}$-indexed filtration $\Fil^{\star}R\Gamma_{\dR}(X):=R\Gamma(X, \Omega_{X_{\eh}}^{\ge \star})$, called \textit{Hodge filtration}.
 \end{df}

 The next result shows in particular that in the smooth case the de Rham cohomology defined above agrees with the usual de Rham cohomology.
 
 \begin{prop}[{\cite[Theorem 4.0.2]{Guo1}}]\label{ehdescent}
  Let $X$ be a smooth rigid-analytic variety over $L$. Let $\pi:X_{\eh}\to X_{\ett}$ be the natural morphism of sites. Then, for each $i\ge 0$, we have
  $$ R\pi_*\Omega_{X_{\eh}}^i=\Omega_{X_{\ett}}^i[0]$$
  as complexes of sheaves with values in $\Vect_L^{\cond}$.
 \end{prop}

 The following boundedness result will also be useful in the sequel.
 
 \begin{prop}[{\cite[Theorem 6.0.2]{Guo1}}]\label{bound0}
  Let $X$ be a qcqs rigid-analytic variety over $L$ of dimension $d$. Then, $H^i(X, \Omega_{X_{\eh}}^j)$ vanishes if $i>d$ or $j>d$.
 \end{prop}

 From the proposition above we deduce the following corollary.

 \begin{cor}\label{dRbound}
  Let $X$ be a qcqs rigid-analytic variety over $L$ of dimension $d$. Then, the de Rham cohomology complex $R\Gamma_{\dR}(X)$ lies in $D^{\le 2d}(\Vect_{L}^{\ssolid})$.
 \end{cor}

  \subsection{Period sheaves}\label{logpro}\label{petsh}

 In this subsection, we first recall the definitions of the pro-étale sheaf-theoretic version of the classical period rings of Fontaine, and we introduce a log-variant of the pro-\'etale sheaf-theoretic version $\Bb$ of the ring $B$ of analytic functions on $Y_{\FF}$, i.e. the \textit{log-crystalline pro-étale period sheaf} $\Bb_{\log}$. Then, after some preliminary complementary results on the pro-étale period sheaves (and condensed period rings), we recall that, thanks to results of Scholze \cite{Scholze3}, the pro-étale period sheaves satisfy $v$-descent. 

 \subsubsection{Pro-étale period sheaves}\label{petsheaves}
 
 \begin{df}\label{gyu}
    Let $X$ be an analytic adic space over $\Spa(\Zz_p, \Zz_p)$. We define the \textit{integral $\pet$-structure sheaf} $\widehat{\cl O}_X^+$ and the \textit{$\pet$-structure sheaf} $\widehat{\cl O}_X$ as the sheaves on $X_{\pet}$ satisfying respectively $$\widehat{\cl O}_X^+(Y):=\cl O^+_{Y^{\sharp}}(Y^\sharp),\;\;\;\;\;\;\;\; \widehat{\cl O}_X(Y):=\cl O_{Y^{\sharp}}(Y^\sharp)$$
   for all perfectoid spaces $Y\in X_{\pet}$.
  \end{df}

  We recall that, thanks to \cite[Theorem 8.7]{Scholze3}, $\widehat{\cl O}_X^+$ and $\widehat{\cl O}_X$ are indeed sheaves.

 \begin{df}\label{periodsheaves} Let $X$ be an analytic adic space over $\Spa(\Qq_p, \Zz_p)$.  The following are defined to be sheaves on $X_{\pet}$.
 \begin{enumerate}[(i)] 
   \item The \textit{tilted integral $\pet$-structure sheaf} $\widehat{\cl O}_X^{\flat +}=\varprojlim_{\varphi} \widehat{\cl O}_{X}^+/p$, where the inverse limit is taken along the Frobenius map $\varphi$.
  \item The sheaves $\Aa_{\inf}=W(\widehat{\cl{O}}_X^{\flat +})$ and $\Bb_{\inf}=\Aa_{\inf}[1/p]$. We have a morphism of pro-\'etale sheaves $\theta: \Aa_{\inf}\to\widehat{\cl{O}}^+_{X}$ that extends to $\theta: \Bb_{\inf}\to \widehat{\cl{O}}_{X}$.
  \item We define the \textit{positive de Rham sheaf} $\Bb_{\dR}^+=\varprojlim_{n\in \Nn} \Bb_{\inf}/(\ker\theta)^n$, with filtration given by $\Fil^r \Bb_{\dR}^+=(\ker \theta)^r \Bb_{\dR}^+$.
  \item\label{periodsheaves:last} Let $t$ be a generator of $\Fil^1 \Bb_{\dR}^+$.\footnote{Such a generator exists locally on $X_{\pet}$, it is a non-zero-divisor and unique up to unit, by \cite[Lemma 6.3]{Scholze}.} We define the \textit{de Rham sheaf} $\Bb_{\dR}=\Bb_{\dR}^+[1/t]$, with filtration $\Fil^r \Bb_{\dR}=\sum_{j\in \Zz} t^{-j} \Fil^{r+j}\Bb_{\dR}^+$.
 \end{enumerate}
\end{df}

\begin{notation}\label{hui}
 In the following, we denote by $v(-)$ the valuation on $\cl O_{C^\flat}$ defined as follows: for $x\in \cl O_{C^\flat}$, we define $v(x)$ as the $p$-adic valuation of $x^{\sharp}\in \cl O_{C}$.
\end{notation}

 \begin{df}\label{I}
   Let $X$ an analytic adic space over $\Spa(C, \cl O_C)$.  Let $I=[s, r]$ be an interval of $(0, \infty)$ with rational endpoints, and let $\alpha, \beta\in \cl O_{C^{\flat}}$ with valuation $v(\alpha)=1/r$ and $v(\beta)=1/s$. We define the following sheaves on $X_{\pet}$
  $$\Aa_{\inf, I}=\Aa_{\inf}\left[\frac{p}{[\alpha]}, \frac{[\beta]}{p}\right],\;\;\;\; \Aa_I=\varprojlim_n\Aa_{\inf, I}/p^n,\;\;\;\;\Bb_I=\Aa_I[1/p].$$
  Moreover, we define the sheaf on $X_{\pet}$
  $$\Bb=\varprojlim _{I\subset (0, \infty)} \Bb_I$$
  where $I$ runs over all the compact intervals of $(0, \infty)$ with rational endpoints.
 \end{df}
 
 We recall the following interpretation of the latter period sheaves defined above in terms of the curves $Y_{\FF, S}$ (see \S \ref{convent} for the notation).
 
 \begin{lemma}[{\cite[Lemma 4.14]{Bosco}}]
  Let $S^{\sharp}$ be an affinoid perfectoid space over $\Spa(C, \cl O_C)$, and let $S=(S^{\sharp})^\diamondsuit$.
  Let $I=[s, r]\subset (0, \infty)$ be an interval with rational endpoints. Then,  we have
  $$\Bb_I(S^{\sharp})=\cl O(Y_{\FF, S, I}),\;\;\;\;\;\;\;\;\Bb(S^{\sharp})=\cl O(Y_{\FF, S}).$$
 \end{lemma}

 The following fundamental exact sequences of $p$-adic Hodge theory summarize the relevant relations between the various rational period sheaves.

 \begin{prop}\label{suites} Let $X$ an analytic adic space over $\Spa(C, \cl O_C)$. Let $i\ge 0$ be an integer.
 We have the following exact sequences of sheaves on $X_{\pet}$
 \begin{equation}\label{suite2}
  0\to \Bb^{\varphi=p^i}\to \Bb \xrightarrow{\varphi p^{-i}-1} \Bb\to 0 
 \end{equation}
 
 \begin{equation}\label{suite3}
 0\to \Qq_p(i)\to \Bb^{\varphi=p^i}\to \Bb_{\dR}^+/\Fil^i\Bb_{\dR}^+\to 0.
 \end{equation}
 \end{prop}
 \begin{proof}
  See e.g. \cite[Proposition 4.16]{Bosco}.
 \end{proof}

  \begin{cor}\label{profundexact}
  Let $X$ be an analytic adic space over $\Spa(C, \cl O_C)$. We have the following exact sequences of sheaves on $X_{\pet}$
  \begin{equation}\label{fundexacttard}
     0\to \Bb_e\to \Bb[1/t]\overset{\varphi-1}{\to} \Bb[1/t]\to 0
  \end{equation}
  
  \begin{equation}\label{fundexact}
   0\to \Qq_p\to \Bb_e\to \Bb_{\dR}/\Bb_{\dR}^+\to 0
  \end{equation}
  where $\Bb_e:=\Bb[1/t]^{\varphi=1}$.
 \end{cor}
 \begin{proof}
  See \cite[Proposition 8.5]{LeBras1} and \cite[Corollary 4.18]{Bosco}.
 \end{proof}

 \subsubsection{Log-crystalline period sheaves}
   
   We recall that Fargues--Fontaine defined in \cite[\S 10.3.1]{FF} the \textit{log-crystalline period ring}
   $$B_{\log}:=B\otimes_{\Sym_{\Zz} \mathcal{O}_{C^\flat}^{\times}}\Sym_{\Zz}(C^\flat)^\times$$
   where $\Sym_{\Zz}(-)$ denotes the symmetric algebra over $\Zz$. The ring $B_{\log}$ is endowed with an action of the Galois group $\mathscr{G}_K$, a Frobenius $\varphi$, and a monodromy operator $N$ for which $B_{\log}^{N=0}=B$. Moreover, there is a (non-canonical) isomorphism of rings
   $$B[U]\overset{\sim}{\to} B_{\log},\;\;\;\; U\mapsto \log[p^\flat]$$
   where $B[U]$ denotes the ring of polynomials over $B$ in the variable $U$.
   \medskip
   
   Now, keeping the notation of Definition \ref{I}, we introduce a pro-\'etale sheaf-theoretic version of the ring $B_{\log}$. 
 
  \begin{df}\label{deflogcrys}
  Let $X$ be an analytic adic space over $\Spa(C, \cl O_C)$. Let $I=[s, r]$ be an interval of $(0, \infty)$ with rational endpoints, and let $\alpha, \beta\in \cl O_{C^{\flat}}$ with valuation $v(\alpha)=1/r$ and $v(\beta)=1/s$.  We define the following sheaves on $X_{\pet}$
  $$\Bb_{\log}:=\Bb[U],\;\;\;\;\;\;\;\; \Bb_{\log, I}:=\Bb_I[U].$$
   
  We endow $\Bb_{\log}$ (resp. $\Bb_{\log, I}$) with a Frobenius $\varphi$ and a Galois action extending the ones on $\Bb$ (resp. $\Bb_I$) by setting $\varphi(U):=pU$, and, for $g \in \mathscr{G}_K$, 
  \begin{equation}\label{exgal}
   g(U):= U+ \log[g(p^\flat)/p^\flat].
  \end{equation}
  Moreover, we equip $\Bb_{\log}$ and $\Bb_{\log, I}$ with a monodromy operator $N:=-\frac{d}{dU}$.
 \end{df}
 
 Let us list some useful basic properties of $\Bb_{\log}$.
 
 \begin{rem}
  The action of $\mathscr{G}_K$ on $\Bb_{\log}$ defined above commutes with $\varphi$ and $N$, and we have $N\varphi =p\varphi N$. 
 \end{rem}

 \begin{rem}
   We have the following exact sequence of sheaves on $X_{\pet}$
   $$
  0\to \Bb\to \Bb_{\log}\xrightarrow{N} \Bb_{\log}\to 0.
  $$
 \end{rem}

 \begin{rem}\label{twistt}
  For $I=[s, r]$ an interval of $(0, \infty)$ with rational endpoints such that $s\le 1\le r$, we have a natural inclusion
  $$\Bb_{I}\hookrightarrow \Bb_{\dR}^+.$$
  The induced inclusion $\Bb\hookrightarrow \Bb_{\dR}^+$ extends to a $\mathscr{G}_K$-equivariant injection $\Bb_{\log}\hookrightarrow \Bb_{\dR}^+$, via sending $U$ to $\log([p^\flat]/p)$ (see the proof of \cite[Proposition 10.3.15]{FF}).
 \end{rem}

 In this article, we will adopt the following notation and conventions.
  
  \begin{convnot}\label{condp}\
  \begin{description}
   \item[Condensed period rings] We denote by $A_{\inf}$, $B_{\dR}^+$, $B_{\dR}$, $B$, $B_{\log}$, the condensed rings given respectively by the sheaves $\Aa_{\inf}$, $\Bb_{\dR}^+$, $\Bb_{\dR}$, $\Bb$, $\Bb_{\log}$ on the site $\Spa(C, \cl O_C)_{\pet}$ and, for any compact interval $I\subset(0, \infty)$ with rational endpoints, we similarly define $A_{\inf, I}$, $A_I$, $B_I$, $B_{\log, I}$.\footnote{See \cite[Corollary 4.9, Example 4.10]{Bosco} for the relation to the classical topological period rings.}
   In addition, we denote by $A_{\cris}$, $B_{\cris}^+$, $B_{\st}^+$ the condensed version\footnote{We take Fontaine's definitions in condensed sets.} of the crystalline and semistable period rings of Fontaine (relative to $\cl O_C$), \cite{Font-periodes}. \medskip
   \item[Orientation] We fix a compatible system $(1, \varepsilon_p, \varepsilon_{p^2}, \ldots)$ of $p$-th power roots of unity in $\cl O_C$, which defines an element $\varepsilon\in \cl O_C^{\flat}$. We denote by $[\varepsilon]\in A_{\inf}$ its Teichm\"uller lift and $\mu=[\varepsilon]-1\in A_{\inf}$. Furthermore, we let $\xi=\mu/\varphi^{-1}(\mu)\in A_{\inf}$ and $t=\log[\varepsilon]\in B$.
  \end{description}     
   \end{convnot}

 Let us collect some useful facts on the above-defined (condensed) period rings, that we will repeatedly use in the following.
 
 \begin{rem}\label{st-log}
  Let us recall that, for a compact interval $I\subset[1/(p-1), \infty)$ with rational endpoints, we have that $A_{\cris}\subset A_I$ (see e.g. \cite[\S 2.4.2]{CN1}). In particular, for any such interval $I$, we also have $B_{\st}^+ \subset B_{\log, I}$, via the (non-canonical) identification
  \begin{equation}\label{noncann}
   B_{\cris}^+[U]\overset{\sim}{\to} B_{\st}^+,\;\;\;\; U\mapsto \log[p^\flat]
  \end{equation}
  where we endow $B_{\cris}^+[U]$ with a Frobenius $\varphi$ extending the one on $B_{\cris}^+$ by setting $\varphi(U):=pU$, a Galois action extending the one on $B_{\cris}^+$ as in (\ref{exgal}), and we equip it with a monodromy operator $N:=-\frac{d}{dU}$.
 \end{rem}
 
 We will also need the following result.
 
  \begin{lemma}\label{intert}
     Let $I\subset (0, \infty)$ be a compact interval with rational endpoints. The system of ideals of the ring $A_I$ defined by $(p^n A_I)_{n\ge 1}$ and $(\{x\in A_I: \mu x\in p^n A_I\})_{n\ge 1}$ are intertwined.
   \end{lemma}
   \begin{proof}
    We will proceed by noetherian approximation, adapting the proof of \cite[Lemma 12.8 (ii)]{BMS1}. We define $\Lambda:=\Zz_p\llbracket T_1, T_2\rrbracket$, and we regard $A_{\inf}$ as a $\Lambda$-module via the $\Zz_p$-linear map
    \begin{equation}\label{lam}
     \Lambda\to A_{\inf},\;\;\;\; T_1\mapsto [\varepsilon],\; T_2\mapsto [p^\flat].
    \end{equation}
   First, note that $\mu$ is the image of $T_1-1$ under the map (\ref{lam}). By setting
   $$\Lambda_{\inf,I}:=\Lambda\left[\frac{p}{T_2^{1/r}}, \frac{T_2^{1/s}}{p}\right]$$
   we have that $A_{\inf, I}=A_{\inf}\otimes_{\Lambda}\Lambda_{\inf,I}$. In particular, denoting by $\Lambda_I$ the $p$-adic completion of $\Lambda_{\inf, I}$, we have $A_{I}=A_{\inf}\widehat \otimes_{\Lambda}\Lambda_I$, where the latter completion is $p$-adic. Then, it suffices to prove the statement with the ring $\Lambda_I$ in place of $A_I$. For this, observing that $\Lambda_I$ is noetherian, we conclude by the Artin-Rees lemma (\cite[Tag 00IN]{Thestack}) for the system of ideals $(p^n \Lambda_I)_{n\ge 1}$ of $\Lambda_I$, and $(T_1-1)\Lambda_I\subset \Lambda_I$.
   \end{proof}


 \subsubsection{$v$-descent}\label{vdesc}
 As announced, our next goal is to state a consequence of Scholze's $v$-descent results in \cite{Scholze3}, which will serve as a tool to prove the main comparison results of this paper for singular rigid-analytic varieties. To state the desired result we need some preliminary definitions.

 \begin{df}
   Let $X$ be an analytic adic space defined over $\Spa(\Zz_p, \Zz_p)$. We denote by $$X_{v}:=X^\diamondsuit_{v}$$ its \textit{$v$-site}. 
  \end{df}
  
  We recall from \cite[Theorem 8.7]{Scholze3} that the presheaves  $\cl O^+: Y\mapsto \cl O_{Y}^+(Y)$ and  $\cl O: Y\mapsto \cl O_{Y}(Y)$ on the $v$-site of all ($\kappa$-small) perfectoid spaces are sheaves. Then, similarly to Definition \ref{gyu}, we can give the following definition.
  
 \begin{df}
   Let $X$ be an analytic adic space defined over $\Spa(\Zz_p, \Zz_p)$. We define the \textit{integral $v$-structure sheaf} $\widehat{\cl O}_X^+$ and the \textit{$v$-structure sheaf} $\widehat{\cl O}_X$ on $X_v$ by setting respectively $$\widehat{\cl O}_X^+(Y):=\cl O^+_{Y^{\sharp}}(Y^\sharp),\;\;\;\;\;\;\;\; \widehat{\cl O}_X(Y):=\cl O_{Y^{\sharp}}(Y^\sharp)$$
   for all perfectoid spaces $Y\in X_v$.
 \end{df}

 Then, we can introduce the following notation.
 
 \begin{notation} For $X$ an analytic adic space over $\Spa(C, \cl O_C)$, starting from the integral $v$-structure sheaf $\widehat{\cl O}_X^+$, we define an analogue of the pro-étale period sheaves in \S \ref{petsheaves} on the $v$-site $X_v$. 
  By a slight abuse of notation, we denote such $v$-sheaves with the same symbol as the respective pro-étale period sheaves, adding a subscript $(-)_v$, resp. $(-)_{\pet}$, in case of potential confusion.
 \end{notation}

 \begin{prop}\label{compv}
    Let $I\subset (0, \infty)$ be a compact interval with rational endpoints, and let $m\ge 1$ be an integer. Let 
   \begin{equation}\label{im}
    \mathbf{B}\in \{\Bb_{I}, \Bb, \Bb_{\dR}^+, \Bb_{\dR}^+/\Fil^m\}.
   \end{equation}
   \begin{enumerate}[(i)]
    \item\label{compv:1} For any $Z$ affinoid perfectoid space over $\Spa(C, \cl O_C)$, we have $H^i_v(Z, \mathbf{B})=0$ for all $i>0$.
    \item\label{compv:2} Let $X$ an analytic adic space over $\Spa(C, \cl O_C)$.  Let $\lambda: X_v\to X_{\pet}$ denote the natural morphism of sites. Then, we have
   $$R\lambda_* \mathbf{B}_v= \mathbf{B}_{\pet}.$$
   In particular, the pro-étale cohomology of $\mathbf{B}$ satisfies $v$-hyperdescent.
   \end{enumerate}
 \end{prop}
 \begin{proof} By standard reduction steps (see e.g. \cite[Proposition 4.7]{Bosco} and the references therein), part \listref{compv:1} follows from the almost vanishing of $H_v^i(Z, \widehat{\cl O}_X^+)$ for $i>0$, which is proven in greater generality in \cite[Proposition 8.8]{Scholze3}. 
 
 For part \listref{compv:2}, by definition we have $\lambda_* \mathbf{B}_v= \mathbf{B}_{\pet}$. Then, we want to show that $R^i\lambda_*\mathbf{B}_v=0$ for $i>0$. This follows from part \listref{compv:1} recalling that $R^i\lambda_*\mathbf{B}_v$ is the sheafification of the presheaf $U\mapsto H^i_v(U, \mathbf{B})$ on $X_{\pet}$, and affinoid perfectoid spaces in $X_{\pet}$ form a basis of the site.
 \end{proof}

 \subsection{$B$-cohomology and $B_{\dR}^+$-cohomology} Now, we can finally define the $B$-cohomology and the $B_{\dR}^+$-cohomology theories for rigid-analytic varieties over $C$. \medskip
 
 In the following, for $X$ a rigid-analytic variety over $C$, we denote by $X_{\ett, \cond}$ the site introduced in \cite[Definition 2.13]{Bosco}, and similarly we define the site $X_{\eh, \cond}$. Note that we have a natural morphism of sites
 $$\alpha: X_v\to X_{\eh, \cond}.$$
 A feature of the site $X_{\eh, \cond}$ (as opposed to the site $X_{\eh}$) is that the pushfoward along $\alpha$ retains the information captured by profinite sets.\footnote{We refer the reader to \cite[\S 2.3]{Bosco} for a more detailed discussion.}

 \begin{df}\label{defdr}
   Let $X$ be a rigid-analytic variety over $C$. We denote by $\alpha:X_{v}\to X_{\eh, \cond}$ the natural morphism of sites. Let $I\subset (0, \infty)$ be a compact interval with rational endpoints, and let $m\ge 1$ be an integer. Given 
   \begin{equation*}
    \mathbf{B}\in \{\Bb_{I}, \Bb, \Bb_{\dR}^+, \Bb_{\dR}^+/\Fil^m\}
   \end{equation*}
   we write $\mathscr{B}=\mathbf{B}_{\Spa(C)_{\pet}}$ for the corresponding condensed period ring. 
   
   We define the \textit{$\mathscr{B}$-cohomology} of $X$ as the complex of $D(\Mod^{\cond}_{\mathscr{B}})$
   $$R\Gamma_{\mathscr{B}}(X):=R\Gamma_{\eh, \cond}(X, L\eta_tR\alpha_*\mathbf{B}).$$
   We endow $R\Gamma_{\mathscr{B}}(X)$ with the filtration induced by the filtration décalée of Definition \ref{beilifildef}.
  \end{df}

 \begin{rem}[Frobenius on the $B$-cohomology]\label{frb}
  Since $\varphi(t)=pt$, the Frobenius automorphism of $\Bb$ induces a $\varphi_B$-semilinear automorphism
  $$\varphi: R\Gamma_B(X)\to R\Gamma_B(X)$$
  which preserves the filtration décalée.
 \end{rem}

  Next, we begin to study some basic properties of the $B$-cohomology theory. As a preparation, we state the following boundedness result which relies on an improved version of the almost purity theorem recently proved by Bhatt--Scholze, \cite[Theorem 10.9]{Prisms}.
  
  \begin{prop}[cf. {\cite[Proposition 7.5.2]{Guo1}}, {\cite[Theorem 6.10.3]{Zavyalov1}}]\label{bddd}
   Let $X$ be a rigid-analytic variety over $C$ of dimension $d$.  Let $\nu: X_{\pet}\to X_{\ett, \cond}$ denote the natural morphisms of sites. Let $\mathbf{B}$ any of the period sheaves of  (\ref{im}). Then, $R^i\nu_* \mathbf{B}$ vanishes for all $i>d$.
   \end{prop}
  \begin{proof}
  We will show that for any affinoid rigid space $X$ over $C$ of dimension $d$, we have $H^i_{\pet}(X, \mathbf{B})=0$ for all $i>d$. By Noether normalization lemma, \cite[\S 3.1, Proposition 3]{Bosch},  there exists a finite morphism $f:X\to \Dd_C^d$, where the target denotes the $d$-dimensional unit closed disk over $C$. Then, by \cite[Lemma 6.10.2]{Zavyalov1},\footnote{This lemma relies on \cite[Theorem 10.11]{Prisms}, and hence on \cite[Theorem 10.9]{Prisms}.} the diamond $X^\diamondsuit$ admits a $\Zz_p(1)^d$-torsor for the $v$-topology
  \begin{equation}\label{torz}
   \widetilde{X}^\diamondsuit \to X^\diamondsuit
  \end{equation}
  where $\widetilde{X}^\diamondsuit$ is a diamond representable by an affinoid perfectoid space.
  Considering the Cartan--Leray spectral sequence associated to (\ref{torz}), by \cite[Proposition 4.12]{Bosco} and Proposition \ref{compv}, we have an isomorphism
    $$R\Gamma_{\cond}(\Zz_p(1)^d, H^0(\widetilde{X}^\diamondsuit, \mathbf{B}))\overset{\sim}{\to}R\Gamma_{\pet}(X, \mathbf{B}).$$
   Then, the statement follows from \cite[Proposition B.3]{Bosco}, which implies that $\Zz_p(1)^d\cong \Zz_p^d$ has cohomological dimension $d$.
  \end{proof}

  The following lemma will be useful to reduce the study of the $B$-cohomology theory to the study of the $B_I$-cohomology theories for suitable intervals $I\subset (0, \infty)$.
  
  \begin{lemma}[cf. {\cite[Lemma 4.3]{LeBras2}}]\label{dinverse}
   With notation as in Definition \ref{defdr}, the natural maps
   \begin{equation}\label{fromlb}
    L\eta_tR\alpha_*\Bb\to R\varprojlim_I L\eta_tR\alpha_* \Bb_I\;\;\;\;\;\;\;\;\;\; L\eta_t R\alpha_*\Bb_{\dR}^+\to R\varprojlim_m L\eta_t R\alpha_*(\Bb_{\dR}^+/\Fil^m).
   \end{equation}
   are isomorphisms compatible with the filtration décalée.
  \end{lemma}
  \begin{proof}
  We prove that the left map in (\ref{fromlb}) is an isomorphism compatible with filtrations (for the right map in (\ref{fromlb}) the proof is similar and easier). By \cite[Lemma 5.2, (1)]{BMS2}, it suffices to show that the limit of the filtrations of the source and the target agree, and that such map is an isomorphism on graded pieces. 
  
  For the first assertion, by the uniform boundedness of the complexes $R\alpha_*\Bb$ and $R\alpha_* \Bb_I$ for varying compact intervals $I\subset (0, \infty)$ with rational endpoints, which follows from Proposition \ref{bddd}, we can reduce to showing that the natural map
 $$R\alpha_*\Bb\to R\varprojlim_{I}R\alpha_* \Bb_I$$
 is an isomorphism. This follows recalling that the natural map $\Bb\to R\varprojlim_I \Bb_I$ is an isomorphism: in fact, using \cite[Lemma 3.18]{Scholze}, we can reduce to checking this on each affinoid perfectoid space $Z$ over $\Spa(C, \cl O_C)$, where it follows from the topological Mittag-Leffler property of the countable inverse system $\{\Bb_I(Z)\}_I$ which implies that, for all $j>0$, we have $R^j\varprojlim_I \Bb_I(Z)=0$ (\cite[Lemma 4.8]{Bosco}).
  
   Then, by Proposition \ref{bbe}\listref{bbe:2} (and by twisting) it remains to prove that, for each $i\ge 0$, the natural map
 \begin{equation}\label{zyzz}
   \tau^{\le i}R\alpha_*(\Bb/t)\to R\varprojlim_I \tau^{\le i}R\alpha_*(\Bb_I/t)
 \end{equation}
 is an isomorphism. For this, we observe that, by \cite[Theorem II.0.1]{FS}, for any compact interval $I\subset (0, \infty)$ with rational endpoints, we have 
 \begin{equation}\label{BIprod}
  \Bb_I/t=\prod_{y\in |Y_{\FF, I}|^{\ccl}}\Bb_{\dR}^+/t^{\ord_y(t)}\Bb_{\dR}^+
 \end{equation}
  where $|Y_{\FF, I}|^{\ccl}\subset |Y_{\FF, I}|$ denotes the subset of classical points (we note that, by compactness of the interval $I$, the latter product is a finite direct product of copies of $\widehat{\cl O}$).\footnote{For $y\in |Y_{\FF}|^{\ccl}$, we have $\ord_y(t)\in \{0, 1 \}$: in fact, $t$ has a simple zero at $\infty$ on $\FF=Y_{\FF}/\varphi^{\Zz}$.} Then, using again the topological Mittag-Leffler property of the countable inverse system $\{\Bb_I(Z)\}_{I}$ for each $Z$ affinoid perfectoid spaces over $\Spa(C, \cl O_C)$, we have that
  \begin{equation}\label{Bprod}
   \Bb/t=\prod_{y\in |Y_{\FF}|^{\ccl}}\Bb_{\dR}^+/t^{\ord_y(t)}\Bb_{\dR}^+
  \end{equation}
  where $|Y_{\FF}|^{\ccl}\subset |Y_{\FF}|$ denotes the subset of classical points (cf. with \cite[\S 2.6]{FF}). Moreover, by (\ref{BIprod}) we have that
  
  \begin{equation}\label{0001}
  R\varprojlim_I \tau^{\le i}R\alpha_*(\Bb_I/t)=\prod_I \tau^{\le i}R\alpha_*(\Bb_I/t).
  \end{equation}
  We conclude that the natural map (\ref{zyzz}) is an isomorphism, combining (\ref{BIprod}), (\ref{Bprod}), (\ref{0001}), and fact that cohomology commutes with direct products.
  \end{proof}

  The next proposition gives in particular a convenient local description of the $B$-cohomology theory on a smooth affinoid rigid space over $C$.
  
  \begin{prop}\label{3.11}
   With notation as in Definition \ref{defdr}, let $\nu :X_{\pet}\to X_{\ett, \cond}$ denote  the natural morphism of sites.
   \begin{enumerate}[(i)]
    \item\label{3.11.1} If $X$ is smooth, we have a natural identification in $D(\Mod^{\cond}_{\mathscr{B}})$
    $$R\Gamma_{\mathscr{B}}(X)=R\Gamma(X, L\eta_tR\nu_*\mathbf{B}).$$
    \item\label{3.11.2} If $X$ is a smooth affinoid over $C$, the natural map of complexes of condensed $\mathscr{B}$-modules
   \begin{equation}\label{funda}
    L\eta_{t}R\Gamma_v(X, \mathbf{B})\to R\Gamma(X, L\eta_tR\alpha_*\mathbf{B})=R\Gamma_{\mathscr{B}}(X)
   \end{equation}
   is a filtered quasi-isomorphism. Here, on both sides, the filtration on $L\eta_t(-)$ is the filtration décalée of Definition \ref{beilifildef}.
   \end{enumerate}  
  \end{prop}
  \begin{proof}
   We first prove part \listref{3.11.2}, adapting the proof of \cite[Proposition 3.11]{LeBras2}, and using Proposition \ref{beilifil} for the compatibility with the filtrations statement.
   
  Thus, let $X$ be a smooth affinoid over $C$. To show that (\ref{funda}) is a filtered quasi-isomorphism, similarly to the proof of Lemma \ref{dinverse}, by \cite[Lemma 5.2, (1)]{BMS2}, it suffices to show that the limit of the filtrations of the source and the target of (\ref{funda}) agree, and that such map is an isomorphism on graded pieces. The former statement follows from Proposition \ref{bddd}. Then, by Proposition \ref{bbe}\listref{bbe:2} (and by twisting) it suffices to prove that, for each $i\ge 0$, the natural map
 \begin{equation*}
   \tau^{\le i}R\Gamma(X, \mathbf{B}/t)\to R\Gamma(X, \tau^{\le i}R\alpha_*(\mathbf{B}/t))
 \end{equation*}
 is a quasi-isomorphism. 
 By (the proof of) Lemma \ref{dinverse}, we can reduce to the case $\mathbf{B}\in \{\Bb_I, \Bb_{\dR}^+\}$. Then, recalling that $\Bb_I/t$ is isomorphic to a finite direct product of copies of $\widehat{\cl O}$ (by (\ref{BIprod}) and the compactness of the interval $I$), we can further reduce to showing that, for each $i\ge 0$, the natural map
 \begin{equation*}\label{ahh}
   \tau^{\le i}R\Gamma(X, \widehat{\cl O})\to R\Gamma(X, \tau^{\le i}R\alpha_*\widehat{\cl O})
 \end{equation*}
 is a quasi-isomorphism. For this, considering the spectral sequences
 $$H^{j-k}(X, H^k(R\alpha_*\widehat{\cl O}))\implies H^j(X, \widehat{\cl O})$$
 $$H^{j-k}(X, H^k(\tau^{\le i}R\alpha_*\widehat{\cl O}))\implies H^j(X, \tau^{\le i}R\alpha_*\widehat{\cl O})$$
 it suffices to show that, for $j>i$ and $k\le i$, we have $H^{j-k}(X, R^k\alpha_* \widehat{\cl O})=0$, or more generally that
 \begin{equation}\label{combb}
  H^r(X, R^k\alpha_* \widehat{\cl O})=0, \text{ for all } r>0.
 \end{equation}
 By (the proof of) \cite[Proposition 3.23]{ScholzeSurvey}, for any $Y$ smooth rigid-analytic variety over $\Spa(C, \cl O_C)$, denoting by $\nu: Y_{\pet}\to Y_{\ett, \cond}$ the natural morphism of sites, we have a natural isomorphism
 $\Omega_{Y_{\ett}}^k(-k)\overset{\sim}{\to} R^k\nu_* \widehat{\cl O}$ of sheaves with values in $\Vect_C^{\cond}$.\footnote{Working with condensed group cohomology instead of continuous group cohomology in the proof of \textit{loc. cit.}.} Then, by $\eh$-sheafification, we have a natural isomorphism of sheaves with values in $\Vect_C^{\cond}$
 \begin{equation}\label{pz1}
  \Omega_{X_{\eh}}^k(-k)\overset{\sim}{\to} R^k\alpha_* \widehat{\cl O}.
 \end{equation}
 Denoting by $\pi: X_{\eh}\to X_{\ett}$ the natural morphism of sites, by Proposition \ref{ehdescent} (using that $X$ is smooth), we have that
 \begin{equation}\label{pz2}
  R\pi_*\Omega_{X_{\eh}}^k=\Omega_{X_{\ett}}^k[0]
 \end{equation}
 as complexes of sheaves with values in $\Vect_C^{\cond}$. Then, combining (\ref{pz1}) and (\ref{pz2}), we deduce (\ref{combb}) from the condensed version of Tate's acyclicity theorem (see \cite[Lemma 5.6(i)]{Bosco}). This concludes the proof of part \listref{3.11.2}.
 
 For part \listref{3.11.1}, as the statement is étale local, we can reduce to the case when $X$ is a smooth affinoid rigid space over $C$. In this case, similarly to part \listref{3.11.2}, the natural map 
  \begin{equation}\label{hopfin}
   L\eta_{t}R\Gamma_{\pet}(X, \mathbf{B})\to R\Gamma(X, L\eta_tR\nu_*\mathbf{B})
  \end{equation}
  is a quasi-isomorphism. Hence, combining (\ref{hopfin}) and (\ref{funda}), the statement follows from Proposition \ref{compv}.
  \end{proof}

  \section{\textbf{Hyodo--Kato cohomology}}\label{sHK}
   \sectionmark{}

  In this section, following Colmez--Nizioł, \cite{CN4}, we define the \textit{Hyodo--Kato cohomology theory} for rigid-analytic varieties over $C$, simplifying the topological treatment given in \textit{op. cit.} and extending it to the singular case.
 
   \subsection{Local Hyodo--Kato morphism}\label{localHK} We begin by revisiting in the condensed setup the Hyodo--Kato morphism constructed by Beilinson and Colmez--Nizioł.
 
  \begin{convnot}\label{logcriss} Let $n\ge 1$ be an integer. For a condensed ring $R$, we denote by $R_n$ the reduction of $R$ modulo $p^n$. \medskip
  \begin{description}
   \item[Log structures] We define a (pre-)log structure on a given condensed ring as a (pre-)log structure on the underlying ring. \medskip   
   
   For $\cl O$ a discrete valuation ring, we denote by $\cl O^\times$ (resp. $\cl O_{n}^\times$) the canonical log structure on $\cl O$ (resp. its pullback on $\cl O_{n}$), and  we denote by $\cl O^0$ (resp. $\cl O_{n}^0$) the log structure on $\cl O$ associated to $(\Nn\to \cl O, 1\mapsto 0)$  (resp. its pullback on $\cl O_{n}$). \medskip
   
   We denote by $\cl O_C^\times$ (resp. $\cl O_{C, n}^\times$) the canonical log structure on $\cl O_C$ (resp. its pullback on $\cl O_{C, n}$). We write $A_{\cris, n}^\times$ for the unique quasi-coherent, integral, log structure on $A_{\cris, n}$ lifting $\cl O_{C, n}^\times$  (see e.g. \cite[\S 5.2]{CK}). We denote by $A_{\cris}^\times$ the log structure on $A_{\cris}$ associated to the pre-log structure
   $$\cl O_{C^\flat}\setminus \{0\}\to A_{\cris},\; x\mapsto [x].$$
   Note that the log structure $A_{\cris, n}^\times$ is the pullback of the log structure $A_{\cris}^\times$.
   \medskip 
   
   \item[Log-crystalline cohomology] We refer the reader to \cite[\S 1]{Beili} for a review of log-crystalline cohomology, and the terminology used in the following. We write PD as a shortening of \textit{divided power}.\medskip
   
   Let $(\cl Y, M_{\cl Y}, \cl I, \gamma)$ be a ($p$-adic formal) log PD scheme such that $(\cl Y, M_{\cl Y})$ is quasi-coherent. Let $(\cl X, M_{\cl X})$ be an integral quasi-coherent ($p$-adic formal) log scheme over $(\cl Y, M_{\cl Y}, \cl I, \gamma)$.
   We write $$((\cl X, M_{\cl X})/(\cl Y, M_{\cl Y}))_{\cris}$$ for the log-crystalline site of $(\cl X, M_{\cl X})$ over $(\cl Y, M_{\cl Y}, \cl I, \gamma)$, \cite[\S 1.12]{Beili}, we denote by $\cl O_{\cris}$ its structure sheaf, regarded as a sheaf with values in condensed abelian groups, and we define the log-crystalline cohomology
   $$R\Gamma_{\cris}((\cl X, M_{\cl X})/(\cl Y, M_{\cl Y})):= R\Gamma(((\cl X, M_{\cl X})/(\cl Y, M_{\cl Y}))_{\cris}, \cl O_{\cris})\in D(\CondAb).$$
   In the case the relevant log structures are fixed, they are omitted from the notation.
   \medskip
   
   \item[Condensed period rings] Recall from \S \ref{convent} that $\cl O_F=W(k)$ and $\cl O_{\breve F}=W(\bar k)$, where $\bar k$ is a fixed algebraic closure of $k$. We fix the unique Frobenius equivariant section $k\to \cl O_K/p$ of $\cl O_K/p\to k$, in order to regard $\cl O_K$ as a $\cl O_F$-algebra, and $\cl O_C$, as well as the condensed period rings of \ref{condp}, as a $\cl O_{\breve F}$-algebra.\medskip
   
   We denote $r_{\varpi}^+=\cl O_F\llbracket T \rrbracket$ and we equip it with the log structure associated to $T$. We write $r^{\PD}_{\varpi}$ for the $p$-adic log PD envelope of $r_{\varpi}^+$ with respect to the kernel of the morphism $$r_{\varpi}^+\to \cl O_K^\times,\; T\mapsto \varpi$$
   and we endow it with a Frobenius induced by $T\mapsto T^p$, and a monodromy defined by $T\mapsto T$.
   Then, we define the condensed period rings
   $$\widehat A_{\st, n}:=H^0_{\cris}(\cl O^\times_{C, n}/r^{\PD}_{\varpi, n})\simeq R\Gamma_{\cris}(\cl O^\times_{C, n}/r^{\PD}_{\varpi, n}), \;\;\;\;\;\; \widehat A_{\st}:=\varprojlim_n \widehat A_{\st, n},\;\;\;\;\;\; \widehat B_{\st}^+:=\widehat A_{\st}[1/p]$$
   and we equip them with their natural action of $\mathscr{G}_K$, Frobenius $\varphi$, and monodromy $N$, \cite[\S 4.6]{Tsuji-Cst}, \cite[\S 3.2.1]{CDN1}. \medskip
  \end{description}
  \end{convnot}

  \begin{theorem}[Beilinson, Colmez--Nizioł, {\cite[Theorem 2.12, Corollary 2.20]{CN4}}]\label{bcn}
  Let $\cl X$ be an integral quasi-coherent log scheme over $\cl O_{C, 1}^\times$. Denote by $\cl X^0$ the pullback of $\cl X$ to $\cl O_{\breve F, 1}^0$.
   Assume that $\cl X$ has a descent to $\cl Z$ a qcqs, fine, log-smooth, log scheme over $\cl O_{L, 1}^\times$ of Cartier type,\footnote{See \cite[Definition 4.8]{Kato-log} for the definition of \textit{Cartier type}.} for some finite extension $L/K$.
    \begin{enumerate}[(i)]
      \item\label{bcn:1} There exists a natural isomorphism in $D(\Vect_{\Qq_p}^{\ssolid})$
   \begin{equation}\label{beilin}
   R\Gamma_{\cris}(\cl X^0/\cl O_{\breve F}^0)\dsolid_{\cl O_{\breve F}}B_{\st}^+\overset{\sim}{\to} R\Gamma_{\cris}(\cl X/A_{\cris}^\times)\dsolid_{A_{\cris}}B_{\st}^+
   \end{equation}
   independent of the descent, and compatible with the actions of Galois, Frobenius $\varphi$ and monodromy $N$.\footnote{On the right-hand side of (\ref{beilin}) the operator $N$ is the monodromy of $B_{\st}^+$, and on the left-hand side of (\ref{beilin}) it combines the monodromy of both factors of the tensor product.}
   \item\label{bcn:2} There exists a natural isomorphism in $D(\Vect_{\Qq_p}^{\ssolid})$
   \begin{equation}\label{iHK}
   R\Gamma_{\cris}(\cl X^0/\cl O_{\breve F}^0)\dsolid_{\cl O_{\breve F}}C\overset{\sim}{\to} R\Gamma_{\cris}(\cl X/\cl O_C^\times)_{\Qq_p}
   \end{equation}
    independent of the descent, and compatible with the actions of Galois, Frobenius $\varphi$ and with the quasi-isomorphism (\ref{beilin}) via the morphism $R\Gamma_{\cris}(\cl X/A_{\cris}^\times)\to R\Gamma_{\cris}(\cl X/\cl O_C^\times)$ induced by Fontaine's map $\theta: A_{\cris}\to \cl O_C$.
   \end{enumerate}
   \end{theorem}
  \begin{proof}
  For part \listref{bcn:1}, the desired morphism (\ref{beilin}), satisfying the stated properties, is constructed in \cite[Theorem 2.12]{CN4}, and we only need to carry \textit{loc. cit.} to solid $\Qq_p$-vector spaces. By the independence of the descent proven in \textit{loc. cit.} we can assume for simplicity that $L=K$.
  Relying on \cite[\S 2.3.3]{CN4}, we will construct (\ref{beilin}) as the composite $$\varepsilon_{\st}:=\delta^{-1}\circ(\widehat\varepsilon_{\st})^{N-\nilp}\circ\delta$$
  where $M^{N-\nilp}:=\colim_{r\in \Nn}M^{N^r=0}$, we denote by $\delta:B_{\st}^+\overset{\sim}{\to}\widehat{B}_{\st}^{+, N-\nilp}$ the natural $B_{\cris}^+$-linear isomorphism,\footnote{Which is compatible with Galois, Frobenius, and monodromy actions, \cite[Theorem 3.7]{Katolast}.} and 
  $$\widehat\varepsilon_{\st}:R\Gamma_{\cris}(\cl X^0/\cl O_{\breve F}^0)\dsolid_{\cl O_{\breve F}}\widehat{B}_{\st}\to R\Gamma_{\cris}(\cl X/A_{\cris}^\times)\dsolid_{A_{\cris}}\widehat{B}_{\st}$$ is defined as follows. Considering the morphisms of PD thickenings
  \begin{center}
  \begin{tikzcd}
  r^{\PD}_{\varpi, n} \arrow[d]\arrow[twoheadrightarrow]{r} & \cl O_{K, 1}^\times \arrow[d] \\
   \widehat A_{\st, n} \arrow[twoheadrightarrow]{r} & \cl O_{C, 1}^\times
  \end{tikzcd}
  \end{center}
   for varying $n\ge 1$, by base change, \cite[(1.11.1)]{Beili}, we have quasi-isomorphisms
    \begin{equation}\label{pdt}
     R\Gamma_{\cris}(\cl Z/(r^{\PD}_{\varpi, n}, \cl O_{K, 1}^\times))\otimes_{r^{\PD}_{\varpi, n}}^{\LL}\widehat A_{\st, n}\overset{\sim}{\to}R\Gamma_{\cris}(\cl X/(\widehat A_{\st, n}, \cl O_{C, 1}^\times)).
    \end{equation}
   Now, denoting by $\widehat{\otimes}^{\LL}$ the derived $p$-adic completion, again by base change \cite[(1.11.1)]{Beili}, after taking the derived inverse limit over $n\ge 1$, and then inverting $p$ in (\ref{pdt}), the right-hand side identifies with $(R\Gamma_{\cris}(\cl X/A_{\cris}^\times)\widehat{\otimes}_{A_{\cris}}^{\LL}\widehat{A}_{\st})_{\Qq_p}$, and, by \cite[Theorem 2.6]{CN4}, the left-hand side is quasi-isomorphic to 
   $$(R\Gamma_{\cris}(\cl Z^0/\cl O_{F}^0)\widehat{\otimes}_{\cl O_{F}}^{\LL}\widehat{A}_{\st})_{\Qq_p}\simeq ( R\Gamma_{\cris}(\cl X^0/\cl O_{\breve F}^0)\widehat{\otimes}_{\cl O_{\breve F}}^{\LL}\widehat{A}_{\st})_{\Qq_p}.$$
   where $\cl Z^0$ is the pullback of $\cl Z$ to $\cl O_{F, 1}^0$.  We denote by $\widehat\varepsilon_{\st}$ the induced from (\ref{pdt}) morphism
   \begin{equation}\label{mmort}
     \widehat\varepsilon_{\st}:(R\Gamma_{\cris}(\cl X^0/\cl O_{\breve F}^0)\widehat{\otimes}_{\cl O_{\breve F}}^{\LL}\widehat{A}_{\st})_{\Qq_p}\overset{\sim}{\to}(R\Gamma_{\cris}(\cl X/A_{\cris}^\times)\widehat{\otimes}_{A_{\cris}}^{\LL}\widehat{A}_{\st})_{\Qq_p}.
   \end{equation}
  
   To write the target of (\ref{mmort}) in terms of the derived solid tensor product, we note that, since the $A_{\cris}$-algebra $\widehat{A}_{\st}$ is isomorphic to the $p$-adic completion of a divided power polynomial algebra of the form $A_{\cris}\langle x\rangle$, applying Proposition \ref{solid-vs-padic} with $M=R\Gamma_{\cris}(\cl X/A_{\cris}^\times)$ and $N=A_{\cris}\langle x\rangle$ regarded in $D(\Mod_{A_{\cris}}^{\ssolid})$, we obtain the identification
   \begin{equation}\label{avs1}
     R\Gamma_{\cris}(\cl X/A_{\cris}^\times)\widehat{\otimes}_{A_{\cris}}^{\LL}\widehat{A}_{\st}=R\Gamma_{\cris}(\cl X/A_{\cris}^\times)\dsolid_{A_{\cris}}\widehat{A}_{\st}.
   \end{equation}
   Next, we want to show that
   \begin{equation}\label{avs2}
    (R\Gamma_{\cris}(\cl X^0/\cl O_{\breve F}^0)\widehat{\otimes}_{\cl O_{\breve F}}^{\LL}\widehat{A}_{\st})_{\Qq_p}=( R\Gamma_{\cris}(\cl X^0/\cl O_{\breve F}^0)\dsolid_{\cl O_{\breve F}}\widehat{A}_{\st})_{\Qq_p}.
   \end{equation}
   We note that, choosing a basis for the  $\breve F$-Banach space $\widehat{B}_{\st}^+$, we can can identify it with $N^\wedge_p[1/p]$, where $N=\bigoplus_I\cl O_{\breve F}$ for some set $I$.\footnote{In fact, as $\breve F$ is discretely valued, combining \cite[Lemma A.30, Proposition A.55(i), Lemma A.52]{Bosco}, any $\breve F$-Banach space is isomorphic to $(\bigoplus_I \Zz_p)^{\wedge}_p\solid_{\Zz_p}\breve F$, for some set $I$, and the latter is isomorphic to $(\bigoplus_I\cl O_{\breve F})^{\wedge}_p[1/p]$ by Proposition \ref{solid-vs-padic}.}
   Since $\widehat{A}_{\st}$ is a lattice in $\widehat{B}_{\st}^+$, there exist $n, m\in \Zz$ such that $p^nN^\wedge_p\subset \widehat{A}_{\st}\subset p^m N^\wedge_p$. Then, (\ref{avs2}) follows applying Proposition \ref{solid-vs-padic} with $M= R\Gamma_{\cris}(\cl X^0/\cl O_{\breve F}^0)$ and $N=\bigoplus_I\cl O_{\breve F}$ regarded in $D(\Mod_{\cl O_{\breve F}}^{\ssolid})$.
   
   Therefore, in view of (\ref{avs1}) and (\ref{avs2}), using that the derived solid tensor product commutes with filtered colimits, the composite $\varepsilon_{\st}=\delta^{-1}\circ(\widehat\varepsilon_{\st})^{N-\nilp}\circ\delta$ is given by
   \begin{align}
    R\Gamma_{\cris}(\cl X^0/\cl O_{\breve F}^0)\dsolid_{\cl O_{\breve F}}B_{\st}^+ \label{boh1}
    &\overset{\sim}{\to}( R\Gamma_{\cris}(\cl X^0/\cl O_{\breve F}^0)\dsolid_{\cl O_{\breve F}}\widehat{B}_{\st}^+)^{N-\nilp}  \\ 
    &\overset{\sim}{\to}(R\Gamma_{\cris}(\cl X/A_{\cris}^\times)\dsolid_{\cl O_{\breve F}}\widehat{B}_{\st}^+)^{N-\nilp} \\
    &\overset{\sim}{\leftarrow} R\Gamma_{\cris}(\cl X/A_{\cris}^\times)\dsolid_{A_{\cris}}B_{\st}^+ \label{boh2}
   \end{align}
   where in (\ref{boh1}) we used that the monodromy operator $N$ on $ R\Gamma_{\cris}(\cl X^0/\cl O_{\breve F}^0)$ is nilpotent by Lemma \ref{nilpo} (and base change), and in (\ref{boh2}) we used the triviality of the action of $N$ on $R\Gamma_{\cris}(\cl X/A_{\cris}^\times)$. This shows part \listref{bcn:1}.
   
   Part \listref{bcn:2} follows from \listref{bcn:1}. In fact, under the (non-canonical) identification $B_{\st}^+=B_{\cris}^+[U]$, given by (\ref{noncann}), applying to (\ref{beilin}) the (non-Galois-equivariant) map $B_{\st}^+\to B_{\cris}^+: U\mapsto 0$, and then Fontaine's map $\theta: B_{\cris}^+\to C$, by base change we get (\ref{iHK}).  The compatibility of (\ref{iHK}) with the Galois action is checked in the proof of \cite[Corollary 2.20]{CN4}.
  \end{proof}

  \begin{lemma}\label{nilpo}
   Let $\cl Z$ be a quasi-separated, fine, saturated, log-smooth, locally of finite type log scheme over $\cl O_{K, 1}^\times$ of dimension $d$. Then, the monodromy operator $N$ on $R\Gamma_{\cris}(\cl Z/W(k)^\times)$ is nilpotent with nilpotency index bounded above by a function depending on $d$.
  \end{lemma}
  \begin{proof}
   By \cite[Theorem 5.10]{Niziol_toric} there exists a log-blow-up $\cl Y\to \cl Z$ over $\cl O_{K, 1}^\times$ that resolves singularities, and by the proof of \cite[Proposition 2.3]{Niziol_semi} we have a natural quasi-isomorphism
   $$R\Gamma_{\cris}(\cl Z/W(k)^\times)\overset{\sim}{\to}R\Gamma_{\cris}(\cl Y/W(k)^\times)$$
   compatible with monodromy $N$. Then, the statement follows from \cite[\S 3]{Mokrane}.
  \end{proof}

  \subsection{Beilinson bases and  $\infty$-categories of hypersheaves}\label{bbases} In this subsection, we collect some $\infty$-categorical tools that we will need to extend to rigid-analytic varieties over $C$ the local Hyodo--Kato morphism of \S \ref{localHK}.
 
  \begin{notation}[Hypersheaves and hypercompletion]
   Let $\cl C$ be a site, and let $\cl D$ be a presentable $\infty$-category.
   We denote by $\Shv(\cl C, \cl D)$ the $\infty$-category of sheaves on $\cl C$ with values in $\cl D$.
   
   We recall that $$\Shv(\cl C, \cl D)=\Shv(\cl C, \Ani)\otimes \cl D$$
  \cite[Remark 1.3.1.6]{LurSpectral}, where $\otimes$ denotes the tensor product of $\infty$-categories \cite[\S 4.8.1]{HA}.

  We denote by $\Shv^{\hyp}(\cl C, \Ani)$ the full $\infty$-subcategory of $\Shv(\cl C, \Ani)$ spanned by the hypercomplete objects, \cite[\S 6.5]{HTT}, and we define the $\infty$-category of \textit{hypersheaves on $\cl C$ with values in $\cl D$} as
  $$\Shv^{\hyp}(\cl C, \cl D):=\Shv^{\hyp}(\cl C, \Ani)\otimes \cl D.$$
  The inclusion $\Shv^{\hyp}(\cl C, \cl D)\hookrightarrow \Shv(\cl C, \cl D)$ admits a left adjoint
  \begin{equation}\label{hypercompl}
   (-)^{\hyp}:\Shv(\cl C, \cl D)\to \Shv^{\hyp}(\cl C, \cl D)
  \end{equation}
  called \textit{hypercompletion}, \cite[Remark 1.3.3.2]{LurSpectral}.
  \end{notation}

 The following generalization of the notion of Grothendieck basis for a site is due to Beilinson, \cite[\S 2.1]{BeiliDer}.
 
 \begin{df}
  Ler $\cl C$ be a small site. A \textit{Beilinson basis} for $\cl C$ is a pair $(\cl B, \beth)$ where $\cl B$ is a small category and $\beth: \cl B\to \cl C$ is a faithful functor satisfying the following property:
   \begin{enumerate}[]
    \item for any $V\in \cl C$ and any finite family of pairs $\{(U_{\alpha}, f_{\alpha})\}$ with $U_{\alpha}\in \cl B$ and $f_{\alpha}:V\to \beth(U_{\alpha})$, there exists a family $\{U'_{\beta}\}$ with $U'_{\beta}\in\cl B$ and a covering family $\{\beth(U_{\beta}')\to V\}$ such that each composition $$\beth(U'_{\beta})\to V\to \beth(U_{\alpha})$$
  lies in the image of $\Hom(U'_{\beta}, U_{\alpha})\hookrightarrow \Hom(\beth(U'_{\beta}), \beth(U_{\alpha}))$.
   \end{enumerate}

 \end{df}

 We endow $\cl B$ with the Grothendieck topology induced from that of $\cl C$: a sieve in $\cl B$ is a covering sieve if its image under $\beth:  \cl B\to \cl C$ generates a covering sieve in $\cl C$. \medskip

 We will use repeatedly the following result.

  \begin{lemma}\label{beihyp}
   Let $\cl C$ be a small site, and let $(\cl B, \beth)$ be a Beilinson basis for $\cl C$. For any presentable $\infty$-category $\cl D$, the functor $\beth: \cl B\to \cl C$ induces an equivalence of $\infty$-categories 
   $$\Shv^{\hyp}(\cl B, \cl D)\overset{\sim}{\to}\Shv^{\hyp}(\cl C, \cl D):\; \cl F\mapsto\cl F^\beth $$
   where the hypersheaf $\cl F^\beth$ is defined via sending $V\in \cl C$ to
   $$\cl F^\beth(V)=\colim\limits_{U_{\bullet}}\lim_{[n]\in \Delta}\cl F(U_n)$$
   the colimit running over all simplicial objects $U_\bullet$ of $\cl B$ such that $\beth(U_{\bullet})\to V$ is a hypercover; furthermore, chosen such a $U_{\bullet}$, the natural map
   $$\cl F^\beth(V)\to \lim_{[n]\in \Delta}\cl F(U_n)$$
   is an isomorphism in $\cl D$.
  \end{lemma}
  \begin{proof}
  It suffices to show the statement in the case $\cl D$ is the $\infty$-category of anima $\Ani$. By \cite[\S 2.1 Proposition]{BeiliDer}, the functor $\beth: \cl B\to \cl C$ is continuous and induces an equivalence of topoi $\cl B^{\sim}\to \cl C^{\sim}$. Then, the statement follows interpreting the notion of hypercompleteness in terms of the Brown--Joyal--Jardine theory of simplicial presheaves, via \cite[Proposition 6.5.2.14]{HTT}.
  \end{proof}

  \begin{rem}\label{indeedpre}
   We will often apply Lemma \ref{beihyp} in the case $\cl D=D(\Mod_A^{\cond})$ is the derived $\infty$-category of $A$-modules in $\CondAb$, for a given condensed ring $A$. Note that such $\cl D$ is indeed presentable, since it is compactly generated, as it follows from \cite[Theorem 2.2]{Scholzecond}.\footnote{Recall also our set-theoretic conventions in \S\ref{convent}.} Moreover, by \cite[Corollary 2.1.2.3]{LurSpectral}, we have an equivalence of $\infty$-categories
   $$D(\Shv(\cl C, \Mod_A^{\cond}))\overset{\sim}{\to} \Shv^{\hyp}(\cl C, D(\Mod_A^{\cond}))$$
   sending $M\in D(\Shv(\cl C, \Mod_A^{\cond}))$ to the hypersheaf $U\mapsto R\Gamma(U, M)$.
  \end{rem}

  \subsection{Globalization}\label{globals}
  
  In this subsection, we extend to rigid-analytic varieties over $C$ the local Hyodo--Kato morphism of \S \ref{localHK}, from a suitable Beilinson basis for the site $\Rig_{C, \eh}$. 
  
  \begin{notation}\label{notss}\
  \begin{description}
   \item[Semistable formal schemes] 
  For each prime  $\ell$,  we fix a compatible system $(p, p^{1/\ell}, p^{\ell^2}, \ldots)$ of $\ell$-th power roots of $p$ in $\cl O_C$.
   We denote by $\cl M_{\sss}$ the category of \textit{semistable $p$-adic formal schemes over $\Spf(\cl O_C)$}, that is the category of $p$-adic formal schemes over $\Spf(\cl O_C)$ having in the Zariski topology a covering by open affines $\fr U$ with \textit{semistable coordinates}, i.e. admitting an étale $\Spf(\cl O_C)$-morphism $\fr U\to \Spf(R^\square)$ with 
   $$R^{\square}:=\cl O_C\{t_0, \ldots, t_r, t^{\pm 1}_{r+1}, \ldots, t_d^{\pm 1}\}/(t_0\cdots t_r-p^q)$$
   for some $0\le r\le d$, and $q\in \Qq_{>0}$ (that may depend on $\fr U$).
   We denote by $\cl M_{\sss, \qcqs}$ the subcategory of $\cl M_{\sss}$ consisting of the qcqs formal schemes.
   We write $$(-)_{\eta}: \cl M_{\sss}\to \Rig_C$$ for the generic fiber functor. \medskip
   \item[Log structures] Unless stated otherwise, we equip $\fr X\in \cl M_{\sss}$ (resp. $\fr X_{\cl O_C/p^n}$) with the canonical log structure, \cite[\S 1.6]{CK}, i.e. the log structure given by the subsheaf associated to the subpresheaf $\cl O_{\fr X, \ett}\cap (\cl O_{\fr X, \ett}[1/p])^\times \hookrightarrow \cl O_{\fr X, \ett}$ (resp. its pullback).
   For $\fr X\in \cl M_{\sss}$, we denote by $\fr X_{\cl O_C/p}^0$ the pullback to $\cl O_{\breve F, 1}^0$ of the log scheme $\fr X_{\cl O_C/p}$ over $\cl O_{C, 1}^\times$. \medskip
   \item[Log de Rham cohomology] For $\fr X\in \cl M_{\sss}$, we denote by $\Omega_{\fr X, \log}^\bullet$ the logarithmic de Rham complex of $\fr X$ over $\cl O_C$, and we define the \textit{log de Rham cohomology} of $X$ (over $\cl O_C$) as
  $$R\Gamma_{\log\dR}(\fr X):=R\Gamma(\fr X, \Omega_{\fr X, \log}^{\bullet})\in D(\Mod_{\cl O_C}^{\cond}).$$
  
  \end{description}
  \end{notation} 

  \begin{rem}\label{bbcc}\
  \begin{enumerate}[(i)]
   \item\label{bbcc:1}  For any affine $\fr X \in\cl M_{\sss}$ with semistable coordinates there exists a finite extension $L/K$ and a $p$-adic formal scheme $\fr X '\to \Spf(\cl O_L)$ admitting an étale $\Spf(\cl O_L)$-morphism $\fr X'\to \Spf(R')$ with 
   $$R':=\cl O_L\{t_0, \ldots, t_r, t^{\pm 1}_{r+1}, \ldots, t_d^{\pm 1}\}/(t_0\cdots t_r-p^q)$$
   for some $0\le r\le d$, and $q\in \Qq_{>0}$, such that $\fr X=\fr X'\times_{\Spf(\cl O_L)}\Spf(\cl O_C)$: this follows from \cite[04D1, 00U9]{Thestack}. By \cite[Claims 1.6.1 and 1.6.2]{CK}, the $p$-adic formal scheme $\fr X'$ can be endowed with a fine log structure, whose base change to $\Spf(\cl O_C)$ gives the log structure on $\fr X$ we started with.
   \item\label{bbcc:2} For any $\fr X\in \cl M_{\sss, \qcqs}$ there exist a finite extension $L/K$ and a descent of $\fr X_{\cl O_C/p}$ to a qcqs, fine, log-smooth, log scheme over $\cl O_{L,1}^\times$ of Cartier type: covering $\fr X$ by a finite number of open affines with semistable coordinates, this follows from part (\ref{bbcc:1}) and fact that morphisms of Cartier type are stable under base change.
  \end{enumerate}
   \end{rem}
 
  The following Beilinson basis will be used to define the Hyodo--Kato cohomology for rigid-analytic varieties over $C$ starting from the semistable reduction case.
   
  \begin{prop}\label{propobv}
   The pair $(\cl M_{\sss}, (-)_{\eta})$ is a Beilinson basis for the site $\Rig_{C, \eh}$.
  \end{prop}
  \begin{proof}
   By Proposition \ref{baseh} (and Remark \ref{oyt}), it suffices to show that  $(\cl M_{\sss}, (-)_{\eta})$ is a Beilinson basis for $\RigSm_{C, \ett}$, i.e. the big étale site of smooth rigid-analytic varieties over $C$ . This follows from Temkin's alteration theorem \cite[Theorem 3.3.1]{Temk2}, as shown in \cite[Proposition 2.8]{CN}.
  \end{proof}

  \subsubsection{\normalfont{\textbf{Condensed $(\varphi, N)$-modules}}}
  
  Before defining the Hyodo--Kato cohomology for rigid-analytic varieties over $C$, we need to establish the following terminology.
  
  \begin{df} Let $\varphi: \breve F\to \breve F$ denote the automorphism induced by the $p$-th power Frobenius on the residue field.
   \begin{enumerate}[(i)]
    \item A \textit{condensed $\varphi$-module over $\breve F$} is a pair $(V, \varphi_V)$ with $V\in \Vect_{\breve F}^{\cond}$ and $\varphi_V:V\to V$ a $\varphi$-semilinear automorphism, called \textit{Frobenius}.  A morphism of condensed $\varphi$-modules over $\breve F$ is a morphism of condensed modules over $\breve F$, which is compatible with the Frobenius.
    \item  A \textit{condensed $(\varphi, N)$-module over $\breve F$} is a triple $(V, \varphi_V, N_V)$ with $(V, \varphi_V)$ a condensed $\varphi$-module over $\breve F$ and 
    $N_V:V\to V$ a $\breve F$-linear endomorphism, called \textit{monodromy operator}, such that $N_V\varphi_V=p\varphi_V N_V$ (by abuse of notation, we often denote $\varphi=\varphi_V$ and $N=N_V$). 
    A morphism of condensed $(\varphi, N)$-modules over $\breve F$ is a morphism of condensed modules over $\breve F$, which is compatible with the Frobenius and the monodromy operator.
   \end{enumerate}
   
   Note that the category of condensed $(\varphi, N)$-modules over $\breve F$ is an abelian category. We denote by $D_{(\varphi, N)}(\Vect_{\breve F}^{\cond})$ the corresponding derived $\infty$-category; we abbreviate $D_{(\varphi, N)}(\breve F)=D_{(\varphi, N)}(\Vect_{\breve F}^{\cond})$.
  \end{df}

  \begin{rem}\label{automf}
   For $\fr X\in \cl M_{\sss}$, we have that $R\Gamma_{\cris}(\fr X_{\cl O_C/p}^0/\cl O_{\breve F}^0)_{\Qq_p}$ lies in  $D_{(\varphi, N)}(\Vect_{\breve F}^{\cond})$, in fact,  by \cite[Proposition 2.24]{HK} the Frobenius is an automorphism on it.
  \end{rem}
  
  Then, we are ready to give the following definition, which is based on Lemma \ref{beihyp} and Proposition \ref{propobv}.
  
  \begin{df}[Hyodo--Kato cohomology]\label{defhk}
   We denote by $\cl F_{\HK}$ the hypersheaf on $\Rig_{C, \eh}$ with values in $D_{(\varphi, N)}(\Vect_{\breve F}^{\cond})$ associated to the presheaf
   \begin{equation}\label{prehk}
    (\cl M_{\sss})^{\op}\to D_{(\varphi, N)}(\Vect_{\breve F}^{\cond}): \fr X\mapsto R\Gamma_{\cris}(\fr X_{\cl O_C/p}^0/\cl O_{\breve F}^0)_{\Qq_p},
   \end{equation}
   For $X$ a rigid-analytic variety over $C$, we define the \textit{Hyodo--Kato cohomology} of $X$ as
   $$R\Gamma_{\HK}(X):=R\Gamma(X, \cl F_{\HK})\in D_{(\varphi, N)}(\Vect_{\breve F}^{\cond}).$$
  \end{df}
  \medskip

  The following result shows in particular that the Hyodo--Kato cohomology of $X$ is a refinement of the de Rham cohomology of $X$ (Definition \ref{dRsingg}).\medskip
  
  We refer the reader to \cite[\S A.6]{Bosco} and \S \ref{nuccomp} for a review of the notion of \textit{nuclearity}, introduced by Clausen--Scholze, used in the following statement and the rest of the paper.

  \begin{theorem}\label{mainHK} Let $X$ be a rigid-analytic variety over $C$.
  \begin{enumerate}[(i)]
   \item\label{mainHK:1}{\normalfont (Local-global compatibility)} Assume $X$ is the generic fiber of $\fr X\in \cl M_{\sss}$, then the natural map
   $$R\Gamma_{\cris}(\fr X_{\cl O_C/p}^0/\cl O_{\breve F}^0)_{\Qq_p}\to R\Gamma_{\HK}(X)$$
   is an isomorphism in $D_{(\varphi, N)}(\Vect_{\breve F}^{\cond})$.
   \item\label{mainHK:2} {\normalfont (Boundedness and nuclearity)} If $X$ is qcqs of dimension $d$, then $R\Gamma_{\HK}(X)$ is represented by a complex of nuclear (solid) $\breve F$-vector spaces, and it lies in $D^{\le 2d}(\Vect_{\breve F}^{\ssolid})$. Moreover, the monodromy operator $N$ on $R\Gamma_{\HK}(X)$ is nilpotent with nilpotency index bounded above by a function depending on $d$.
   \item\label{mainHK:3} {\normalfont (Hyodo--Kato isomorphism)} Assume $X$ is connected and paracompact, then we have a natural isomorphism in $D(\Vect_{\Qq_p}^{\ssolid})$
   $$\iota_{\HK}: R\Gamma_{\HK}(X)\dsolid_{\breve F} C\overset{\sim}{\to} R\Gamma_{\dR}(X).$$
  \end{enumerate}
  \end{theorem}
  \begin{proof}
  We start with some preliminary observations.
  First, we observe that,  for any $\fr X\in \cl M_{\sss, \qcqs}$, by \cite[(1.8.1)]{Beili}, we have a natural quasi-isomorphism
  \begin{equation}\label{todR0}
    R\Gamma_{\cris}(\fr X_{\cl O_C/p}/\cl O_C^\times)_{\Qq_p}\simeq R\Gamma_{\log\dR}(\fr X)_{\Qq_p}\simeq R\Gamma_{\dR}(\fr X_C)
  \end{equation}
  and then, by Theorem \ref{bcn}, which applies thanks to Remark \ref{bbcc}, we have a natural quasi-isomorphism
  \begin{equation}\label{todR}
   R\Gamma_{\cris}(\fr X_{\cl O_C/p}^0/\cl O_{\breve F}^0)_{\Qq_p}\dsolid_{\breve F}C\overset{\sim}{\to}  R\Gamma_{\dR}(\fr X_C).
  \end{equation}  
  Moreover, we claim that $R\Gamma_{\cris}(\fr X_{\cl O_C/p}^0/\cl O_{\breve F}^0)_{\Qq_p}$ is represented by a complex of $\breve F$-Banach spaces. For this, we note that, as $\breve F$ is discretely valued we can choose a basis of the $\breve F$-Banach space $C$, and then there exists a $\breve F$-Banach space $V$ and an isomorphism 
  \begin{equation}\label{funiso}
   C\cong \breve F\oplus V \;\text{ in }\; \Vect_{\breve F}^{\ssolid}.
  \end{equation}
  Now, the claim follows using the quasi-isomorphism (\ref{todR}) combined with the isomorphism (\ref{funiso}), observing that, as $\fr X_C$ is qcqs, $R\Gamma_{\dR}(\fr X_C)$ is represented by a complex $C$-Banach spaces (and hence $\breve F$-Banach spaces), and a direct summand of a Banach space is a Banach space.

  For part \listref{mainHK:1}, it suffices to show that given $\fr X\in \cl M_{\sss, \qcqs}$ with generic fiber $X$, then, for any simplicial object $\fr U_\bullet$ of $\cl M_{\sss, \qcqs}$ such that $\fr U_{\bullet, \eta}\to X$ is a $\eh$-hypercover, the natural map 
  \begin{equation}\label{natty}
   R\Gamma_{\cris}(\fr X_{\cl O_C/p}^0/\cl O_{\breve F}^0)_{\Qq_p}\to \lim_{[n]\in \Delta}R\Gamma_{\cris}(\fr U_{n, \cl O_C/p}^0/\cl O_{\breve F}^0)_{\Qq_p}
  \end{equation}
  is a quasi-isomorphism. First, we note that the map (\ref{natty}) is compatible with the Frobenius and the monodromy operator. Next, we will use an idea from the proof of \cite[Proposition 3.5]{CN4}. By (\ref{todR}) we have the following commutative diagram
  \begin{center}
  \begin{tikzcd}
  R\Gamma_{\cris}(\fr X_{\cl O_C/p}^0/\cl O_{\breve F}^0)_{\Qq_p}\dsolid_{\breve F}C \arrow[d, "\wr"]\arrow[r]&  \lim_{[n]\in \Delta}(R\Gamma_{\cris}(\fr U_{n, \cl O_C/p}^0/\cl O_{\breve F}^0)_{\Qq_p}\dsolid_{\breve F} C) \arrow[d, "\wr"] \\
   R\Gamma_{\dR}(\fr X_C) \arrow[r] &  \lim_{[n]\in \Delta}R\Gamma_{\dR}(\fr U_{n, C}).
  \end{tikzcd}
  \end{center}
  The bottom horizontal arrow is a quasi-isomorphism, as $R\Gamma_{\dR}(-)$ satisfies $\eh$-hyperdescent, hence the top horizontal arrow is a quasi-isomorphism too. Moreover, setting $M_n:=R\Gamma_{\cris}(\fr U_{n, \cl O_C/p}^0/\cl O_{\breve F}^0)_{\Qq_p}$ we have
  $$\lim_{[n]\in \Delta}(M_n\dsolid_{\breve F} C)=(\lim_{[n]\in \Delta}M_n)\dsolid_{\breve F} C$$
  as it follows from \cite[Corollary A.67(ii)]{Bosco}, recalling that each $M_n$ is represented by a complex of $\breve F$-Banach spaces (and hence nuclear $\breve F$-vector spaces by \cite[Corollary A.61]{Bosco}), and using that\footnote{Here, all the limits are derived.}
  \begin{equation}\label{finoo}
   \lim_{[n]\in \Delta}M_n=\varprojlim_{m\in \Nn}\lim_{[n]\in \Delta_{\le m}}M_n.
  \end{equation}
  Then, considering the fibers of the horizontal arrows in following commutative diagram
   \begin{center}
  \begin{tikzcd}
  R\Gamma_{\cris}(\fr X_{\cl O_C/p}^0/\cl O_{\breve F}^0)_{\Qq_p} \arrow[d]\arrow[r]&   \lim_{[n]\in \Delta}R\Gamma_{\cris}(\fr U_{n, \cl O_C/p}^0/\cl O_{\breve F}^0)_{\Qq_p}\arrow[d] \\
  R\Gamma_{\cris}(\fr X_{\cl O_C/p}^0/\cl O_{\breve F}^0)_{\Qq_p}\dsolid_{\breve F}C  \arrow[r, "\sim"] & \lim_{[n]\in \Delta}(R\Gamma_{\cris}(\fr U_{n, \cl O_C/p}^0/\cl O_{\breve F}^0)_{\Qq_p}\dsolid_{\breve F} C)
  \end{tikzcd}
  \end{center}
  in order to show that the top horizontal arrow is a quasi-isomorphism, it suffices to prove that, for any $M\in D(\Vect_{\breve F}^{\ssolid})$, $$M\dsolid_{\breve F} C \text{ acyclic } \implies  M \text{ acyclic}.$$ This immediately follows using the isomorphism (\ref{funiso}).

  For part \listref{mainHK:2}, to show that $R\Gamma_{\HK}(X)$ lies in $D^{\le 2d}(\Vect_{\breve F}^{\ssolid})$, using the quasi-isomorphism (\ref{todR}) combined with the isomorphism (\ref{funiso}) and $\eh$-hyperdescent, we can reduce to the analogous statement for the de Rham cohomology, which follows from Corollary \ref{dRbound}. To show that $R\Gamma_{\HK}(X)$ is represented by a complex of nuclear $\breve F$-vector spaces, taking a simplicial object $\fr U_\bullet$ of $\cl M_{\sss, \qcqs}$ such that $\fr U_{\bullet, \eta}\to X$ is a $\eh$-hypercover, by $\eh$-hyperdescent and part \listref{mainHK:1}, we have
  \begin{equation}\label{disceso}
   R\Gamma_{\HK}(X)=\lim_{[n]\in \Delta}R\Gamma_{\cris}(\fr U_{n, \cl O_C/p}^0/\cl O_{\breve F}^0)_{\Qq_p}
  \end{equation}
  and then, by \cite[Corollary A.61, Theorem A.43(i)]{Bosco}, we can reduce to the fact that each complex $R\Gamma_{\cris}(\fr U_{n, \cl O_C/p}^0/\cl O_{\breve F}^0)_{\Qq_p}$ is represented by a complex of $\breve F$-Banach spaces, which was shown above. The last statement of part \listref{mainHK:2} follows from Lemma \ref{nilpo}.
 
  For part \listref{mainHK:3}, we first assume $X$ qcqs. In this case, using (\ref{disceso}), the statement follows (\ref{todR}) and \cite[Corollary A.67(ii)]{Bosco}, which applies thanks to part \listref{mainHK:2}. For a general $X$ connected and paracompact, choosing a quasi-compact admissible covering $\{U_n\}_{n\in \Nn}$ of $X$ such that $U_n\subseteq U_{n+1}$, the statement follows from the previous case, using again \cite[Corollary A.67(ii)]{Bosco} and part \listref{mainHK:2}.
  \end{proof}

  As a consequence of the Hyodo--Kato isomorphism, we have the following result.

  \begin{cor}\label{samedim}
   Let $X$ be a connected, paracompact, rigid-analytic variety over $C$. Then, the Hyodo--Kato complex $R\Gamma_{\HK}(X)$ and the de Rham complex $R\Gamma_{\dR}(X)$ have the same cohomological dimension.
  \end{cor}
  \begin{proof}
   By Theorem \ref{mainHK}\listref{mainHK:3} and the flatness of $C$ for the solid tensor product $\solid_{\breve F}$ (\cite[Corollary A.65]{Bosco}), for any $i\ge 0$, we have an isomorphism
   $$H^i_{\HK}(X)\solid_{\breve F} C\cong H^i_{\dR}(X).$$
   Therefore, if $H^i_{\HK}(X)$ vanishes then $H^i_{\dR}(X)$ vanishes as well, and the converse statement follows using the isomorphism (\ref{funiso}).
  \end{proof}

  \subsection{Finiteness in the overconvergent case}
  In this subsection, we extend the Hyodo--Kato morphism to dagger varieties over $C$. As we will see, this will follow easily from the results of the previous subsection, using that the solid tensor product commutes with colimits.
  Moreover, we will prove a finiteness result for the Hyodo--Kato cohomology of qcqs dagger varieties over $C$, generalizing already known results to the singular case.
 
 \subsubsection{\normalfont{\textbf{Hyodo--Kato cohomology of dagger varieties over $C$}}}
  We begin with a general construction that will allow us to canonically define a cohomology theory on $\Rig_{L, \eh}^\dagger$ starting from a cohomology theory defined on $\Rig_{L, \eh}$. Then, we will specialize this construction to the Hyodo--Kato and the de Rham cohomology theories.
  
  \begin{notation}
    In the following, we keep notation and conventions from \ref{othernot}. In particular, we denote by  $L$ a characteristic $0$ complete valued field with a non-archimedean valuation of rank $1$ and residue characteristic $p$.
  \end{notation}
  
  \begin{construction}\label{consthyp}
   Let $\cl D$ be a presentable $\infty$-category. The continuous functor 
   \begin{equation}\label{contlim}
    l:\Rig_{L, \eh}^\dagger \to \Rig_{L, \eh}: X\mapsto \widehat X
   \end{equation}
   given by sending a dagger variety $X$ to its limit $\widehat X$, induces an adjunction
   $$l_*:\Shv^{\hyp}(\Rig_{L, \eh}^\dagger, \cl D)\rightleftarrows \Shv^{\hyp}(\Rig_{L, \eh}, \cl D):l^{* \hyp}$$
   where $l^{*\hyp}$ is given by the composite of the pullback functor $l^*:\Shv(\Rig_{L, \eh}, \cl D)\to \Shv(\Rig_{L, \eh}^\dagger, \cl D)$ and the hypercompletion functor (\ref{hypercompl}).
   For $\cl F\in \Shv(\Rig_{L, \eh}, \cl D)$, we denote
   $$\cl F^\dagger:=l^{* \hyp}\cl F\in \Shv^{\hyp}(\Rig_{L, \eh}^\dagger, \cl D).$$
   
  \end{construction}

  Now, using Construction \ref{consthyp} in the case $\cl D=D(A)=D(\Mod_A^{\ssolid})$ (see Remark \ref{indeedpre}), we can give the following definition.
  
  \begin{df}[de Rham and Hyodo--Kato cohomology of dagger varieties]\label{dhkdef}\
  \begin{enumerate}[(i)]
   \item  Let $X$ be a dagger variety over $L$. Denote by $$\cl F_{\dR}\in \Shv^{\hyp}(\Rig_{C, \eh}, D(L))$$ the hypersheaf given by $R\Gamma_{\dR}(-)$.
   We define the \textit{de Rham cohomology} of $X$ as
   $$R\Gamma_{\dR}(X):=R\Gamma(X, \cl F_{\dR}^\dagger)\in D(L).$$
   \item Let $X$ be a dagger variety over $C$. Consider the hypersheaf $$\cl F_{\HK}\in \Shv^{\hyp}(\Rig_{C, \eh}, D_{(\varphi, N)}(\breve F))$$ introduced in Definition \ref{defhk}.
   We define the \textit{Hyodo--Kato cohomology} of $X$ as
   $$R\Gamma_{\HK}(X):=R\Gamma(X, \cl F_{\HK}^\dagger)\in D_{(\varphi, N)}(\breve F).$$
  \end{enumerate}
  
  \end{df}

  \subsubsection{\normalfont{\textbf{Presentation of a dagger structure}}}\label{present} 
  In order to construct the Hyodo--Kato morphism for dagger varieties over $C$, we will rely on its analogue for rigid-analytic varieties over $C$, that is Theorem \ref{mainHK}.  For this, we will need to express more explicitly the Hyodo--Kato/de Rham cohomology of a smooth dagger affinoid over $C$ in terms of the respective cohomology of smooth affinoid rigid spaces over $C$. This is our next goal. \medskip
 
   We recall from \ref{othernot} that given a dagger variety $X=(\widehat X, \cl O^{\dagger})$ over $L$ with underlying rigid-analytic variety $\widehat X$ we say that $X$ is a dagger structure on $\widehat X$.
   
   We have the following important example of dagger structure.
   
  \begin{rem}[Dagger structure on smooth affinoid rigid spaces]\label{presmooth}
   We note that any smooth affinoid rigid space $\widehat X=\Spa(R, R^\circ)$ over $L$ has a dagger structure. In fact, by \cite[Theorem 7 and Remark 2]{Elkik}, there exist $f_1, \ldots, f_m$ elements of the Washnitzer algebra  $L\langle \underline{T}\rangle^\dagger$ such that $R\cong L\langle \underline{T}\rangle/(f_1, \ldots, f_m)$.\footnote{Here, we write $\underline{T}$ for $T_1, \ldots, T_n$ where $n$ is the dimension of $\widehat X$.} In particular, the dagger variety associated to the dagger algebra $L\langle \underline{T}\rangle^\dagger/(f_1, \ldots, f_m)$ defines a dagger structure on $\widehat X$. 
  \end{rem}
 
  Next, we recall the following convenient definition.
  
  \begin{df}[{\cite[Definition A.19]{Vezzani1}}]
  Let $X$ be an affinoid rigid space over $L$. A \textit{presentation of a dagger structure} on $X$ is a pro-(affinoid rigid space over $L$) $\varprojlim_{h\in \Nn} X_h$ with $X$ and $X_h$ rational subspaces of $X_1$, such that $X\Subset X_{h+1}\Subset X_h$,\footnote{For $Y\subset Z$ an open immersion of rigid-analytic varieties over $L$, we write $Y\Subset Z$ if the inclusion map of $Y$ into $Z$ factors over the adic compactification of $Y$ over $L$.} and this system is coinitial among rational subspaces containing $X$. 
  
  A morphism of presentations of a dagger structure on an affinoid rigid space over $L$ is a morphism of pro-objects.
  \end{df}

  The next lemma relates affinoid dagger spaces to presentations of a dagger structure on an affinoid rigid space.
  
  \begin{lemma}\label{dagglemma}
  Let $\widehat X$ be an affinoid rigid space over $L$, and let $\varprojlim_h X_h$ be a presentation of a dagger structure on $\widehat X$. We denote by $X^\dagger$ the dagger affinoid over $L$ associated to the dagger algebra $R=\varinjlim_h \cl O(X_h)$.
   \begin{enumerate}[(i)]
    \item\label{dagglemma:1} The functor 
    $$\varprojlim\nolimits_h X_h\mapsto X^\dagger$$
   from the category of presentations of a dagger structure on an affinoid rigid space over $L$ to the category of affinoid dagger spaces over $L$ is an equivalence.
    
    \item\label{dagglemma:2} Let $(\varprojlim\nolimits_h X_h)_{\ett}$ denote the (small) étale site of $\varprojlim\nolimits_h X_h$, \cite[Definition A.24]{Vezzani1}. We have natural morphisms of sites $$\widehat X_{\ett}\to (X^\dagger)_{\ett} \to (\varprojlim\nolimits_h X_h)_{\ett}$$  which induce an equivalence on the associated topoi.
   \end{enumerate}
  \end{lemma}
  \begin{proof}
   Part \listref{dagglemma:1} is \cite[Proposition A.22(2)]{Vezzani1}, and part \listref{dagglemma:2} is \cite[Corollary A.28]{Vezzani1}.
  \end{proof}

  \begin{rem}\label{lanuit}
  Entering the proof of \cite[Proposition A.22(2)]{Vezzani1}, we see that given an affinoid dagger space $(\widehat X, \cl O^\dagger)$ over $L$, associated to a dagger algebra $L\langle \underline{T}\rangle^\dagger/(f_1, \ldots, f_m)$, then the corresponding presentation of a dagger structure $\varprojlim\nolimits_h X_h$ on $\widehat X$ can be defined as follows: since $L\langle \underline{T}\rangle^\dagger$ is noetherian, \cite[\S 1.4]{GK1}, there exists an integer $H$ sufficiently big such that $f_1, \ldots, f_m\in L\langle \pi^{1/H}\underline{T}\rangle$, where $\pi$ is a pseudo-uniformizer of $\cl O_L$; then, we set $X_h:=\Spa(R_h, R_h^\circ)$ with $$R_h:= L\langle \pi^{1/(h+H)}\underline{T}\rangle/(f_1, \ldots, f_m).$$
  \end{rem}

   Then, the next result follows formally from Lemma \ref{dagglemma} by étale hyperdescent (cf. \cite[Lemma 3.13]{CN}).
   
   \begin{lemma}\label{immHK}
   Fix notation as in \ref{othernot} and Construction \ref{consthyp}. Let $\cl F\in \Shv^{\hyp}(\Rig_{L, \eh}, \cl D)$ that is the pullback of an hypersheaf in 
   $\Shv^{\hyp}(\RigSm_{L,\ett}, \cl D)$. Let $X$ be a smooth dagger affinoid over $L$ with corresponding presentation $\varprojlim\nolimits_h X_h$. Then, we have
   $$R\Gamma(X, \cl F^\dagger)=\colim_{h\in \Nn}R\Gamma(X_h, \cl F).$$ 
  \end{lemma}
    
  The assumptions of the previous lemma are designed to be satisfied by the hypersheaves defining the de Rham and the Hyodo--Kato cohomology:
 
  \begin{rem}\label{remmm}
   We note that Lemma \ref{immHK} applies to $\cl F=\cl F_{\dR}$, thanks to Proposition \ref{ehdescent}, and it applies to $\cl F=\cl F_{\HK}$ thanks to Theorem \ref{mainHK}\listref{mainHK:1}.
  \end{rem}

  Next, we recall that the category of partially proper dagger varieties is equivalent to the category of partially proper rigid-analytic varieties, via the functor (\ref{contlim}), \cite[Theorem 2.27]{GK1}. In the situation of Lemma \ref{immHK}, such equivalence preserves cohomology:
  
  \begin{prop}\label{parpar}
   Fix notation as in \ref{othernot} and Construction \ref{consthyp}. Let $\cl F\in \Shv^{\hyp}(\Rig_{L, \eh}, \cl D)$ that is the pullback of an hypersheaf in 
   $\Shv^{\hyp}(\RigSm_{L,\ett}, \cl D)$. Let $X$ be a partially proper dagger variety over $L$. Then, there exists a natural isomorphism
   $$R\Gamma(X, \cl F^\dagger)\overset{\sim}{\to}R\Gamma(\widehat X, \cl F).$$
  \end{prop}
  \begin{proof}
   Recalling that any partially proper dagger variety admits an admissible covering by Stein spaces (see the proof of \cite[Theorem 2.26]{GK1}), we may assume that $X$ is a Stein space. Then, let $\{U_n\}_{n\in \Nn}$ be a Stein covering of $X$. Writing $R\Gamma(X, \cl F^\dagger)=R\varprojlim_{n\in \Nn}R\Gamma(U_n, \cl F^\dagger)$, and similarly $R\Gamma(\widehat X, \cl F)=R\varprojlim_{n\in \Nn}R\Gamma(\widehat U_n, \cl F)$, it suffices to show that, for a fixed $n\in \Nn$, the natural map $R\Gamma(U_{n+1}, \cl F^\dagger)\to R\Gamma(U_{n}, \cl F^\dagger)$ factors through $R\Gamma(\widehat U_{n+1}, \cl F)$. Renaming $V:=U_n$ and $W:=U_{n+1}$, by Proposition \ref{baseh} (and Remark \ref{oyt}), we can choose an $\eh$-hypercover $W_\bullet\to W$ with each $W_m$ smooth dagger affinoid over $L$; via pullback along the open immersion $V\to W$, we obtain an $\eh$-hypercover $V_\bullet\to V$ with each $V_m$ smooth dagger affinoid over $L$. Then, we may reduce to the case $V$ and $W$ are smooth over $L$, which follows from Lemma \ref{immHK}.
  \end{proof}

   \subsubsection{\normalfont{\textbf{Semistable weak formal schemes}}} In order to study the Hyodo--Kato cohomology of dagger varieties over $C$ (Definition \ref{dhkdef}), we will define a convenient Beilinson basis for the site $\Rig_{C, \eh}^\dagger$. In addition to Notation \ref{notss}, we introduce the following notation. We refer the reader to \cite{Meredith} for the basics on the theory of \textit{weak formal schemes}, and to \cite{Lan-Mur} for an analogue of Raynaud's theorem relating the categories of weak formal schemes and dagger varieties.

  \begin{notation}
    We denote by $\cl M_{\sss}^\dagger$ the category of weak formal schemes over $\Spf(\cl O_C)$ having in the Zariski topology a covering by open affines $\fr U$ with \textit{semistable coordinates}, i.e. admitting an $\Spf(\cl O_C)$-morphism $\fr U\to \Spf(R^{\square \dagger})$ with 
   $$R^{\square \dagger}:=\cl O_C[t_0, \ldots, t_r, t^{\pm 1}_{r+1}, \ldots, t_d^{\pm 1}]^\dagger/(t_0\cdots t_r-p^q)$$
   for some $0\le r\le d$, and $q\in \Qq_{>0}$ (that may depend on $\fr U$). We denote by $\cl M_{\sss, \qcqs}^\dagger$ the subcategory of $\cl M_{\sss}^\dagger$ consisting of the qcqs formal schemes.
   We write $$(-)_{\eta}: \cl M_{\sss}^\dagger\to \Rig_C^\dagger$$ for the generic fiber functor.
  \end{notation}

  \begin{prop}\label{propobv2}
   The pair $(\cl M_{\sss}^\dagger, (-)_{\eta})$ is a Beilinson basis for the site $\Rig_{C, \eh}^\dagger$.
  \end{prop}
  \begin{proof}
  As in the proof of Proposition \ref{propobv}, the statement follows from Proposition \ref{baseh} (and Remark \ref{oyt}) combined with \cite[Proposition 2.13]{CN}.
  \end{proof}

  The following result is an overconvergent version of Theorem \ref{mainHK}.
 
   \begin{theorem}\label{mainHKover} Let $X$ be a dagger variety over $C$.
  \begin{enumerate}[(i)]
   \item\label{mainHKover:1}{\normalfont (Local description)} Assume $X$ is the generic fiber of $\fr X\in \cl M_{\sss}^{\dagger}$, then there is a natural quasi-isomorphism
   $$R\Gamma_{\HK}(X)\simeq R\Gamma_{\rig}(\fr X_{\bar k}/\cl O_{\breve F}^0)$$
   compatible with Frobenius $\varphi$ and monodromy $N$. Here, the right-hand side denotes the (rational) log-rigid cohomology of $\fr X_{\bar k}$ over $\cl O_{\breve F}^0$, \cite[\S 1]{GK-HK}, \cite[\S 3.1.2]{CDN1}.
   \item\label{mainHKover:3} {\normalfont (Hyodo--Kato isomorphism)} Assume $X$ is connected and paracompact, then we have a natural isomorphism in $D(\Vect_{\Qq_p}^{\ssolid})$
   $$\iota_{\HK}: R\Gamma_{\HK}(X)\dsolid_{\breve F} C\overset{\sim}{\to} R\Gamma_{\dR}(X).$$
  \end{enumerate}
  \end{theorem}
  \begin{proof}
  Part \listref{mainHKover:1} follows from \cite[\S 4.2.1, (iv)]{CN4}.\footnote{Note that Theorem \ref{mainHK}\listref{mainHK:1} implies that, for $X$ a smooth rigid-analytic/dagger variety over $C$, the Hyodo--Kato cohomology $R\Gamma_{\HK}(X)$ agrees with the one defined in \cite{CN4}  considered in $D(\CondAb)$.} Part \listref{mainHKover:3} for $X$ smooth affinoid follows from Theorem \ref{mainHK}\listref{mainHK:2} and Lemma \ref{immHK} (together with Remark \ref{remmm}), using that the tensor product $\dsolid_{\breve F}$ commutes with filtered colimits. From Lemma \ref{immHK} we also deduce that, for $X$ smooth affinoid, $R\Gamma_{\HK}(X)$ is represented by a complex of nuclear $\breve F$-vector spaces (recall that the category of nuclear $\breve F$-vector spaces is closed under colimits). Therefore, the same argument used in the proof of Theorem \ref{mainHK}\listref{mainHK:3} shows part \listref{mainHKover:3} in general.
  \end{proof}


  \subsubsection{\normalfont{\textbf{Finiteness}}}
  Now, we state the promised finiteness result for the Hyodo--Kato cohomology groups of a qcqs dagger variety over $C$, and we give a bound on the slopes of such cohomology groups regarded as $\varphi$-modules.

   \begin{theorem}\label{slopp}
  Let $X$ be a qcqs dagger variety over $C$. Let $i\ge 0$.
  \begin{enumerate}[(i)]
   \item\label{slopp:1} The condensed cohomology group $H_{\HK}^i(X)$ (resp. $H_{\dR}^i(X)$) is a finite-dimensional condensed vector space over $\breve F$ (resp. over $C$).
   \item\label{slopp:2}  The vector bundle on $\FF$ associated to the finite  $\varphi$-module $H_{\HK}^i(X)$ over $\breve F$ has Harder--Narasimhan slopes $\ge -i$.
  \end{enumerate}
  \end{theorem}
  \begin{proof}
   In the case when $X$ is the generic fiber of $\fr X\in \cl M_{\sss, \qcqs}^{\dagger}$ by Theorem \ref{mainHKover}\listref{mainHKover:1} part \listref{slopp:1} follows from a result of Grosse-Kl\"onne, \cite[Theorem 5.3]{GK-HK} (and base change), and part \listref{slopp:2} follows from \cite[Théorème 3.1.2]{CLS}. In the general case, we take a simplicial object $\fr U_\bullet$ of $\cl M_{\sss, \qcqs}^\dagger$ such that $\fr U_{\bullet, \eta}\to X$ is a $\eh$-hypercover, and we consider the spectral sequence
   \begin{equation}\label{spectralHK}
    E_1^{j, i-j}=H_{\HK}^{i-j}(\fr U_{j, \eta})\implies H_{\HK}^{i}(X).
   \end{equation}
   Then, part \listref{slopp:1} for the Hyodo--Kato cohomology follows immediately from the previous case, the spectral sequence (\ref{spectralHK}), and \cite[Lemma A.33]{Bosco}. Similarly, part \listref{slopp:1} for the de Rham cohomology follows from the previous case and an analogous spectral sequence for the de Rham cohomology.\footnote{Alternatively, part \listref{slopp:1} for the de Rham cohomology follows from part \listref{slopp:1} for the de Hyodo--Kato cohomology and the Hyodo--Kato isomorphism, Theorem \ref{mainHKover}\listref{mainHKover:3}.} \smallskip
   
   For part \listref{slopp:2}, applying to (\ref{spectralHK}) the exact functor $\cl E(-)$ sending a finite $\varphi$-module over $\breve F$ to the associated vector bundle on $\FF$, and then twisting by $\cl O(i)$, we deduce that the vector bundle $\cl E(H_{\HK}^i(X))\otimes \cl O(i)$ has non-negative Harder--Narasimhan slopes: in fact, by the previous case, for all $j$, the vector bundle $\cl E(H_{\HK}^{i-j}(\fr U_{j, \eta}))\otimes \cl O(i)$ has non-negative Harder--Narasimhan slopes, and then the claim follows from the classification of vector bundles on $\FF$.
  \end{proof}

  \section{\textbf{$B$-cohomology}}\label{Bleit}
   \sectionmark{}

  This section is devoted to the proof of the following main result, which compares the $B$-cohomology with the Hyodo--Kato cohomology.

   \begin{theorem}\label{B=HK}
   Let $X$ be a connected, paracompact, rigid-analytic variety defined over $C$. Then, we have a natural isomorphism in $D(\Mod_{B}^{\ssolid})$
   \begin{equation}\label{decaB}
    R\Gamma_{B}(X)\simeq (R\Gamma_{\HK}(X)\dsolid_{\breve F}B_{\log})^{N=0}
   \end{equation}
   compatible with the action of Frobenius $\varphi$. If $X$ is the base change to $C$ of a rigid-analytic variety defined over $K$, then (\ref{decaB}) is $\mathscr{G}_K$-equivariant.
  \end{theorem}
  
  We will first prove Theorem \ref{B=HK} in the case when $X$ has semistable reduction. This will be done in two main steps: we first compare, in \S \ref{ae1}, the $B$-cohomology with the log-crystalline cohomology over $A_{\cris}$, and then, in \S \ref{ae2}, we relate the latter with the Hyodo--Kato cohomology.
  
  \begin{convnot}
   In the following, we keep the notation and conventions introduced in \ref{logcriss} and \ref{notss}.
  \end{convnot}

  \subsection{The comparison with the log-crystalline cohomology over $A_{\cris}$}\label{ae1}
  
  We begin by comparing the $B$-cohomology  with the log-crystalline cohomology over $A_{\cris}$.
  
  \begin{theorem}\label{firstep}
   Let $\fr X$ be a qcqs semistable $p$-adic formal scheme over $\Spf(\cl O_C)$ and let $I\subset[1/(p-1), \infty)$ be a compact interval with rational endpoints. Then, there is a natural isomorphism in $D(\Vect_{\Qq_p}^{\ssolid})$
   \begin{equation}\label{B-cris}
    R\Gamma_{B_{I}}(\fr X_C)\simeq R\Gamma_{\cris}(\mathfrak X_{\cl O_C/p}/A_{\cris}^\times)\dsolid_{A_{\cris}}B_{I}
   \end{equation}
   compatible with the action of Frobenius $\varphi$.
  \end{theorem}
  
  In the first instance, we prove a local version of Theorem \ref{firstep}, and then we globalize the result. Therefore, we begin by defining the local setting in which we will work.
  
  \begin{notation}\label{notaz}
   Let $\mathfrak X=\Spf(R)$ be a connected affine $p$-adic formal scheme over $\Spf(\cl O_C)$ admitting an étale $\Spf(\cl O_C)$-morphism $\fr X\to \Spf(R^\square)$ with 
   $$R^{\square}:=\cl O_C\{t_0, \ldots, t_r, t^{\pm 1}_{r+1}, \ldots, t_d^{\pm 1}\}/(t_0\cdots t_r-p^q)$$
   for some $0\le r\le d$, and $q\in \Qq_{>0}$. 
   
   We denote by $R_{\infty}^\square$ the perfectoid $R^\square$-algebra defined by $R_{\infty}^\square:=(\varinjlim_m R_m^{\square})^\wedge_p$ with
   $$R^{\square}_m:=\cl O_C\{t_0^{1/p^m}, \ldots, t_r^{1/p^m}, t^{\pm 1/p^m}_{r+1}, \ldots, t_d^{\pm 1/p^m}\}/(t_0\cdots t_r-p^{q/p^m})$$
   and we put $\fr X_{C, \infty}^{\square}:=\Spa(R_{\infty}^{\square}[1/p], R_{\infty}^{\square})$. We set $$R_{\infty}:=(R\otimes_{R^{\square}} R_{\infty}^\square)^\wedge_p$$ and we note that (see also \cite[\S 3.2]{CK})
   \begin{equation}\label{ccover}
    \fr X_{C, \infty}:=\Spa(R_{\infty}[1/p], R_{\infty})\to \fr X_C
   \end{equation}
   is an affinoid perfectoid pro-étale cover of $\fr X_C$ with Galois group $$\Gamma:=\Zz_p(1)^d\cong \Zz_p^d$$ where the latter isomorphism is given by the choice of a compatible system of $p$-th power roots of unity in $\cl O_C$ (see \ref{condp}). We denote by $\gamma_1, \ldots, \gamma_d$ the generators of $\Gamma$ defined by
   $$\gamma_i:=(\varepsilon^{-1}, 1, \ldots, 1, \varepsilon, 1, \ldots, 1)\;\; \text{ for } i=1, \ldots, r$$ 
   $$\gamma_i:=(1, \ldots, 1, \varepsilon, 1, \ldots, 1)\;\;\; \text{ for } i=r+1, \ldots, d$$
   where $\varepsilon$ sits on the $i$-th entry.
   
  \end{notation}

  \subsubsection{\normalfont{\textbf{The condensed ring $\Bb_I(R_{\infty})$}}}
  
   In the setting of Notation \ref{notaz}, given $\Mm$ any pro-étale period sheaf of \S \ref{petsheaves}, we put 
   $$\Mm(R_{\infty}^{\square}):=\Mm(\fr X_{C, \infty}^{\square})\;\;\;\;\;\;\;\;\;\;\Mm(R_{\infty}):=\Mm(\fr X_{C, \infty})$$
   which we regard as condensed rings.
   
  \begin{rem}\label{form}
    We recall from \cite[\S 3.14]{CK} that we have the following decomposition of $\Aa_{\inf}(R_{\infty}^\square)$ 
    \begin{equation}\label{squareint-nint}
    \Aa_{\inf}(R_{\infty}^\square)\cong A_{\inf}(R^\square)\oplus \Aa_{\inf}(R_{\infty}^\square)^{\nonint}
   \end{equation}
   where $A_{\inf}(R^\square)$ denotes the ``integral'' part, and $\Aa_{\inf}(R_{\infty}^\square)^{\nonint}$  the ``nonintegral part''. We have
   \begin{equation}\label{tofixx}
    A_{\inf}(R^\square)\cong A_{\inf}\{X_0, \ldots, X_r, X_{r+1}^{\pm 1}, \ldots, X_d^{\pm 1}\}/(X_0\cdots X_r-[p^{\flat}]^q)
   \end{equation}
  where $X_i:=[t_i^\flat]$, and the convergence is $(p, \mu)$-adic.  Such decomposition lifts to $\Aa_{\inf}(R_{\infty})$ as follows
   \begin{equation}\label{int-nint} 
    \Aa_{\inf}(R_{\infty})\cong A_{\inf}(R)\oplus \Aa_{\inf}(R_{\infty})^{\nonint}
   \end{equation}
  where $A_{\inf}(R)$ is the unique lift of the étale $(R^\square/p)$-algebra $R/p$, along $\theta: A_{\inf}(R^\square)\twoheadrightarrow R^\square$, to a $(p, \mu)$-adically complete, formally étale $A_{\inf}(R^\square)$-algebra.
  \end{rem}
  
  \begin{rem}\label{form2}
   Given a compact interval $I\subset (0, \infty)$ with rational endpoints, by \cite[Proposition 4.7(ii)]{Bosco}, we have
   \begin{equation}\label{Acompl}
    \Aa_I(R_{\infty})\cong \Aa_{\inf}(R_{\infty})\widehat \otimes_{A_{\inf}}A_I
   \end{equation}
    where the completion $\widehat \otimes_{A_{\inf}}$ is $p$-adic. Then, one has similar decompositions as (\ref{int-nint}) replacing $\Aa_{\inf}$ with $\Aa_I$, resp. $\Bb_I$, and $A_{\inf}(R)$ with $A_I(R):=A_{\inf}(R)\widehat \otimes_{A_{\inf}}A_I$, resp. $B_I(R):=A_I(R)[1/p]$ (where the completion $\widehat \otimes_{A_{\inf}}$ is $p$-adic).
  \end{rem}

  \begin{rem}\label{remess}
   Let $I\subset (0, \infty)$ be a compact interval with rational endpoints. We claim that we have a natural isomorphism
   \begin{equation}\label{azzz}
     \Aa_I(R_{\infty})\cong \Aa_{\inf}(R_{\infty})\solid_{A_{\inf}}A_I
   \end{equation}
    In particular, inverting $p$, using that the solid tensor product commutes with filtered colimits, we have an isomorphism
   $$\Bb_I(R_{\infty})\cong \Aa_{\inf}(R_{\infty})\solid_{A_{\inf}}B_I.$$

   To show (\ref{azzz}), up to twisting by the Frobenius, we can assume that $I\subset[1/(p-1), \infty)$. Now, we use the isomorphism (\ref{Acompl}), and then we apply Proposition \ref{solid-vs-padic} taking $M=\Aa_{\inf}(R_{\infty})$ and $N=A_{\inf, I}$ (see \ref{condp} for the notation), regarded as objects of $\Mod_{A_{\inf}}^{\ssolid}$, thus obtaining that 
   \begin{equation}\label{azzz2}
    \Aa_{\inf}(R_{\infty})\dsolid_{A_{\inf}}(A_{\inf, I})^\wedge_p\cong (\Aa_{\inf}(R_{\infty})\otimes_{A_{\inf}}A_{\inf, I})^\wedge_p
   \end{equation}
   where $(-)^\wedge_p$ denotes the derived $p$-adic completion. Since $A_{\inf, I}$ is $p$-torsion-free, thanks to Lemma \ref{condtors} the derived $p$-adic completion $(A_{\inf, I})^\wedge$ identifies with $A_I$. Then, it remains to show that the derived $p$-adic completion appearing on the right-hand side of (\ref{azzz2}) is underived: by \cite[Lemma 12.2]{BMS1} and Remark \ref{st-log}, we have  $\mu^{p-1}/p\in A_{\cris}\subset A_I$, and therefore, for any integer $n\ge 1$, we have that $A_{\inf, I}/p^n=A_I/p^n\cong A_I/(p^n, \mu^{n'})$ for a large enough integer $n'$;\footnote{In fact, one can take $n':=(p-1)n$.} now, it suffices to observe that, by \cite[Lemma 3.13]{CK}, $(p^n, \mu^{n'})$ is an $\Aa_{\inf}(R_{\infty})$-regular sequence and $\Aa_{\inf}(R_{\infty})/(p^n, \mu^{n'})$ is flat over $A_{\inf}/(p^n, \mu^{n'})$,\footnote{In fact, \textit{loc. cit.} translates to the condensed setting, for the ideal sheaf $(p^n, \mu^{n'})$ in the condensed ring $\Aa_{\inf}(R_{\infty})$, observing that both $\Aa_{\inf}(R_{\infty})/(p^n, \mu^{n'})$ and $A_{\inf}/(p^n, \mu^{n'})$ are discrete.} hence
   $$\Aa_{\inf}(R_{\infty})\otimes^{\LL}_{A_{\inf}}A_{\inf, I}/(p^n, \mu^{n'})$$
   is concentrated in degree 0, and the claim follows.
   \end{rem}

   \subsubsection{\normalfont{\textbf{Local computations}}}\label{locop}
   
   Next, by a standard argument, we express locally the $B$-cohomology and the $B_{\dR}^+$-cohomology in terms of Koszul complexes.
   
   \begin{lemma}\label{condstep}
     Let $I\subset (0, \infty)$ be a compact interval with rational endpoints, and let $m\ge 1$ be an integer. Given $$\mathbf{B}\in \{\Bb_{I}, \Bb, \Bb_{\dR}^+, \Bb_{\dR}^+/\Fil^m\}$$ we write  $\mathscr{B}=\mathbf{B}_{\Spa(C)_{\pet}}$ for the corresponding condensed period ring. In the setting of Notation \ref{notaz}, we have a natural isomorphism in $D(\Mod^{\ssolid}_{\mathscr{B}})$
     \begin{equation}\label{frfr}
       R\Gamma_{\mathscr{B}}(\fr X_C)\simeq L\eta_{t}\Kos_{\mathbf{B}(R_{\infty})}(\gamma_1-1, \ldots, \gamma_d-1)
     \end{equation}
     compatible with the filtration décalée of Definition \ref{beilifildef}.
   \end{lemma}
   \begin{proof}
   Using Proposition \ref{3.11}\listref{3.11.2}, it remains to check that 
   $$R\Gamma_{\pet}(\fr X_C, \mathbf{B})\simeq \Kos_{\mathbf{B}(R_{\infty})}(\gamma_1-1, \ldots, \gamma_d-1).$$
   Considering the Cartan--Leray spectral sequence associated to the affinoid perfectoid pro-étale cover $ \fr X_{C, \infty}\to \fr X_C$ of (\ref{ccover}) with Galois group $\Gamma$ (\cite[Proposition 4.12]{Bosco}), we have the following natural isomorphism in $D(\Mod^{\ssolid}_{\mathscr{B}})$
   $$R\Gamma_{\cond}(\Gamma, \mathbf{B}(R_{\infty}))\overset{\sim}{\to}R\Gamma_{\pet}(\fr X_C, \mathbf{B}).$$
   Then, the statement follows from \cite[Proposition B.3]{Bosco}.
   \end{proof}

   Our next goal is to express the right-hand side of (\ref{frfr}) in terms of differential forms. For this, recalling the notation introduced in Remark \ref{form} and Remark \ref{form2}, in the setting of Notation \ref{notaz} we denote the dual basis of the log $A_{\inf}$-derivations (see \cite[\S 5.10]{CK}) as follows:
   \begin{equation}\label{derder}
    \partial_i:=\frac{\partial}{\partial\log(X_i)}:A_{\inf}(R)\to A_{\inf}(R)
   \end{equation}
   for $1\le i\le d$. Given $I\subset (0, \infty)$ a compact interval with rational endpoints, by slight abuse of notation, we will also denote by $\partial_i$ the extension of the derivatives (\ref{derder}) to $A_I(R)$ or $B_I(R)$.

  \begin{lemma}\label{primitive}
   Let $I\subset[1/(p-1), \infty)$ be a compact interval with rational endpoints. In the setting of Notation \ref{notaz}, we have a $B_I$-linear quasi-isomorphism
   \begin{equation}\label{crucc}
    \Kos_{A_{\inf}(R)}(\partial_1, \ldots,\partial_d)\dsolid_{A_{\inf}}B_{I}\overset{\sim}{\to}\Kos_{B_I(R)}(\partial_1, \ldots,\partial_d)\overset{\sim}{\to}L\eta_{t}\Kos_{\Bb_I(R_{\infty})}(\gamma_1-1, \ldots, \gamma_d-1)
   \end{equation}
   compatible with the action of Frobenius $\varphi$.
   \end{lemma}
  
   \begin{proof}
    We will generalize the proof of \cite[Proposition 7.13]{Bosco}.  Since $\mu$ divides $\gamma_i-1$ in $B_I(R)$ for all $i$, i.e. $\Gamma$ acts trivially on $B_I(R)/\mu$, and since, by the choice of $I$, the elements $\mu$ and $t$ differ by a unit in $B_I$, by \cite[Lemma 7.8]{Bosco} we have that
    \begin{equation}\label{2255}
     \eta_t\Kos_{B_I(R)}(\gamma_1-1, \ldots, \gamma_d-1)\simeq \Kos_{B_I(R)}\left(\frac{\gamma_1-1}{t}, \ldots, \frac{\gamma_d-1}{t}\right).
    \end{equation}
   Using that $A_{\cris}\subset B_I$, by the choice of $I$, the arguments in \cite[Lemma 12.5]{BMS1} and \cite[Lemma 5.15]{CK} show that, for each $i$, we have the following Taylor expansion in $B_I(R)$
   \begin{equation*}\label{1+H}
    \frac{\gamma_i-1}{t}=\frac{\partial}{\partial\log(X_i)}\cdot h, \; \text{  with  }\;  h:=1+\sum_{j\ge 1}\frac{t^{j}}{(j+1)!}\left(\frac{\partial}{\partial\log(X_i)}\right)^{j}
   \end{equation*}
   where $h-1$ is topologically nilpotent, in particular the factor $h$ is an automorphism of $B_I(R)$; furthermore, the latter automorphism is $\varphi$-equivariant.\footnote{To check this one can argue as in the proof of \cite[Proposition 7.13]{Bosco}.} Then, recalling the notation (\ref{derder}), we deduce that the maps
   \begin{equation*}\label{2term}
    \left((B_I(R)\overset{\partial_i}{\to}B_I(R)\right)\overset{(\id,\; h)}{\longrightarrow}\left(B_I(R)\overset{\gamma_i-1}{\to}B_I(R)\right)
   \end{equation*}
   for $1\le i\le d$, induce a $\varphi$-equivariant quasi-isomorphism
    \begin{equation}\label{kosdec}
     \Kos_{B_I(R)}(\partial_1, \ldots, \partial_d)\overset{\sim}{\to}\eta_{t}\Kos_{B_I(R)}(\gamma_1-1, \ldots, \gamma_d-1).
    \end{equation}
    Next, we show that the natural map
    \begin{equation}\label{nex}
      \eta_{t}\Kos_{B_I(R)}(\gamma_1-1, \ldots, \gamma_d-1)\to \eta_{t}\Kos_{\Bb_I(R_{\infty})}(\gamma_1-1, \ldots, \gamma_d-1)
    \end{equation}
    is a quasi-isomorphism. For this, recalling Remark \ref{form2}, we have
    $$\Aa_I(R_{\infty})\cong A_I(R)\oplus \Aa_I(R_{\infty})^{\nonint} $$
    where $\Aa_I(R_{\infty})^{\nonint}$ denotes the ``nonintegral part'' of $\Aa_I(R_{\infty})$. Then, as $L\eta_{\mu}(-)$ commutes with filtered colimits (hence with inverting $p$), it suffices to show that 
    \begin{equation}\label{killA_I}
    L\eta_\mu \Kos_{\Aa_I(R_{\infty})^{\nonint}}(\gamma_1-1, \ldots, \gamma_d-1)\simeq 0.
    \end{equation}
    In order to show (\ref{killA_I}) we need to prove that $\mu$ kills $H^i_{\cond}(\Gamma, \Aa_I(R_{\infty})^{\nonint})$ for all $i\in \Zz$. By \cite[Proposition 3.25]{CK} (and \cite[Proposition B.3]{Bosco}), the element $\mu$ kills $H^i_{\cond}(\Gamma, \Aa_{\inf}(R_{\infty})^{\nonint})$ for all $i \in \Zz$, then we conclude by Corollary \ref{b2} below.
    
    Now, combining (\ref{kosdec}) with (\ref{nex}), to prove the statement it remains to check that we have an isomorphism
    \begin{equation}\label{bc1}
      \Kos_{A_{\inf}(R)}(\partial_1, \ldots, \partial_d)\dsolid_{A_{\inf}}B_{I}\overset{\sim}{\to}\Kos_{B_I(R)}(\partial_1, \ldots, \partial_d).
    \end{equation}
    Since the solid tensor product commutes with filtered colimits, it suffices to show (\ref{bc1}) replacing $B_I=A_I[1/p]$ with $A_I$. Then, using Proposition \ref{solid-vs-padic}, we reduce to showing that we have an isomorphism
    $$ \Kos_{A_{\inf}(R)}(\partial_1, \ldots, \partial_d)\widehat\otimes^{\LL}_{A_{\inf}}A_{I}\overset{\sim}{\to}\Kos_{A_I(R)}(\partial_1, \ldots, \partial_d)$$
    where the completion $\widehat\otimes^{\LL}_{A_{\inf}}$ is derived $p$-adic. For this, we observe that the latter completion is an underived (termwise) $p$-adic completion: as recalled in Remark \ref{remess}, for any integer $n\ge 1$, we have $A_I/p^n\cong A_I/(p^n, \mu^{n'})$ for a large enough integer $n'$, and, by \cite[Lemma 3.13]{CK}, $(p^n, \mu^{n'})$ is an $A_{\inf}(R)$-regular sequence with $A_{\inf}(R)/(p^n, \mu^{n'})$ flat over $A_{\inf}/(p^n, \mu^{n'})$.
   \end{proof}

   We used crucially the following result.
   
   \begin{lemma}\label{b1}
    Let $A$ be a condensed ring, and let $(f)\subset A$ be a principal ideal sheaf. Let $M$ be an $(f)$-adically complete $A$-module in $\CondAb$. Consider the following condition on a given $A$-module $P$ in $\CondAb$:
    \begin{equation}\label{condit}
     \text{for every } j, n\ge 1 \text{ the map } \Tor_j^{A}(P, M/f^{n'})\to  \Tor_j^{A}(P, M/f^{n}) \text{ vanishes for some } n'>n.
    \end{equation}
    For any bounded complex $P^\bullet$ of $A$-modules in $\CondAb$, with each $P^i$ and $H^i(P^\bullet)$ satisfying (\ref{condit}), for all $i\in \Zz$ we have a natural isomorphism in $\CondAb$
    $$H^i(P^\bullet \widehat \otimes_A M)\cong H^i(P^\bullet)\widehat\otimes_{A}M$$
    where the completion $\widehat \otimes_{A}$ is $(f)$-adic.
   \end{lemma}
    \begin{proof}
     See the proof of \cite[Lemma 3.30]{CK}.
    \end{proof}

   \begin{cor}\label{b2} Let $I\subset [1/(p-1), \infty)$ be a compact interval with rational endpoints. Let us denote $N_{\infty}:= \Aa_{\inf}(R_{\infty})^{\nonint}$. Then, for every $i\in \Zz$, we have a natural isomorphism in  $\CondAb$
    $$H^i_{\cond}(\Gamma, N_{\infty}\widehat\otimes_{A_{\inf}}A_I)\cong H^i_{\cond}(\Gamma, N_{\infty})\widehat\otimes_{A_{\inf}}A_I$$
    where the completion $\widehat\otimes_{A_{\inf}}$ is $p$-adic.
   \end{cor}
   \begin{proof}
    We will show that the condition (\ref{condit}) of Lemma \ref{b1} holds for $A=A_{\inf}$, $f=p$, $M=A_I$, and each $P\in \{\Aa_{\inf}(R_{\infty}),\; H^i_{\cond}(\Gamma, N_{\infty}),\; H^i_{\cond}(\Gamma, N_{\infty}/\mu)\}$, adapting the argument of \cite[Lemma 3.31]{CK} to our setting. 
    
    For $P=\Aa_{\inf}(R_{\infty})$, we even have that $P\otimes_{A_{\inf}}^{\LL}A_I/p^n\in D(\CondAb)$ is concentrated in degree 0 for all $n\ge 1$: in fact, recalling that $A_I/p^n\cong A_I/(p^n, \mu^{n'})$ for a large enough integer $n'$ (see Remark \ref{remess}), it suffices to apply \cite[Lemma 3.13]{CK}, which implies that $(p^n, \mu^{n'})$ is an $\Aa_{\inf}(R_{\infty})$-regular sequence with $\Aa_{\inf}(R_{\infty})/(p^n, \mu^{n'})$ flat over $A_{\inf}/(p^n, \mu^{n'})$ (noting that the latter two condensed rings are discrete).
    
    Next, we claim that the case $P=H^i_{\cond}(\Gamma, N_{\infty})$ follows from the case $P=H^i_{\cond}(\Gamma, N_{\infty}/\mu)$. In fact, by \cite[Proposition 3.25]{CK} (and \cite[Proposition B.3]{Bosco}), $\mu$ kills every $H^j_{\cond}(\Gamma, N_{\infty})$, therefore, the long exact sequence in condensed group cohomology associated to the short exact sequence $0\to N_{\infty}\overset{\mu}{\to}N_{\infty}\to N_{\infty}/\mu\to 0$, gives the short exact sequences
    $$0\to H^i_{\cond}(\Gamma, N_{\infty})\to H^i_{\cond}(\Gamma, N_{\infty}/\mu)\to H^{i+1}_{\cond}(\Gamma, N_{\infty})\to 0.$$
    By \cite[Proposition B.3]{Bosco}, $H^i_{\cond}(\Gamma, N_{\infty})$ vanishes for a large enough integer $i$, hence the claim follows by descending induction on $i$.
    
    Finally, for $P=H^i_{\cond}(\Gamma, N_{\infty}/\mu)$, we claim that $P$ satisfies the following conditions that imply the condition (\ref{condit}) of Lemma \ref{b1}: 
    \begin{enumerate}[(i)]\label{afr}
     \item\label{afr:1} $P$ is $p$-torsion-free;
     \item\label{afr:2} for every $n\ge 1$, $P/p^n$ is a filtered colimit of $A_{\inf}$-modules isomorphic to $A_{\inf}/(\varphi^{-r}(\mu), p^n)$ for variable $r\ge 0$.
    \end{enumerate}
    In fact, recalling \cite[Proposition B.3]{Bosco}, condition \listref{afr:1} follows from \cite[Proposition 3.19]{CK}, and condition \listref{afr:2} follows from \cite[Proposition 3.19, Corollary 3.23]{CK} and Lazard's theorem (\cite[058G]{Thestack}), observing that $P/p^n$ is discrete. It remains to show  that such conditions on $P$ imply the condition (\ref{condit}) of Lemma \ref{b1} in our case. For this, by condition \listref{afr:1} and Lemma \ref{condtors}, we have that $P\otimes_{A_{\inf}}^{\LL}A_{\inf}/p^n= P\otimes_{\Zz}^{\LL}\Zz/p^n$ is concentrated in degree 0, and then by condition \listref{afr:2} and the $\mu$-torsion-freeness of $A_{\inf}/p^n$, we can reduce to checking that, for every $n\ge 1$ and $r\ge 0$, the map
    \begin{equation}\label{nonec}
     \Tor_1^{A_{\inf}/p^{n'}}(A_{\inf}/(\varphi^{-r}(\mu), p^{n'}), A_I/p^{n'})\to \Tor_1^{A_{\inf}/p^n}(A_{\inf}/(\varphi^{-r}(\mu), p^{n}), A_I/p^{n})
    \end{equation}
    vanishes for some $n'>n$. In order to check this, we observe that the source of the map (\ref{nonec}) identifies with $(A_I/p^{n'})[\varphi^{-r}(\mu)]$ (and the target with $(A_I/p^{n})[\varphi^{-r}(\mu)]$), and we conclude observing that $\varphi^{-r}(\mu)$ divides $\mu$ in $A_{\inf}$,\footnote{In fact, for any $r\ge 1$, one has $\mu=(\prod_{j=0}^{r-1}\varphi^{-j}(\xi))\cdot \varphi^{-r}(\mu)$.} and, by Lemma \ref{intert}, the map $(A_I/p^{n'})[\mu]\to A_I/p^n$ vanishes for some $n'>n$.
    \end{proof}

  \subsubsection{\normalfont{\textbf{The functorial isomorphism}}}\label{functisom}

   As we will see in this subsection, the source of the quasi-isomorphism (\ref{crucc}) of Lemma \ref{primitive} computes $$R\Gamma_{\cris}(\mathfrak X_{\cl O_C/p}/A_{\cris}^\times)\dsolid_{A_{\cris}}B_{I}.$$ Therefore, Lemma \ref{primitive}, combined with Lemma \ref{condstep}, already provides a local version of the desired Theorem \ref{firstep}. However, such quasi-isomorphism depends on the choice of the coordinates $\fr X\to \Spf(R^\square)$, introduced in Notation \ref{notaz}. To have a functorial quasi-isomorphism, we will rely on a modification of the method of ``all possible coordinates'' pioneered by Bhatt--Morrow--Scholze in \cite[\S 12.2]{BMS1}, and used by Česnavičius--Koshikawa in \cite[\S 5]{CK} to prove the functoriality with respect to étale maps of the absolute crystalline comparison isomorphism for the $A_{\inf}$-cohomology in the semistable case.
    \medskip
    
   Let us begin resuming the setting of \cite[\S 5.17]{CK}.

   \begin{notation}\label{notknock}
    We denote by $k_C$ the residue field of $\cl O_C$.
    \begin{itemize}
     \item Let $\mathfrak X=\Spf(R)$ be an affine $p$-adic formal scheme over $\Spf(\cl O_C)$, such that every two irreducible components of $\Spec(R\otimes_{\cl O_C}k_C)$ intersect, and such that there exist finite sets $\Sigma$ and $\Lambda\neq \emptyset$, and a closed immersion of $p$-adic formal schemes over $\Spf(\cl O_C)$
    \begin{equation}\label{mane}
     \fr X\hookrightarrow \Spf(R_{\Sigma}^\square)\times_{\Spf(\cl O_C)} \textstyle\prod_{\lambda\in \Lambda}\Spf(R_{\lambda}^\square)=:\Spf(R_{\Sigma, \Lambda}^\square)
    \end{equation}
    where 
    $$R_{\Sigma}^\square:=\cl O_C\{t_{\sigma}^{\pm 1}:\sigma\in \Sigma\}$$
    the induced map $\fr X\to \Spf(R_{\Sigma}^\square)$ is a closed immersion, and, for each $\lambda\in \Lambda$,
    $$R_{\lambda}^\square:=\cl O_C\{t_{\lambda, 0}, \ldots, t_{\lambda, r_{\lambda}}, t_{\lambda, r_{\lambda+1}}^{\pm 1}, \ldots, t_{\lambda, d}^{\pm 1} \}/(t_{\lambda, 0}\cdots t_{\lambda, r_{\lambda}}-p^{q_{\lambda}})$$
    for some $0\le r_{\lambda}\le d$, and  $q_{\lambda}\in \Qq_{>0}$, and the induced map $\fr X\to \Spf(R_{\lambda}^\square)$ is étale. 
    
    We will denote by $M_R$ the canonical log structure on $R$, Notation \ref{notss}.
    \item We define
    $$A_{\inf, \Sigma, \Lambda}^\square:=A_{\inf}(R_{\Sigma}^\square)\widehat\otimes_{A_{\inf}} \textstyle\widehat\bigotimes_{\lambda\in \Lambda}A_{\inf}(R_{\lambda}^\square)$$
    where the completions are $(p, \mu)$-adic.
    
    We will denote by $M_{\inf, \Sigma, \Lambda}^{\square}$ the log structure on $A_{\inf, \Sigma, \Lambda}^\square$ associated to the log structures on $A_{\inf}(R_{\Sigma}^\square)$ and $A_{\inf}(R_{\lambda}^\square)$, for varying $\lambda \in \Lambda$, defined in \cite[\S 5.9]{CK}.
    
    \item Similarly to Notation \ref{notaz}, we define $R_{\Sigma, \Lambda, \infty}^\square$ a perfectoid $R_{\Sigma, \Lambda}^\square$-algebra  such that
   \begin{equation}\label{bcc}
    \Spa(R_{\Sigma, \Lambda, \infty}^\square[1/p], R_{\Sigma, \Lambda, \infty}^\square)\to \fr X_C
   \end{equation}
   is an affinoid perfectoid pro-étale cover of $\fr X_C$ with Galois group 
   $$\Gamma_{\Sigma, \Lambda}:=\Gamma_{\Sigma}\times \textstyle \prod_{\lambda\in \Lambda}\Gamma_{\lambda}\cong \Zz_p^{|\Sigma|}\times \prod_{\lambda\in \Lambda}\Zz_p^d$$
   (see also \cite[\S 5.18]{CK}).
   
   We denote by $(\gamma_{\sigma})_{\sigma\in \Sigma}, (\gamma_{\lambda, i})_{\lambda\in \Lambda, 1\le i\le d}$ the  generators of $\Gamma_{\Sigma, \Lambda}$ defined by
    $$\gamma_{\sigma}:=(1, \ldots, 1, \varepsilon, 1, \ldots, 1)\;\; \text{ for } \sigma\in \Sigma$$
     where $\varepsilon$ sits on the $\sigma$-th entry, and
    $$\gamma_{\lambda, i}:=(\varepsilon^{-1}, 1, \ldots, 1, \varepsilon, 1, \ldots, 1)\;\; \text{ for } i=1, \ldots, r_{\lambda}$$
    $$\gamma_{\lambda, i}:=(1, \ldots, 1, \varepsilon, 1, \ldots, 1)\;\;\; \text{ for } i=r_{\lambda}+1, \ldots, d$$
   where $\varepsilon$ sits on the $i$-th entry.
   
   The base change of (\ref{bcc}) along the generic fiber of (\ref{mane}) defines an affinoid perfectoid pro-étale cover of $\fr X_C$
   $$\fr X_{C, \Sigma, \Lambda, \infty}:=\Spa(R_{\Sigma, \Lambda, \infty}[1/p], R_{\Sigma, \Lambda, \infty})\to \fr X_C$$
   with Galois group $\Gamma_{\Sigma, \Lambda}$.
   
   \item  Given $\Mm$ any pro-étale period sheaf of \S \ref{petsheaves}, we set $$\Mm(R_{\Sigma, \Lambda, \infty}):=\Mm(\fr X_{C, \Sigma, \Lambda, \infty})$$
   and we regard it as a condensed ring.

   \end{itemize}
    \end{notation}
    
    \begin{rem}\label{doit}
    In Notation \ref{notknock}, the assumption on the special fiber $\Spec(R\otimes_{\cl O_C}k_C)$ guarantees that each irreducible component of such special fiber is cut out by a unique $t_{\lambda, i}$ with $0\le i\le r_{\lambda}$ (see also \cite[\S 5.17]{CK}).
     \end{rem}

     We note that Remark \ref{form2} and Remark \ref{remess} hold in the setting of Notation \ref{notknock}, i.e. with $R_{\Sigma}^{\square}$ and $R_{\lambda}^{\square}$ in place of $R^{\square}$. In particular, to fix the notation, we let 
    \begin{equation*}
    A_{\inf}(R_{\Sigma}^\square)\cong A_{\inf}\{X_\sigma^{\pm 1}|\sigma\in \Sigma\}
   \end{equation*}
   \begin{equation*}
    A_{\inf}(R_{\lambda}^\square)\cong A_{\inf}\{X_{\lambda, 0}, \ldots, X_{\lambda, r_{\lambda}}, X_{r_{\lambda}+1}^{\pm 1}, \ldots, X_{\lambda, d}^{\pm 1}\}/(X_{\lambda, 0}\cdots X_{\lambda, r_{\lambda}}-[p^{\flat}]^{q_{\lambda}})
   \end{equation*}
   be the isomorphisms (\ref{tofixx}) for $R_{\Sigma}^{\square}$ and $R_{\lambda}^{\square}$. \medskip

  For the proof of Theorem \ref{firstep}, we will need the next result from \cite[\S 5.22]{CK} on the log-crystalline cohomology over $A_{\cris}$, that we translate here in the condensed setting.
  
  \begin{lemma}\label{cruc}
   In the setting of Notation \ref{notknock}, we have a $\varphi$-equivariant identification
   \begin{equation}\label{coorcompl1}
    R\Gamma_{\cris}(\mathfrak X_{\cl O_C/p}/A_{\cris}^\times)\simeq \Omega_{D_{\Sigma, \Lambda}(R)}^\bullet:=\Kos_{D_{\Sigma, \Lambda}(R)}\left((\partial_{\sigma})_{\sigma\in \Sigma},(\partial_{\lambda, i})_{\lambda\in \Lambda, 1\le i\le d}\right)
   \end{equation}
   where $D_{\Sigma, \Lambda}(R)$ is an $A_{\cris}$-algebra in $\CondAb$ characterized by the following properties: $D_{\Sigma, \Lambda}(R)$ is $p$-adically complete, and, for each integer $n\ge 1$, $D_{\Sigma, \Lambda}(R)/p^n$ is the log PD envelope of \footnote{Here, $R/p$ is equipped with the pullback of the canonical log structure $M_R$ on $R$, and  $A_{\inf, \Sigma, \Lambda}^{\square}\otimes_{A_{\inf}}A_{\cris}/p^n$ is endowed with the pullback of the log structure $M_{\inf, \Sigma, \Lambda}^{\square}$ on $A_{\inf, \Sigma, \Lambda}^{\square}$, Notation \ref{notknock}.}
   $$(A_{\inf, \Sigma, \Lambda}^\square\otimes_{A_{\inf}}A_{\cris}/p^n, M_{\inf, \Sigma, \Lambda}^{\square}) \twoheadrightarrow (R/p,  M_R)\;\; \text{ over }\;\; A_{\cris}^\times/p^n \twoheadrightarrow \cl O_C^\times/p.$$
   Here, $\partial_{\sigma}:=\frac{\partial}{\partial\log(X_\sigma)}$, resp. $\partial_{\lambda, i}:=\frac{\partial}{\partial\log(X_{\lambda, i})}$, are as in (\ref{derder}) with $R_{\Sigma}^{\square}$, resp. $R_{\lambda}^{\square}$, in place of $R$.
  \end{lemma}
  \begin{proof}
   This is the content of \cite[Proposition 5.23]{CK}, which relies on \cite[(1.8.1)]{Beili}. The characterization of $D_{\Sigma, \Lambda}(R)$ follows from \cite[\S 5.22, Lemma 5.29]{CK}.
  \end{proof}

  We note that the action of $\Gamma_{\Sigma, \Lambda}$ on $A_{\inf, \Sigma, \Lambda}^\square$ induces a natural action of $\Gamma_{\Sigma, \Lambda}$ on $D_{\Sigma, \Lambda}(R)$. \medskip
  
  The next lemma, which is a semistable version of \cite[Corollary 12.7]{BMS1}, expresses the complex $\Omega_{D_{\Sigma, \Lambda}(R)}^\bullet$ of Lemma \ref{cruc}, which computes the log-crystalline cohomology over $A_{\cris}$,  in terms of condensed group cohomology $R\Gamma_{\cond}(\Gamma_{\Sigma, \Lambda}, D_{\Sigma, \Lambda}(R))$ via passage through Lie algebra cohomology. Cf. \cite[\S 4.3]{CN1}.

 \begin{lemma}\label{ghui} In the setting of Notation \ref{notknock}, we denote by $\Lie \Gamma_{\Sigma, \Lambda}$ the Lie algebra of $\Gamma_{\Sigma, \Lambda}$, and we write $\exp:\Lie \Gamma_{\Sigma, \Lambda} \cong \Gamma_{\Sigma, \Lambda}$ for the  exponential isomorphism. Then, there is a natural action of $\Lie \Gamma_{\Sigma, \Lambda}$ on $D_{\Sigma, \Lambda}(R)$ defined for $g\in \Lie \Gamma_{\Sigma, \Lambda}$, with $\exp(g)=\gamma\in \Gamma_{\Sigma, \Lambda}$, by
 \begin{equation}\label{2310}
  g=\log(\gamma):=\sum_{n\ge 1}\frac{(-1)^{n-1}}{n}(\gamma-1)^n.
  \end{equation}
 We write $U_{\Sigma, \Lambda}$ for the universal enveloping algebra of $\Lie \Gamma_{\Sigma, \Lambda}$, and we denote by $$R\Gamma(\Lie \Gamma_{\Sigma, \Lambda}, D_{\Sigma, \Lambda}(R)):=R\underline{\Hom}_{U_{\Sigma, \Lambda}}(\Zz_p, D_{\Sigma, \Lambda}(R))\in D(\CondAb)$$  the Lie group cohomology.
 \begin{enumerate}[(i)]
  \item\label{ghui:1} There is a quasi-isomorphism
  $$L\eta_\mu R\Gamma(\Lie \Gamma_{\Sigma, \Lambda}, D_{\Sigma, \Lambda}(R))\simeq \Omega_{D_{\Sigma, \Lambda}(R)}^\bullet.$$
  \item\label{ghui:2} There is a quasi-isomorphism
  $$R\Gamma(\Lie \Gamma_{\Sigma, \Lambda}, D_{\Sigma, \Lambda}(R))\simeq R\Gamma_{\cond}(\Gamma_{\Sigma, \Lambda}, D_{\Sigma, \Lambda}(R)).$$
 \end{enumerate}
 \end{lemma}
 \begin{proof}
 We first need to check that, for $g\in \Lie \Gamma_{\Sigma, \Lambda}$, the series (\ref{2310}) converges to an endomorphism of $D_{\Sigma, \Lambda}(R)$. For this, it suffices to prove that the action of $\gamma-1$ on $D_{\Sigma, \Lambda}(R)$ takes values in $([\varepsilon]-1)D_{\Sigma, \Lambda}(R)$: in fact, by \cite[Lemma 12.2]{BMS1}, this implies that the action of $\frac{(\gamma-1)^n}{n}$ on $D_{\Sigma, \Lambda}(R)$ has values in $D_{\Sigma, \Lambda}(R)$, and such values converge to 0 as $n\to \infty$. Thus, following the proof of \cite[Lemma 12.6]{BMS1} and using \cite[Lemma 5.29]{CK}, we can reduce to checking that the action of $\gamma-1$ on $A_{\inf, \Sigma, \Lambda}^{\square}$ takes values in $([\varepsilon]-1)A_{\inf, \Sigma, \Lambda}^{\square}$: this is clear for $\gamma$ being one of the generators $(\gamma_{\sigma})_{\sigma\in \Sigma}, (\gamma_{\lambda, i})_{\lambda\in \Lambda, 1\le i\le d}$ of $\Gamma_{\Sigma, \Lambda}$.\footnote{Note that, for $1\le i\le d$, the element $\gamma_{\lambda, i}-1$ acts as follows on $X_j$: it sends $X_i\mapsto ([\varepsilon]-1)X_i$;  $X_j\mapsto 0$ if $0<j\neq i$; $X_0\mapsto ([\varepsilon^{-1}]-1)X_0=-([\varepsilon]-1)[\varepsilon^{-1}]X_0$ if $i\le r_{\lambda}$; and $X_0\mapsto 0$ if $i>r_{\lambda}$.}
 
 Next, to prove both part \listref{ghui:1} and part \listref{ghui:2}, we denote $T(\Sigma, \Lambda):= \Sigma \cup \{(\lambda, i): \lambda\in \Lambda, 1\le i\le d\}$. By (the proof of) \cite[Lemma 5.33]{CK}, the element $\gamma_\tau\in \Gamma_{\Sigma, \Lambda}$ acts on $D_{\Sigma, \Lambda}(R)$ as the  endomorphism $\exp(\log([\varepsilon])\cdot \partial_\tau)$, for varying $\tau \in T(\Sigma, \Lambda)$. Therefore, denoting $g_\tau:=\log([\varepsilon])\cdot \partial_\tau\in \Lie \Gamma_{\Sigma, \Lambda}$, we have a quasi-isomorphism 
  \begin{equation}\label{atete}
   R\Gamma(\Lie \Gamma_{\Sigma, \Lambda}, D_{\Sigma, \Lambda}(R))\simeq\Kos_{D_{\Sigma, \Lambda}(R)}(g_{\tau\in T(\Sigma, \Lambda)})
  \end{equation}
  Since $\log([\varepsilon])$ and $\mu$ differ by a unit in $A_{\cris}$, part \listref{ghui:1} follows from the quasi-isomorphism (\ref{atete}) and \cite[Lemma 7.8(ii)]{Bosco}. For part \listref{ghui:2}, we note that $g_\tau=\log(\gamma_\tau)=(\gamma_\tau-1)\cdot h_\tau$ for $\tau \in T(\Sigma, \Lambda)$, with $h_\tau$ automorphisms of $D_{\Sigma, \Lambda}(R)$  commuting with each other and with $\gamma_\tau-1$, therefore 
  \begin{equation}\label{nocturne}
   \Kos_{D_{\Sigma, \Lambda}(R)}(g_{\tau\in T(\Sigma, \Lambda)})\cong \Kos_{D_{\Sigma, \Lambda}(R)}((\gamma_\tau-1)_{\tau\in T(\Sigma, \Lambda)})
  \end{equation}
  (cf. the proof of \cite[Lemma 12.5]{BMS1}). Then, the statement of part \listref{ghui:2} follows combining (\ref{atete}), (\ref{nocturne}),  and \cite[Proposition B.3]{Bosco}.
 \end{proof}
 
 We deduce the following result.

 \begin{cor}\label{bgim}
  In the setting of Notation \ref{notknock}, we have a natural (in $\fr X$ and the datum (\ref{mane})) quasi-isomorphism
  \begin{equation}\label{themo}
   R\Gamma_{\cris}(\mathfrak X_{\cl O_C/p}/A_{\cris}^\times)\overset{\sim}{\longrightarrow} L\eta_\mu R\Gamma_{\cond}(\Gamma_{\Sigma, \Lambda}, D_{\Sigma, \Lambda}(R))
  \end{equation}
 \end{cor}
 \begin{proof}
  We first construct the desired natural morphism (\ref{themo}). By the proof of \cite[(1.8.1)]{Beili} and by \cite[Lemma 5.29]{CK}, we have the following \v{C}ech-Alexander computation of the log-crystalline cohomology over $A_{\cris}$:
  \begin{equation}\label{iden1}
   R\Gamma_{\cris}(\mathfrak X_{\cl O_C/p}/A_{\cris}^\times)\simeq  \left(D_{\Sigma, \Lambda}(R)(0)\to D_{\Sigma, \Lambda}(R)(1)\to D_{\Sigma, \Lambda}(R)(2)\to \cdots\right)
  \end{equation}
  where $D_{\Sigma, \Lambda}(R)(n):=\varprojlim_{m\ge 1} D_{\Sigma, \Lambda, m}(R)(n)$ with $\Spec (D_{\Sigma, \Lambda, m}(R)(n))$ the $(n+1)$-fold product of $\Spec (D_{\Sigma, \Lambda}(R)/p^m)$ in $(\mathfrak X_{\cl O_C/p}/A_{\cris, m}^\times)_{\cris}$ (we recall \ref{logcriss} for the notation). On the other hand, by \cite[Proposition B.2(i)]{Bosco}, the condensed group cohomology is computed by
  \begin{equation}\label{iden2}
   R\Gamma_{\cond}(\Gamma_{\Sigma, \Lambda}, D_{\Sigma, \Lambda}(R))\simeq\left( D_{\Sigma, \Lambda}(R)\to \underline{\Hom}(\Zz[\Gamma_{\Sigma, \Lambda}], D_{\Sigma, \Lambda}(R))\to \underline{\Hom}(\Zz[\Gamma_{\Sigma, \Lambda}^2], D_{\Sigma, \Lambda}(R))\to \cdots\right).
  \end{equation}
  Under the identifications (\ref{iden1}) and (\ref{iden2}), we define the morphism
  \begin{equation}\label{last??}
   R\Gamma_{\cris}(\mathfrak X_{\cl O_C/p}/A_{\cris}^\times) \to R\Gamma_{\cond}(\Gamma_{\Sigma, \Lambda}, D_{\Sigma, \Lambda}(R))
  \end{equation}
   induced, in degree $n\ge 0$, by the composite of the termwise action $\Gamma_{\Sigma, \Lambda}^n\times D_{\Sigma, \Lambda}(R)(n)\to D_{\Sigma, \Lambda}(R)(n)$ with the co-diagonal map $D_{\Sigma, \Lambda}(R)(n)\to D_{\Sigma, \Lambda}(R)$. By \cite[Lemma 6.10]{BMS1} there is a natural map
   \begin{equation}\label{crisalid}
    L\eta_\mu R\Gamma_{\cris}(\mathfrak X_{\cl O_C/p}/A_{\cris}^\times)\to R\Gamma_{\cris}(\mathfrak X_{\cl O_C/p}/A_{\cris}^\times)
   \end{equation}
   which is a  quasi-isomorphism, as it follows combining Lemma \ref{cruc} and Lemma \ref{ghui}\listref{ghui:1}. Using the quasi-isomorphism (\ref{crisalid}), applying the décalage functor $L\eta_\mu(-)$ to (\ref{last??}), we obtain the desired morphism (\ref{themo}), which is a quasi-isomorphism thanks to Lemma \ref{cruc} (which relies on \cite[(1.8.1)]{Beili}, as (\ref{iden1})) and Lemma \ref{ghui}.
 \end{proof}

 As a next step toward Theorem \ref{firstep}, we want to construct a comparison morphism from the log-crystalline cohomology over $A_{\cris}$ to the $B_I$-cohomology. For this, we will rely on the morphism (\ref{themo}), defined in Corollary \ref{bgim}, and on the following construction (which is inspired from the proof of \cite[Theorem 8.1]{logprism}), that will allow us to compare the target of (\ref{themo}) with the $B_I$-cohomology.
  
  \begin{lemma}\label{tricky} In the setting of Notation \ref{notknock}, there is a natural (in $R$ and the datum (\ref{mane})) $\Gamma_{\Sigma, \Lambda}$-equivariant map
    \begin{equation}\label{funn}
     D_{\Sigma, \Lambda}(R)\to \Aa_{\cris}(R_{\Sigma, \Lambda, \infty})
    \end{equation}
   where $\Aa_{\cris}(R_{\Sigma, \Lambda, \infty}):=\Aa_{\inf}(R_{\Sigma, \Lambda, \infty})\widehat\otimes_{A_{\inf}}A_{\cris}$ and the completion $\widehat\otimes_{A_{\inf}}$ is $p$-adic.
  \end{lemma}
  \begin{proof}
   Since both $\Aa_{\cris}(R_{\Sigma, \Lambda, \infty})$ and $D_{\Sigma, \Lambda}(R)$ are $p$-adically complete (for the latter, recall \cite[Lemma 5.29]{CK}), it suffices to construct, for each $n\ge 1$, a natural map $D_{\Sigma, \Lambda}(R)/p^n\to \Aa_{\cris}(R_{\Sigma, \Lambda, \infty})/p^n$.  We observe that $\Aa_{\cris}(R_{\Sigma, \Lambda, \infty})/p^n=\Aa_{\inf}(R_{\Sigma, \Lambda, \infty})\otimes_{A_{\inf}}A_{\cris}/p^n$. Then, we consider the following commutative diagram of log rings
  \begin{equation}\label{Q}
  \begin{tikzcd}
   (A_{\inf, \Sigma, \Lambda}^{\square}\otimes_{A_{\inf}}A_{\cris}/p^n, M_{\inf, \Sigma, \Lambda}^{\square}) \arrow[two heads]{r}\arrow[d] &  (R/p, M_R) \arrow[d] \\
  (\Aa_{\inf}(R_{\Sigma, \Lambda, \infty})\otimes_{A_{\inf}}A_{\cris}/p^n, N)^a \arrow[two heads]{r} &  (R_{\Sigma, \Lambda, \infty}/p, M_R).
  \end{tikzcd}
  \end{equation}
  Here,  $R/p$ and  $R_{\Sigma, \Lambda, \infty}/p$ are equipped with the pullback of the canonical log structure $M_R$ on $R$, and  $A_{\inf, \Sigma, \Lambda}^{\square}\otimes_{A_{\inf}}A_{\cris}/p^n$ is endowed with the pullback of the log structure $M_{\inf, \Sigma, \Lambda}^{\square}$ on $A_{\inf, \Sigma, \Lambda}^{\square}$, as in Lemma \ref{cruc}. Moreover, $\Aa_{\inf}(R_{\Sigma, \Lambda, \infty})\otimes_{A_{\inf}}A_{\cris}/p^n$ is equipped with the pullback of the log structure on $\Aa_{\inf}(R_{\Sigma, \Lambda, \infty})$ associated to the pre-log structure defined as follows: we set $N:=(h^{\gp})^{-1}(M_{\inf, \Sigma, \Lambda}^{\square})$ where $h^{\gp}$ denotes the morphism of groups associated to the natural morphism of monoids $h:M_{\inf, \Sigma, \Lambda}^{\square}\to M_R$; the argument of \cite[Lemma 5.37]{CK} shows that the natural map $M_{\inf, \Sigma, \Lambda}^{\square}\to \Aa_{\inf}(R_{\Sigma, \Lambda, \infty})$ uniquely extends to a map $N\to \Aa_{\inf}(R_{\Sigma, \Lambda, \infty})$.\footnote{In fact, in the notation of \cite[\S 5.25, \S 5.26]{CK}, it suffices to prove that $\Aa_{\inf}(R_{\Sigma, \Lambda, \infty})$ is naturally an $(A_{\inf, \Sigma, \Lambda}^{\square}\otimes_{\Zz[Q]}\Zz[P_{\lambda_0}])$-algebra compatibly with the change of $\lambda_0\in \Lambda$, which is shown in \cite[Lemma 5.37]{CK}.} The resulting surjective map of log rings at the bottom of the diagram (\ref{Q}) is exact by construction, hence, the universal property of the log PD envelope $D_{\Sigma, \Lambda}(R)/p^n$ gives the desired natural map $D_{\Sigma, \Lambda}(R)/p^n\to \Aa_{\cris}(R_{\Sigma, \Lambda, \infty})/p^n$.  
  \end{proof}

  We are now ready to prove Theorem \ref{firstep}.
  
  \begin{proof}[Proof of Theorem \ref{firstep}] It suffices to prove the statement Zariski locally on $\fr X$ in a functorial way, as $\fr X$ is assumed to be qcqs and the derived tensor product $\dsolid_{A_{\cris}}$ commutes with finite limits. 
  
  Thus, let $\fr X=\Spf(R)$ as in Notation \ref{notknock}, with fixed finite sets $\Sigma$ and $\Lambda$. We will denote $T(\Sigma, \Lambda):=\Sigma \cup \{(\lambda, i): \lambda\in \Lambda, 1\le i\le d\}$.
  
   First,  we note that, similarly to Lemma \ref{condstep}, we have a $\varphi$-equivariant identification
   \begin{equation}\label{moio}
    R\Gamma_{B_{I}}(\fr X_C)\simeq L\eta_t R\Gamma_{\cond}(\Gamma_{\Sigma, \Lambda},\Bb_{I}(R_{\Sigma, \Lambda, \infty}))\simeq L\eta_{t}\Kos_{\Bb_I(R_{\Sigma, \Lambda, \infty})}((\gamma_{\tau}-1)_{\tau\in T(\Sigma, \Lambda)}).
   \end{equation}   

   Then, using the natural map (\ref{themo}) constructed in Corollary \ref{bgim}, together with the natural morphism (\ref{funn}) constructed in Lemma \ref{tricky}, we have a natural (in $\fr X$ and the datum (\ref{mane})) morphism
   \begin{equation}\label{lstm}
    R\Gamma_{\cris}(\mathfrak X_{\cl O_C/p}/A_{\cris}^\times)\overset{\sim}{\to} L\eta_\mu R\Gamma_{\cond}(\Gamma_{\Sigma, \Lambda}, D_{\Sigma, \Lambda}(R))\to L\eta_\mu R\Gamma_{\cond}(\Gamma_{\Sigma, \Lambda},\Bb_{I}(R_{\Sigma, \Lambda, \infty})).
   \end{equation}
   Using that $\mu$ and $t$ differ by a unit in $B_I$, (\ref{lstm}) induces a natural morphism 
   \begin{equation}\label{theright}
    f_{\fr X, \Sigma, \Lambda}: R\Gamma_{\cris}(\mathfrak X_{\cl O_C/p}/A_{\cris}^\times)\dsolid_{A_{\cris}}B_{I}\to R\Gamma_{B_{I}}(\fr X_C)
   \end{equation}
   that we claim to be a quasi-isomorphism. For this, it suffices to show that, for a fixed $\lambda\in \Lambda$, we have a commutative diagram as follows, whose arrows are quasi-isomorphisms
  \begin{equation}\label{diaprinc}
  \begin{tikzcd}
  \Kos_{A_{\cris}(R)_{\lambda}}\left((\partial_{\lambda, i})_{1\le i\le d} \right)\dsolid_{A_{\cris}}B_{I} \arrow[r]\arrow[d] & \eta_{t}\Kos_{\Bb_I(R_{\lambda, \infty})}((\gamma_{\lambda, i}-1)_{1\le i\le d})\arrow[d] \\
  \Kos_{D_{\Sigma, \Lambda}(R)}\left((\partial_{\tau})_{\tau\in T(\Sigma, \Lambda)}\right)\dsolid_{A_{\cris}}B_{I} \arrow[r] & \eta_{t}\Kos_{\Bb_I(R_{\Sigma, \Lambda, \infty})}((\gamma_{\tau}-1)_{\tau\in T(\Sigma, \Lambda)})
  \end{tikzcd}
  \end{equation}
  with the bottom arrow of (\ref{diaprinc}) compatible with the morphism (\ref{theright}), under the identifications (\ref{coorcompl1}) and (\ref{moio}). Here, we denote $A_{\cris}(R)_{\lambda}:=A_{\inf}(R)_{\lambda}\widehat\otimes_{A_{\inf}} A_{\cris}$, where the completion $\widehat\otimes_{A_{\inf}}$ is $p$-adic, and $A_{\inf}(R)_{\lambda}$ is the unique lift of the étale $(R_{\lambda}^\square/p)$-algebra $R/p$, along $\theta: A_{\inf}(R_{\lambda}^\square)\twoheadrightarrow R_{\lambda}^\square$, to a $(p, \mu)$-adically complete, formally étale $A_{\inf}(R_{\lambda}^\square)$-algebra.
  \begin{enumerate}[(i)]
   \item The right vertical map of (\ref{diaprinc}) is a quasi-isomorphism since, by Lemma \ref{condstep} and (\ref{moio}), both the target and the source are quasi-isomorphic to $R\Gamma_{B_{I}}(\fr X_C)$. 
   \item The top horizontal arrow of (\ref{diaprinc}) is induced by the quasi-isomorphism constructed in Lemma \ref{primitive}, observing that we have the following identifications:
  $$\Kos_{A_{\inf}(R)_{\lambda}}\left((\partial_{\lambda, i})_{1\le i\le d} \right)\dsolid_{A_{\inf}}B_{I}=
  \left(\Kos_{A_{\inf}(R)_{\lambda}}\left((\partial_{\lambda, i})_{1\le i\le d} \right)\dsolid_{A_{\inf}}A_{\cris}\right)\dsolid_{A_{\cris}}B_{I}$$
  and, by Proposition \ref{solid-vs-padic}, 
  \begin{equation}\label{afa}
   \Kos_{A_{\inf}(R)_{\lambda}}\left((\partial_{\lambda, i})_{1\le i\le d} \right)\dsolid_{A_{\inf}}A_{\cris} \simeq 
   \Kos_{A_{\cris}(R)_{\lambda}}\left((\partial_{\lambda, i})_{1\le i\le d} \right)
  \end{equation}
  as we now explain. In fact, Proposition \ref{solid-vs-padic} implies that the derived solid tensor product $\dsolid_{A_{\inf}}$ appearing in (\ref{afa}) can be replaced by the derived $p$-adic completion $\widehat\otimes^{\LL}_{A_{\inf}}$, and then it remains to observe that the latter completion identifies with the underived (termwise) $p$-adic completion: for this, denoting by $A_{\cris}^{(m)}\subset A_{\cris}$ the $p$-adic completion of the $A_{\inf}$-subalgebra generated by $\xi^j/j!$ for varying $j\le m$,  we note that, for any integer $n\ge 1$, we have $A_{\cris}/p^n\cong \varinjlim_{m\ge p} A_{\cris}^{(m)}/p^n$, with $A_{\cris}^{(m)}/p^n\cong A_{\cris}^{(m)}/(p^n, \mu^{n'})$ for $m\ge p$, and a large enough integer $n'$ (see \cite[\S 3.26]{CK});\footnote{In fact, for $m\ge p$, one has $\mu^p/p!\in A_{\cris}^{(m)}$, therefore one can take $n':=pn$.} then,  we recall that $(p^n, \mu^{n'})$ is an $A_{\inf}(R)$-regular sequence with $A_{\inf}(R)/(p^n, \mu^{n'})$ flat over $A_{\inf}/(p^n, \mu^{n'})$ (see \cite[Lemma 3.13]{CK}).
  \item  The left vertical map of (\ref{diaprinc}) is constructed as follows. Since $D_{\Sigma, \Lambda}(R)$ is a $p$-adically complete pro-nilpotent thickening of $R/p$ (here we use \cite[Lemma 5.29]{CK}), by the infinitesimal lifting criterion for the $p$-adic formally étale map $A_{\cris}(R_{\lambda}^\square)\to A_{\cris}(R)_{\lambda}$, we deduce that in the following diagram 
  \begin{center}
  \begin{tikzcd}
  A_{\cris}(R_{\lambda}^\square) \arrow[r]\arrow[d] &  A_{\cris}(R)_{\lambda} \arrow[dl, dashrightarrow]\arrow[d] \\
   D_{\Sigma, \Lambda}(R) \arrow[two heads]{r} &  R/p 
  \end{tikzcd}
  \end{center}
  there exists a unique dotted arrow making the diagram commute.
  The resulting map 
  $$\Kos_{A_{\cris}(R)_{\lambda}}\left((\partial_{\lambda, i})_{1\le i\le d} \right)\to \Kos_{D_{\Sigma, \Lambda}(R)}\left((\partial_{\tau})_{\tau\in T(\Sigma, \Lambda)}\right)$$
  is a quasi-isomorphism, as both the source and the target compute $R\Gamma_{\cris}(\mathfrak X_{\cl O_C/p}/A_{\cris}^\times)$, by \cite[Proposition 5.13]{CK} and Lemma \ref{cruc}, respectively.
  \item  The bottom horizontal arrow of (\ref{diaprinc}) is constructed as follows. Similarly to the map (\ref{2term}) in Lemma \ref{primitive}, for varying $\tau \in T(\Sigma, \Lambda)$, we have $\varphi$-equivariant maps
  \begin{equation}\label{2term2}
    \left((D_{\Sigma, \Lambda}(R)\overset{\partial_\tau}{\to}D_{\Sigma, \Lambda}(R)\right)\overset{(\id,\; h_\tau)}{\longrightarrow}\left(D_{\Sigma, \Lambda}(R)\overset{\gamma_\tau-1}{\to}D_{\Sigma, \Lambda}(R)\right)
   \end{equation}
   with $h_\tau:=1+\sum_{j\ge 1}\frac{t^{j}}{(j+1)!}(\partial{_\tau})^{j}$, which induce a $\varphi$-equivariant quasi-isomorphism
   \begin{equation}\label{kospass}
    \Kos_{D_{\Sigma, \Lambda}(R)}\left((\partial_{\tau})_{\tau\in T(\Sigma, \Lambda)}\right)\overset{\sim}{\to} \eta_\mu \Kos_{D_{\Sigma, \Lambda}(R)}\left((\gamma_\tau-1)_{\tau\in T(\Sigma, \Lambda)}\right)
   \end{equation}
  (cf. \cite[Proposition 5.34]{CK}). Moreover, the $\Gamma_{\Sigma, \Lambda}$-equivariant map (\ref{funn}) constructed in Lemma \ref{tricky}, induces a morphism
  \begin{equation}\label{lastcomp}
   \eta_\mu \Kos_{D_{\Sigma, \Lambda}(R)}\left((\gamma_\tau-1)_{\tau\in T(\Sigma, \Lambda)}\right)\to \eta_{\mu}\Kos_{\Bb_I(R_{\Sigma, \Lambda, \infty})}((\gamma_{\tau}-1)_{\tau\in T(\Sigma, \Lambda)}).
  \end{equation}
  Then, we define the bottom horizontal arrow of (\ref{diaprinc}) as the morphism induced by the composite of (\ref{kospass}) and (\ref{lastcomp}). The so constructed morphism makes the diagram (\ref{diaprinc}) commute, it is compatible with the morphism (\ref{theright}), under the identifications (\ref{coorcompl1}) and (\ref{moio}), and, by the previous points, it is a quasi-isomorphism.
  
  \end{enumerate}
  
  Taking the filtered colimit $\varinjlim_{\Sigma, \Lambda}$ of the quasi-isomorphisms $f_{\fr X, \Sigma, \Lambda}$ constructed in (\ref{theright}), putting everything together, using that $\dsolid_{A_{\cris}}$ commutes with filtered colimits, we obtain the desired quasi-isomorphism
  $$f_{\fr X}:R\Gamma_{\cris}(\mathfrak X_{\cl O_C/p}/A_{\cris}^\times)\dsolid_{A_{\cris}}B_{I}\overset{\sim}{\longrightarrow} \varinjlim_{\Sigma, \Lambda}K(R_{\Sigma, \Lambda, \infty})\simeq R\Gamma_{B_{I}}(\fr X_C).$$
  where we denoted $K(R_{\Sigma, \Lambda, \infty}):=L\eta_t R\Gamma_{\cond}(\Gamma_{\Sigma, \Lambda},\Bb_{I}(R_{\Sigma, \Lambda, \infty}))$. \medskip
  
  It remains to show that the morphism $f_{\fr X}$ depends functorially on $\fr X=\Spf(R)$. For this, it suffices to prove that the filtered colimit $\varinjlim_{\Sigma, \Lambda}K(R_{\Sigma, \Lambda, \infty})$ depends functorially on $R$, compatibly with the constructed morphism from the complex $R\Gamma_{\cris}(\mathfrak X_{\cl O_C/p}/A_{\cris}^\times)$ (i.e. the filtered colimit $\varinjlim_{\Sigma, \Lambda}$ of the morphisms defined in (\ref{lstm})). We study separately the latter filtered colimit  in the case $\fr X$ is smooth or non-smooth, and show that it reduces to a simpler filtered colimit.\footnote{We thank Teruhisa Koshikawa for suggesting the following idea.}
  \begin{itemize}
   \item Suppose $\fr X$ is smooth. In this case, given a finite set $\Lambda$ as in Notation \ref{notknock}, for each pair $(\lambda, i)$ with $\lambda\in \Lambda$ and $0\le i \le d$, we have
   \begin{equation}\label{rell1}
    t_{\lambda, i}=(p^q)^{n_{\lambda, i}}\cdot u_{\lambda, i}\;\; \text{  for unique } n_{\lambda, i}\in \Zz_{\ge 0} \text{ and } u_{\lambda, i}\in R^\times
   \end{equation}
   where $q\in \Qq_{>0}$ is the unique element such that $\Zz\cdot q=\sum_{\lambda\in \Lambda}\Zz\cdot q_{\lambda}$ inside $\Qq$ (see \cite[(5.26.2)]{CK}). We recall that $p$-th power roots of $p$ in $\cl O_C$ are fixed (Notation \ref{notss}). Then, for sufficiently large finite sets $\Sigma\subset R^{\times}$ (containing the $u_{\lambda, i}$'s as in (\ref{rell1})), denoting $R_{\Sigma, \infty}:=R_{\Sigma, \emptyset, \infty}$, we have natural surjections 
   \begin{equation}\label{natsurj0}
    R_{\Sigma, \Lambda, \infty}\twoheadrightarrow R_{\Sigma, \infty}
   \end{equation}
    given by assigning the images of $p$-th power roots of $u_{\lambda, i}$. The maps (\ref{natsurj0}) induce an isomorphism
   \begin{equation}\label{finall1}
    \varinjlim_{\Sigma, \Lambda}K(R_{\Sigma, \Lambda, \infty})\overset{\sim}{\longrightarrow}\varinjlim_{\Sigma}K(R_{\Sigma, \infty})
   \end{equation}
   as both sides of (\ref{finall1}) compute $R\Gamma_{B_{I}}(\fr X_C)$.
   \item Suppose $\fr X$ is non-smooth. We choose an étale map $\fr X\to \Spf(R_{\lambda_0}^{\square})$ as in Notation \ref{notknock}. Then, by Remark \ref{doit}, given a finite set $\Lambda$  containing $\lambda_0$,  for each pair $(\lambda, i)$ with $\lambda\in \Lambda$ and $0\le i\le d$,  either the image of $t_{\lambda, i}$ in $R$ lies in $R^{\times}$, or it cuts out an irreducible component of the special fiber of $\fr X$, and  there exists a unique $0\le i'\le r_{\lambda_0}$ (depending on $i$) such that
   \begin{equation}\label{rell2}
    t_{\lambda, i}=u_{\lambda, \lambda_0, i}\cdot t_{\lambda_0, i'}\;\; \text{  for a unique } u_{\lambda, \lambda_0, i}\in R^\times
   \end{equation}
   (see \cite[(5.27.1)]{CK}). Then, for sufficiently large finite sets $\Sigma\subset R^{\times}$ (containing the invertible $t_{\lambda, i}$'s in $R$ and the $u_{\lambda, \lambda_0, i}$'s as in (\ref{rell2})), we have natural surjections 
   \begin{equation}\label{natsurj}
    R_{\Sigma, \Lambda, \infty}\twoheadrightarrow R_{\Sigma, \{\lambda_0\}, \infty}
   \end{equation}
   given by assigning the images of $p$-th power roots of $t_{\lambda, i}$ in the case the image of $t_{\lambda, i}$ in $R$ lies in $R^\times$, or by assigning the images of $p$-th power roots of $u_{\lambda, \lambda_0, i}$ in the complementary case.    
   The maps (\ref{natsurj}) induce an isomorphism
   \begin{equation}\label{finall2}
    \varinjlim_{\Sigma, \Lambda}K(R_{\Sigma, \Lambda, \infty})\overset{\sim}{\longrightarrow}\varinjlim_{\Sigma}K(R_{\Sigma, \{\lambda_0\}, \infty}).
   \end{equation}
   For later reference, we note here that for a finite set $\Lambda$ as in Notation \ref{notknock} containing $\lambda_0$, and for sufficiently large finite sets $\Sigma$,  for any $\lambda\in \Lambda$ we have a commutative diagram
   \begin{equation}\label{finallab}
   \begin{tikzcd}[row sep=0.5em,column sep=3em]
   & R_{\Sigma, \{\lambda\}, \infty} \arrow[dd, "\wr"]  \\
   R_{\Sigma, \Lambda, \infty} \arrow[ur, twoheadrightarrow] \arrow[dr, twoheadrightarrow]\\
   &  R_{\Sigma, \{\lambda_0\}, \infty}.  \\
   \end{tikzcd}
   \end{equation}
   Here, the bottom diagonal arrow is (\ref{natsurj}), the top diagonal arrow is defined similarly to the latter (with $\lambda$ in place of $\lambda_0$), and the vertical arrow is the  isomorphism defined as follows. For $(\lambda, i)$ such that the image of $t_{\lambda, i}$ in $R$ does not lie in $R^\times$, the relation (\ref{rell2}) implies that  there exists a unique $0\le i'\le r_{\lambda_0}$ (depending on $i$) such that, for any integer $m\ge 0$,  we have
   \begin{equation*}
    t_{\lambda, i}^{1/p^m}=u_{\lambda, \lambda_0, i}^{(m)}\cdot t_{\lambda_0, i'}^{1/p^{m}}\;\; \text{ in } R_{\Sigma, \Lambda, \infty } \text{  for a unique } u_{\lambda, \lambda_0, i}^{(m)}\in R_{\Sigma, \Lambda, \infty }^\times
   \end{equation*}
    using that $t_{\lambda_0, i'}^{1/p^{m}}$ is a unit in $R_{\Sigma, \Lambda, \infty }[1/p]$ and $R_{\Sigma, \Lambda, \infty }$ is integrally closed in $R_{\Sigma, \Lambda, \infty }[1/p]$. Then, the vertical arrow of (\ref{finallab}) is defined by sending, for $m\ge 0$ and $(\lambda, i)$ as before, the element $t_{\lambda, i}^{1/p^m}$ in $R_{\Sigma, \{\lambda\}, \infty}$ to the image of $u_{\lambda, \lambda_0, i}^{(m)}\cdot t_{\lambda_0, i'}^{1/p^{m}}$ in $R_{\Sigma, \{\lambda_0\}, \infty}$; this defines an isomorphism as the $u_{\lambda, \lambda_0, i}^{(m)}$'s are invertible.
   \end{itemize}
   
    Now, let $g:\fr X'=\Spf(R')\to \fr X=\Spf(R)$ be a map of affine $p$-adic formal schemes over $\Spf(\cl O_C)$, with $\fr X$ and $\fr X'$ equipped with data as in Notation \ref{notknock}, for some sets $\Sigma$ and $\Lambda_0$ with $\Lambda_0=\{\lambda_0\}$ in the case $\fr X$ is non-smooth and $\Lambda=\emptyset$ in the smooth case, resp. $\Sigma'$ and $\Lambda'_0$ with $\Lambda'_0=\{\lambda_0'\}$ in the case $\fr X'$ is non-smooth and $\Lambda=\emptyset$ in the smooth case. Then, for a sufficiently large finite set $\Sigma'\subset {R'}^{\times}$, there exists a unique map $g^{\square}:\Spf({R'}_{\Sigma', \Lambda'_0}^\square)\to \Spf(R_{\Sigma, \Lambda_0}^\square)$ of $p$-adic formal schemes over $\Spf(\cl O_C)$ making the following diagram commute
   \begin{equation}\label{diav}
   \begin{tikzcd}
   \fr X' \arrow[r, "g"]\arrow[d] & \fr X\arrow[d] \\
   \Spf({R'}_{\Sigma', \Lambda'_0}^\square) \arrow[r, "g^{\square}"] & \Spf(R_{\Sigma, \Lambda_0}^\square).
   \end{tikzcd}
   \end{equation}
   In fact, suppose $\fr X$ and $\fr X'$ are non-smooth (in the other case one can argue similarly), then by Remark \ref{doit}, for $0\le i\le d$, we have the following relation in $R'$
   $$t_{\lambda_0, i}=(p^{q_{\lambda'_0}})^{n_{\lambda_0, i}}\cdot u_{\lambda_0, \lambda'_0, i}\cdot \prod_{0\le j\le r_{\lambda_0'}}( t_{\lambda'_0, j})^{a_{\lambda_0, i, j}}$$
   for unique $n_{\lambda_0, i}\in \Zz_{\ge 0}$, $u_{\lambda_0, \lambda'_0, i}\in R'^\times$, and $a_{\lambda_0, i, j}\in \Zz_{\ge 0}$ not all positive. 
   
   For a sufficiently large finite set $\Sigma'\subset {R'}^{\times}$, the diagram (\ref{diav}) induces a morphism of Galois covers
   \begin{equation*}\label{diav2}
   \begin{tikzcd}
   \fr X'_{C, \Sigma', \Lambda'_0, \infty} \arrow[r]\arrow[d] & \fr X_{C, \Sigma, \Lambda_0, \infty} \arrow[d] \\
   \fr X'_C \arrow[r] & \fr X_C
   \end{tikzcd}
   \end{equation*}
  which in turn induces a map  $\varinjlim_{\Sigma}K(R_{\Sigma, \Lambda_0, \infty})\to \varinjlim_{\Sigma}K(R'_{\Sigma', \Lambda'_0, \infty})$. The latter map composed with the isomorphism (\ref{finall1}) (in the case $R$ is smooth, and similarly for $R'$) or the isomorphism (\ref{finall2}) (in the case $R$ is non-smooth, and similarly for $R'$), induces a map 
   \begin{equation}\label{avee}
    \varinjlim_{\Sigma, \Lambda}K(R_{\Sigma, \Lambda, \infty})\to \varinjlim_{\Sigma', \Lambda'}K(R'_{\Sigma', \Lambda', \infty})
   \end{equation}
   which does not depend on the choices of $\lambda_0$ and $\lambda_0'$: this follows from the commutative diagram (\ref{finallab}) (and a similar diagram for $R'$).  By construction, the map (\ref{avee}) fits in the following commutative diagram
   \begin{equation*}
   \begin{tikzcd}
   R\Gamma_{\cris}(\mathfrak X_{\cl O_C/p}/A_{\cris}^\times) \arrow[r]\arrow[d] & \varinjlim_{\Sigma, \Lambda}K(R_{\Sigma, \Lambda, \infty})\arrow[d] \\
   R\Gamma_{\cris}(\mathfrak X'_{\cl O_C/p}/A_{\cris}^\times) \arrow[r] & \varinjlim_{\Sigma', \Lambda'}K(R'_{\Sigma', \Lambda', \infty})
   \end{tikzcd}
   \end{equation*}
   where the top, resp. bottom,  horizontal arrow is the filtered colimit $\varinjlim_{\Sigma, \Lambda}$, resp. $\varinjlim_{\Sigma', \Lambda'}$, of the morphisms defined in (\ref{lstm}). This concludes the proof.
  \end{proof}

  \subsection{The comparison with the Hyodo--Kato cohomology}\label{ae2}

 
 
  Now, we have all the ingredients to conclude the proof of Theorem \ref{B=HK}. 
  
  \begin{proof}[Proof of Theorem \ref{B=HK}]
  First, let $\fr X\in \cl M_{\sss, \qcqs}$, i.e. $\fr X$ a qcqs semistable $p$-adic formal scheme over $\Spf(\cl O_C)$. Combining Theorem \ref{firstep} with Theorem \ref{bcn}\listref{bcn:1}, and recalling Remark \ref{st-log}, for any compact interval $I\subset[1/(p-1), \infty)$ with rational endpoints, we have a natural isomorphism
  \begin{equation}\label{firstI}
   R\Gamma_{B_I}(\fr X_C)\simeq (R\Gamma_{\HK}(\fr X_C)\dsolid_{\breve F} B_{\log, I})^{N=0}.
  \end{equation}
  Twisting by the Frobenius, and recalling that $\varphi$ is an automorphism on the Hyodo--Kato cohomology, we deduce that (\ref{firstI}) extends to any compact interval $I\subset (0, \infty)$ with rational endpoints. Then, passing to the derived limit over all such $I$, by Lemma \ref{dinverse} we have a natural isomorphism
  $$R\varprojlim_I R\Gamma_{B_I}(\fr X_C)\simeq R\Gamma_{B}(\fr X_C)$$
   and then, by \cite[Corollary A.67(ii)]{Bosco}, which applies observing that $B_{\log, I}$ is a nuclear $\breve F$-vector space, and  recalling that $R\Gamma_{\HK}(\fr X_C)$ is representable by a bounded complex of $\breve F$-Banach spaces (see Theorem \ref{mainHK}\listref{mainHK:2} and its proof), we obtain a natural isomorphism
  \begin{equation*}
   R\Gamma_{B}(\fr X_C)\simeq (R\Gamma_{\HK}(\fr X_C)\dsolid_{\breve F} \widehat B_{\log})^{N=0}
  \end{equation*}
  where $\widehat B_{\log}:=R\varprojlim_I B_{\log, I}$. Now, we claim that the natural map 
  \begin{equation}\label{oiuy}
   (R\Gamma_{\HK}(\fr X_C)\dsolid_{\breve F} B_{\log})^{N=0} \to (R\Gamma_{\HK}(\fr X_C)\dsolid_{\breve F} \widehat B_{\log})^{N=0}
  \end{equation}
  is an isomorphism. In fact, recalling that $B_{\log}=B[U]$ and $B_{\log, I}=B_I[U]$, both the source and the target of (\ref{oiuy}) identify with the complex $R\Gamma_{\HK}(\fr X_C)\dsolid_{\breve F} B$ via the operator $\exp(N\cdot U)$ on the latter tensor product (cf. Lemma \ref{unex2}; note that the monodromy $N$ is nilpotent on $R\Gamma_{\HK}(\fr X_C)$, by Theorem \ref{mainHK}\listref{mainHK:2}, in particular such operator is well-defined).\footnote{We observe that, in the case when $\fr X_C$ is the base change to $C$ of a rigid-analytic variety defined over $K$, such identifications may be not $\mathscr{G}_K$-equivariant. However, in this case, the map (\ref{oiuy}) is $\mathscr{G}_K$-equivariant. \label{againn}}
  
  Next, let $X$ be a qcqs rigid-analytic variety over $C$. Consider a simplicial object $\fr U_\bullet$ of $\cl M_{\sss, \qcqs}$ such that $\fr U_{\bullet, \eta}\to X$ is a $\eh$-hypercover. To show that for such $X$ we have an isomorphism as in (\ref{decaB}), since the $B$-cohomology satisfies $\eh$-hyperdescent, it suffices to show that the natural map at the top of the following commutative diagram is an isomorphism
 \begin{center}
  \begin{tikzcd}
  \lim_{[n]\in \Delta}(R\Gamma_{\cris}(\fr U_{n, \cl O_C/p}^0/\cl O_{\breve F}^0)_{\Qq_p}\dsolid_{\breve F} B_{\log})^{N=0} \arrow[r]  & (R\Gamma_{\HK}(X)\dsolid_{\breve F} B_{\log})^{N=0}   \\
   \lim_{[n]\in \Delta}(R\Gamma_{\cris}(\fr U_{n, \cl O_C/p}^0/\cl O_{\breve F}^0)_{\Qq_p}\dsolid_{\breve F} B) \arrow[u]\arrow[r] & R\Gamma_{\HK}(X)\dsolid_{\breve F}B \arrow[u].
  \end{tikzcd}
 \end{center}
 Here, arguing as above, the vertical arrows are the isomorphisms defined by the operator $\exp(N\cdot U)$ (recall that the monodromy $N$ is nilpotent on $R\Gamma_{\HK}(X)$, by Theorem \ref{mainHK}\listref{mainHK:2}, in particular such operator is well-defined).\footnote{Similarly to Footnote \ref{againn}, in the case when $X$ is the base change to $C$ of a rigid-analytic variety defined over $K$, such isomorphisms may be not $\mathscr{G}_K$-equivariant. However, in this case, the top horizontal arrow of the diagram above is $\mathscr{G}_K$-equivariant.} We note that the bottom horizontal arrow of the diagram above is an isomorphism by $\eh$-hyperdescent of the Hyodo--Kato cohomology, thanks to \cite[Corollary A.67(ii)]{Bosco} (which applies recalling that $B$ is a Fréchet space and each $R\Gamma_{\cris}(\fr U_{n, \cl O_C/p}^0/\cl O_{\breve F}^0)_{\Qq_p}$ is representable by a complex of nuclear $\breve F$-vector spaces, and using (\ref{finoo})). Then, we deduce that the top horizontal arrow is an isomorphism as well, as desired.
   
   Lastly, for a $X$ connected, paracompact, rigid-analytic variety over $C$, choosing a quasi-compact admissible covering $\{U_n\}_{n\in \Nn}$ of $X$ such that $U_n\subseteq U_{n+1}$, using Theorem \ref{mainHK}\listref{mainHK:2}, the same argument used above reduces us to the previous case.\footnote{Recall that any rigid-analytic variety is assumed to quasi-separated and of finite dimension, \S \ref{convent}.}
  \end{proof}

  \section{\textbf{$B_{\dR}^+$-cohomology}}\label{Bleit2}
   \sectionmark{}

 Our first goal in this section is to compare the $B_{\dR}^+$-cohomology with the de Rham cohomology.

   \subsection{The comparison with the de Rham cohomology}

  In the smooth case, the following result is essentially already contained in \cite[\S 6]{Bosco}, and relies on Scholze's Poincaré lemma for $\Bb_{\dR}^+$. \medskip
  
  In the following, we denote by $R\Gamma_{\dR}(X)$ the de Rham cohomology endowed with the Hodge filtration (Definition \ref{dRsingg}).
  
  \begin{theorem}\label{B_dR=dR}
   Let $X$ be a  connected, paracompact, rigid-analytic variety defined over $K$. Then, we have a natural isomorphism in $D(\Mod_{B_{\dR}^+}^{\ssolid})$
   \begin{equation}\label{decaBdR}
    R\Gamma_{B_{\dR}^+}(X_C)\simeq R\Gamma_{\dR}(X)\dsolid_K B_{\dR}^+
   \end{equation}
   compatible with filtrations, and the action of $\mathscr{G}_K$.
  \end{theorem}
 \begin{proof}
  Assume first that $X$ is smooth. By \cite[Lemma 2.17, Corollary 6.12, Lemma 6.13, Lemma 5.6(ii)]{Bosco} we have an isomorphism in $D(\Mod_{B_{\dR}^+}^{\ssolid})$
  $$R\Gamma_{B_{\dR}^+}(X_C)\simeq R\Gamma(X, \Omega_X^\bullet \solid_K B_{\dR}^+)$$
  compatible with filtrations. Then, in this case, the statement follows from \cite[Theorem 5.20]{Bosco}.
  
  The general case follows by $\eh$-hyperdescent, using \cite[Corollary A.67(ii)]{Bosco}, and observing, for the compatibility with filtrations part, that the filtration on $R\Gamma_{\dR}(X)$ is finite, Proposition \ref{bound0} (as $X$ is also assumed to be of finite dimension, \S \ref{convent}).
 \end{proof}
 
 The same argument and references used in the proof of Theorem \ref{B_dR=dR} also show the following result, which generalizes \cite[Theorem 1.8(ii)]{Bosco} to the singular case:
 
 \begin{theorem}\label{tosingg}
  Let $X$ be a connected, paracompact, rigid-analytic variety defined over $K$. Then, we have a natural isomorphism in $D(\Mod_{B_{\dR}^+}^{\ssolid})$
  $$R\Gamma_{\pet}(X_C, \Bb_{\dR}^+)\simeq \Fil^0(R\Gamma_{\dR}(X)\dsolid_K B_{\dR}).$$
 \end{theorem}

 \subsubsection{Compatibility}
 
 Our next goal is to prove that the comparison of Theorem \ref{B_dR=dR} is compatible with the comparison of Theorem \ref{B=HK}.
 
 \begin{theorem}\label{compatib}
   Let $X$ be a connected, paracompact, rigid-analytic variety defined over $K$. 
   \begin{enumerate}[(i)]
    \item\label{compatib:1}{\normalfont (Hyodo--Kato isomorphism over $B_{\dR}^+$)}
   We have a natural isomorphism in $D(\Mod^{\ssolid}_{B_{\dR}^+})$
   \begin{equation}\label{HKB}
    R\Gamma_{\HK}(X_C)\dsolid_{\breve F} B_{\dR}^+\overset{\sim}{\longrightarrow} R\Gamma_{\dR}(X)\dsolid_{K}B_{\dR}^+.
   \end{equation}
  \item\label{compatib:2}{\normalfont (Compatibility)} The isomorphism (\ref{decaBdR}) of Theorem \ref{B_dR=dR} is compatible with the isomorphism (\ref{decaB}) of Theorem \ref{B=HK}, i.e. we have a commutative diagram as follows
 \begin{center}
  \begin{tikzcd}
  R\Gamma_{B}(X_C) \arrow[d]\arrow[r, "\sim"]  & (R\Gamma_{\HK}(X_C)\dsolid_{\breve F}B_{\log})^{N=0} \arrow[d] \\
    R\Gamma_{B_{\dR}^+}(X_C) \arrow[r, "\sim"] & R\Gamma_{\dR}(X)\dsolid_{K}B_{\dR}^+ 
  \end{tikzcd}
 \end{center}
  where the left vertical map is induced by the inclusion $\Bb\hookrightarrow \Bb_{\dR}^+$, and the right vertical map is induced by (\ref{HKB}).
   \end{enumerate}
 \end{theorem}

  We refer the reader to Theorem \ref{compatib2} for a version of Theorem \ref{compatib} for rigid-analytic varieties defined over $C$.
 
 \subsection{The comparison with the infinitesimal cohomology over $B_{\dR}^+$}
 
  To prove Theorem \ref{compatib} we will relate the $B_{\dR}^+$-cohomology of Definition \ref{defdr} to the infinitesimal cohomology over $B_{\dR}^+$ in a way that is compatible with the comparison between the $B$-cohomology and the log-crystalline cohomology over $A_{\cris}$ (following from Theorem \ref{firstep}). \medskip
  
  \subsubsection{The infinitesimal cohomology over $B_{\dR}^+$ and its filtrations}\label{guoinf}
  
  We begin with some recollections on the infinitesimal cohomology over $B_{\dR}^+$ of rigid-analytic varieties, introduced in \cite{Guo2}.

 \begin{df}[Infinitesimal cohomology over $B_{\dR}^+$]\label{infccc}
 Let $X$ be a rigid-analytic variety over $C$. Given an integer $m\ge 1$, denote $B_{\dR, m}^+=B_{\dR}^+/\xi^m$. 
 
 The \textit{infinitesimal site $(X/B_{\dR, m}^+)_{\inf}$ of $X$ over $B_{\dR, m}^+$} is defined as follows: 
  \begin{itemize}
   \item  
  the underlying category has objects the pairs $(U, T)$ where $U$ is an open subspace of $X$ and $U\to T$ is an infinitesimal thickening of adic spaces with $T$ an adic space of topological finite presentation over $B_{\dR, m}^+$, and morphisms $(U, T)\to (U', T')$ with $U\to U'$ an open immersion and $T\to T'$ a compatible map of adic spaces over $B_{\dR, m}^+$;
  \item the coverings are given by the families of morphisms $\{(U_i, T_i)\to (U, T)\}$ with $U_i\to U$ and $T_i\to T$ coverings for the analytic topology.
  \end{itemize}
  The \textit{infinitesimal site of $X$ over $B_{\dR}^+$} is defined as the direct limit of sites (in the sense of \cite[Exposé VI, \S 8]{SGA4})
  $$(X/B_{\dR}^+)_{\inf}:=\varinjlim_{m}(X/B_{\dR, m}^+)_{\inf}.$$
  The \textit{infinitesimal structure sheaf $\cl O_{\inf}$ on $(X/B_{\dR}^+)_{\inf}$} is the sheaf with values in $\Mod_{B_{\dR}^+}^{\cond}$ sending $(U, T)$ to $\cl O_T(T)$, and the \textit{infinitesimal cohomology of $X$ over $B_{\dR}^+$} is defined as
  \begin{equation*}\label{infBdr^+}
   R\Gamma_{\inf}(X/B_{\dR}^+):=R\Gamma((X/B_{\dR}^+)_{\inf}, \cl O_{\inf})\in D(\Mod_{B_{\dR}^+}^{\cond}).
  \end{equation*}
 \end{df}
 
  One can also define a \textit{big infinitesimal site} version of the (small) infinitesimal sites defined above (cf. \cite[Definition 2.1.2]{Guo2}). \medskip

 \begin{df}[Infinitesimal filtration]\label{definf} Let $X$ be a rigid-analytic variety over $C$. We define the \textit{infinitesimal filtration on the infinitesimal cohomology of $X$ over $B_{\dR}^+$} as the $\Nn^{\op}$-indexed filtration
 $$\Fil_{\inf}^{\star}R\Gamma_{\inf}(X/B_{\dR}^+):=R\Gamma((X/B_{\dR}^+)_{\inf},\cl J_{\inf}^{\star})$$
  induced on the $i$-th level by the subsheaf $\cl J_{\inf}^i \subset \cl O_{\inf}$ on $(X/B_{\dR}^+)_{\inf}$, where $\cl J_{\inf}$ is the kernel ideal of the natural map from the infinitesimal structure sheaf $\cl O_{\inf}$ to the pullback, on the infinitesimal site, of the analytic structure sheaf of $X$.
 \end{df}

 We recall that the infinitesimal cohomology over $B_{\dR}^+$ satisfies $\eh$-hyperdescent:

  \begin{lemma}\label{ehhyp}
   The presheaf
   $$\Rig_C^{\op}\to D(\Mod_{B_{\dR}^+}^{\cond}): Y\mapsto R\Gamma_{\inf}(Y/B_{\dR}^+)$$
   satisfies $\eh$-hyperdescent.
  \end{lemma}
  \begin{proof}
   The statement follows from \cite[Theorem 6.2.5]{Guo2}.
  \end{proof}

  On the other hand, the pieces of the infinitesimal filtration on the infinitesimal cohomology over $B_{\dR}^+$ (Definition \ref{definf}) do not satisfy $\eh$-hyperdescent in general.\footnote{Similarly, the pieces of the infinitesimal filtration on the infinitesimal cohomology over $C$ do satisfy $\eh$-hyperdescent. In fact, supposing the contrary, by \cite[Theorem 1.2.1(i)]{Guo2} and Proposition \ref{baseh} we would have, for any rigid-analytic variety $X$ over $C$, a natural filtered isomorphism between the infinitesimal cohomology of $X$ over $C$ and the de Rham cohomology of $X$ over $C$ (Definition \ref{dRsingg}). Now, assuming that $X$ is a complete intersection affinoid, there is a natural filtered isomorphism between the infinitesimal cohomology of $X$ over $C$ and the cohomology of the analytic derived de Rham complex of $X$ over $C$, \cite[Corollary 5.5.2]{Guo2}; but, the graded pieces of the latter do not vanish if $X$ is not smooth (recalling that the $i$-th graded piece of the analytic derived de Rham complex of $X$ over $C$ can be identified with a shift of the $i$-fold wedge product of the analytic cotangent complex of $X$ over $C$), instead, by Proposition \ref{bound0}, the graded pieces of the de Rham cohomology of $X$ over $C$ eventually vanish.} For this reason, we introduce the following filtration on the infinitesimal cohomology over $B_{\dR}^+$ that is closer to the Hodge filtration on the de Rham cohomology (Definition \ref{dRsingg}) and it will be crucial in the formulation of the semistable conjecture for proper (possibly singular) rigid-analytic varieties over $C$ (see Theorem \ref{propsing}). The following definition is based on Proposition \ref{baseh} (and Remark \ref{oyt}).

  \begin{df}[Hodge filtration]\label{hdg}  Let $X$ be a rigid-analytic variety over $C$. We define the \textit{Hodge filtration on the infinitesimal cohomology of $X$ over $B_{\dR}^+$} as the $\Nn^{\op}$-indexed filtration
 $$\Fil_{\Hdg}^{\star}R\Gamma_{\inf}(X/B_{\dR}^+)$$
  given on the $i$-th level by the cohomology on $X$ of the hypersheaf on $\Rig_{C, \eh}$  associated to the presheaf
  $$\RigSm_{C}^{\op}\to D(\Mod_{B_{\dR}^+}^{\cond}):Y \mapsto \Fil_{\inf}^{i}R\Gamma_{\inf}(Y/B_{\dR}^+).$$
  \end{df}

  In the smooth case the Hodge filtration on the infinitesimal cohomology over $B_{\dR}^+$ (Definition \ref{hdg}) agrees with the infinitesimal filtration (Definition \ref{definf}).
  
  \begin{prop}\label{ultimo} {\normalfont(Hodge filtration in the smooth case)}  Let $X$ be a smooth rigid-analytic variety over $C$. We have a natural isomorphism of filtered objects
  $$\Fil_{\inf}^{\star}R\Gamma_{\inf}(X/B_{\dR}^+)\overset{\sim}{\longrightarrow}\Fil_{\Hdg}^{\star}R\Gamma_{\inf}(X/B_{\dR}^+).$$
  \end{prop}
  
  We will prove Proposition \ref{ultimo} in the next subsection, together with the following already announced comparison of the $B_{\dR}^+$-cohomology with the infinitesimal cohomology over $B_{\dR}^+$.

  \begin{theorem}\label{secondstep1}
    For any rigid-analytic variety $X$ over $C$,  we have a natural isomorphism in $D(\Mod^{\cond}_{B_{\dR}^+})$
    \begin{equation}\label{BdR-cris}
     R\Gamma_{B_{\dR}^+}(X)\simeq R\Gamma_{\inf}(X/B_{\dR}^+)
    \end{equation}
   compatible with filtrations, endowing the $B_{\dR}^+$-cohomology with the filtration decalée (Definition \ref{beilifildef}) and the infinitesimal cohomology over $B_{\dR}^+$ with the Hodge filtration (Definition \ref{hdg}).
  \end{theorem}

  \subsubsection{Proofs}
  
  We want to prove Theorem \ref{compatib}, Proposition \ref{ultimo} and Theorem \ref{secondstep1}. \medskip

  As a first step toward Theorem \ref{compatib}, we need to construct a natural map from the log-crystalline cohomology over $A_{\cris}$ to the infinitesimal cohomology over $B_{\dR}^+$.
  
  \begin{lemma}\label{cristodR}
   Let $\fr X$ be a semistable $p$-adic formal scheme over $\cl O_C$, and let $X=\fr X_C$ denote its generic fiber. Then, there exists a natural morphism
   \begin{equation}\label{logtodR}
    R\Gamma_{\cris}(\fr X_{\cl O_C/p}/A_{\cris}^\times)\to R\Gamma_{\inf}(X/B_{\dR}^+).
   \end{equation}
  \end{lemma}
  \begin{proof}
  We can assume that $\fr X$ is affine. We note that we have a natural isomorphism $$R\Gamma_{\cris}(\fr X_{\cl O_C/p}/A_{\cris}^\times)\simeq R\Gamma_{\cris}(\fr X/A_{\cris}^\times).$$ Then, it suffices to construct, for each integer $m\ge 1$,  a morphism of big sites
  $$f:  (X/B_{\dR, m}^+)_{\INF}\to (\fr X/A_{\cris}^\times)_{\CRIS}$$
  recalling that the restriction functor from the big topos the small one preserves cohomology (see \cite[Corollary 2.2.8]{Guo2} for the infinitesimal topos). We define $f$ via the continuous functor sending  $\fr U\to \Spf(A)$ in the big log-crystalline site $(\fr X/A_{\cris}^\times)_{\CRIS}$ to  $\fr U_C\to \Spa(A\otimes_{A_{\cris}}B_{\dR, m}^+)$ in the big infinitesimal site $(X/B_{\dR, m}^+)_{\INF}$ (forgetting the log structures). One checks that $f$ is a well-defined morphism of sites, with the help of \cite[00X4, 00X5]{Thestack}.
  \end{proof}

  Before proving Theorem \ref{compatib}, we will show the following intermediate compatibility result.

  \begin{prop}\label{secondstep}
    Let $X$ be the generic fiber of a qcqs semistable formal scheme $\fr X$ defined over $\Spf(\cl O_C)$. Let $I=[1, r]\subset (0, \infty)$ be an interval with rational endpoints. Then, the isomorphism (\ref{B-cris}) is compatible with the isomorphism (\ref{BdR-cris}),  i.e. we have a commutative diagram as follows
 \begin{center}
  \begin{tikzcd}
  R\Gamma_{B_I}(X) \arrow[d]\arrow[r, "\sim"]  & R\Gamma_{\cris}(\fr X_{\cl O_C/p}/A_{\cris}^\times)\dsolid_{A_{\cris}}B_I \arrow[d] \\
    R\Gamma_{B_{\dR}^+}(X) \arrow[r, "\sim"] & R\Gamma_{\inf}(X/B_{\dR}^+) 
  \end{tikzcd}
 \end{center}
  where the left vertical map is induced by the inclusion $\Bb_I\hookrightarrow \Bb_{\dR}^+$, and the right vertical map is induced by the morphism (\ref{logtodR}) constructed in Lemma \ref{cristodR}.
  \end{prop}
  
  To show Proposition \ref{secondstep}, we will prove Theorem \ref{secondstep1} going over the same steps as in the proof of Theorem \ref{firstep}. We begin with the first step, corresponding to Lemma \ref{primitive}.

  \begin{lemma}\label{same}
   In the setting of Notation \ref{notaz}, for any integer $m\ge 1$, we have a $B_{\dR, m}^+$-linear quasi-isomorphism
   $$\Kos_{A_{\inf}(R)}(\partial_1, \ldots,\partial_d)\otimes_{A_{\inf}}^{\LL}B_{\dR}^+/\xi^m\overset{\sim}{\to}\Kos_{B_{\dR}^+/\xi^m(R)}(\partial_1, \ldots,\partial_d)\overset{\sim}{\to}L\eta_{t}\Kos_{(\Bb_{\dR}^+/\Fil^m)(R_{\infty})}(\gamma_1-1, \ldots, \gamma_d-1)$$
   compatible with the quasi-isomorphism (\ref{crucc}).
   \end{lemma}
   \begin{proof}
    Since $\mu$ and $t$ differ by a unit in $B_{\dR}^+/\xi^m$, as in the proof of Lemma \ref{primitive} we can reduce to showing that the element $\mu$ kills $H^i_{\cond}(\Gamma, \Aa_{\inf}(R_{\infty})^{\nonint}\otimes_{A_{\inf}}B_{\dR}^+/\xi^m)$ for all $i\in \Zz$.  Let $N_{\infty}:=\Aa_{\inf}(R_{\infty})^{\nonint}$. We proceed by induction on $m\ge 1$. Since $\xi$ is a non-zero-divisor in $\Aa_{\inf}(R_{\infty})\supset  N_{\infty}$, we have the following exact sequence
    $$0\to N_{\infty}\otimes_{A_{\inf}}C(m)\to N_{\infty}\otimes_{A_{\inf}}B_{\dR}^+/\xi^{m+1}\to N_{\infty}\otimes_{A_{\inf}}B_{\dR}^+/\xi^m\to 0$$
   which allows us to reduce to the case $m=1$. Then, it suffices to show that the element $\varepsilon_p-1$ kills $H^i_{\cond}(\Gamma, \widehat{\cl O}^+(R_{\infty})^{\nonint})$ for all $i \in \Zz$, which follows from \cite[Lemma 5.5]{Scholze} (and \cite[Proposition B.3]{Bosco}).
   \end{proof}
  
   The following byproduct of Lemma \ref{same} will be useful later on.
  
  \begin{cor}\label{mainfil}
   In the setting of Notation \ref{notaz}, for any $i\ge 0$, we have a $B_{\dR}^+$-linear quasi-isomorphism
  \begin{equation}\label{qisos}
   \xi^{\max(i-\bullet, 0)}\Omega_{B_{\dR}^+(R)}^\bullet \overset{\sim}{\to}\Fil^i R\Gamma_{B_{\dR}^+}(X)
  \end{equation}
  where $\Omega_{B_{\dR}^+(R)}^\bullet:=\Kos_{B_{\dR}^+(R)}(\partial_1, \ldots,\partial_d)$.
  \end{cor}
  \begin{proof}
   The statement follows combining Lemma \ref{condstep}, Lemma \ref{same}, and Lemma \ref{dinverse}.  In fact, by induction on $i\ge 0$, we can reduce to showing (\ref{qisos}) on graded pieces.
  \end{proof}

  As done in \S \ref{functisom}, in order to construct a functorial isomorphism, we need to introduce more general coordinates. For this, we resume here the setting of \cite[\S 6.3]{CK}.\footnote{We warn the reader that our notation slightly differs from \textit{loc. cit.}.}

  \begin{notation}\label{setse}\
  \begin{itemize}
   \item Let $X=\Spa(A, A^\circ)$ an affinoid space over $\Spa(C, \cl O_C)$ that is the base change to $\Spa(C, \cl O_C)$ of an affinoid space $X_0=\Spa(A_0, A_0^\circ)$ defined over $\Spa(L, \cl O_L)$, for some finite subextension $\breve F\subset L\subset C$, and admitting an étale $\Spa(L, \cl O_L)$-morphism
    \begin{equation*}
     X_0\to \Spa(L\langle T_0, \ldots, T_r, T_{r+1}^{\pm 1}, \ldots, T_d^{\pm 1} \rangle /(T_0\cdots T_r -p^q), \cl O_L\langle T_0, \ldots, T_r, T_{r+1}^{\pm 1}, \ldots, T_d^{\pm 1} \rangle /(T_0\cdots T_r -p^q))
    \end{equation*}
    for some $0\le r\le d$, and  $q\in \Qq_{>0}$.
    
    Assume, in addition, that there are finite subsets $\Psi_0\subset (A_0^\circ)^\times$ and $\Xi_0\subset A_0^\circ\cap A_0^\times$ such that the $L$-linear map
    \begin{equation}\label{bccc0}
     L\langle (X_u^{\pm 1})_{u\in \Psi_0}, (X_a)_{a\in \Xi_0}\rangle\to A_0, \;\;\;\;\; X_u\mapsto u,\; X_a\mapsto a,
    \end{equation}
    is surjective. In particular, there are finite subsets $\Psi\subset (A^\circ)^\times$ and $\Xi\subset A^\circ\cap A^\times$ such that the $C$-linear map
    \begin{equation}\label{bccc}
     C\langle (X_u^{\pm 1})_{u\in \Psi}, (X_a)_{a\in \Xi}\rangle\to A, \;\;\;\;\; X_u\mapsto u,\; X_a\mapsto a,
    \end{equation}
    is surjective.\footnote{The descent (\ref{bccc0}) of (\ref{bccc}) to $L$  is needed in \cite[Lemma 6.3.8]{CK}, which we will use in Lemma \ref{ainn} below.}
   \item   We consider the affinoid perfectoid cover 
    \begin{equation}\label{md}
     C\langle (X_u^{\pm 1})_{u\in \Psi}, (X_a)_{a\in \Xi}\rangle\to C\langle (X_u^{\pm 1/p^{\infty}})_{u\in \Psi}, (X_a^{1/p^{\infty}})_{a\in \Xi}\rangle
    \end{equation}
    with Galois group
    $$\Gamma_{\Psi, \Xi}:=\textstyle \prod_{\Psi\sqcup \Xi}\Zz_p(1)\cong \Zz_p^{|\Psi\sqcup \Xi|}.$$
    We denote by $(\gamma_{u})_{u\in \Psi\sqcup \Xi}$ the canonical generators of $\Gamma_{\Psi, \Xi}$. 
    
    The base change of (\ref{md}) along the surjection (\ref{bccc}) gives the following affinoid perfectoid pro-étale cover of $X$
   $$X_{\Psi, \Xi, \infty}:=\Spa(A_{\Psi, \Xi, \infty}, A_{\Psi, \Xi, \infty}^+)\to X$$
    with Galois group $\Gamma_{\Psi, \Xi}$.
    \item   
   Given $\Mm$ any pro-étale period sheaf of \S \ref{petsheaves}, we set $$\Mm(A_{\Psi, \Xi, \infty}^+):=\Mm(X_{\Psi, \Xi, \infty})$$
   and we regard it as a condensed ring.
  \end{itemize}
   \end{notation}
   
   \begin{rem}\label{almob}
    As observed in \cite[\S 6.2 and \S 6.3]{CK}, for any smooth rigid-analytic variety $X$ over $C$, the affinoid spaces $\Spa(A, A^\circ)$ of Notation \ref{setse} form an basis for the analytic topology of $X$.
   \end{rem}

  The following result should be regarded as the analogue of Lemma \ref{cruc} over $B_{\dR}^+$.
  
  \begin{lemma}\label{primitive2}
   In the setting of Notation \ref{setse}, we have a filtered quasi-isomorphism
   \begin{equation}\label{coorcompl2}
    R\Gamma_{\inf}(X/B_{\dR}^+)\simeq \Omega^\bullet_{D_{\Psi, \Xi}(A)/B_{\dR}^+}:=\Kos_{D_{\Psi, \Xi}(A)}\left((\partial_{u})_{u\in \Psi},(\partial_{a})_{a\in \Xi}\right)
   \end{equation}
   where $D_{\Psi, \Xi}(A)=\varprojlim_{m\ge 1}D_{\Psi, \Xi, m}(A)$, and, for each $m\ge 1$, $D_{\Psi, \Xi, m}(A)$ is the $B_{\dR}^+/\xi^m$-algebra representing the envelope of
   \begin{equation}\label{envi}
    \Spa(A)\hookrightarrow \Spa(B_{\dR}^+/\xi^m\langle (X_u^{\pm 1})_{u\in \Psi}, (X_a)_{a\in \Xi}\rangle) \text{ over } \Spa(C)\hookrightarrow \Spa(B_{\dR}^+/\xi^m).
   \end{equation}
   Here,  $\partial_u:=\frac{\partial}{\partial\log(X_u)}=X_u\cdot\frac{\partial}{\partial X_u}$ for $a\in \Psi\cup \Xi$, and the right-hand side of (\ref{coorcompl2}) in endowed with the infinitesimal filtration, defined on the $i$-th level, for $i\ge 0$, as follows:
   \begin{equation}\label{inffil}
    \Fil^i\Omega^\bullet_{D_{\Psi, \Xi}(A)/B_{\dR}^+}:=J^{\max(i-\bullet, 0)}\Omega^\bullet_{D_{\Psi, \Xi}(A)/B_{\dR}^+}
   \end{equation}
   where $J:=\varprojlim_m J_m$ with $J_m$ the ideal corresponding to the closed immersion (\ref{envi}).
  \end{lemma}
  \begin{proof}
   The statement follows from \cite[Theorem 4.1.1, Theorem 7.2.3]{Guo2}.
  \end{proof}

  Keeping the notation of the above lemma, we state the following result.
  
  \begin{lemma}\label{ainn}
   In the setting of Notation \ref{setse}, for any sufficiently large $\Psi$ and $\Xi$, we have an isomorphism in $D(\Mod^{\ssolid}_{B_{\dR}^+})$
  \begin{equation}\label{todeRham}
    \Omega^\bullet_{D_{\Psi, \Xi}(A)/B_{\dR}^+}\simeq \Omega^\bullet_{A_0/K}\solid_L B_{\dR}^+
  \end{equation}
  compatible with filtrations, where the left-hand side is endowed with the infinitesimal filtration (\ref{inffil}), and the right-hand side is endowed with the tensor product filtration.
  \end{lemma}
  \begin{proof} We consider the $B_{\dR}^+$-algebra $\cl O D_{\Psi, \Xi}(A)$ defined as the completion of 
  $$(A_0\solid_L B_{\dR}^+)\solid_{B_{\dR}^+}D_{\Psi, \Xi}(A)\to A_0$$
  along its kernel. Then, we have the following natural 
  \begin{equation}\label{zigzag0}
   D_{\Psi, \Xi}(A)\rightarrow \cl O D_{\Psi, \Xi}(A)\leftarrow A_0\solid_L B_{\dR}^+
  \end{equation}
  and, denoting by $\Omega^\bullet_{\cl O D_{\Psi, \Xi}(A)/B_{\dR}^+}$ the de Rham complex associated to $\cl O D_{\Psi, \Xi}(A)$ over $B_{\dR}^+$, we consider the natural maps of complexes
  \begin{equation}\label{zigzag}
     \Omega^\bullet_{D_{\Psi, \Xi}(A)/B_{\dR}^+}\rightarrow \Omega^\bullet_{\cl O D_{\Psi, \Xi}(A)/B_{\dR}^+}\leftarrow \Omega^\bullet_{A_0/K}\solid_L B_{\dR}^+.
   \end{equation}
  We claim that the zigzag (\ref{zigzag}) is a filtered quasi-isomorphism. As in \cite[Lemma 13.13]{BMS1}, it suffices to check that,  for sufficiently large $\Psi$ and $\Xi$, we have
  $$D_{\Psi, \Xi}(A)\cong (A_0\solid_L B_{\dR}^+)\llbracket(X_u-\widetilde u)_{u\in (\Psi\sqcup \Xi)\setminus\{T_1, \ldots, T_d\}}\rrbracket,\;\;\;\;\;\cl O D_{\Psi, \Xi}(A)\cong (A_0\solid_L B_{\dR}^+)\llbracket(X_u-\widetilde u)_{u\in \Psi\sqcup \Xi}\rrbracket.$$
  The first isomorphism follows from \cite[Lemma 6.3.8]{CK}, observing that the completed tensor product of Tate $L$-algebras, appearing in \textit{loc. cit.}, agrees with the solid tensor product $\solid_L$. The second isomorphism follows from the first one, using that the completion of $A_0\solid_L A_0\to A_0$ along its kernel is isomorphic to $A_0\llbracket(u\otimes 1-1\otimes u)_{u\in \{T_1, \ldots, T_d\}}\rrbracket$.
  \end{proof}

  We proceed by constructing \textit{in coordinates} a natural map from the log-crystalline cohomology over $A_{\cris}$ to the infinitesimal cohomology over $B_{\dR}^+$, with an eye to the compatibility between (\ref{B-cris}) and (\ref{BdR-cris}) that we want to prove.  For this, we first need to relate the setting of Notation \ref{notknock} to the one of Notation \ref{setse}.

  \begin{notation}\label{relnot}
   In the setting of Notation \ref{notknock}, we put $\Spa(A, A^\circ)=X:=\fr X_C$, and
   $$\Psi:=\textstyle \{t_\sigma\}_{\sigma\in \Sigma}\cup \bigcup_{\lambda\in \Lambda}\{t_{\lambda, r_{\lambda}+1}, \ldots, t_{\lambda, d}\}, \;\;\;\;\; \Xi:=\bigcup_{\lambda\in \Lambda}\{t_{\lambda, 1}, \ldots, t_{\lambda, r_{\lambda}}\}.$$
   These choices satisfy the assumptions of Notation \ref{setse}, therefore in the following we can retain the notation of \textit{loc. cit.} using such choices.
  \end{notation}

  \begin{lemma}\label{ccost}
   In the setting of Notation \ref{relnot}, there is a natural morphism
   \begin{equation*}\label{cristobdr}
    \Omega^\bullet_{D_{\Sigma, \Lambda}(R)/A_{\cris}}\to \Omega^\bullet_{D_{\Psi, \Xi}(A)/B_{\dR}^+}
   \end{equation*}
   which, under the isomorphisms (\ref{coorcompl1}) and (\ref{coorcompl2}), is compatible with the morphism (\ref{logtodR}) constructed in Lemma \ref{cristodR}.
  \end{lemma}
  \begin{proof}
  We need to construct, in the setting of Notation \ref{relnot}, for each integer $m\ge 1$, a natural map
  \begin{equation}\label{des}
   D_{\Sigma, \Lambda}(R)\to D_{\Psi, \Xi, m}(A)
  \end{equation}
  compatible with the natural maps to $A$.\footnote{This is done in \cite[\S 6.5]{CK}, which we rephrase here in a slightly different way, proceeding as in the proof of Lemma \ref{tricky}.}
  By \cite[\S 6.5]{CK}, there exist a $p$-adically complete ring of definition $D_{\Psi, \Xi, m}(A)_0$ of $D_{\Psi, \Xi, m}(A)$, and a commutative diagram of log rings as follows
  \begin{equation*}\label{Q2}
  \begin{tikzcd}
   (A_{\inf, \Sigma, \Lambda}^{\square}\otimes_{A_{\inf}}A_{\cris}/p^n, M_{\inf, \Sigma, \Lambda}^{\square}) \arrow[two heads]{r}\arrow[d] &  (R/p, M_R) \arrow[-,double line with arrow={-,-}]{d} \\
  (D_{\Psi, \Xi, m}(A)_0/p^n, N)^a \arrow[two heads]{r} &  (R/p, M_R).
  \end{tikzcd}
  \end{equation*}
  Here,  $R/p$ is equipped with the pullback of the canonical log structure $M_R$ on $R$, and the ring  $A_{\inf, \Sigma, \Lambda}^{\square}\otimes_{A_{\inf}}A_{\cris}/p^n$ is endowed with the pullback of the log structure $M_{\inf, \Sigma, \Lambda}^{\square}$ on $A_{\inf, \Sigma, \Lambda}^{\square}$. Moreover, $D_{\Psi, \Xi, m}(A)_0/p^n$ is equipped with the pullback of the log structure on $D_{\Psi, \Xi, m}(A)_0$ associated to the following pre-log structure: as in Lemma \ref{tricky}, we set $N:=(h^{\gp})^{-1}(M_{\inf, \Sigma, \Lambda}^{\square})$ where $h^{\gp}$ denotes the morphism of groups associated to the natural morphism of monoids $h:M_{\inf, \Sigma, \Lambda}^{\square}\to M_R$; the argument of \cite[\S 6.5]{CK} shows that the natural map $M_{\inf, \Sigma, \Lambda}^{\square}\to D_{\Psi, \Xi, m}(A)_0$ uniquely extends to a map $N\to D_{\Psi, \Xi, m}(A)_0$. The resulting surjective map of log rings at the bottom of the diagram (\ref{Q}) is exact by construction, hence, the universal property of the log PD envelope $D_{\Sigma, \Lambda}(R)/p^n$ gives the desired natural map $D_{\Sigma, \Lambda}(R)/p^n\to D_{\Psi, \Xi, m}(A)_0/p^n$.  Since both $D_{\Psi, \Xi, m}(A)_0$ and $D_{\Sigma, \Lambda}(R)$ are $p$-adically complete (for the latter, recall \cite[Lemma 5.29]{CK}), we obtain a map (\ref{des}) as desired. 
  \end{proof}

  We can finally prove the main results of this section.
   
   \begin{proof}[Proof of Theorem \ref{secondstep1}]
    The argument of will be analogous to the one in the proof of Theorem \ref{firstep}, in order to prove the compatibility stated in Proposition \ref{secondstep}. 
    
    As both the $B_{\dR}^+$-cohomology, together with its filtration decalée, and the infinitesimal cohomology over $B_{\dR}^+$, together with its Hodge filtration, satisfy $\eh$-hyperdescent (for the infinitesimal cohomology, see Lemma \ref{ehhyp} and Definition \ref{hdg}), it suffices to prove the statement $\eh$-locally on $X$ in a functorial way. Thus, by Proposition \ref{propobv}, we can place ourselves in the setting of Notation \ref{relnot}, with sufficiently large $\Psi$ and $\Xi$.
    
    Fix $m\ge 1$. As in Lemma \ref{condstep}, we have a quasi-isomorphism 
   \begin{equation}\label{prevghost}
    R\Gamma(\fr X_C, L\eta_t R\nu_*(\Bb_{\dR}^+/\Fil^m))\simeq L\eta_{t}\Kos_{(\Bb_{\dR}^+/\Fil^m)(A^+_{\Psi, \Xi, \infty})}((\gamma_u-1)_{u\in \Psi\sqcup \Xi}).
   \end{equation}   
   Next, we claim that we have a commutative diagram as follows, whose arrows are natural quasi-isomorphisms
  \begin{equation}\label{diaprinc2}
  \begin{tikzcd}
   \Kos_{A_{\cris}(R)_{\lambda}}\left((\partial_{\lambda, i})_{1\le i\le d} \right)\dsolid_{A_{\cris}}B_{\dR}^+/\xi^m \arrow[r]\arrow[d] & \eta_{t}\Kos_{(\Bb_{\dR}^+/\Fil^m)(R_{\lambda, \infty})}((\gamma_{\lambda, i}-1)_{1\le i\le d})\arrow[d] \\
  \Kos_{D_{\Psi, \Xi, m}(A)}\left((\partial_{u})_{u\in \Psi\sqcup \Xi}\right) \arrow[r] & \eta_{t}\Kos_{(\Bb_{\dR}^+/\Fil^m)(A^+_{\Psi, \Xi, \infty})}((\gamma_u-1)_{u\in \Psi\sqcup \Xi}).
  \end{tikzcd}
  \end{equation}
  
  \begin{enumerate}[(i)]
   \item The right vertical map of (\ref{diaprinc2}) is a quasi-isomorphism since, by Lemma \ref{condstep}, both the target and the source are quasi-isomorphic to the complex  $R\Gamma(\fr X_C, L\eta_t R\nu_*(\Bb_{\dR}^+/\Fil^m))$. 
   \item\label{num:2} The top horizontal arrow of (\ref{diaprinc2}) is the quasi-isomorphism obtained combining Lemma \ref{same} with (\ref{afa}).
  \item  The left vertical map of (\ref{diaprinc2}) is induced by the one constructed in Lemma \ref{ccost}. We observe that both the target and the source of this map are derived $\xi$-adically complete.\footnote{For the latter, one can use for example the quasi-isomorphism of part \listref{num:2} combined with \cite[Proposition 3.4.4, Lemma 3.4.14]{BS}.}
  Then, by the derived Nakayama lemma, it suffices to show that such a map is a quasi-isomorphism for $m=1$. In this case, both the target and the source compute the de Rham cohomology $R\Gamma_{\dR}(\fr X_C)$. The former follows reducing the quasi-isomorphism (\ref{todeRham}) mod $\xi$. For the latter, we observe that (for $m=1$) by \cite[Proposition 5.13]{CK}, and Proposition \ref{solid-vs-padic}, the source computes $(R\Gamma_{\cris}(\fr X_{\cl O_C/p}/A_{\cris}^\times)\widehat\otimes^{\LL}_{A_{\cris}}\cl O_C)_{\Qq_p}$, and by \cite[(1.11.1) and (1.8.1)]{Beili} we have a quasi-isomorphism
  $$(R\Gamma_{\cris}(\fr X_{\cl O_C/p}/A_{\cris}^\times)\widehat\otimes^{\LL}_{A_{\cris}}\cl O_C)_{\Qq_p}\overset{\sim}{\to}R\Gamma_{\log\dR}(\fr X)_{\Qq_p}\overset{\sim}{\to}R\Gamma_{\dR}(\fr X_C).$$
  \item To construct the bottom horizontal arrow of (\ref{diaprinc2}), we first note that, similarly to (\ref{kospass}), we have a  quasi-isomorphism
    \begin{equation}\label{kospass2}
     \Kos_{D_{\Psi, \Xi, m}(A)}\left((\partial_{u})_{u\in \Psi\sqcup \Xi}\right)\overset{\sim}{\to}\eta_{t}\Kos_{D_{\Psi, \Xi, m}(A)}\left((\gamma_{u}-1)_{u\in \Psi\sqcup \Xi}\right).
    \end{equation}
    Next, we define a natural $\Gamma_{\Psi, \Xi}$-equivariant $B_{\dR}^+/\xi^m$-linear map
    \begin{equation}\label{amam}
      D_{\Psi, \Xi, m}(A)\to (\Bb_{\dR}^+/\Fil^m)(A^+_{\Psi, \Xi, \infty})
    \end{equation}
   via sending $X_u$ to $[(u, u^{1/p}, \ldots)]$, which induces a morphism
  \begin{equation}\label{simbaa}
   \Kos_{D_{\Psi, \Xi, m}(A)}\left((\gamma_{u}-1)_{u\in \Psi\sqcup \Xi}\right)\to \eta_{t}\Kos_{(\Bb_{\dR}^+/\Fil^m)(A^+_{\Psi, \Xi, \infty})}((\gamma_u-1)_{u\in \Psi\sqcup \Xi}).
  \end{equation}
  Then, we define the bottom horizontal arrow of (\ref{diaprinc2}) as  the composite of (\ref{kospass2}) and (\ref{simbaa}). The so constructed map makes the diagram (\ref{diaprinc2}) commute and, by the previous points, it is a quasi-isomorphism.
  \end{enumerate}
  
   Now, taking the filtered colimit $\varinjlim_{\Psi, \Xi}$, over all sufficiently large $\Psi$ and $\Xi$, of the constructed bottom horizontal quasi-isomorphism of (\ref{diaprinc2}), and then passing to the limit $R\varprojlim_{m}$, using  the quasi-isomorphism (\ref{prevghost}) combined with Lemma \ref{dinverse}, and recalling Lemma \ref{primitive2}, we obtain the desired quasi-isomorphism: such morphism is functorial since taking instead the filtered colimit of $\varinjlim_{\Psi}$, over all sufficiently large $\Psi$ (and $\Xi=\emptyset$), of the bottom horizontal quasi-isomorphism of (\ref{diaprinc2}), we obtain the same morphism. This finishes the proof of (\ref{BdR-cris}).
   
   For the compatibility with filtrations part, we need to use in addition the compatibility with filtration stated in Lemma \ref{primitive2}, Lemma \ref{ainn}, and Corollary \ref{mainfil}. 
   \end{proof}

   \begin{proof}[Proof of Proposition \ref{secondstep}]
   
   By the proof of Theorem \ref{secondstep1} above, we can reduce to checking the commutativity of the following diagram
  \begin{center}
  \begin{tikzcd}
  D_{\Lambda, \Sigma}(R) \arrow[d]\arrow[r]&  \Aa_{\cris}(R_{\Sigma, \Lambda, \infty}) \arrow[d] \\
   D_{\Psi, \Xi, m}(A) \arrow[r] &  (\Bb_{\dR}^+/\Fil^m)(A^+_{\Psi, \Xi, \infty}) 
  \end{tikzcd}
  \end{center}
  where the top horizontal arrow is (\ref{funn}), the left vertical arrow is (\ref{des}), the bottom horizontal arrow is (\ref{amam}), and the right vertical arrow is induced by the composition $A_{\cris}\hookrightarrow B_{\dR}^+\twoheadrightarrow B_{\dR}^+/\Fil^m$.
  In turn, we can reduce to verifying the commutativity of the following diagram
  \begin{center}
  \begin{tikzcd}
  A_{\inf, \Sigma, \Lambda}^\square \arrow[d]\arrow[r]&  \Aa_{\inf}(R_{\Sigma, \Lambda, \infty}) \arrow[d] \\
   A_{\inf}\langle (X_u^{\pm 1})_{u\in \Psi}, (X_a)_{a\in \Xi}\rangle \arrow[r] &  \Aa_{\inf}(A^+_{\Psi, \Xi, \infty}).
  \end{tikzcd}
  \end{center}
  This is clear as both the composition maps from $A_{\inf, \Sigma, \Lambda}^\square$ to $\Aa_{\inf}(A^+_{\Psi, \Xi, \infty})$ send $X_\tau$ to $[(X_{\tau}, X_{\tau}^{1/p}, \ldots)]$, for any $\tau\in \Sigma \cup \{(\lambda, i): \lambda\in \Lambda, 1\le i\le d\}$.
  \end{proof}

   \begin{proof}[Proof of Proposition \ref{ultimo}]
    We may assume $X$ qcqs. We want to show that, given $i\ge 0$, for any $\eh$-hypercover $Y_{\bullet}\to X$ of qcqs smooth rigid-analytic varieties over $C$, the natural map
   \begin{equation}\label{gyuu}
    \Fil_{\inf}^{i}R\Gamma_{\inf}(X/B_{\dR}^+)\to \lim_{[n]\in \Delta}\Fil_{\inf}^{i}R\Gamma_{\inf}(Y_n/B_{\dR}^+)
   \end{equation}
   is an isomorphism. For this, we observe that, by the proof of Theorem \ref{secondstep1} (recalling Remark \ref{almob}), for smooth rigid-analytic varieties over $C$ the infinitesimal filtration on the infinitesimal cohomology over $B_{\dR}^+$ naturally identifies with the filtration decalée on the $B_{\dR}^+$-cohomology, and the latter satisfies $\eh$-hyperdescent.
   \end{proof}


 Before proving Theorem \ref{compatib}, we state and prove a version of the latter for rigid-analytic varieties defined over $C$.
 
 \begin{theorem}\label{compatib2}
   Let $X$ be a connected, paracompact, rigid-analytic variety defined over $C$. 
   \begin{enumerate}[(i)]
    \item\label{compatib2:1}{\normalfont(Hyodo--Kato isomorphism over $B_{\dR}^+$)}
   We have a natural isomorphism in $D(\Mod^{\ssolid}_{B_{\dR}^+})$
   \begin{equation}\label{HKB2}
    R\Gamma_{\HK}(X)\dsolid_{\breve F} B_{\dR}^+\overset{\sim}{\longrightarrow} R\Gamma_{\inf}(X/B_{\dR}^+).
   \end{equation}
  \item\label{compatib2:2}{\normalfont (Compatibility)} The isomorphism (\ref{decaBdR}) of Theorem \ref{B_dR=dR} is compatible with the isomorphism (\ref{decaB}) of Theorem \ref{B=HK}, i.e. we have a commutative diagram as follows
 \begin{center}
  \begin{tikzcd}
  R\Gamma_{B}(X) \arrow[d]\arrow[r, "\sim"]  & (R\Gamma_{\HK}(X)\dsolid_{\breve F}B_{\log})^{N=0} \arrow[d] \\
    R\Gamma_{B_{\dR}^+}(X) \arrow[r, "\sim"] &  R\Gamma_{\inf}(X/B_{\dR}^+)
  \end{tikzcd}
 \end{center}
  where the left vertical map is induced by the inclusion $\Bb\hookrightarrow \Bb_{\dR}^+$, and the right vertical map is induced by (\ref{HKB2}).
   \end{enumerate}
 \end{theorem}
 \begin{proof}
 For part \listref{compatib2:1}, as in the proof of Theorem \ref{mainHK}\listref{mainHK:3}, we can reduce to showing the statement locally for $X$ the generic fiber of $\fr X\in \cl M_{\sss, \qcqs}$. For this, by Theorem \ref{bcn}, which applies thanks to Remark \ref{bbcc}, we have a natural morphism in $D(\Mod_{B_{\dR}^+}^{\ssolid})$
  \begin{equation}\label{beforebn}
   R\Gamma_{\cris}(\fr X_{\cl O_C/p}^0/\cl O_{\breve F}^0)_{\Qq_p}\dsolid_{\breve F}B_{\dR}^+\overset{\sim}{\to}  R\Gamma_{\cris}(\fr X_{\cl O_C/p}/A_{\cris}^\times)\dsolid_{A_{\cris}}B_{\dR}^+\to R\Gamma_{\inf}(X/B_{\dR}^+)
  \end{equation}  
  which is an isomorphism modulo $\xi$. Here, the right arrow of (\ref{beforebn}) is the one induced by (\ref{logtodR}). Since both the source and the target of (\ref{beforebn}) are derived $\xi$-adically complete, we conclude, by the derived Nakayama lemma, that (\ref{beforebn}) is an isomorphism, as desired.
  
  Now, part \listref{compatib2:2} is clear from Proposition \ref{secondstep} and the construction of (\ref{HKB2}) in part \listref{compatib2:1}.
 \end{proof}

 We are ready to prove Theorem \ref{compatib}.
 
 \begin{proof}[Proof of Theorem \ref{compatib}] For part \listref{compatib:1}, we first observe that we have a natural isomorphism
  \begin{equation}\label{infdR}
    R\Gamma_{\inf}(X_C/B_{\dR}^+)\simeq R\Gamma_{\dR}(X)\dsolid_K B_{\dR}^+.
   \end{equation}
  For this, with an eye to the compatibility of part \listref{compatib:2}, by the same ingredients used in the proof of Theorem \ref{B_dR=dR}, we can reduce to showing (\ref{infdR}) $\eh$-locally, using Lemma \ref{primitive2} and Lemma \ref{ainn}.
  Then, part \listref{compatib:1} follows from Theorem \ref{compatib2}\listref{compatib2:1}\footnote{To avoid confusion, we warn the reader that in Theorem \ref{compatib} the rigid-analytic variety $X$ is defined over $K$, instead in Theorem \ref{compatib2} the rigid-analytic variety $X$ is defined over $C$.} combined with the isomorphism (\ref{infdR}).
  
  For part \listref{compatib:2}, by Theorem \ref{compatib2}\listref{compatib2:2}, we are reduced to showing the compatibility between (\ref{decaBdR}) and (\ref{BdR-cris}), under the isomorphism (\ref{infdR}). For this, in the setting of Notation \ref{notaz}, one readily reduces to check the commutativity of the following diagram
 \begin{equation}\label{seeinp}
  \begin{tikzcd} 
  D_{\Psi, \Xi}(A)\arrow[r]\arrow[d] &\cl O D_{\Psi, \Xi}(A) \arrow[d] & A_0\solid_K B_{\dR}^+ \arrow[l] \arrow[-,double line with arrow={-,-}]{d} \\
  \Bb_{\dR}^+(A^+_{\Psi, \Xi, \infty}) \arrow[r] & \cl O \Bb_{\dR}^+(A^+_{\Psi, \Xi, \infty}) & A_0\solid_K B_{\dR}^+ \arrow[l].
  \end{tikzcd}
  \end{equation}
  Here, the top row is the zigzag (\ref{zigzag0}) of Lemma \ref{ainn}, and the bottom row, coming from Scholze's Poincaré lemma, is constructed, in the condensed setting, in \cite[Corollary 6.12]{Bosco}. 
  
 \end{proof}

 \section{\textbf{Syntomic Fargues--Fontaine cohomology}}\label{synleit}
  \sectionmark{}

 In this section, we define a cohomology theory for rigid-analytic varieties over $C$, called \textit{syntomic Fargues--Fontaine cohomology}, which is close in spirit to Bhatt--Morrow--Scholze's syntomic cohomology theory for smooth $p$-adic formal schemes over $\cl O_C$. In a stable range, we compare it with the rational $p$-adic pro-étale cohomology. Our definition is global in nature, it doesn't require neither the existence of nice formal models, nor smoothness, and it extends to coefficients.  Moreover, it has a close relationship with the Fargues--Fontaine curve, as we show in  \S\ref{FFperf}. \medskip

 \begin{df}\label{synFF}
  Let $X$ be a rigid-analytic variety over $C$.  Let $i\ge 0$ be an integer. We define the \textit{syntomic Fargues--Fontaine cohomology of $X$ with coefficients in $\Qq_p(i)$} as the complex of $D(\Vect_{\Qq_p}^{\cond})$  given by the fiber
  $$
   R\Gamma_{\syn, \FF}(X, \Qq_p(i)):=\Fil^iR\Gamma_B(X)^{\varphi=p^i}:=\fib(\Fil^iR\Gamma_B(X)\xrightarrow{\varphi p^{-i}-1} \Fil^i R\Gamma_{B}(X))
  $$
  where $R\Gamma_{B}(X)$ is endowed with the filtration décalée from Definition \ref{beilifildef}.
 \end{df}

 \subsection{The comparison with the $p$-adic pro-\'{e}tale cohomology}
 
 The announced comparison between the syntomic Fargues--Fontaine cohomology and the $p$-adic pro-\'{e}tale cohomology will rely on the following result.
 
  \begin{lemma}  Let $X$ an analytic adic space over $\Spa(C, \cl O_C)$. Let $i\ge 0$ be an integer. We have the following exact sequence of sheaves on $X_{\pet}$
    \begin{equation}\label{suite4}
 0\to \Qq_p(i)\to \Fil^i\Bb\xrightarrow{\varphi p^{-i}-1} \Fil^i\Bb\to 0
 \end{equation}
 where  $\Fil^i\Bb=t^i\Bb$.
  \end{lemma}
  \begin{proof}
   The statement follows from the combination of (\ref{suite2}) and (\ref{suite3}) for $i=0$, recalling that $\varphi(t^i)=p^it$.
  \end{proof}

 Now, let $X$ be a rigid-analytic variety over $C$.
 Via the exact sequence (\ref{suite4}) of sheaves on the pro-étale site $X_{\pet}$ we have a natural morphism
 \begin{equation}\label{toto}
 R\Gamma_{\syn, \FF}(X, \Qq_p(i))\to R\Gamma_{\pet}(X, \Qq_p(i)).
 \end{equation}
 where we are implicitly using $v$-descent, Proposition \ref{compv}.

 \begin{theorem}\label{BK=pet}
  Let $X$ be a rigid-analytic variety over $C$.  Let $i\ge 0$ be an integer.
  \begin{enumerate}[(i)]
   \item\label{BK=pet:1} The truncation $\tau^{\le i}$ of (\ref{toto}) is an isomorphism in $D(\Vect_{\Qq_p}^{\cond})$, i.e. we have
  $$\tau^{\le i}R\Gamma_{\syn, \FF}(X, \Qq_p(i))\overset{\sim}{\longrightarrow} \tau^{\le i}R\Gamma_{\pet}(X, \Qq_p(i)).$$
   \item\label{BK=pet:2} We have a natural isomorphism in $D(\Vect_{\Qq_p}^{\cond})$
   $$ R\Gamma_{\syn, \FF}(X, \Qq_p(i))\simeq \fib(R\Gamma_B(X)^{\varphi=p^i}\to R\Gamma_{B_{\dR}^+}(X)/\Fil^i).$$
  \end{enumerate}
 \end{theorem}
 
 \begin{proof}
 For part \listref{BK=pet:1}, recalling Definition \ref{beilifildef}, we have a natural isomorphism
  $$\tau^{\le i}\Fil^i L\eta_{t}R\alpha_*\Bb\overset{\sim}{\to} \tau^{\le i}R\alpha_*\Fil^i \Bb$$
  which, taking cohomology, induces
  $$\tau^{\le i}\Fil^i R\Gamma_B(X)\overset{\sim}{\to} \tau^{\le i}R\Gamma_v(X, \Fil^i \Bb).$$
  Then, the statement follows from the exact sequence (\ref{suite4}).
  
  For part \listref{BK=pet:2}, considering the natural map of fiber sequences
 \begin{equation}
  \begin{tikzcd}
   R\Gamma_{\syn, \FF}(X, \Qq_p(i)) \arrow[r] \arrow[d] & \Fil^iR\Gamma_B(X) \arrow[r, "\varphi p^{-i}-1"] \arrow[d] & \Fil^iR\Gamma_B(X)  \arrow[d] \\
   R\Gamma_B(X)^{\varphi=p^i} \arrow[r] & R\Gamma_B(X)\arrow[r, "\varphi p^{-i}-1"]  & R\Gamma_B(X)
  \end{tikzcd}
  \end{equation}
  we deduce that it suffices to construct a natural map
  \begin{equation}\label{indcof}
   R\Gamma_{B_{\dR}^+}(X)/\Fil^i\to \fib(R\Gamma_B(X)/\Fil^i\xrightarrow{\varphi p^{-i}-1}R\Gamma_B(X)/\Fil^i)
  \end{equation}
  and show that it is an isomorphism. Endowing the source, resp. the target, of (\ref{indcof}) with its natural (finite) filtration, induced by the filtration decalée on $R\Gamma_{B_{\dR}^+}(X)$, resp. $R\Gamma_B(X)$, we see that we can reduce to showing that, for $0\le j < i$, there is a natural map 
  $$\gr^j L\eta_tR\alpha_*\Bb_{\dR}^+\to \fib(\gr^j L\eta_tR\alpha_*\Bb\xrightarrow{\varphi p^{-i}-1}\gr^j L\eta_tR\alpha_*\Bb)$$
  that is an isomorphism, or equivalently, applying Proposition \ref{bbe}\listref{bbe:2}, that there is a natural map
  \begin{equation}\label{asdes}
   \tau^{\le j}R\alpha_*\gr^j\Bb_{\dR}^+\to  \fib(\tau^{\le j}R\alpha_*\gr^j\Bb\xrightarrow{\varphi p^{-i}-1} \tau^{\le j}R\alpha_*\gr^j\Bb)
  \end{equation}
  that is an isomorphism. For this, combining the exact sequences of sheaves (\ref{suite2}), (\ref{suite3}), and (\ref{suite4}), we have the following exact sequence of sheaves on $X_{\pet}$
  \begin{equation}\label{suite5}
  0\to \Bb_{\dR}^+/\Fil^i \Bb_{\dR}^+\to \Bb/\Fil^i\Bb \xrightarrow{\varphi p^{-i}-1} \Bb/\Fil^i\Bb\to 0
  \end{equation}
  where $\Fil^i\Bb_{\dR}^+=t^i\Bb_{\dR}^+$ and $\Fil^i\Bb=t^i\Bb$. Then, observing that (\ref{suite5}) is compatible with filtrations, we have, for $0\le j< i$, the following exact sequence of sheaves on $X_{\pet}$
  \begin{equation}\label{ongrr}
   0\to \gr^j\Bb_{\dR}^+\to \gr^j\Bb\xrightarrow{\varphi p^{-i}-1} \gr^j\Bb\to 0.
  \end{equation}
  Now, the exact sequence (\ref{ongrr}) induces a natural map (\ref{asdes}) as desired, that we want to prove to be an isomorphism. Thus, we want to show that the map $\varphi p^{-i}-1:R^j\alpha_*\gr^j\Bb\to R^j\alpha_*\gr^j \Bb$ is surjective, or equivalently that the map $R^{j+1}\alpha_*\gr^j\Bb_{\dR}^+\to R^{j+1}\alpha_*\gr^j\Bb$, induced by (\ref{ongrr}), is injective: up to twisting, it suffices to observe that, for $j=0$, the left map in (\ref{ongrr}) is given by the inclusion in the direct product (recall (\ref{Bprod}))
  $$\gr^0 \Bb_{\dR}^+\hookrightarrow \gr^0 \Bb=\prod_{y\in |Y_{\FF}|^{\ccl}}\Bb_{\dR}^+/t^{\ord_y(t)}\Bb_{\dR}^+$$
  corresponding to the classical point $V(\xi)\in |Y_{\FF}|^{\ccl}$. This concludes the proof of part \listref{BK=pet:2}.
  
 \end{proof}

 \subsection{Fargues--Fontaine cohomology and nuclear complexes on the curve}\label{FFperf}
 In this subsection, we first reinterpret the main comparison theorems proved in the previous sections in terms of the Fargues--Fontaine curve (Theorem \ref{lb}): such results were conjectured by Le Bras, and proven by him in some special cases, cf. \cite[Remark 1.2, Conjecture 6.3]{LeBras2}, and are related to work of Le Bras--Vezzani, \cite{LBV}. Then, in Theorem \ref{synlift}, we naturally attach to any qcqs rigid-analytic variety over $C$ a quasi-coherent complex on the Fargues--Fontaine curve $\FF$, whose cohomology is the syntomic Fargues--Fontaine cohomology (Definition \ref{synFF}); we conclude by showing that in the proper case such complex is perfect.
 
 \subsubsection{\normalfont{\textbf{Quasi-coherent, nuclear, and perfect complexes on $\FF$}}}
 
 We will rely on results of Clausen--Scholze, \cite{Scholzeanalytic}, \cite{CS}, and Andreychev, \cite{Andr}, to talk about quasi-coherent, nuclear, and perfect complexes on the adic Fargues--Fontaine curve $\FF$. \medskip
 
 Let us start by recalling some notations from \S \ref{andr}.
 
 \begin{notation} Given a condensed ring $R$, we denote by $\Perf_R\subset D(\Mod_R^{\cond})$ the $\infty$-subcategory of perfect complexes over $R$, \cite[Definition 5.1]{Andr}. \medskip
 
 Given a pair $(A, A^+)$ with $A$ a complete Huber ring and $A^+$ a subring of $A^\circ$, we denote by $(A, A^+)_{\solidif}$ the associated analytic ring, \cite[\S 3.3]{Andr}. We write $D((A, A^+)_{\solidif})$ for the derived $\infty$-category of $(A, A^+)_{\solidif}$-complete complexes, and $\Nuc((A, A^+)_{\solidif})$ for the $\infty$-subcategory of nuclear complexes. Note that, in the case $A^+=\Zz$, we have $D((A, \Zz)_{\solidif})=D(\Mod_{A}^{\ssolid})$ in the notation of \S\ref{convent}. \medskip
 
 Given $Y$ an analytic adic space, we denote by $\QCoh(Y)$ the $\infty$-category of quasi-coherent complexes on $Y$, we write $\Nuc(Y)$ for the $\infty$-subcategory of nuclear complexes on $Y$, and $\Perf(Y)$ for the $\infty$-category of perfect complexes on $Y$, \S \ref{andr}.
 \end{notation}

 The following construction will be used to define a lift of the $B$-cohomology theory to a cohomology theory with values in the $\infty$-category of quasi-coherent complexes on the adic Fargues--Fontaine curve $\FF=Y_{\FF}/{\varphi^{\Zz}}$, called \textit{Fargues--Fontaine cohomology}.

 \begin{construction}
  We write $$Y_{\FF}=\bigcup_{I\subset (0, \infty)}Y_{\FF, I}$$
  with $Y_{\FF, I}=\Spa(B_I, B_I^+)$ for varying $I\subset (0, \infty)$ compact intervals with rational endpoints.
  
  By \cite[Theorem 3.27]{Andr}, we have natural maps of analytic rings
  $$(B_I, \Zz)_{\solidif}\to (B_I, B_I^+)_{\solidif}$$
  for varying $I\subset (0, \infty)$ compact intervals with rational endpoints. Such maps induce base change functors
  $$-\otimes_{(B_I, \Zz)_{\solidif}}^{\LL}(B_I, B_I^+)_{\solidif}: D((B_I, \Zz)_{\solidif})\to D((B_I, B_I^+)_{\solidif})$$
  and then, by analytic descent for quasi-coherent complexes, Theorem \ref{4.1}, they induce a functor
  \begin{equation}\label{mainfun0}
   \CoAdm(\Solid_B):=\lim_{I\subset(0, \infty)}D((B_I, \Zz)_{\solidif})\longrightarrow\lim_{I\subset(0, \infty)}\QCoh(Y_{\FF, I})=\QCoh(Y_{\FF})
  \end{equation}
  with source the $\infty$-category of \textit{coadmissible solid modules over $B$},\footnote{The terminology adopted here comes from \cite[\S 3]{STJ}.} and target the $\infty$-category of quasi-coherent complexes on $Y_{\FF}$.
 \end{construction}
 
 In order to move to the study of quasi-coherent complexes on the Fargues--Fontaine curve, we need to make the following definitions.

  \begin{df}
   We define the $\infty$-category of \textit{quasi-coherent $\varphi$-complexes over $Y_{\FF}$} as the equalizer
  $$ 
  \begin{tikzcd}
   \QCoh(Y_{\FF})^\varphi:=\eq ( \QCoh(Y_{\FF}) \ar[r, shift left,"\varphi^*"] \ar[r, shift right, swap,"\id"]
   &    \QCoh(Y_{\FF}) )
 \end{tikzcd}
 $$
  that is the $\infty$-category of the pairs $(\cl E, \varphi_{\cl E})$, where $\cl E$ is a quasi-coherent complex on $Y_{\FF}$, and $\varphi_{\cl E}: \varphi^*\cl E\simeq \cl E$ is an isomorphism. 
  
  We define $\Nuc(Y_{\FF})^\varphi$ (resp. $\Perf(Y_{\FF})^\varphi$) as the full $\infty$-subcategory of $\QCoh(Y_{\FF})^\varphi$ spanned by the pairs $(\cl E, \varphi_{\cl E})$, with $\cl E$ a nuclear (resp. perfect) complex on $Y_{\FF}$.
  \end{df}
  
  \begin{df}
   We define the $\infty$-category of \textit{coadmissible solid $\varphi$-modules over $B$} as the equalizer
   $$ 
  \begin{tikzcd}
  \CoAdm(\Solid_B)^\varphi:=\eq ( \CoAdm(\Solid_B) \ar[r, shift left,"\varphi^*"] \ar[r, shift right, swap,"\id"]
   &    \CoAdm(\Solid_B) )
   \end{tikzcd}
 $$
 that is the $\infty$-category of the pairs $(\cl M, \varphi_{\cl M})$, where $\cl M$ is a coadmissible solid module over $B$, and $\varphi_{\cl M}: \varphi^*\cl M\simeq \cl M$ is an isomorphism. 
 
 We define $\CoAdm(\Nuc_B)^\varphi$ (resp. $\CoAdm(\Perf_B)^\varphi$) the $\infty$-category of \textit{coadmissible nuclear ({\normalfont resp.} perfect) $\varphi$-modules over $B$} as the full $\infty$-subcategory of $\CoAdm(\Solid_B)^\varphi$ spanned by the pairs $(\cl M, \varphi_{\cl M})$ with $\cl M$ in $\lim_{I\subset (0, \infty)}\Nuc((B_I, \Zz)_{\solidif})$ (resp. in $\lim_{I\subset (0, \infty)}\Perf_{B_I}$).
  \end{df}

 Now, recalling that the action of $\varphi$ on $Y_{\FF}$ is free and totally discontinuous, \cite[Proposition II.1.16]{FS}, it follows formally from the analytic descent for quasi-coherent complexes, Theorem \ref{4.1}, that we have an equivalence of $\infty$-categories 
 \begin{equation}\label{indie}
  \QCoh(Y_{\FF})^\varphi\simeq \QCoh(Y_{\FF}/\varphi^{\Zz}).
 \end{equation}
 Thus, from (\ref{mainfun0}) we obtain a functor 
  \begin{equation}\label{mainfun}
   \cl E_{\FF}(-):\CoAdm(\Solid_B)^{\varphi}\to \QCoh(\FF)
  \end{equation}
  with target the $\infty$-category of quasi-coherent complexes on the Fargues--Fontaine curve. \medskip
  
  Next, we focus on nuclear and perfect complexes on $\FF$. We invite the reader to compare the following result with \cite[Theorem 7.18]{FF2}.

 \begin{prop}[Nuclear complexes on the curve as coadmissible $\varphi$-modules over $B$]\label{perff}\
 \begin{enumerate}[(i)]
  \item\label{perff:1}  The functor $\cl E_{\FF}(-)$, defined in (\ref{mainfun}), induces an equivalence of $\infty$-categories
  \begin{equation}\label{mainfunnuc}
   \CoAdm(\Nuc_B)^{\varphi}\overset{\sim}{\longrightarrow} \Nuc(\FF)
  \end{equation}
  which restricts to an equivalence of $\infty$-categories
    \begin{equation}\label{mainfunperf}
   \CoAdm(\Perf_B)^{\varphi}\overset{\sim}{\longrightarrow} \Perf(\FF).
  \end{equation}
  \item\label{perff:2} Given $\cl E\in  \Nuc(\FF)$, let $((M_I(\cl E))_{I\subset (0, \infty)}, \varphi)$ be the coadmissible nuclear $\varphi$-module over $B$ corresponding to $\cl E$ via the equivalence (\ref{mainfunnuc}). Let $M(\cl E):=R\lim_{I\subset (0, \infty)}M_I(\cl E)$. Then, there is a natural identification in $D(\Mod_{\Qq_p}^{\ssolid})$
  $$R\Gamma(\FF, \cl E)=\fib(M(\cl E)\xrightarrow{\varphi-1}M(\cl E)).$$
 \end{enumerate}
 \end{prop}
 \begin{proof}
  For part \listref{perff:1}, we first observe that, for any compact interval $I\subset (0, \infty)$ with rational endpoints, the base change functor $-\otimes_{(B_I, \Zz)_{\solidif}}^{\LL}(B_I, B_I^+)_{\solidif}: D((B_I, \Zz)_{\solidif})\to D((B_I, B_I^+)_{\solidif})$ induces an equivalence of $\infty$-categories\footnote{Such equivalence follows from Theorem \ref{unp}, however we give here a self-contained proof.}
  \begin{equation}\label{anvedii}
   \Nuc((B_I, \Zz)_{\solidif})\overset{\sim}{\longrightarrow}\Nuc((B_I, B_I^+)_{\solidif}).
  \end{equation}
  In fact, by Theorem \ref{nuclearbanach2}\listref{nuclearbanach:22}, $\Nuc((B_I, \Zz)_{\solidif})$ is generated, under shifts and colimits, by the objects $\underline{\Hom}(\Zz[S], B_I)$, for varying $S$ profinite sets; then, as, thanks to Corollary \ref{bcnuc}, nuclearity is preserved under base change, by Proposition \ref{incu2} we deduce that  such objects also generate, under shifts and colimits, $\Nuc((B_I, B_I^+)_{\solidif})$.
  Then, we have an equivalence of $\infty$-categories
  \begin{equation}\label{implfunnuc}
   \CoAdm(\Nuc_B)^{\varphi}\overset{\sim}{\longrightarrow} \Nuc(Y_{\FF})^{\varphi}.
  \end{equation}
 Now, by analytic descent for nuclear complexes, Theorem \ref{5.41}, we have 
 \begin{equation}\label{nuci}
  \Nuc(Y_{\FF})^{\varphi}\simeq \Nuc(Y_{\FF}/\varphi^{\Zz})
 \end{equation}
  which combined with (\ref{implfunnuc}) implies the equivalence (\ref{mainfunnuc}). Such equivalence restricts to (\ref{mainfunperf}) by analytic descent for perfect complexes, Theorem \ref{5.3}.
  
  Part \listref{perff:2} follows from part \listref{perff:1}, as we now explain. In fact, for any $\cl E\in  \QCoh(\FF)$, using the equivalence (\ref{indie}), we have the following identification
  \begin{align*}
   R\Gamma(\FF, \cl E)&=\Hom_{\QCoh(\FF)}(\cl O, \cl E)\\
   &=\fib(\Hom_{\QCoh(Y_{\FF})}(\cl O, \cl E|_{Y_{\FF}})\xrightarrow{\varphi-1}\Hom_{\QCoh(Y_{\FF})}(\cl O, \cl E|_{Y_{\FF}})) \\
   &=\fib(R\Gamma(Y_{\FF}, \cl E|_{Y_{\FF}})\xrightarrow{\varphi-1}R\Gamma(Y_{\FF}, \cl E|_{Y_{\FF}}))
  \end{align*}
  where
  $$R\Gamma(Y_{\FF}, \cl E|_{Y_{\FF}})=R\lim\nolimits_{I\subset (0, \infty)}\Hom_{\QCoh(Y_{\FF, I})}(\cl O, \cl E|_{Y_{\FF, I}}).$$
  Now, assuming $\cl E\in \Nuc(\FF)$ as in the statement, thanks to the equivalence (\ref{anvedii}), we have
 $$\Hom_{\QCoh(Y_{\FF, I})}(\cl O, \cl E|_{Y_{\FF, I}})=\Hom_{\Nuc(Y_{\FF, I})}(\cl O, \cl E|_{Y_{\FF, I}})=\Hom_{\Nuc((B_I, \Zz)_{\solidif})}(B_I, M_I(\cl E))=M_I(\cl E).$$
 Hence, the statement follows, observing that all the $\infty$-categories appearing above are naturally enriched over $D(\Mod_{\Qq_p}^{\ssolid})$.
 \end{proof}

 Perfect complexes on $\FF$ are well-understood. Let us recall the following characterization due to Anschütz--Le Bras.
 
 \begin{prop}[{\cite[Proposition 2.6]{AL}}]\label{perffALB}
   Each $\cl E\in \Perf(\FF)$ is quasi-isomorphic to bounded complex of vector bundles on $\FF$.
 \end{prop}

 In addition, perfect complexes on $\FF$ are closely related to Banach--Colmez spaces, as we now recall.
  
  \begin{df}[{\cite[Définition 2.11]{LeBras1}}]\label{defbc}
  The category $\cl{BC}$ of \textit{Banach--Colmez spaces over $C$} (for short, \textit{BC spaces}) is the smallest abelian subcategory of sheaves of $\underline{\Qq_p}$-modules on the $v$-site $\Spa(C, \cl O_C)_v$ stable under extensions and containing the $v$-sheaves $\underline{\Qq_p}$ and $(\Aa_C^{1})^\diamondsuit$.
 \end{df}
  
  In the following, we denote by
 $$\tau: \FF_v\to \Spa(C, \cl O_C)_v$$
 the natural morphism of $v$-sites, sending an affinoid perfectoid $S\in \Spa(C, \cl O_C)_v$ to the relative Fargues--Fontaine curve $\FF_S\in \FF_v$.
  \begin{prop}[{\cite[Corollary 3.10, Remark 3.12]{AL}}]\label{perffALB2} We have an equivalence of categories
  \begin{equation}\label{equivder}
   R\tau_*: \Perf(\FF)\overset{\sim}{\longrightarrow} D^b(\cl{BC})
  \end{equation}
  where the right-hand side denotes the bounded derived category of Banach--Colmez spaces over $C$. 
  \end{prop}

  \subsubsection{\normalfont{\textbf{Fargues--Fontaine cohomology}}} We are almost ready to define the Fargues--Fontaine cohomology together with its filtration. First, in order to prove \cite[Conjecture 6.3]{LeBras2}, we want to extend the definition of the $B$-cohomology to dagger varieties over $C$. \medskip
  
  In the following, we abbreviate $D(B)=D(\Mod_B^{\cond})$.

  \begin{df}[$B$-cohomology of dagger varieties]
   Let $X$ be a dagger variety over $C$.
    Denote $$\cl F_{B}\in \Shv^{\hyp}(\Rig_{C, \eh}, D(B))$$ the hypersheaf defined by the $B$-cohomology $R\Gamma_{B}(-)$.  Via Construction \ref{consthyp}, we define
    $$R\Gamma_{B}(X):=R\Gamma(X, \cl F_{B}^\dagger)\in D(B)$$
   and we endow it with the filtration induced by the filtration décalée on the $B$-cohomology of rigid-analytic varieties over $C$.
   We give analogous definitions replacing the $B$-cohomology with the $B_I$-cohomology, for varying $I\subset (0, \infty)$ compact intervals with rational endpoints.
  \end{df}

  \begin{rem}[Comparison with the Hyodo--Kato cohomology in the dagger case]\label{HKdagg} We note that the main comparison theorems proved in \S \ref{Bleit} for the $B$-cohomology (and the $B_I$-cohomology) of rigid-analytic varieties extend to the dagger case, thanks to the properties of the solid tensor product. In particular, a version of Theorem \ref{B=HK} holds true for $X$ a qcqs dagger variety over $C$: in fact, by the proof of \textit{loc. cit.} we can reduce to the case $X$ is a smooth dagger affinoid over $C$, which follows from the statement of \textit{loc. cit.} using Lemma \ref{immHK} (which applies thanks to Proposition \ref{3.11}\listref{3.11.1}) and the fact that the solid tensor product commutes with filtered colimits.
  \end{rem}

  In order to define the Fargues--Fontaine cohomology, we will apply the functor (\ref{mainfun}) in the following situation.
  
  \begin{lemma}[Coadmissibility of $B$-cohomology]\label{ko}
   Let $X$ be a qcqs rigid-analytic/dagger variety over $C$. Let $i\ge 0$. The pair 
   \begin{equation}\label{thepair}
    \left((\Fil^i R\Gamma_{B_{I}}(X))_{I\subset(0, \infty)}, \varphi \right)
   \end{equation}
   defines a coadmissible nuclear $\varphi$-module over $B$. Here, the complexes $R\Gamma_{B_I}(X)$, for $I\subset (0, \infty)$ compact intervals with rational endpoints, are endowed with the filtration décalée (Definition \ref{beilifildef}).
  \end{lemma}
  \begin{proof}
   First, we recall that $\varphi(t)=pt$, and $p$ is invertible in $B_I$. Next, let us assume $i=0$. To show that the pair (\ref{thepair}) is a coadmissible solid $\varphi$-module over $B$, we need to check that 
   \begin{equation}\label{120}
    R\Gamma_{B_J}(X)\dsolid_{B_{J}}B_I\simeq R\Gamma_{B_I}(X)
   \end{equation}
   for any $I\subset J\subset (0, \infty)$ compact intervals with rational endpoints: this follows for example from Theorem \ref{B=HK} (and Remark \ref{HKdagg}).\footnote{Trivializing the monodromy action.} 
   
   In order to prove the nuclearity of (\ref{thepair}) for $i=0$, we can use again Theorem \ref{B=HK}, Theorem \ref{mainHK}\listref{mainHK:2}, and the fact that nuclearity is preserved under base change, Corollary \ref{bcnuc}. Alternatively, and more directly, the nuclearity can be checked as follows. By Theorem \ref{nuclearbanach2}\listref{nuclearbanach:12} and $\eh$-hyperdescent, we can reduce to the case when $X$ is a smooth affinoid rigid space over $C$. Then, by Proposition \ref{3.11}, and the fact that nuclearity can be checked on cohomology groups thanks to Theorem \ref{nuclearbanach2}\listref{nuclearbanach:32}, using \cite[Lemma 6.4]{BMS1} we can reduce to the assertion that each complex of solid $B_I$-modules $R\Gamma_{\pet}(X, \Bb_I)$ is nuclear, which was shown in Lemma \ref{nucb}.
   
   It remains to show that the isomorphism (\ref{120}) is compatible with the filtration décalée, and to prove the nuclearity of the filtered pieces. Proceeding by induction on $i\ge 0$, we can check this on graded pieces using Proposition \ref{bbe}\listref{bbe:2}, again with the help of Theorem \ref{nuclearbanach2}, Proposition \ref{3.11}, and Lemma \ref{nucb} for the nuclearity.
  \end{proof}

  The following lemma on the nuclearity of the cohomology of the rational pro-étale period sheaves was used above.
  

  \begin{lemma}\label{nucb} Let $X$ be a qcqs analytic adic space over $\Spa(C, \cl O_C)$.  Let $I\subset (0, \infty)$ be a compact interval with rational endpoints, and let $m\ge 1$ be an integer. Given $\mathbf{B}\in \{\Bb_{I}, \Bb_{\dR}^+/\Fil^m\}$
   we write $\mathscr{B}=\mathbf{B}_{\Spa(C)_{\pet}}$ for the corresponding condensed period ring. Then, we have $$R\Gamma_{\pet}(X, \mathbf{B})\in \Nuc((\mathscr{B}, \Zz)_{\solidif}).$$
  \end{lemma}
  \begin{proof}
   Picking a simplicial $\pet$-hypercover $U_\bullet\to X$ such that all $U_n$ are affinoid perfectoid spaces over $\Spa(C, \cl O_C)$, by $\pet$-hyperdescent and Theorem \ref{nuclearbanach2}\listref{nuclearbanach:12} (applied to the Banach $\Qq_p$-algebra $\mathscr{B}$), it suffices to check that for any affinoid perfectoid space $U$ over $\Spa(C, \cl O_C)$ we have that $R\Gamma_{\pet}(U, \mathbf{B})\in \Nuc((\mathscr{B}, \Zz)_{\solidif})$. We note that, By \cite[Proposition 4.7]{Bosco} and \cite[Corollary 4.9]{Bosco}, we have
  $R\Gamma_{\pet}(U,\mathbf{B})=\mathbf{B}(U)[0]$ concentrated in degree 0 (as it can be checked on on $S$-valued points, for any $\kappa$-small extremally disconnected set $S$). Thus, by Remark \ref{A=F_nuc}, we are reduced to prove that $\mathbf{B}(U)[0]\in \Nuc((\Qq_p, \Zz)_{\solidif})$, which follows \cite[Corollary A.50]{Bosco} observing that $\mathbf{B}(U)$ is a $\Qq_p$-Banach space.
 \end{proof}

  Finally, we can give the following definition.
  
  \begin{df}\label{drff}
  Let $X$ be a qcqs rigid-analytic/dagger variety over $C$. We define the \textit{Fargues--Fontaine cohomology} of $X$, denoted $$\cl H_{\FF}(X)\in \QCoh(\FF)$$ as the quasi-coherent complex on the Fargues--Fontaine curve $\FF$, endowed with filtration $\Fil^\star \cl H_{\FF}(X)$,  associated, via the functor $\cl E_{\FF}(-)$ defined in (\ref{mainfun}), to the filtered coadmissible solid $\varphi$-module over $B$ given by (\ref{thepair}). For $i\in \Zz$, we denote by $\cl H^i_{\FF}(X)$ its $i$-th cohomology group.
  \end{df}


 Using Definition \ref{drff} we can now reformulate, in terms of the Fargues--Fontaine curve, the comparison theorems proven in the previous sections.
 
 \begin{theorem}[cf. {\cite[Conjecture 6.3]{LeBras2}}]\label{lb}
  Let $X$ be a qcqs dagger variety over $C$. Let $i\ge 0$. 
  \begin{enumerate}[(i)] 
   \item\label{lb:1} The quasi-coherent complex $\cl H_{\FF}(X)$ on $\FF$ is perfect, and its cohomology groups are vector bundles on $\FF$. We have a natural isomorphism
  \begin{equation}\label{fff}
     \cl H_{\FF}^i(X)\cong \cl E(H_{\HK}^i(X))
  \end{equation}
  where $\cl E(H_{\HK}^i(X))$ is the vector bundle on $\FF$ associated to the finite $(\varphi, N)$-module $H_{\HK}^i(X)$ over $\breve F$.  If $X$ is the base change to $C$ of a rigid-analytic variety defined over $K$, then (\ref{fff}) is $\mathscr{G}_K$-equivariant.
  \item\label{lb:2}  The completion at $\infty$ of (\ref{fff}) gives a natural isomorphism
  \begin{equation}\label{fffinf}
    \cl H_{\FF}^i(X)^\wedge_{\infty}\cong H_{\inf}^i(X/B_{\dR}^+)
  \end{equation}
  where $R\Gamma_{\inf}(X/B_{\dR}^+)$ is defined, via Construction \ref{consthyp}, from Definition \ref{infccc}.
  
  \end{enumerate}
 \end{theorem}
 \begin{proof}
 Part \listref{lb:1} follows from Theorem \ref{B=HK} (combined with Remark \ref{HKdagg}), Theorem \ref{slopp}\listref{slopp:1} and Lemma \ref{unex2}. Part \listref{lb:2} follows from  Theorem \ref{compatib2} (which extends to the dagger case similarly to Remark \ref{HKdagg}).
 \end{proof}
 
 \begin{rem}\label{f-mod}
  Let $X$ be a qcqs dagger variety over $C$ and let $i\ge 0$. Recalling that the functor from finite $\varphi$-modules over $\breve F$ to vector bundles on $\FF$ induces a bijection on isomorphism classes, \cite[Theorem II.0.3]{FS}, we deduce from Theorem \ref{lb}\listref{lb:1} that the vector bundle $\cl H_{\FF}^i(X)$ determines, up to isomorphisms, the $\varphi$-module structure on $H_{\HK}^i(X)$.
   
   Using Theorem \ref{lb}\listref{lb:1}, we can also recover from $\cl H_{\FF}^i(X)$ the $(\varphi, N)$-module structure on $H_{\HK}^i(X)$, up to isomorphisms: in fact, denoting $\mathscr{G}_{\breve F}=\Gal(C/\breve F)$,  by \cite[Proposition 10.3.20(2)]{FF} we have a natural isomorphism of $(\varphi, N)$-modules over $\breve F$
   $$H_{\HK}^i(X)\cong \left(H^0(\FF\setminus\{\infty\}, \cl H_{\FF}^i(X))\otimes_{B_e}B_{\log}[1/t]\right)^{\mathscr{G}_{\breve F}}$$  
   observing that, since (\ref{fff}) is an isomorphism of $\mathscr{G}_{\breve F}$-equivariant vector bundles on $\FF$, we have a $\mathscr{G}_{\breve F}$-equivariant isomorphism 
   $H^0(\FF\setminus\{\infty\}, \cl H_{\FF}^i(X))\cong (H_{\HK}^i(X)\otimes_{\breve F}B_{\log}[1/t])^{\varphi=1, N=0}$.
 \end{rem}

 Next, we state the main result of this subsection.
 
 \begin{theorem}\label{synlift}
   Let $X$ be a qcqs rigid-analytic/dagger variety over $C$. Let $i\ge 0$.  Consider the quasi-coherent complex on $\FF$ defined by $$\cl H_{\syn}(X)(i):=\Fil^i\cl H_{\FF}(X)\otimes \cl O(i).$$ 
   We have
   \begin{equation}\label{qpos}
    R\Gamma(\FF, \cl H_{\syn}(X)(i))=R\Gamma_{\syn, \FF}(X, \Qq_p(i)).
   \end{equation}
   If $X$ is a proper rigid-analytic variety over $C$, the complex $\cl H_{\syn}(X)(i)$ is perfect, in particular the complex $R\Gamma_{\syn, \FF}(X, \Qq_p(i))$ identifies with the $C$-points of a bounded complex of Banach--Colmez spaces.
   
 \end{theorem}
  \begin{proof}
  By Proposition \ref{perff}\listref{perff:2}, which applies thanks to Lemma \ref{ko}, we have
  \begin{align*}
   R\Gamma(\FF, \cl H_{\syn}(X)(i))&=\fib(R\Gamma(Y_{\FF}, \cl H_{\syn}(X)(i)|_{Y_{\FF}})\xrightarrow{\varphi-1}R\Gamma(Y_{\FF}, \cl H_{\syn}(X)(i)|_{Y_{\FF}})) \\
   &=\fib(R\Gamma(Y_{\FF}, \Fil^i\cl H_{\FF}(X)|_{Y_{\FF}})\xrightarrow{\varphi p^{-i}-1} R\Gamma(Y_{\FF}, \Fil^i\cl H_{\FF}(X)|_{Y_{\FF}})) \\
   &=\fib(\Fil^iR\Gamma_B(X)\xrightarrow{\varphi p^{-i}-1}\Fil^iR\Gamma_B(X))
  \end{align*}
  where in the last step we used in addition Lemma \ref{dinverse}. This shows (\ref{qpos}). 
  
  Next, assume $X$ proper. To show that $\cl H_{\syn}(X)(i)$ is a perfect complex, using the derived Beauville--Laszlo gluing, Lemma \ref{dbl}, we can reduce to showing that,  for $I\subset (0, \infty)$ a compact interval with rational endpoints, the complexes $R\Gamma_{\pet}(X, \Bb_I[1/t])$ and $\Fil^iR\Gamma_{B_{\dR}^+}(X)$ are perfect. First, we note that such complexes are bounded thanks to Proposition \ref{bddd}. Then, to prove that the complex $R\Gamma_{\pet}(X, \Bb_I[1/t])$ is perfect, we can apply Theorem \ref{B=HK} combined with Theorem \ref{slopp}\listref{slopp:1} (or, alternatively, \cite[Theorem 3.17]{ScholzeSurvey}), and for the complex $\Fil^iR\Gamma_{B_{\dR}^+}(X)$ we use  Proposition \ref{fillemma} below.

  Finally, we note that, by Proposition \ref{perffALB2}, the complex $R\Gamma_{\syn, \FF}(X, \Qq_p(r))$ identifies with the $C$-points of a bounded complex of Banach--Colmez spaces, as desired (in fact, with the same notation as in Proposition \ref{perffALB2}, for $\cl E\in \Perf(\FF)$, we have $R\tau_*(\cl E)(\Spa(C, \cl O_C))=R\Gamma(\FF, \cl E)$, by the $v$-descent result \cite[Proposition II.2.1]{FS} combined with Proposition \ref{perffALB}).
  \end{proof}

 We used the following derived version of the classical Beauville--Laszlo gluing.
 
 \begin{lemma}[Derived Beauville--Laszlo gluing]\label{dbl}
 Let $R$ be a commutative condensed ring, let $f\in R$ be a non-zero-divisor, and denote by $\widehat R$ the $f$-adic completion of $R$. We have a natural equivalence of $\infty$-categories $$\Perf_R\simeq \Perf_{R[1/f]}\times_{\Perf_{\widehat R[1/f]}}\Perf_{\widehat R}.$$ 
 \end{lemma}
 \begin{proof}
  By \cite[Proposition 5.6(2), Example 5.10]{BhattTannaka}, we have a natural equivalence of $\infty$-categories $\Perf_{R(*)}\simeq \Perf_{R[1/f](*)}\times_{\Perf_{\widehat R[1/f](*)}}\Perf_{\widehat R(*)}.$
  In order to carry such equivalence into the condensed setting, we recall that, for any commutative condensed ring $A$, we define the \textit{condensification functor}, \cite[Definition 5.8]{Andr}, as the composite
  $$\Cond_A: D(\Mod_{A(*)})\hookrightarrow D(\Mod_{\underline{A(*)}}^{\cond})\xrightarrow{\otimes_{\underline{A(*)}}^{\LL}A}D(\Mod_{A}^{\cond}).$$
  Then, the statement follows observing that $\Cond_A$ preserves perfect complexes, and applying \cite[Lemma 5.10]{Andr}.
 \end{proof}

 We also used the following finiteness/degeneration result on the  $B_{\dR}^+$-cohomology of proper rigid-analytic varieties over $C$. As we will see, part \listref{fillemma:1} follows from results of Guo, instead part \listref{fillemma:2} relies crucially on a combination of Conrad--Gabber's spreading out for proper rigid-analytic varieties and a generic smoothness result recently proved by Bhatt--Hansen, \cite{Bhatt-Hansen}, which allow us to reduce the statement to the case when $X$ is the base change to $C$ of a proper smooth rigid-analytic variety over a discretely valued subfield of $C$.

 
   \begin{prop}\label{fillemma}
   Let $X$ be a proper rigid-analytic variety over $C$. 
   \begin{enumerate}[(i)]
    \item\label{fillemma:1} For all $i\in \Zz$, the cohomology group $H^i_{B_{\dR}^+}(X)$ is a finite free module over $B_{\dR}^+$.
    \item\label{fillemma:2}  For all $i, r\in \Zz$, the natural map
   \begin{equation}\label{natlast}
    H^i(\Fil^r R\Gamma_{B_{\dR}^+}(X))\to H^i_{B_{\dR}^+}(X)
   \end{equation}
   is injective. Equivalently, for all $i, r\in \Zz$, the natural map
   \begin{equation*}\label{surjdR}
    H^i_{B_{\dR}^+}(X)/\Fil^r\to  H^i(R\Gamma_{B_{\dR}^+}(X)/\Fil^r)
   \end{equation*}
   is an isomorphism, where $\Fil^r H^i_{B_{\dR}^+}(X):=\im(H^i(\Fil^r R\Gamma_{B_{\dR}^+}(X))\to H^i_{B_{\dR}^+}(X))$.
   \end{enumerate}
  \end{prop}
  \begin{proof}  
  Part \listref{fillemma:1} follows from \cite[Theorem 1.2.7(vi)]{Guo2} combined with Theorem \ref{secondstep1}. Next, we prove part \listref{fillemma:2} and give at the same time an alternative proof of part \listref{fillemma:1}. We first note that the equivalence of the two assertions in part \listref{fillemma:2} immediately follows considering the long exact sequence in cohomology associated to the exact triangle
   $$\Fil^r R\Gamma_{B_{\dR}^+}(X)\to R\Gamma_{B_{\dR}^+}(X)\to R\Gamma_{B_{\dR}^+}(X)/\Fil^r.$$ 
   
   We will use that, by \cite[Corollary 13.16]{BMS1}, there exists a proper flat morphism $f: \cl X\to S$ of rigid-analytic varieties over a discretely valued subfield  $L\subset C$ such that $X$ is the fibre of $f$ over a point $\eta\in S(C)$; we note that we may assume $S=\Spa(A, A^\circ)$ to be a smooth affinoid over $L$. Then, the map $\eta$ corresponds to a map $A\to C$ of affinoid $L$-algebras, and, by the (formal) smoothness of $A$ over $L$, the latter map lifts to a map $A\to B_{\dR}^+$. Now, we divide the argument in several cases.
  \begin{enumerate}[(a)]
   \item\label{keya} In the case $S$ is a point, we have that $X$ is the base change to $C$ of a proper rigid-analytic variety $\cl X$ defined over a discretely valued subfield  $L\subset C$. Then, by Theorem \ref{B_dR=dR}, we have a natural filtered quasi-isomorphism
   \begin{equation}\label{ump}
    R\Gamma_{\dR}(\cl X)\otimes_L^{\LL} B_{\dR}^+\simeq R\Gamma_{B_{\dR}^+}(X).
   \end{equation}
   Here, we used that, as $\cl X$ is proper, the cohomology groups of $R\Gamma_{\dR}(X)$ and its filtered pieces, are finite-dimensional over $L$:  in fact, by Proposition \ref{baseh}, there exists a proper $\eh$-hypercover $\cl X_\bullet\to \cl X$ with each $\cl X_n$ smooth over $L$; then, by cohomological descent, we can reduce to the case $\cl X$ is smooth, which follows from \cite{Kiehl0}. Then, part \listref{fillemma:1} is clear from the quasi-isomorphism (\ref{ump}), and part \listref{fillemma:2} follows from the compatibility with filtrations of (\ref{ump}) and the degeneration of the Hodge-de Rham spectral sequence for $\cl X$
   \begin{equation}\label{ehsp}
    H^{i-j}(\cl X, \Omega^j_{\cl X_{\eh}})\implies H^i_{\dR}(\cl X)
   \end{equation}
   (see \cite[Corollary 1.8]{Scholze} for the case $\cl X$ is smooth, and \cite[Proposition 8.0.8]{Guo1} for the case $\cl X$ is singular).
   \item\label{keyb}  In the case $f:\cl X\to S$ is smooth, denoting by $Rf_{\dR *}\cl O_{\cl X}:=Rf_{*}\Omega_{\cl X/S}^{\bullet}$ the relative de Rham cohomology of $f$ endowed with its Hodge filtration $Rf_{*}\Omega_{\cl X/S}^{\ge \star}$, we claim that we have a filtered quasi-isomorphism
   \begin{equation}\label{spread2}
    Rf_{\dR *}\cl O_{\cl X}\otimes_A^{\LL}B_{\dR}^+\simeq R\Gamma_{B_{\dR}^+}(X)
   \end{equation}
   compatible with (\ref{ump}) in the case $S$ is a point. For this, by Theorem \ref{secondstep1} and Proposition \ref{ultimo}, it suffices to show that we have a $B_{\dR}^+$-linear map
   \begin{equation}\label{spread}
     Rf_{\dR *}\cl O_{\cl X}\otimes_A^{\LL}B_{\dR}^+ \to  R\Gamma_{\inf}(X/B_{\dR}^+)
   \end{equation}
   which is a quasi-isomorphism compatible with filtrations, where the right-hand side is endowed with the infinitesimal filtration. Arguing as in the proof of \cite[Theorem 13.19]{BMS1}, we can construct (\ref{spread}) on a hypercover of $\cl X$ by \textit{very small} smooth affinoid spaces over $L$ in the sense of \cite[Definition 13.5]{BMS1},\footnote{I.e. the affinoid spaces $\Spa(A_0, A_0^{\circ})$ of Notation \ref{setse} for $r=0$, $q=1$, and $\Xi_0=\emptyset$.} using in addition Lemma \ref{primitive2}.\footnote{In the case $S$ is a point, the stated compatibility with (\ref{ump}) follows from the proof of Theorem \ref{compatib} (see in particular the commutative diagram (\ref{seeinp})).}
    We note that the constructed map (\ref{spread}) is a quasi-isomorphism after applying the functor $-\otimes^{\LL}_{B_{\dR}^+}B_{\dR}^+/\xi$, recalling that we have a  natural quasi-isomorphism $$R\Gamma_{\inf}(X/B_{\dR}^+)\otimes^{\LL}_{B_{\dR}^+}B_{\dR}^+/\xi\simeq R\Gamma_{\dR}(X)$$
     by \cite[Theorem 1.2.7(i), Theorem 1.2.1(i)]{Guo2}; then, in order to show that (\ref{spread}) is a quasi-isomorphism, using the derived Nakayama lemma, it suffices to check that both the source and the target of (\ref{spread}) are derived $\xi$-adically complete: for the source, we note that each $R^if_{\dR *}\cl O_{\cl X}$ is a coherent $\cl O_S$-module with (integrable) connection, in particular it is a locally free $\cl O_S$-module (see \cite[Corollaire 2.5.2.2]{Andre}); for the target, we can use for example the \v{C}ech-Alexander complex computing the infinitesimal cohomology over $B_{\dR}^+$, \cite[Proposition 4.1.3]{Guo2}. To prove that the quasi-isomorphism (\ref{spread}) is compatible with filtrations, proceeding by induction on the index $i\ge 0$ of the filtrations, it suffices to check the compatibility on graded pieces, which follows from Lemma \ref{gradedul}.
     
   Now, to prove part \listref{fillemma:1}, using the quasi-isomorphism (\ref{spread2}), it suffices to recall that each $R^if_{\dR *}\cl O_{\cl X}$ is a finite projective $A$-module. For part \listref{fillemma:2}, using the filtered quasi-isomorphism (\ref{spread2}), it suffices to recall that, thanks to \cite[Theorem 8.8]{Scholze} and \cite[Theorem 10.5.1]{SW}, for all $i, j\in \Zz$ the relative Hodge cohomology $R^{i-j}f_{*}\Omega_{\cl X/S}^j$ is a finite projective $A$-module, and that the relative Hodge-de Rham spectral sequence
   $$R^{i-j}f_{*}\Omega_{\cl X/S}^j\implies R^if_{\dR *}\cl O_{\cl X}$$
   degenerates.
   
   \item In the general case, we will use that $Rf_{\pet *}\Zz_p$ is a bounded Zariski-constructible complex of sheaves on $S$, thanks to \cite[Theorem 3.10, Theorem 3.36]{Bhatt-Hansen}. Denoting by $\nu: S_{\pet}\to S_{\ett}$ the natural morphism of sites, and setting
   $$D_{\dR}(Rf_{\pet *}\Zz_p):=R^0\nu_*(Rf_{\pet *}\Zz_p\otimes_{\Zz_p}\cl O\Bb_{\dR, S})$$
   (see e.g. \cite[Definition 6.6]{Bosco} for the definition of Scholze's pro-étale sheaf $\cl O\Bb_{\dR, S}$), we claim that we have a natural filtered quasi-isomorphism
   \begin{equation}\label{ryui}
    D_{\dR}(Rf_{\pet *}\Zz_p)\otimes_A^{\LL}B_{\dR}^+\simeq R\Gamma_{B_{\dR}^+}(X).
   \end{equation}
   For this, we argue by induction on $\dim(S)$. For the base case $\dim(S)=0$, we have that $S$ is a disjoint union of points, hence we can reduce to the case $S$ is a point; then, by the de Rham comparison theorem for proper (possibly singular) rigid-analytic varieties defined over $L$, \cite[Theorem 1.1.4]{Guo1}, we have a natural filtered quasi-isomorphism  $D_{\dR}(Rf_{\pet *}\Zz_p)\simeq R\Gamma_{\dR}(\cl X)$, and the claim follows from the filtered quasi-isomorphism (\ref{ump}) of part \listref{keya}.
   
   For the inductive step, using Proposition \ref{baseh}, picking a proper $\eh$-hypercover $\cl X_\bullet\to \cl X$ with each $\cl X_n$ smooth over $L$, we can reduce to the case $\cl X$ is smooth over $L$. In the latter case, by \cite[Theorem 2.29]{Bhatt-Hansen}, the maximal open subset $S'\subset S$ such that $f':f^{-1}(S')\to S'$ is smooth is a dense Zariski-open subset; in particular, the Zariski-closed complement $Z:=S\setminus S'$ is nowhere dense in $S$, and hence we have $\dim(Z)< \dim(S)$. Then, if $\eta$ is contained in $Z$, restricting $f$ over $Z$, the claim follows from the inductive hypothesis, by proper base change \cite[Theorem 3.15, Theorem 3.36]{Bhatt-Hansen}; if $\eta$ is contained in $S'$, by the smoothness of the restriction $f'$ of $f$ over $S'$, the claim follows again by proper base change from the filtered quasi-isomorphism (\ref{spread2}) of part \listref{keyb}, recalling that in this case, by the relative de Rham comparison (\cite[Theorem 8.8]{Scholze} combined with \cite[Theorem 10.5.1]{SW}) we have a filtered quasi-isomorphism $D_{\dR}(Rf'_{\pet *}\Zz_p)\simeq Rf'_{\dR *}\cl O_{\cl X}$. 
   
   Now, part \listref{fillemma:1} follows from the quasi-isomorphism (\ref{ryui}), recalling that $Rf_{\pet *}\Zz_p$ is a bounded Zariski-constructible complex of sheaves on $S$. For part \listref{fillemma:2}, using the filtered quasi-isomorphism (\ref{ryui}), it suffices to check that, up to replacing $S$ by a suitable Zariski locally closed subset containing $\eta$, the spectral sequence associated to the filtered complex $D_{\dR}(Rf_{\pet *}\Zz_p)$, i.e.
   $$H^{i-j}(\gr^jD_{\dR}(f_{\pet *}\Zz_p))\implies H^i(D_{\dR}(Rf_{\pet *}\Zz_p)),$$
   degenerates. For this, using again that $Rf_{\pet *}\Zz_p$ is a bounded Zariski-constructible complex of sheaves on $S$, by \cite[Theorem 3.9, (i) and (iii)]{LZ}, we can suppose that all the terms of the spectral sequence above are vector bundles on $S$, in which case the degeneration can be checked on stalks at classical points, where it follows from the degeneration of the spectral sequence (\ref{ehsp}) of part \listref{keya}.
   \end{enumerate}
   \end{proof}
   
   The following lemma was used in the proof above.
   
   \begin{lemma}\label{gradedul}
   Let $X$ be a smooth rigid-analytic variety over $C$. For any $r\ge 0$, we have a natural quasi-isomorphism
     $$\gr^r R\Gamma_{B_{\dR}^+}(X)\simeq \bigoplus_{0\le i\le r}R\Gamma(X, \Omega_X^i)(r-i)[-i].$$
   \end{lemma}
   \begin{proof}
   First, we note that, by Proposition \ref{bbe}\listref{bbe:2}, we have $\gr^r R\Gamma_{B_{\dR}^+}(X)\simeq R\Gamma(X, \tau^{\le r}R\alpha_* \widehat{\cl O}(r))$. We want to show that we have a natural identification
   \begin{equation}\label{2359}
    \tau^{\le r}R\alpha_* \widehat{\cl O}(r)\simeq \bigoplus_{0\le i\le r}\Omega_X^i(r-i)[-i].
   \end{equation}
   We may assume that $X$ is a smooth affinoid over $C$, and then, by \cite[Theorem 7 and Remark 2]{Elkik}, we can further assume that $X$ is the base change to $C$ of a smooth affinoid  defined over a finite extension of $K$. In this case, (\ref{2359}) follows from \cite[Corollary 6.12]{Bosco}.
   \end{proof}

 \subsection{Comparison with \cite{BMS2}}\label{nygaa} In this subsection, we compare the syntomic Fargues--Fontaine cohomology for rigid-analytic varieties over $C$, Definition \ref{synFF}, and the syntomic cohomology for semistable $p$-adic formal schemes over $\cl O_C$ defined (in the smooth case) by Bhatt--Morrow--Scholze, \cite[\S 10]{BMS2}. In particular, we show that the syntomic Fargues--Fontaine cohomology can be locally recovered from the $A_{\inf}$-cohomology together with its Nygaard filtration.
 The results proved here are not used in the rest of the paper, but we hope they will be useful for future reference. \medskip
 
 As the main comparison results of this subsection will be proven in the semistable case, we begin by recalling the definition of the $A_{\inf}$-cohomology, as well as the Nygaard filtration on it, in the latter setting. We will phrase the definition of the Nygaard filtration in terms of the décalage functors of Definition \ref{twistdec}, as this will be convenient for the desired comparison.

 \begin{df}[Nygaard filtration on $A_{\inf}$-cohomology]\label{defnyg}
  Let $\mathfrak X$ be a semistable $p$-adic formal scheme over $\cl O_C$, and let $X$ denote its generic fiber, regarded as an adic space over $\Spa(C, \cl O_C)$.  Denote by $\nu':X_{\pet}\to \mathfrak X_{\ett, \cond}$ the natural morphism of sites.\footnote{Here, the site $\mathfrak X_{\ett, \cond}$ is defined similarly to \cite[Definition 2.13]{Bosco}.} 
  \begin{enumerate}[(i)]
   \item We define the \textit{$A_{\inf}$-cohomology} of $\fr X$ as the complex of $D(\Mod^{\cond}_{A_{\inf}})$
   $$R\Gamma_{A_{\inf}}(\fr X):=R\Gamma(\fr X, A\Omega_{\fr X}), \;\;\text{ where } \;\;\; A\Omega_{\fr X}:=L\eta_{\mu}R\nu'_*\Aa_{\inf}.$$ 
   \item  Given an integer $i\ge 0$, consider the function $\delta_i:\Zz\to \Zz, j\mapsto \max(i-j, 0).$
  We endow the $A_{\inf}$-cohomology of $\fr X$  with the \textit{Nygaard filtration} whose $i$-th level is given by
   $$ \Fil_{\cl N}^iR\Gamma_{A_{\inf}}(\fr X):=R\Gamma(\fr X, \Fil_{\cl N}^i A\Omega_{\fr X}), \;\; \text{ where } \;\;\; \Fil_{\cl N}^i A\Omega_{\fr X}:= L(\eta_{\delta_i,\xi}\circ\eta_\mu)R\nu'_*\Aa_{\inf}.$$  
  \end{enumerate}
 \end{df}
 
 \begin{rem}\label{remnyg0}
  In the Definition \ref{nyg} above we implicitly used that the functor $\eta_{\delta_i,\xi}\circ\eta_\mu(-)$ preserves quasi-isomorphisms. To check the latter assertion we observe that, as $\xi=\mu/\varphi^{-1}(\mu)$, we have
  $$\eta_{\delta_i,\xi}\circ\eta_\mu=\eta_{\delta_i, \mu}\circ \eta_{-\delta_i, \varphi^{-1}(\mu)}\circ \eta_{\mu}=\eta_{\varepsilon_i, \mu}\circ \eta_{-\delta_i, \varphi^{-1}(\mu)}$$
  where $\varepsilon_i:\Zz\to \Zz,j\mapsto \max(i, j)$. Then, since both $\varepsilon_i$ and $-\delta_i$  are non-decreasing functions, we conclude by Proposition \ref{nondec}.
 \end{rem}
 
  The lemma below, combined with \cite[Proposition 9.10]{BMS2}, shows that the Nygaard filtration defined above is equivalent to the one of \cite{BMS2} (in the smooth case). \medskip
 
 First, we recall that $A\Omega_{\fr X}$ comes equipped with a Frobenius, \cite[(2.2.5)]{CK}: in fact, the Frobenius automorphism of $\Aa_{\inf}$ induces a $\varphi_{A_{\inf}}$-semilinear map $$\varphi:A\Omega_{\fr X}\to A\Omega_{\fr X}.$$
 
 \begin{lemma}\label{ease}
 Let $\tilde \xi:=\varphi(\xi)$. Under the notation of Definition \ref{defnyg}, the Frobenius $\varphi: A\Omega_{\fr X}\to A\Omega_{\fr X}$ factors functorially over a $\varphi_{A_{\inf}}$-semilinear quasi-isomorphism
 $$A\Omega_{\fr X}\overset{\sim}{\to}  L\eta_{\tilde \xi}A\Omega_{\fr X}$$
 sending the Nygaard filtration (Definition \ref{defnyg}) on the source to the filtration décalée (Definition \ref{beilifildef}) on the target.
 \end{lemma}
 \begin{proof}
  Let $A\Omega_{\fr X}^{\psh}$ denote the presheaf version of $A\Omega_{\fr X}$, \cite[\S 4.1]{CK}. Given an integer $i\ge 0$, it suffices to observe that the Frobenius automorphism of $\Aa_{\inf}$ induces a quasi-isomorphism
  $$\varphi^*\Fil_{\cl N}^i A\Omega_{\fr X}^{\psh}\simeq L(\eta_{\delta_i, \tilde \xi}\circ \eta_{\tilde \xi})A\Omega_{\fr X}^{\psh}\simeq L\eta_{\varepsilon_i, \tilde \xi}A\Omega_{\fr X}^{\psh}$$
  where we used that $\varphi(\mu)=\tilde \xi \cdot \mu$ and we denoted $\varepsilon_i:\Zz\to \Zz,j\mapsto \max(i, j)$.
 \end{proof}
 
  \begin{rem}[Frobenius action on Nygaard filtration]\label{n8}
 The Frobenius automorphism of $\Aa_{\inf}$ induces a $\varphi_{A_{\inf}}$-semilinear map $$\Fil_{\cl N}^\star(\varphi):\Fil_{\cl N}^\star A\Omega_{\fr X}\to \tilde \xi^\star\otimes A\Omega_{\fr X}.$$
 \end{rem}

 We are almost ready to define the syntomic cohomology of Bhatt--Morrow--Scholze in the semistable reduction case. We shall use the following notation.
 
 \begin{notation} We consider the Breuil--Kisin--Fargues module over $A_{\inf}$ 
 $$A_{\inf}\{1\}:=\frac{1}{\mu}(A_{\inf}\otimes_{\Zz_p}\Zz_p(1))$$
 (\cite[Example 4.24]{BMS1}) and, given $i\in \Zz$, for any $A_{\inf}$-module $M$ we denote by $$M\{i\}:=M\otimes_{A_{\inf}}A_{\inf}\{1\}^{\otimes i}$$ its $i$-th \textit{Breuil--Kisin--Fargues twist}.
 \end{notation}

 \begin{df}[Bhatt--Morrow--Scholze's syntomic cohomology, cf. {\cite[\S 10]{BMS2}}]\label{BMSsyn} Fix notation as in Definition \ref{defnyg}. Let $i\ge 0$ be an integer. We define
 $$R\Gamma_{\syn, \BMS}(\fr X, \Zz_p(i)):=\fib(\Fil_{\cl N}^iR\Gamma_{A_{\inf}}(\fr X)\{i\}\xrightarrow{\varphi\{i\}-1}R\Gamma_{A_{\inf}}(\fr X)\{i\})$$
 where $\Fil_{\cl N}^i R\Gamma_{A_{\inf}}(\fr X)\{i\}$ denotes the Breuil--Kisin--Fargues twisted $i$-th level of the Nygaard filtration on the $A_{\inf}$-cohomology of $\fr X$, and we write $\varphi\{i\}$ for the tensor product of $\Fil_{\cl N}^i(\varphi)$ (Remark \ref{n8}) with the Frobenius of $A_{\inf}\{i\}$.
 \end{df}

  Next, we want to compare Definition \ref{BMSsyn} with Definition \ref{synFF}. Recalling that the Fargues--Fontaine curve $\FF$ has a presentation given by the quotient of $Y_{\FF, [1, p]}$ via the identification $\varphi: Y_{\FF, S, [1, 1]}\cong Y_{\FF, S, [p, p]}$, we will define a \textit{Nygaard filtration} on the $B_{[1, p]}$-cohomology of rigid-analytic varieties over $C$, and we will explain how to recover the syntomic Fargues--Fontaine cohomology from the latter. \medskip
  
   Similarly to Definition \ref{defnyg}, we can give the following.

 \begin{df}[Nygaard filtration on $B_I$-cohomology]\label{nyg}
  Let $X$ be a rigid-analytic variety over $C$ and denote by $\alpha:X_{v}\to X_{\eh, \cond}$ the natural morphism of sites.  Let  $I=[1, r]\subset (0, \infty)$ be an interval with rational endpoints. 
  Given $\mathbf{B}\in \{\Bb_{I}, \Bb_{\dR}^+\}$, we write $\mathscr{B}=\mathbf{B}_{\Spa(C)_{\pet}}$.
  
 We endow  the $\mathscr{B}$-cohomology of $X$ with the \textit{Nygaard filtration} whose $i$-th level is given by
  $$\Fil_{\cl N}^i R\Gamma_{\mathscr{B}}(X):=R\Gamma(X, \Fil_{\cl N}^i L\eta_{t}R\alpha_*\mathbf{B}),\;\; \text{ where } \;\;\; \Fil_{\cl N}^i L\eta_{t}R\alpha_*\mathbf{B}:=L(\eta_{\delta_i,\xi}\circ\eta_{t})R\alpha_*\mathbf{B}.$$
 \end{df}

  \begin{rem}
  In the Definition \ref{nyg} above, we used Remark \ref{remnyg0} together with the fact that, by the choice of the interval $I$, the elements $t$ and $\mu$ differ by a unit in $B_I$.
  \end{rem}

  We note that the Nygaard filtration on the $B_{\dR}^+$-cohomology agrees with its filtration décalée, since the elements $t$ and $\xi$ generate the same ideal in $B_{\dR}^+$. Moreover, we recall that the latter filtration has a more explicit expression in coordinates, as shown in Corollary \ref{mainfil}. In a similar vein, we have the following result.

  \begin{lemma}\label{BIP}
  Let $\mathfrak X=\Spf(R)$ be a semistable $p$-adic formal scheme over $\cl O_C$ as in Notation \ref{notaz}, and let $X= \mathfrak X_C$ denote its generic fiber. Let $I=[1, r]\subset (0, \infty)$ be an interval with rational endpoints. For any $i\ge 0$, we have a $B_I$-linear quasi-isomorphism, compatible with Frobenius,
  \begin{equation}\label{beffr}
   \xi^{\max(i-\bullet, 0)}\Omega^\bullet_{B_{I}(R)}\overset{\sim}{\to} \Fil_{\cl N}^i R\Gamma_{B_I}(X)
  \end{equation}
  where $\Omega_{B_{I}(R)}^\bullet:=\Kos_{B_{I}(R)}(\partial_1, \ldots,\partial_d)$, in the notation of \S\ref{locop}.
 \end{lemma}
 \begin{proof}
 For $i=0$, the statement follows combining Lemma \ref{condstep} and Lemma \ref{primitive}.
 Arguing by induction on $i\ge 0$, to show the statement in general it suffices to check (\ref{beffr}) on graded pieces.
 
 We begin by observing that the natural map
  \begin{equation}\label{axx}
   (L\eta_tR\nu_*\Bb_I)/\Fil_{\cl N}^i\to (L\eta_tR\nu_*\Bb_{\dR}^+)/\Fil_{\cl N}^i
  \end{equation}
  is an isomorphism. In fact, we can reduce to showing that, for each $j\ge 0$, the natural map 
  \begin{equation}\label{sides}
  \gr_{\cl N}^j L\eta_tR\nu_*\Bb_I\to \gr_{\cl N}^j L\eta_tR\nu_*\Bb_{\dR}^+
  \end{equation}
  is an isomorphism (here, the graded pieces $\gr_{\cl N}^{\star}$ refer to the Nygaard filtration $\Fil_{\cl N}^{\star}$). For this, since we can replace $L\eta_t$ with $L\eta_{\xi}$ on both sides of (\ref{sides}) (recall that we have an isomorphism $B_I/\xi\overset{\sim}{\to}B_{\dR}^+/\xi$, and the elements $t$ and $\xi$ generate the same ideal in $B_{\dR}^+$), by Proposition \ref{bbe}\listref{bbe:2}, the claim reduces to the isomorphism $$\gr_{\cl N}^j\Bb_I\overset{\sim}{\to}\gr^j\Bb_{\dR}^+$$ 
  where we denote by $\gr_{\cl N}^{\star}$ the graded pieces for the $\xi$-adic filtration $\Fil_{\cl N}^{\star}$ on $\Bb_I$.
  
  Now, the desired statement, i.e. the quasi-isomorphism (\ref{beffr}) on graded pieces, follows from the quasi-isomorphisms (\ref{qisos}) of Corollary \ref{mainfil} using that, for any $j\ge 0$, the natural map $B_{I}(R)/\xi^j\to B_{\dR}^+(R)/\xi^j$ is an isomorphism.
 \end{proof}

 On the other hand, we have the following local description of the Nygaard filtration on the $A_{\inf}$-cohomology.
 
 \begin{lemma}\label{nyainf}
   Let $\mathfrak X=\Spf(R)$ be a semistable $p$-adic formal scheme over $\cl O_C$ as in Notation \ref{notaz}. For any $i\ge 0$, there is an $A_{\inf}$-linear quasi-isomorphism, compatible with the Frobenius,
    $$\Fil_{\cl N}^iR\Gamma_{A_{\inf}}(\fr X)\simeq\xi^{\max(i-\bullet, 0)}q\text{-}\Omega^\bullet_{A_{\inf}(R)}$$ where $q\text{-}\Omega^\bullet_{A_{\inf}(R)}$ denotes the \textit{logarithmic $q$-de Rham complex} defined as\footnote{See also \cite[\S 7.3]{logprism} for the construction of the \textit{logarithmic $q$-de Rham complex} in a more general setting.}
  $$q\text{-}\Omega^\bullet_{A_{\inf}(R)}:=\Kos_{A_{\inf}(R)}\left(\textstyle\frac{\partial_q}{\partial_q\log(X_1)}, \ldots, \frac{\partial_q}{\partial_q\log(X_d)}\right)$$
  with $q=[\varepsilon]\in A_{\inf}$.
 \end{lemma}
 \begin{proof}
 For $i=0$ the statement follows from \cite[Theorem 8.1, Theorem 7.17]{logprism}. The general case follows by induction on $i\ge 0$, computing the graded pieces of the filtrations, using Lemma \ref{ease} and Proposition \ref{bbe}\listref{bbe:2} (cf. \cite[Remark 9.11]{BMS2}).
 \end{proof}

  The following proposition can be regarded as a refinement of Lemma \ref{BIP}. 

 \begin{prop}[cf. {\cite[Proposition 3.12]{LeBras2}}]\label{refi}
  Let $\mathfrak X$ be a qcqs semistable $p$-adic formal scheme over $\cl O_C$, and let $X=\mathfrak X_C$ denote its generic fiber. 
  Denote by  $$\nu:X_{\pet}\to X_{\ett, \cond} \;\;\;\;\;\;\; \lambda: X_{\ett, \cond }\to \mathfrak X_{\ett, \cond}$$ the natural morphisms of sites, and let $\nu':X_{\pet}\to \mathfrak X_{\ett, \cond}$ be their composition. 
  
  Let $I=[1, r]\subset (0, \infty)$ be an interval with rational endpoints. For any $i\ge 0$, the natural map 
  \begin{equation*}\label{ghuo}
   \Fil_{\cl N}^i L\eta_{\mu}R\nu'_*\Aa_{\inf}\to R\lambda_*\Fil_{\cl N}^i L\eta_{\mu}R\nu_*\Bb_{I}
  \end{equation*}
  induces a quasi-isomorphism\footnote{Here, we use the fact that $\mu$ and $t$ differ by a unit in $B_I$.}
  \begin{equation*}\label{igino}
   \Fil_{\cl N}^iR\Gamma_{A_{\inf}}(\fr X)\dsolid_{A_{\inf}}B_{I} \overset{\sim}{\longrightarrow} \Fil_{\cl N}^i R\Gamma_{B_I}(X).
  \end{equation*}

 \end{prop}
 \begin{proof}
  We can reduce to the case $\mathfrak X=\Spf(R)$ is a semistable $p$-adic formal scheme over $\cl O_C$ as in Notation \ref{notaz}. The statement for $i=0$ is essentially contained in Lemma \ref{primitive}. In fact, combining Lemma \ref{condstep} and the quasi-isomorphism (\ref{2255}) in the proof of Lemma \ref{primitive}, we have a quasi-isomorphism
  $$R\Gamma_{B_I}(X)\simeq \Kos_{B_I(R)}\left(\frac{\gamma_1-1}{t}, \ldots, \frac{\gamma_d-1}{t}\right).$$
  Therefore, using Lemma \ref{nyainf} for $i=0$, and recalling that $t$ and $\mu=q-1$ differ by a unit in $B_I$, it suffices to check that the natural map
  $$\Kos_{A_{\inf}(R)}\left(\frac{\gamma_1-1}{q-1}, \ldots, \frac{\gamma_d-1}{q-1}\right)\dsolid_{A_{\inf}}B_I\to \Kos_{B_I(R)}\left(\frac{\gamma_1-1}{q-1}, \ldots, \frac{\gamma_d-1}{q-1}\right)$$
  is a quasi-isomorphism;  this can be done as in the proof of (\ref{bc1}) in Lemma \ref{primitive}. Finally, the statement in general follows arguing by induction on $i\ge 0$, calculating again the graded pieces of the filtrations.
 \end{proof}


 We can finally state and prove the main result of this subsection, which in particular tells us how the syntomic Fargues--Fontaine cohomology can be locally recovered from the $A_{\inf}$-cohomology together with its Nygaard filtration.

 \begin{prop}\label{crucialsyn} Let $X$ be a rigid-analytic variety over $C$. Denote $I=[1, p]$ and $I'=[1, 1]$. Let $i\ge 0$ be an integer. 
 \begin{enumerate}[(i)]
  \item\label{crucialsyn:1} We have a natural isomorphism in $D(\Vect_{\Qq_p}^{\cond})$
  $$R\Gamma_{\syn, \FF}(X, \Qq_p(i))\simeq \fib(\Fil_{\cl N}^iR\Gamma_{B_I}(X)\xrightarrow{\varphi p^{-i}-1} R\Gamma_{B_{I'}}(X)).$$
  \item\label{crucialsyn:2} Assume that $X$ is the generic fiber of  a qcqs semistable $p$-adic formal scheme $\mathfrak X$ over $\cl O_C$. Then, we have a natural isomorphism in $D(\Vect_{\Qq_p}^{\cond})$
  $$R\Gamma_{\syn, \FF}(X, \Qq_p(i))\simeq \fib(\Fil_{\cl N}^iR\Gamma_{A_{\inf}}(\fr X)\{i\}\dsolid_{A_{\inf}}B_{I}\xrightarrow{\varphi\{i\}-1} R\Gamma_{A_{\inf}}(\fr X)\{i\}\dsolid_{A_{\inf}}B_{I'}).$$
  In particular, there is a natural morphism
  $$R\Gamma_{\syn, \BMS}(\fr X, \Zz_p(i))\longrightarrow R\Gamma_{\syn, \FF}(X, \Qq_p(i)).$$
 \end{enumerate}
 \end{prop}
 
 \begin{proof}
  For part \listref{crucialsyn:1}, using the isomorphism
  $$R\Gamma_{B_{I}}(X)/\Fil_{\cl N}^i\overset{\sim}{\to} R\Gamma_{B_{\dR}^+}(X)/\Fil^i$$
  coming from (\ref{axx}), we have a natural isomorphism
  \begin{equation}\label{combob}
    \fib(\Fil_{\cl N}^iR\Gamma_{B_I}(X)\xrightarrow{\varphi p^{-i}-1} R\Gamma_{B_{I'}}(X))\simeq \fib(R\Gamma_{B_I}(X)^{\varphi=p^i}\to R\Gamma_{B_{\dR}^+}(X)/\Fil^i)
  \end{equation}
   where
  $$R\Gamma_{B_I}(X)^{\varphi=p^i}:=\fib(R\Gamma_{B_I}(X)\xrightarrow{\varphi p^{-i}-1} R\Gamma_{B_{I'}}(X)).$$
  Then, combining (\ref{combob}) with Theorem \ref{BK=pet}\listref{BK=pet:2}, it remains to show that the natural map
  \begin{equation}\label{tramm}
   R\Gamma_{B}(X)^{\varphi=p^i}\to R\Gamma_{B_I}(X)^{\varphi=p^i}
  \end{equation}
  is an isomorphism. For this, in the notation of \S \ref{FFperf}, we observe that by Proposition \ref{perff} the source of (\ref{tramm}) is computed by the cohomology of $\cl H_{\FF}(X)\otimes \cl O(i)\in \Nuc(\FF)$. But the latter cohomology also computes the target of (\ref{tramm}), using the presentation of the curve $\FF$ as the quotient of $Y_{\FF, [1, p]}$ via the identification $\varphi: Y_{\FF, [1, 1]}\cong Y_{\FF, [p, p]}$. This concludes the proof of part \listref{crucialsyn:1}.
    
 For part \listref{crucialsyn:2}, by part \listref{crucialsyn:1} and Proposition \ref{refi} we have a natural isomorphism in $D(\Vect_{\Qq_p}^{\cond})$
 \begin{equation}\label{misfib}
  R\Gamma_{\syn, \FF}(X, \Qq_p(i))\simeq \fib(\Fil_{\cl N}^iR\Gamma_{A_{\inf}}(\fr X)\dsolid_{A_{\inf}}B_{I}\xrightarrow{\varphi p^{-i}-1} R\Gamma_{A_{\inf}}(\fr X)\dsolid_{A_{\inf}}B_{I'}).
 \end{equation}
 On the other hand,  trivializing the Breuil--Kisin--Fargues twists we can rewrite the fiber in the statement of part \listref{crucialsyn:2} as
 \begin{equation}\label{misfib2}
  \fib(\Fil_{\cl N}^iR\Gamma_{A_{\inf}}(\fr X)\dsolid_{A_{\inf}}B_{I}\xrightarrow{\varphi \tilde{\xi}^{-i}-1} R\Gamma_{A_{\inf}}(\fr X)\dsolid_{A_{\inf}}B_{I'})
 \end{equation}
 where $\tilde \xi:=\varphi(\xi)$. We conclude observing that, writing $\mu=ut$ with $u$ unit in $B_I$, the multiplication by $u^i$ map induces an isomorphism between the fiber in (\ref{misfib2}) and the fiber in (\ref{misfib}).
 
  \end{proof}

 \subsection{Comparison with \cite{CN}} In this subsection, we show that, in high degrees, the syntomic Fargues--Fontaine cohomology does not agree with the syntomic cohomology for smooth rigid-analytic varieties over $C$ defined by Colmez--Nizioł.
  
  \begin{notation}
   Let $X$ be a smooth rigid-analytic variety over $C$.  Let $i\ge 0$ be an integer. We denote by $R\Gamma_{\syn, \CN}(X, \Qq_p(i))$ the \textit{syntomic cohomology of $X$ with coefficients in $\Qq_p(i)$ of Colmez--Nizioł}, defined in \cite[\S 4.1]{CN}.
  \end{notation}

 \begin{example}\label{FFvsCN}
  Let $X=\Pp^1_{C}$ be the rigid-analytic projective line over $C$. We claim that
  \begin{equation}\label{dd}
   H^3_{\syn, \FF}(X, \Qq_p(0))\cong C/\Qq_p,\;\;\;\;\;\; H^3_{\syn, \CN}(X, \Qq_p(0))\cong 0.
  \end{equation}
  For this, applying Theorem \ref{BK=pet}\listref{BK=pet:1}, combined with Theorem \ref{B=HK}, for the syntomic Fargues--Fontaine cohomology, and \cite[Corollary 5.5]{CN4} for the syntomic cohomology of Colmez--Nizioł, we obtain respectively
  $$H^3_{\syn, \FF}(X, \Qq_p(0))\cong (H_{\HK}^2(X)\otimes_{\breve F} B)/(\varphi-1),\;\;\;\;\;\; H^3_{\syn, \CN}(X, \Qq_p(0))\cong (H_{\HK}^2(X)\otimes_{\breve F} B_{\cris}^+)/(\varphi-1)$$
  where we used that $H_{\HK}^3(X)\cong 0$. Now, the map $\varphi-1$ is surjective on $H_{\HK}^2(X)\otimes_{\breve F} B_{\cris}^+$ (see e.g. \cite[Remark 2.30]{CN1}). Instead, observing that $H_{\HK}^2(X)$ is a one-dimensional $\varphi$-module over $\breve F$ with slope 1 (by Theorem \ref{mainHKover}\listref{mainHKover:1}, and \cite[Théorème 3.1.2]{CLS}\footnote{In this case, in order to deduce that $H_{\HK}^2(X)$ has slope 1, it is sufficient to use the weak Lefschetz theorem for crystalline cohomology, \cite{BerthLef}.}), we deduce that the vector bundle on the Fargues--Fontaine curve $\FF$ associated to $H_{\HK}^2(X)$ is isomorphic to $\cl O(-1)$; hence, we can identify $(H_{\HK}^2(X)\otimes_{\breve F} B)/(\varphi-1)$ with $H^1(\FF, \cl O(-1))$, thus showing (\ref{dd}).
 \end{example}

 \section{\textbf{Applications}}\label{applleit}
  \sectionmark{}

 In this section, we gather the results we have obtained so far, giving some applications.
 
  \subsection{Fundamental diagrams of rational $p$-adic Hodge theory}
 
 \begin{theorem}\label{fund1}
 Let $X$ be a qcqs rigid-analytic variety defined over $K$. We have a $\mathscr{G}_K$-equivariant pullback square in $D(\Vect_{\Qq_p}^{\ssolid})$
 \begin{center}
  \begin{tikzcd}
  R\Gamma_{\pet}(X_C, \Qq_p) \arrow[r] \arrow[d] & (R\Gamma_{\HK}(X_C)\dsolid_{\breve F}B_{\log}[1/t])^{N=0, \varphi=1} \arrow[d]  \\
  \Fil^0(R\Gamma_{\dR}(X)\dsolid_K B_{\dR}) \arrow[r] & R\Gamma_{\dR}(X)\dsolid_K B_{\dR}.
  \end{tikzcd}
 \end{center}
 \end{theorem}
 
 \begin{proof}
  First, we note that we can rewrite the fundamental exact sequence (\ref{fundexact}) of $p$-adic Hodge theory on the pro-étale site $X_{C, \pet}$ as a pullback square
  \begin{center}
  \begin{tikzcd}
  \Qq_p \arrow[r] \arrow[d] & \Bb[1/t]^{\varphi=1} \arrow[d]  \\
  \Bb_{\dR}^+ \arrow[r] & \Bb_{\dR}.
  \end{tikzcd}
 \end{center}
  Then, the statement follows combining  Theorem \ref{B=HK}, \cite[Theorem 6.5]{Bosco} (together with Theorem \ref{tosingg} for the singular case), as well as the compatibility proven in Theorem \ref{compatib}. In fact, by Theorem \ref{B=HK}, using that $X$ is qcqs, we have
  $$R\Gamma_{\pet}(X, \Bb[1/t]^{\varphi=1})\simeq R\Gamma_{B}(X)[1/t]^{\varphi=1}\simeq (R\Gamma_{\HK}(X_C)\dsolid_{\breve F}B_{\log}[1/t])^{N=0, \varphi=1}$$
  where in the last step we used that $\dsolid_{\breve F}$ commutes with filtered colimits.
 \end{proof}

 We invite the reader to compare the following result with \cite{CN4}.

 \begin{theorem}\label{tog}
  Let $X$ be a connected, paracompact, rigid-analytic variety defined over $K$. For any $i\ge 0$, we have a  $\mathscr{G}_K$-equivariant isomorphism in $D(\Vect_{\Qq_p}^{\ssolid})$
  $$\tau^{\le i}R\Gamma_{\pet}(X_C, \Qq_p(i))\simeq \tau^{\le i}\fib((R\Gamma_{\HK}(X_C)\dsolid_{\breve F}B_{\log})^{N=0, \varphi=p^i}\to (R\Gamma_{\dR}(X)\dsolid_K B_{\dR}^+)/\Fil^i).$$
 \end{theorem}
 \begin{proof}
  This follows combining  Theorem \ref{BK=pet}, Theorem \ref{B=HK}, Theorem \ref{B_dR=dR} and Theorem \ref{compatib}.
 \end{proof}


 In some special cases, we can explicitly compute the cohomology groups of the de Rham contribution in the fiber sequence of Theorem \ref{tog}.
 
 \begin{prop}\label{expic}
   Let $X$ be a rigid-analytic variety over $K$. Let $i\ge 0$,  and denote
  $$\dR(X, i):=(R\Gamma_{\dR}(X)\dsolid_K B_{\dR}^+)/\Fil^i.$$
  \begin{enumerate}[(i)]
  \item \label{expic:1} If $X$ is proper, for any $j\ge 0$, we have a $\mathscr{G}_K$-equivariant isomorphism in $\Vect_K^{\ssolid}$ 
   $$H^j(\dR(X, i))\cong (H^j_{\dR}(X)\otimes_K B_{\dR}^+)/\Fil^i.$$
  \item \label{expic:2} If $X$ is a smooth affinoid or Stein space,  for $j \ge i$ we have $H^j(\dR(X, i))=0$, and for $0\le j< i$ we have a $\mathscr{G}_K$-equivariant exact sequence in $\Vect_K^{\ssolid}$
    $$0\to (\Omega^j(X_C)/\im d)(i-j-1) \to H^j(\dR(X, i)) \to H^j_{\dR}(X)\solid_K B_{\dR}^+/t^{i-j-1}\to 0.$$
  \end{enumerate} 
 \end{prop}
 \begin{proof} One can argue similarly to the proof of \cite[Corollary 6.17]{Bosco}. Part \listref{expic:1} follows from (the proof of) Proposition \ref{fillemma}\listref{fillemma:2}. For part \listref{expic:2}, using Tate's acyclicity theorem for affinoid spaces, and Kiehl's ayclicity theorem for Stein spaces (which hold true in the condensed setting,  \cite[Lemma 5.6(i), Lemma 5.9]{Bosco}), and relying crucially on the flatness of the $K$-Fréchet space $B_{\dR}^+$ (and its filtered pieces) for the solid tensor product $\solid_K$ (\cite[Corollary A.65]{Bosco}), we have
 $$\dR(X, i)\simeq [\cl O(X)\solid_K B_{\dR}^+/t^i\to \Omega^1(X)\solid_K B_{\dR}^+/t^{i-1}\to \cdots \to \Omega^{i-1}(X)\solid_K B_{\dR}^+/t]$$
 from which one readily deduces the statement.
 \end{proof}

  \subsection{Proper spaces} In this subsection, we prove a version of the semistable conjecture for proper (possibly singular) rigid-analytic varieties over $C$. We remark that in the smooth case the following result is already known, \cite[Theorem 5.8]{CN4}; however, already in the latter case our proof is different from \textit{loc. cit.} as it does not rely on Fontaine--Messing syntomic cohomology.
 
  \begin{theorem}\label{propsing} Let $X$ be a proper rigid-analytic variety over $C$. For each $i\ge 0$, we have a natural isomorphism
  \begin{equation}\label{ssgen}
   H^i_{\ett}(X, \Qq_p)\otimes_{\Qq_p}B_{\log}[1/t]\cong H^i_{\HK}(X)\otimes_{\breve F}B_{\log}[1/t]
  \end{equation}
  compatible with the actions of the Frobenius $\varphi$ and the monodromy $N$, and inducing a natural isomorphism
  \begin{equation}\label{dRgen}
   H^i_{\ett}(X, \Qq_p)\otimes_{\Qq_p}B_{\dR}\cong H^i_{\inf}(X/B_{\dR}^+)\otimes_{B_{\dR}^+}B_{\dR}
  \end{equation}
  compatible with filtrations. 
  In particular, we have a natural isomorphism
  \begin{equation}\label{bfbeef}
   H^i_{\ett}(X, \Qq_p)\cong (H^i_{\HK}(X)\otimes_{\breve F}B_{\log}[1/t])^{\varphi=1, N=0}\cap \Fil^0(H^i_{\inf}(X/B_{\dR}^+)\otimes_{B_{\dR}^+}B_{\dR}).
  \end{equation}
  Here, the filtration on $H^i_{\inf}(X/B_{\dR}^+)$ is defined by 
  \begin{equation*}
   \Fil^\star H^i_{\inf}(X/B_{\dR}^+):=\im(H^i(\Fil^\star R\Gamma_{\inf}(X/B_{\dR}^+))\to H^i_{\inf}(X/B_{\dR}^+))
  \end{equation*}
  where $R\Gamma_{\inf}(X/B_{\dR}^+)$ is endowed with the Hodge filtration (Definition \ref{hdg}).
   \end{theorem}   
  \begin{proof}
   Let us fix $i\ge 0$. By the properness of $X$, using Scholze's primitive comparison theorem, and the finiteness of the $\Qq_p$-vector space $H_{\ett}^i(X, \Qq_p)$, \cite[Theorem 3.17]{ScholzeSurvey}, we have a natural isomorphism of vector bundles on the Fargues--Fontaine curve $\FF$
   \begin{equation}\label{farguess}
    H^i_{\ett}(X, \Qq_p)\otimes_{\Qq_p}\cl O_{\FF}\cong \cl E(H_{\pet}^i(X, \Bb_e), H_{\pet}^i(X, \Bb_{\dR}^+))
   \end{equation}
   where the right-hand side of (\ref{farguess}) denotes the vector bundle on $\FF$ associated to the $B$-pair $(H_{\pet}^i(X, \Bb_e), H_{\pet}^i(X, \Bb_{\dR}^+))$ with $\Bb_e=\Bb[1/t]^{\varphi=1}$. 
   We claim that $(\ref{farguess})$ is naturally isomorphic to the vector bundle on $\FF$ associated to the $B$-pair
   \begin{equation}\label{bpair1}
    (H_{\HK}^i(X)\otimes_{\breve F}B_{\log}[1/t])^{N=0, \varphi=1}, \Fil^0(H^i(X/B_{\dR}^+)\otimes_{B_{\dR}^+}B_{\dR})).
   \end{equation}
   For this, using Theorem \ref{B=HK}, and the perfectness of $R\Gamma_{\HK}(X)$ over $\breve F$ proven in Theorem \ref{slopp}\listref{slopp:1} (combined with Proposition \ref{parpar}), we have a natural isomorphism
   \begin{equation}\label{tkN}
    R\Gamma_{\pet}(X, \Bb_e)\simeq (R\Gamma_{\HK}(X)\otimes_{\breve F}B_{\log}[1/t])^{N=0, \varphi=1}.
   \end{equation}
   Taking cohomology of (\ref{tkN}), by Lemma \ref{unex} combined with Lemma \ref{unex2}, we have a natural isomorphism
   \begin{equation}\label{isos1}
    H^i_{\pet}(X, \Bb_e)\cong (H_{\HK}^i(X)\otimes_{\breve F}B_{\log}[1/t])^{N=0, \varphi=1}.
   \end{equation}
   Moreover, by Theorem \ref{secondstep1} and Proposition \ref{fillemma}, we have a natural isomorphism
   \begin{equation}\label{isos2}
    H^i_{\pet}(X, \Bb_{\dR}^+)\cong \Fil^0(H^i(X/B_{\dR}^+)\otimes_{B_{\dR}^+}B_{\dR}).
   \end{equation}
   Hence, the desired claim follows combining the isomorphisms (\ref{isos1}) and (\ref{isos2}), and the compatibility shown in Theorem \ref{compatib2}\listref{compatib2:1}.
   
   Now, we are ready to prove that we have a natural isomorphism (\ref{ssgen}) as in the statement. From what we have shown above, applying $H^0(\FF, -)$ to (\ref{farguess}) we obtain (\ref{bfbeef}), from which we deduce that we have a natural $B_{\log}[1/t]$-linear injective map
   \begin{equation}\label{injfirst}
    H^i_{\ett}(X, \Qq_p)\otimes_{\Qq_p}B_{\log}[1/t]\to  H^i_{\HK}(X)\otimes_{\breve F}B_{\log}[1/t]
   \end{equation}
  compatible with the actions of the Frobenius $\varphi$ and the monodromy $N$. To conclude that (\ref{injfirst}) is an isomorphism, we observe that
   $$\dim_{\Qq_p}H^i_{\ett}(X, \Qq_p)=\dim_{\breve F}H_{\HK}^i(X).$$
   For this, we note that $\dim_{\Qq_p}H^i_{\ett}(X, \Qq_p)$ is equal to the rank of the vector bundle (\ref{farguess}) on $\FF$, and hence, from what we have shown above, it is equal to the rank of the vector bundle associated to the $B$-pair (\ref{bpair1}); but the latter is a modification at $\infty$ of the vector bundle on $\FF$ associated to the finite $(\varphi, N)$-module $H_{\HK}^i(X)$ over $\breve F$, whose rank is $\dim_{\breve F}H_{\HK}^i(X)$.
   
   Lastly, we have that the isomorphism (\ref{ssgen}) induces an isomorphism (\ref{dRgen}) which is compatible with filtrations, recalling again Theorem \ref{compatib2}, Theorem \ref{secondstep1}, and Proposition \ref{fillemma}.
   \end{proof}


  We used the following general results.
 
 \begin{lemma}\label{unex}
  For any finite $\varphi$-module $(V, \varphi)$ over $\breve F$, the map
  $$\varphi-1: V\otimes_{\breve F}B[1/t]\to V\otimes_{\breve F}B[1/t]$$
  is surjective.
 \end{lemma}
 \begin{proof}
 It suffices to show that for any sufficiently big integer $m$, the map
 \begin{equation}\label{t-m}
  \varphi-1: V\otimes_{\breve F}t^{-m}B\to V\otimes_{\breve F}t^{-m}B
 \end{equation}
 is surjective. For this, we consider $\cl E:=\cl E(V, \varphi)$ the vector bundle on the Fargues--Fontaine curve $\FF$ associated to $(V, \varphi)$, and, for any integer $m$, the vector bundle $\cl E(m):=\cl E\otimes \cl O(m)$ on $\FF$. Note that we have
 \begin{align*}
   R\Gamma(\FF, \cl E(m))&=[H^0(Y_{\FF}, \cl E|_{Y_{\FF}})\xrightarrow{\varphi p^{-m}-1}H^0(Y_{\FF}, \cl E|_{Y_{\FF}})] \\
                         &=[V\otimes_{\breve F}B\xrightarrow{\varphi p^{-m}-1}V\otimes_{\breve F}B].
 \end{align*}
 Then, since for any integer $m$ sufficiently big such that the vector bundle $\cl E(m)$ has non-negative Harder--Narasimhan slopes, one has $H^1(\FF, \cl E(m))=0$, we deduce that (recalling that $\varphi(t)=pt$) for any such $m$ the map (\ref{t-m}) is surjective, as desired.
 \end{proof}

 \begin{lemma}\label{unex2}
 For any finite $(\varphi, N)$-module $(V, \varphi, N)$ over $\breve F$, we have a short exact sequence
 \begin{equation}\label{NU}
  0\to V\otimes_{\breve F} B \overset{\alpha}{\to} V\otimes_{\breve F} B_{\log} \overset{N}{\to} V\otimes_{\breve F} B_{\log}\to 0
 \end{equation}
 where, recalling that $B_{\log}=B[U]$ (Definition \ref{deflogcrys}), the morphism $\alpha$ is induced by the isomorphism of finite $\varphi$-modules over $\breve F$
 $$\exp(N\cdot U):V\otimes_{\breve F} B \overset{\sim}{\to} (V\otimes_{\breve F} B_{\log})^{N=0}:\;\; x\mapsto \sum_{j\ge 0}\frac{(-1)^j}{j!}N^j(x)\cdot U^j.$$ 
 \end{lemma}
 \begin{proof}
  For any finite $(\varphi, N)$-module $(V, \varphi, N)$ over $\breve F$, the monodromy operator $N$ has finite nilpotency index. If the nilpotency index is 1, i.e. $N=0$ on $V$, the statement follows from the exactness of the sequence (\ref{NU}) for $V=\breve F$. The statement in the general case follows by induction on such nilpotency index (cf. the proof of \cite[Lemma 3.20]{CDN1}).
 \end{proof}

 \subsection{Smooth Stein spaces}\label{sstein}
  
 In this subsection, our goal is to prove the following result, Theorem \ref{steindiagram}. We remark that it could be deduced from Theorem \ref{tog},  Proposition \ref{expic}\listref{expic:2} and the theory of Banach--Colmez spaces, as done in \cite{CDN1}. However, we give here a more direct proof using the relative fundamental exact sequence of $p$-adic Hodge theory.

 \begin{theorem}[cf. \cite{CDN1}, \cite{CN5}]\label{steindiagram}
 Let $X$ be a smooth Stein space over $C$. For any $i\ge 0$, we have a commutative diagram in $\Vect_{\Qq_p}^{\ssolid}$ with exact rows
 \begin{center}
   \begin{tikzcd}\label{bvac}
   0 \arrow[r] 
     &[-1em] \Omega^{i-1}(X)/\ker d \arrow[-,double line with arrow={-,-}]{d} \arrow[r] 
     &[-1em] H^i_{\pet}(X, \Qq_p(i)) \arrow[d] \arrow[r]      
     &[-1em] (H^i_{\HK}(X)\solid_{\breve F}B_{\log})^{N=0, \varphi=p^i} \arrow[d] \arrow[r] 
     &[-1em] 0  \\
   0 \arrow[r] 
     &[-1em] \Omega^{i-1}(X)/\ker d \arrow[r]           
     &[-1em] \Omega^i(X)^{d=0} \arrow[r]  
     &[-1em] H^i_{\dR}(X) \arrow[r]    
     &[-1em] 0.
  \end{tikzcd}
  \end{center}
 \end{theorem}

 \subsubsection{\normalfont{\textbf{Recollections on Banach--Colmez spaces}}}\label{reviewbc}
 
 As a preparation for the proof of Theorem \ref{steindiagram}, we need further reminders on the category of Banach--Colmez spaces  $\cl {BC}$ (Definition \ref{defbc}). \medskip

 Recall that we denote by
 \begin{equation*}\label{preshad}
  \tau: \FF_v\to \Spa(C, \cl O_C)_v
 \end{equation*}
 the natural morphism of $v$-sites.
 
 We remark that the equivalence of derived categories (\ref{equivder}) in Proposition \ref{perffALB2} does not preserve the natural $t$-structures. In fact, Le Bras showed that suitably changing $t$-structure on the source of (\ref{equivder}) one can pass to the category of BC spaces:
 
 \begin{prop}[{\cite[Théorème 1.2]{LeBras1}}]\label{lbc} The category of Banach--Colmez spaces $\cl {BC}$ is equivalent, via the functor $\tau_*$, to the full subcategory of $\Perf(\FF)$ having as objects the perfect complexes $\cl F$ concentrated in cohomological degrees $[-1, 0]$ such that $H^{-1}(\cl F)$ has negative Harder--Narasimhan slopes and $H^0(\cl F)$ has non-negative Harder--Narasimhan slopes.
 \end{prop}

 In the following, for $\cl E$ a vector bundle on $\FF$, we denote the BC spaces
 \begin{equation}\label{expli}
  \cl H^0(\FF, \cl E):=\tau_*\cl E,\;\;\;\;\;\; \cl H^1(\FF, \cl E):=R^1\tau_*\cl E.
 \end{equation}
 
 Let us recall that Colmez defined a \textit{Dimension} function on BC spaces (\cite[\S 6]{Colmez})
 $$\Dim=(\dim, \htt): \cl{BC}\to \Nn\times \Zz$$
 where $\dim$ is called \textit{dimension} and $\htt$ \textit{height}.\footnote{We refer the reader to \cite[\S 7.2]{LeBras1} for the relation between Colmez's original definition of \textit{Espace de Banach de dimension finie} and Definition \ref{defbc}.} In terms of the Fargues--Fontaine curve, the function $\Dim$ is characterized by the following properties:
 \begin{enumerate} [(i)]
  \item  The function $\Dim$ is additive in short exact sequences; 
  \item For any vector bundle $\cl E$ on $\FF$,  denoting by $\deg(\cl E)$ the degree of $\cl E$ and by $\rk(\cl E)$ its rank,
 \begin{equation}\label{ep}
  \chi(\cl E):=\Dim \cl H^0(\FF, \cl E)-\Dim \cl H^1(\FF, \cl E)=(\deg(\cl E), \rk(\cl E))
 \end{equation}
 where $\chi(\cl E)$ is the \textit{Euler-Poincaré characteristic of $\cl E$} (see e.g. \cite[Préface, Remarque 4.6]{FF}).
 \end{enumerate}

 \begin{rem}[BC spaces vs. condensed $\Qq_p$-vector spaces] Let $\cl E$ be a vector bundle on $\FF$. In the following, we will need to consider the cohomology groups of $\cl E$ on $\FF$ as condensed $\Qq_p$-vector spaces. For this, we denote by
 \begin{equation*}\label{shad}
  f: \FF_{\pet}\to \Spa(C, \cl O_C)_{\pet}
 \end{equation*}
 the natural morphism of (small) sites, and we define
 \begin{equation*}\label{expli2}
  H^0(\FF, \cl E):=f_*\cl E,\;\;\;\;\;\; H^1(\FF, \cl E):=R^1f_*\cl E.
 \end{equation*}
 Note that the condensed structure on such cohomology groups is just the ``shadow'' of the structure of BC spaces (\ref{expli}). As an example, we have the identification $\cl H^0(\FF, \cl O(-1))=(\Aa_C^{1})^\diamondsuit/\underline{\Qq_p}$ as BC spaces, which restricted to the site $\Spa(C, \cl O_C)_{\pet}$ gives the identification $H^0(\FF, \cl O(-1))=C/\Qq_p$ as condensed $\Qq_p$-vector spaces.
 \end{rem}

 We are ready to prove the main result of this subsection.
 
 \begin{proof}[Proof of Theorem \ref{steindiagram}]
 Let $X^\dagger$ be the smooth Stein dagger space associated with $X$ (via \cite[Theorem 2.27]{GK1}).
 Choose $\{U_n^\dagger\}_{n\in \Nn}$ a Stein covering of $X^\dagger$, and denote by $\{U_n\}_{n\in \Nn}$ the corresponding Stein covering of $X$ (i.e. set $U_n:=\widehat U_n^{\dagger}$). Fix $n\in \Nn$, and let $V^\dagger:=U_n^\dagger$. Our first goal is to show that we have a diagram as in the statement replacing $X$ with $V^\dagger$. 
 
 Denoting by $\varprojlim_h V_h$ the presentation of a dagger structure on $V$ corresponding to $V^\dagger$ (recall Lemma \ref{dagglemma}\listref{dagglemma:1}), for $\Mm$ a sheaf on the big pro-étale site $\Rig_{C, \pet}$, we set
 $$R\Gamma_{\pet}(V^\dagger, \Mm):=\colim_{h\in \Nn}R\Gamma(V_h, \Mm).$$
 With this definition, by the relative fundamental exact sequence of $p$-adic Hodge theory (\ref{fundexact}), we have the following commutative diagram with exact rows
 \begin{center}
   \begin{tikzcd}
   \cdots \arrow[r] &[-1em] H_{\pet}^i(V^\dagger, \Qq_p) \arrow[r]\arrow[d] &[-1em] H_{\pet}^i(V^\dagger, \Bb_{e}) \arrow[r, "\alpha_i"]\arrow[d] &[-1em]H_{\pet}^i(V^\dagger, \Bb_{\dR}/\Bb_{\dR}^+)\arrow[r]\arrow[-,double line with arrow={-,-}]{d} &[-1em] \cdots \\
   \cdots \arrow[r] &[-1em] H_{\pet}^i(V^\dagger, \Bb_{\dR}^+)\arrow[r]  &[-1em]H_{\pet}^i(V^\dagger, \Bb_{\dR}) \arrow[r, "\beta_i"]  &[-1em] H_{\pet}^i(V^\dagger, \Bb_{\dR}/\Bb_{\dR}^+) \arrow[r]  &[-1em] \cdots
  \end{tikzcd}
  \end{center}
 from which we obtain the following diagram with exact rows
  \begin{equation}\label{anv}
  \begin{tikzcd}
   0 \arrow[r] &[-1em] \coker\alpha_{i-1} \arrow[r]\arrow[d] &[-1em]   H_{\pet}^i(V^\dagger, \Qq_p) \arrow[r]\arrow[d]  &[-1em] \ker \alpha_i \arrow[r]\arrow[d] &[-1em] 0 \\
   0 \arrow[r] &[-1em] \coker\beta_{i-1} \arrow[r]   &[-1em]  H_{\pet}^i(V^\dagger, \Bb_{\dR}^+)\arrow[r]  &[-1em]  \ker \beta_i \arrow[r] &[-1em] 0.
  \end{tikzcd}
  \end{equation}
  By \cite[Theorem 7 and Remark 2]{Elkik}, we may assume that $V^\dagger$ is the base change to $C$ of a smooth dagger affinoid $V_0^\dagger$  defined over a finite extension $L$ of $K$. Then, using Theorem \ref{B=HK} together with the perfectness of $R\Gamma_{\HK}(V^\dagger)$ over $\breve F$ (Theorem \ref{slopp}\listref{slopp:1}), Lemma \ref{unex} and Lemma \ref{unex2} in order to compute $H_{\pet}^i(V^\dagger, \Bb_{e})$, and relying on \cite[Corollary 6.17(ii)]{Bosco} to determine $H_{\pet}^i(V^\dagger, \Bb_{\dR}/\Bb_{\dR}^+)$, recalling the compatibility proven in Theorem \ref{compatib}\listref{compatib:2}, we have the following commutative diagram with exact rows
\begin{center}
   \begin{tikzcd}
   0 \arrow[r] 
     &[-1em] (H^i_{\HK}(V^\dagger)\otimes_{\breve F}B_{\log}[1/t])^{N=0, \varphi=1} \arrow[d, "\gamma_i"] \arrow[r, "\sim"] 
     &[-1em] H_{\pet}^i(V^\dagger, \Bb_{e}) \arrow[d, "\alpha_i"] \arrow[r]      
     &[-1em] 0 \arrow[d] \arrow[r] 
     &[-1em] 0  \\
   0 \arrow[r] 
     &[-1em] H_{\dR}^i(V_0^\dagger)\otimes_L B_{\dR}/t^{-i}B_{\dR}^+ \arrow[r]           
     &[-1em] H_{\pet}^i(V^\dagger, \Bb_{\dR}/\Bb_{\dR}^+) \arrow[r]  
     &[-1em] (\Omega^i(V^\dagger)/\ker d)(-i-1) \arrow[r]    
     &[-1em] 0.
  \end{tikzcd}
  \end{center}
  We claim that
  \begin{equation}\label{claimo}
   \ker \gamma_i= (H^i_{\HK}(V^\dagger)\otimes_{\breve F}B_{\log})^{N=0, \varphi=p^i} \;\;\;\;\; \text{ and } \;\;\;\;\;\coker \gamma_i=0.
  \end{equation}
  For this, we consider the vector bundle $\cl E(H_{\HK}^i(V^\dagger))$ on the Fargues--Fontaine curve $\FF$ associated to the finite $(\varphi, N)$-module $H_{\HK}^i(V^\dagger)$ over $\breve F$. By Theorem \ref{slopp}\listref{slopp:2}, the vector bundle $\cl E=\cl E(H_{\HK}^i(V^\dagger))\otimes \cl O(i)$ has non-negative Harder--Narasimhan slopes, in particular $H^1(\FF, \cl E)=0$, and we have the short exact sequence
  $$0\to H^0(\FF, \cl E)\to H^0(\FF\setminus\{\infty\}, \cl E)\to \cl E^\wedge_{\infty}[1/t]/\cl E^\wedge_{\infty}\to 0.$$
  Note that such short exact sequence identifies with the short exact sequence
  $$0\to (H^i_{\HK}(V^\dagger)\otimes_{\breve F}B_{\log})^{N=0, \varphi=p^i}\to (H^i_{\HK}(V^\dagger)\otimes_{\breve F}B_{\log}[1/t])^{N=0, \varphi=1}\overset{\gamma_i}{\to} H_{\dR}^i(V_0^\dagger)\otimes_L B_{\dR}/t^{-i}B_{\dR}^+\to 0$$
  thus proving the claim (\ref{claimo}).
  
  Then, twisting by $(i)$ the diagram (\ref{anv}), putting everything together we deduce that, for each $n\in \Nn$, we have a commutative diagram with exact rows
   \begin{equation}\label{lastt}
   \begin{tikzcd}
   0 \arrow[r] 
     &[-1em] \Omega^{i-1}(U_n^\dagger)/\ker d \arrow[-,double line with arrow={-,-}]{d} \arrow[r] 
     &[-1em] H^i_{\pet}(U_n^\dagger, \Qq_p(i)) \arrow[d] \arrow[r]      
     &[-1em] (H^i_{\HK}(U_n^\dagger)\solid_{\breve F}B_{\log})^{N=0, \varphi=p^i} \arrow[d] \arrow[r] 
     &[-1em] 0  \\
   0 \arrow[r] 
     &[-1em] \Omega^{i-1}(U_n^\dagger)/\ker d \arrow[r]           
     &[-1em] \Omega^i(U_n^\dagger)^{d=0} \arrow[r]  
     &[-1em] H^i_{\dR}(U_n^\dagger) \arrow[r]    
     &[-1em] 0.
  \end{tikzcd}
  \end{equation}
  Now, since we have
  $$R\Gamma_{\pet}(X, \Qq_p(i))=R\Gamma_{\pet}(X^\dagger, \Qq_p(i))=R\varprojlim_{n}R\Gamma_{\pet}(U_n^{\dagger}, \Qq_p(i))$$
  and similarly
  $$R\Gamma_{\HK}(X)=R\varprojlim_{n}R\Gamma_{\HK}(U_n^{\dagger}),\;\;\;\;\;\;\;R\Gamma_{\dR}(X)=R\varprojlim_{n}R\Gamma_{\dR}(U_n^{\dagger})$$
  (see Proposition \ref{parpar} and cf. \cite[Proposition 3.17]{CN}), recalling the property \cite[Corollary A.67]{Bosco} of the solid tensor product, the statement follows taking the inverse limit of (\ref{lastt}) over $n\in \Nn$, observing the following $R^1\varprojlim$ vanishing statements.
  \begin{itemize}
   \item Using that $\{U_n\}_{n\in \Nn}$ is a Stein covering of $X$, by the condensed version of the Mittag-Leffler criterion for Banach spaces, \cite[Lemma A.37]{Bosco}, we have that
  $$R^1\varprojlim_n \Omega^i(U_n^\dagger)=R^1\varprojlim_n \Omega^i(U_n)=0.$$  
  \item Since $H_{\dR}^i(U_n^\dagger)$, for varying $n\in \Nn$, are finite-dimensional condensed $C$-vector spaces, we have that
  $$R^1\varprojlim_n H_{\dR}^i(U_n^\dagger)=0.$$
  \item Lastly, we have
  \begin{equation}\label{assert}
   R^1\varprojlim_n (H^i_{\HK}(U_n^\dagger)\otimes _{\breve F}B_{\log})^{N=0, \varphi=p^i}=0.
  \end{equation}
 The claim (\ref{assert}) is essentially contained in the proof of \cite[Lemma 3.28]{CDN1} (that we write here in slightly different terms). By the Mittag-Leffler criterion for condensed abelian groups,\footnote{It follows from the Mittag-Leffler criterion for abelian groups, \cite[Proposition 13.2.2]{Groth}, applying the latter to the values on extremally disconnected sets, and using \cite[Lemma 3.18]{Scholze}.} it suffices to show that in the inverse system $\{(H^i_{\HK}(U_n^\dagger)\otimes _{\breve F}B_{\log})^{N=0, \varphi=p^i}, f_{nm}\}$, for each $n\in \Nn$ there exists $k\ge m$ such that, for every $m\ge k$, the image of the maps $f_{nm}$ are equal to the image of $f_{nk}$. For this, recall that, for $\cl E_n= \cl E(H_{\HK}^i(U_n^\dagger))\otimes \cl O(i)$, we have $H^0(\FF, \cl E_n)\cong (H^i_{\HK}(U_n^\dagger)\otimes _{\breve F}B)^{\varphi=p^i}$ (trivializing the monodromy), and $H^1(\FF, \cl E_n)=0$. Considering the Euler-Poincaré characteristic of $\cl E_n$, by (\ref{ep}) we have that 
  $$\Dim \cl H^0(\FF, \cl E_n)=(\deg(\cl E_n), \rk(\cl E_n))$$
  we deduce that the BC space $\cl H^0(\FF, \cl E_n)$ has height $\ge 0$; moreover, by the characterization of BC spaces in terms of
  $\FF$, Proposition \ref{lbc}, any BC subspace of $\cl H^0(\FF, \cl E_n)$ has height $\ge 0$. Thus, in the inverse system $\{\cl H^0(\FF, \cl E_n), f_{nm}\}$, for each $n\in \Nn$ the image of the maps $f_{nm}$ form a chain of BC spaces with decreasing Dimension (for the lexicographic order on $\Nn\times \Zz$) and height $\ge 0$; in particular, such chain eventually stabilizes, as desired.
  \end{itemize}
 \end{proof}
 
 As a consequence of Theorem \ref{steindiagram} we have the following result.
 
 \begin{cor}\label{Steinvanish}
  Let $X$ be a smooth Stein space over $C$. Then, we have $$H^i_{\pet}(X, \Qq_p)=0$$ for all $i>\dim X$.
 \end{cor}
 \begin{proof}
  Using Theorem \ref{steindiagram}, it suffices to prove that $H^i_{\HK}(X)=0$ for all $i>\dim X$. Thus, by Corollary \ref{samedim}, we can reduce to showing that $H^i_{\dR}(X)=0$ for all $i>\dim X$, which holds true by Kiehl's ayclicity theorem for Stein spaces (cf. \cite[Lemma 5.9]{Bosco} for the statement in the condensed setting).
 \end{proof}

  

 \subsection{Smooth affinoid spaces}\label{smlabel}

 In view of Theorem \ref{steindiagram} and Proposition \ref{expic}\listref{expic:2}, it is natural to formulate the following conjecture.

 \begin{conj}\label{conju}
   Let $X$ be a smooth affinoid rigid space over $C$. For any $i\ge 0$, we have a commutative diagram in $\Vect_{\Qq_p}^{\ssolid}$ with exact rows
 \begin{center}
   \begin{tikzcd}\label{bvac2}
   0 \arrow[r] 
     &[-1em] \Omega^{i-1}(X)/\ker d \arrow[-,double line with arrow={-,-}]{d} \arrow[r] 
     &[-1em] H^i_{\pet}(X, \Qq_p(i)) \arrow[d] \arrow[r]      
     &[-1em] (H^i_{\HK}(X)\solid_{\breve F}B_{\log})^{N=0, \varphi=p^i} \arrow[d] \arrow[r] 
     &[-1em] 0  \\
   0 \arrow[r] 
     &[-1em] \Omega^{i-1}(X)/\ker d \arrow[r]           
     &[-1em] \Omega^i(X)^{d=0} \arrow[r]  
     &[-1em] H^i_{\dR}(X) \arrow[r]    
     &[-1em] 0.
  \end{tikzcd}
  \end{center}
 \end{conj}

 The goal of this subsection is to show Conjecture \ref{conju} for curves.

 \begin{theorem}\label{affcurves}
  Conjecture \ref{conju} holds true for $X$ a smooth affinoid rigid space over $C$ of dimension 1.
 \end{theorem}
 \begin{proof}
  By \cite[Theorem 7 and Remark 2]{Elkik}, we can assume that $X$ is the base change to $C$ of a smooth affinoid $X_0$  defined over a finite extension of $K$, and, without loss of generality, we can further assume that $X_0$ is defined over $K$.
  
  We recall that by Theorem \ref{BK=pet}\listref{BK=pet:2} combined with Theorem \ref{B=HK}, Theorem \ref{B_dR=dR} and Theorem \ref{compatib}, we have, for any $i\ge 0$, the following commutative diagram whose rows are exact triangles
  \begin{center}
   \begin{tikzcd}\label{bubble}
     &[-1em] R\Gamma_{\syn, \FF}(X, \Qq_p(i)) \arrow[d] \arrow[r] 
     &[-1em] (R\Gamma_{\HK}(X_C)\dsolid_{\breve F}B_{\log})^{N=0, \varphi=p^i}  \arrow[d] \arrow[r]      
     &[-1em] (R\Gamma_{\dR}(X)\dsolid_K B_{\dR}^+)/\Fil^i \arrow[-,double line with arrow={-,-}]{d}  
     \\

     &[-1em] \Fil^i(R\Gamma_{\dR}(X)\dsolid_K B_{\dR}^+) \arrow[r]           
     &[-1em] R\Gamma_{\dR}(X)\dsolid_K B_{\dR}^+\arrow[r]  
     &[-1em] (R\Gamma_{\dR}(X)\dsolid_K B_{\dR}^+)/\Fil^i   
  \end{tikzcd}
  \end{center}
  and, by Theorem \ref{BK=pet}\listref{BK=pet:1}, we have an isomorphism $\tau^{\le i}R\Gamma_{\syn, \FF}(X, \Qq_p(i))\overset{\sim}{\to} \tau^{\le i}R\Gamma_{\pet}(X, \Qq_p(i))$.
  Moreover, by Tate's acyclicity theorem, the bottom exact triangle of the diagram above maps, via Fontaine's morphism $\theta: B_{\dR}^+\to C$, to the following exact triangle
  \begin{center}
  \begin{tikzcd}
    &[-1em] \Omega^{\ge i}(X)[-i] \arrow[r]           
     &[-1em] \Omega^{\bullet}(X)\arrow[r]  
     &[-1em] \Omega^{\le i-1}(X). 
  \end{tikzcd}
  \end{center}
  Then, reducing to the case $X$ is connected, the statement for $i=0$ follows immediately using that $B_{\log}^{N=0, \varphi=1}=\Qq_p$. 
  For the statement in the case $i=1$, taking cohomology of the diagram above and using Proposition \ref{expic}\listref{expic:2}, it remains to check that
  \begin{equation}\label{labbel}
   H^1(R\Gamma_{\HK}(X)\dsolid_{\breve{F}}B_{\log})^{N=0, \varphi=p})\cong (H^1_{\HK}(X)\solid_{\breve{F}}B_{\log})^{N=0, \varphi=p}.
  \end{equation}
  For this, we note that, since $B_{\log}$ is a flat solid $\breve{F}$-vector space (by \cite[Corollary A.65]{Bosco}, as it is a filtered colimit of $\breve{F}$-Fréchet spaces), we have that $H^1(R\Gamma_{\HK}(X)\dsolid_{\breve{F}}B_{\log})\cong H^1_{\HK}(X)\solid_{\breve{F}}B_{\log}$. Then, the isomorphism (\ref{labbel}) follows from the exactness of the sequences
  $$0\to B\to B_{\log}\overset{N}{\to}B_{\log}\to 0\;\;\;\;\;\;\;\;\;\; 0\to B^{\varphi=p}\to B\overset{\varphi-p}{\to}B\to 0.$$
  Lastly, the statement for $i>1$ follows from Lemma \ref{Artinvanish} below (which implies the vanishing of $H^i_{\pet}(X, \Qq_p(i))$ as $X$ is qcqs), together with the fact that, in this case, $H^i_{\dR}(X)$ vanishes, and, by Corollary \ref{samedim}, $H^i_{\HK}(X)$ vanishes as well.
 \end{proof}

 We used crucially the following result.
 
 \begin{lemma}[Rigid-analytic $p$-adic Artin vanishing]\label{Artinvanish}
  Let $X$ be a smooth affinoid rigid space over $C$. Then, we have $$H^i_{\pet}(X, \Zz_p)=0$$ for all $i>\dim X$.
 \end{lemma}
 \begin{proof}
  Fix an integer $i>\dim X$. By a result of Bhatt--Mathew, \cite[Remark 7.4(2)]{BhattMathew}, for any $n\ge 1$, we have the vanishing $H^i_{\pet}(X, \Zz/p^n)=0$.\footnote{The cited result does not keep track of the condensed structure on $H^i_{\pet}(X, \Zz/p^n)$, however doing this does not pose any problem.} Thus, using the exact sequence
  $$0\to R^1\varprojlim_n H_{\pet}^{i-1}(X, \Zz/p^n)\to H_{\pet}^i(X, \Zz_p)\to \varprojlim_n H_{\pet}^i(X, \Zz/p^n)\to 0$$
  it remains to show that
  \begin{equation}\label{tp}
   R^1\varprojlim_n H^{i-1}_{\pet}(X, \Zz/p^n)=0.
  \end{equation}
  Considering the long exact sequence associated to the short exact sequence of sheaves on $X_{\pet}$
  $$0\to \Zz/p\to \Zz/p^{n+1}\to \Zz/p^{n}\to 0$$
  and using again that $H^i_{\pet}(X, \Zz/p)=0$, we deduce that the transition maps of the inverse system $\{H^{i-1}_{\pet}(X, \Zz/p^n)\}_{n}$ are surjective; then, we conclude by the Mittag-Leffler criterion.
  \end{proof}

 Now, let us at least indicate a possible direction for proving Conjecture \ref{conju} in dimension higher than 1.

 \begin{rem}[On the obstruction to proving Conjecture \ref{conju}]
  Let us assume for simplicity that $X$ is an affinoid rigid space over $C$ having a smooth formal model over $\cl O_C$. In this case, by Theorem \ref{mainHK}\listref{mainHK:1} we have in particular that the monodromy action $N$ on the Hyodo-Kato cohomology of $X$  is trivial. We claim that to prove Conjecture \ref{conju} it suffices to show that
  \begin{equation}\label{hgui}
   H^1(\FF, \cl H_{\FF}^n(X)\otimes \cl O(m))=0\;\;\;\; \text{ for all } 0\le n\le m
  \end{equation}
  where $\cl H_{\FF}^n(X)$ denotes the $n$-th Fargues--Fontaine cohomology group of $X$ (Definition \ref{drff}). We note that, by Proposition \ref{perff} (applied to $\cl H_{\FF}^n(X)\otimes \cl O(m)\in \Nuc(\FF)$ concentrated in degree 0) and (the proof of) Theorem \ref{B=HK}, the condition (\ref{hgui}) is equivalent to the following one\footnote{Here, we recall \cite[Corollary A.65]{Bosco} and we note that we have $R\lim_{I\subset (0, \infty)} H^{n}_{\HK}(X)\solid_{\breve F}B_I\simeq H^{n}_{\HK}(X)\solid_{\breve F}B$ by \cite[Corollary A.67(ii)]{Bosco}, using that $H^{n}_{\HK}(X)$ is a quotient of $\breve F$-Banach spaces.}
  \begin{equation}\label{deuif}
   \varphi-p^m: H^{n}_{\HK}(X)\solid_{\breve{F}}B \to H^{n}_{\HK}(X)\solid_{\breve{F}}B\;\; \text{ is surjective  for all } 0\le n\le m.
  \end{equation}
  Before proving the claim, we pause to remark that (\ref{hgui}) holds replacing $X$ with a dagger structure $X^{\dagger}$ on $X$ (Remark \ref{presmooth}), in fact, thanks to Theorem \ref{lb}\listref{lb:1} and Theorem \ref{slopp}, the vector bundle $\cl H_{\FF}^n(X^{\dagger})$ on $\FF$ has Harder--Narasimhan slopes $\ge -n$. Recalling Theorem \ref{mainHK}\listref{mainHK:1}, this suggests the following question: can one prove (\ref{deuif}) using that the Frobenius on the $n$-th crystalline cohomology group has an inverse up to $p^n$, \cite[Theorem 1.8(6)]{Prisms}?
  
  Now, to prove that the condition (\ref{hgui}) implies Conjecture \ref{conju}, as in the proof of Theorem \ref{affcurves},  by \cite[Theorem 7 and Remark 2]{Elkik}, we can assume that $X$ is the base change to $C$ of a smooth affinoid $X_0$  defined over a finite extension of $K$, and, without loss of generality, we can further assume that $X_0$ is defined over $K$.
   Then, again as in the proof of Theorem \ref{affcurves}, combining Theorem \ref{tog} with Proposition \ref{expic}\listref{expic:2}, for all $i\ge 0$ we obtain the following commutative diagram in $\Vect_{\Qq_p}^{\ssolid}$ with exact rows
    \begin{center}
   \begin{tikzcd}
   (H^{i-1}_{\HK}(X)\solid_{\breve F}B)^{\varphi=p^i} \arrow[d, "\gamma_i"] \arrow[r] 
     &[-1em] \Omega^{i-1}(X)/d\Omega^{i-2}(X) \arrow[-,double line with arrow={-,-}]{d} \arrow[r] 
     &[-1em] H^i_{\pet}(X, \Qq_p(i)) \arrow[d] \arrow[r, "\alpha_i"]      
     &[-1em] (H^i_{\HK}(X)\solid_{\breve F}B)^{\varphi=p^i} \arrow[d] \arrow[r] 
     &[-1em] 0  \\
   0\to H^{i-1}_{\dR}(X) \arrow[r] 
     &[-1em] \Omega^{i-1}(X)/d\Omega^{i-2}(X) \arrow[r]           
     &[-1em] \Omega^i(X)^{d=0} \arrow[r]  
     &[-1em] H^i_{\dR}(X) \arrow[r]    
     &[-1em] 0.
  \end{tikzcd}
  \end{center}
  Here, we used that, for any $0\le j\le i$, we have
  \begin{equation*}
   H^j((R\Gamma_{\HK}(X)\dsolid_{\breve{F}}B)^{\varphi=p^i})\cong (H^j_{\HK}(X)\solid_{\breve{F}}B)^{\varphi=p^i}.
  \end{equation*}
   In fact, since $B$ is a flat solid $\breve{F}$-vector space (by \cite[Corollary A.65]{Bosco}), we have an exact sequence
  $$0\to (H^{j-1}_{\HK}(X)\solid_{\breve{F}}B)/\im(\varphi-p^i)\to H^j((R\Gamma_{\HK}(X)\dsolid_{\breve{F}}B)^{\varphi=p^i})\to (H^j_{\HK}(X)\solid_{\breve{F}}B)^{\varphi=p^i}\to 0$$
  and the left term vanishes thanks to (\ref{deuif}).
   Now, it remains to show that we have $$\ker \alpha_i \cong \Omega^{i-1}(X)/\ker d.$$
   From the diagram above, we obtain the following commutative diagram with exact rows
   \begin{equation}
   \begin{tikzcd}
   (H^{i-1}_{\HK}(X)\solid_{\breve F}B)^{\varphi=p^i} \arrow[d, "\gamma_i"] \arrow[r] 
     &[-1em] \Omega^{i-1}(X)/d\Omega^{i-2}(X) \arrow[-,double line with arrow={-,-}]{d} \arrow[r]     
     &[-1em] \ker \alpha_i \arrow[d] \arrow[r] 
     &[-1em] 0  \\
   0\to H^{i-1}_{\dR}(X) \arrow[r] 
     &[-1em] \Omega^{i-1}(X)/d\Omega^{i-2}(X) \arrow[r]           
     &[-1em] \Omega^{i-1}(X)/\ker d \arrow[r]    
     &[-1em] 0.
  \end{tikzcd}
  \end{equation}
   Then, using the snake lemma, we deduce that we need to show that the map $\gamma_i$ is surjective. For this, by \cite[Proposition II.2.3]{FS} we have the following exact triangle in $\QCoh(\FF)$
  \begin{equation}\label{fig}
   \cl H_{\FF}^{i-1}(X)\otimes \cl O(i-1)\overset{\cdot t}{\to} \cl H_{\FF}^{i-1}(X)\otimes \cl O(i)\to \iota_{\infty, *}(H_{\dR}^{i-1}(X))
  \end{equation}
   where $\iota_{\infty, *}:\QCoh(\Spa(C))\to \QCoh(\FF)$ denotes the pushforward functor. Here, we used Theorem \ref{mainHK}\listref{mainHK:3} and the flatness of $C$ for the solid tensor product $\solid_{\breve F}$ (\cite[Corollary A.65]{Bosco}). Taking the long exact sequence in cohomology associated to (\ref{fig}) on $\FF$, recalling Proposition \ref{perff}, from (\ref{hgui}) we deduce that $\gamma_i$ is surjective, as desired.
 \end{rem}

 \subsection{Remarks about coefficients}\label{coeffi}
 
 In this subsection, we indicate a partial extension of Theorem \ref{tog} to coefficients. For simplicity, we restrict ourselves to smooth rigid-analytic varieties over $C$.

 \begin{df}
  Let $X$ be a smooth rigid-analytic variety over $C$. Denote by $\nu:X_{\pet}\to X_{\ett, \cond}$ the natural morphism of sites. For $\Mm$ a $\Bb$-local system on $X_{\pet}$, i.e. a sheaf of $\Bb$-modules that is locally on $X_{\pet}$ free of finite rank, we define the \textit{$B$-cohomology of $X$ with coefficients in $\Mm$} as the complex of $D(\Mod^{\cond}_{B})$
  $$R\Gamma_{B}(X, \Mm):=R\Gamma_{\ett, \cond}(X, L\eta_tR\nu_*\Mm).$$
  and we endow it with the filtration décalée. 
 \end{df}

  With the definition above, given $\Ll$ a $\Qq_p$-local system on $X_{\pet}$ with associated $\Bb$-local system $\Mm=\Ll\otimes_{\Qq_p}\Bb$, tensoring with $\Ll\otimes_{\Qq_p}-$ the exact sequence (\ref{suite4}), the same argument used in the proof of Theorem \ref{BK=pet}\listref{BK=pet:1} shows that we have a natural isomorphism in $D(\Vect_{\Qq_p}^{\cond})$
  $$\tau^{\le i}\Fil^iR\Gamma_B(X, \Mm)^{\varphi=p^i}\overset{\sim}{\to} \tau^{\le i}R\Gamma_{\pet}(X, \Ll(i)).$$
  Similarly, we have a version with coefficients of Theorem \ref{BK=pet}\listref{BK=pet:2}. Then, combining the same results cited in the proof of Theorem \ref{B_dR=dR} (in the smooth case), which rely on Scholze’s Poincaré lemma for $\Bb_{\dR}^+$, we obtain the following result. \medskip

 \begin{theorem}
  Let $X$ be a connected, paracompact, smooth rigid-analytic variety defined over $K$. Let $\Ll$ be a de Rham $\Qq_p$-local system on $X_{\pet}$, with associated $\Bb$-local system $\Mm=\Ll\otimes_{\Qq_p}\Bb$, and associated  filtered $\cl O_X$-module with integrable connection $(\cl E, \nabla, \Fil^\bullet)$. For any $i\ge 0$, we have a  $\mathscr{G}_K$-equivariant isomorphism in $D(\Vect_{\Qq_p}^{\ssolid})$
  $$\tau^{\le i}R\Gamma_{\pet}(X_C, \Ll(i))\simeq \tau^{\le i}\fib(R\Gamma_{B}(X_C, \Mm)^{\varphi=p^i}\to (R\Gamma_{\dR}(X, \cl E)\dsolid_K B_{\dR}^+)/\Fil^i).$$
 \end{theorem}

  \clearpage

 \appendix

 \addtocontents{toc}{\protect\vskip5pt}

 \section{\textbf{Complements on condensed mathematics}}\label{complem}
 \sectionmark{}

 This appendix consists of a miscellaneous collection of results on condensed mathematics that we use in the main body of the paper.

 \begin{convnot}
  In this appendix, we adopt the same notation and set-theoretic conventions of \cite[Appendix A]{Bosco}. All condensed rings will be assumed to be commutative and unital. 
 \end{convnot}

 \subsection{Derived $p$-adic completion and solidification}
 
  In this section, we compare the derived $p$-adic completion to the solidification. \medskip
 
 \begin{notation}
 Let $p$ be a prime number. In the following, for $M\in D(\CondAb)$, we denote by
 $$M^\wedge_p:=R\varprojlim_{n\in \Nn} (M\otimes_{\Zz}^{\LL}\Zz/p^n)\in D(\CondAb)$$
 its \textit{derived $p$-adic completion}. We say that an object $M\in D(\CondAb)$ is \textit{derived $p$-adically complete} if the natural map $$M\to M^\wedge_p$$ is an isomorphism of $ D(\CondAb)$.
 
 \smallskip
 
  For $A$ a solid ring, we write $\Mod_A^{\ssolid}$ for the symmetric monoidal category of $A$-modules in $\Solid$, endowed with the solid tensor product $\solid_A$.

 \end{notation}

  \begin{prop}[cf. \cite{Cop}]\label{solid-vs-padic}
   Let $A$ be a derived $p$-adically complete solid ring. For all cohomologically bounded above complexes $M, N\in D(\Mod_A^{\ssolid})$, we have a natural isomorphism
   \begin{equation}\label{guyr}
    M^\wedge_p\dsolid_A N^\wedge_p\overset{\sim}{\longrightarrow} (M\dsolid_A N)^\wedge_p.
   \end{equation} 
  \end{prop}
  \begin{proof}
  First, we recall that the category $\Mod_A^{\ssolid}$ is generated under colimits by the compact projective objects $(\prod_I\Zz)\solid_{\Zz}A=(\prod_I\Zz_p)\solid_{\Zz_p}A$, for varying sets $I$. By hypothesis, we  may assume that $M$ and $N$ are connective; then, writing $M=\colim_{[n]\in \Delta^{\op}}M_n$ in $D(\Mod_A^{\ssolid})$ with $M_n$ a direct sum of objects of the form $(\prod_I\Zz_p)\solid_{\Zz_p}A$,  the natural map
  $\colim_{[n]\in\Delta^{\op}}(M_n)^\wedge_p\to M^\wedge_p$
  is an isomorphism (as it can be checked via a spectral sequence),  and similarly for $N$.
 Therefore, using that the solid tensor product commutes with colimits, it suffices to prove the statement for $M=M'\solid_{\Zz_p}A$ and $N=N'\solid_{\Zz_p}A$ (concentrated in degree 0), with $M'$ and $N'$ objects of $\Mod_{\Zz_p}^{\ssolid}$ of the form $\bigoplus_{j\in J} (\prod_{I_j}\Zz_p)$ for varying sets $J$ and $I_j$, for $j\in J$. 
  
  In the case $A=\Zz_p$, using Lemma \ref{craz} the statement readily reduces (cf. the proof of \cite[Proposition A.49]{Bosco}) to the isomorphism 
  $
   \prod_I\Zz_p \dsolid_{\Zz_p}\prod_{I'} \Zz_p=\prod_{I\times I'} \Zz_p
  $
  which holds for any sets $I$ and $I'$ (see e.g. \cite[Remark A.18]{Bosco}).
  
 In general, we want to show that we can reduce to the case $A=\Zz_p$. We will use that $\prod_I\Zz_p$ is flat for the tensor product $\solid_{\Zz_p}$, i.e. for any $Q\in \Mod_{\Zz_p}^{\ssolid}$ we have that $\prod_I\Zz_p\dsolid_{\Zz_p}Q$ is concentrated in degree 0: for this, writing $Q$ as a filtered colimit of quotients of objects of the form $\prod_J \Zz_p$, we can reduce to the case $Q$ is derived $p$-adically complete; in this case, using (\ref{guyr}) for $A=\Zz_p$, by the derived Nakayama lemma, we can reduce to checking the claim modulo $p$, in which case it follows from \cite[Lemma A.19]{Bosco}, using that any solid $\Ff_p$-module can be written as a filtered colimit of profinite $\Ff_p$-vector spaces, \cite[Proposition 2.8]{Scholzeanalytic}. Then, from (\ref{guyr}) for $A=\Zz_p$, using that $A$ is derived $p$-adically complete, we deduce that
  $$M^\wedge_p\cong (M'\dsolid_{\Zz_p}A)^\wedge_p\cong (M')^\wedge_p\dsolid_{\Zz_p}A$$
  and similarly for $N$. Hence, we have a natural isomorphism
  $$M^\wedge_p\dsolid_A N^\wedge_p \cong (M')^\wedge_p\dsolid_{\Zz_p} (N')^\wedge_p\dsolid_{\Zz_p} A\cong (M'\dsolid_{\Zz_p} N'\dsolid_{\Zz_p}A)^\wedge_p\cong (M\dsolid_A N)^\wedge_p$$
  where we used again (\ref{guyr}) for $A=\Zz_p$.
 \end{proof}
 
 We used crucially the following result.
 
 \begin{lemma}\label{craz} For $c\in [0, 1]$ we denote $(\Zz_{p})_{\le c}:=\{x\in \Zz_p: |x|\le c\}$. For any sets $J$ and $I_j$, for $j\in J$, we have
  \begin{equation}\label{explicit}
  (\bigoplus_{j\in J} \prod_{I_j}\Zz_p)^\wedge_p=\varinjlim_{f:J\to [0, 1], f\to 0}\;\prod_{j\in J}\prod_{I_j}(\Zz_{p})_{\le f(j)}.
  \end{equation}
 where the colimit runs over the functions $f:J\to [0, 1]$ tending to $0$ (i.e. for every $\varepsilon >0$, the set $\{j\in J: |f(j)|\ge \varepsilon\}$ is finite) partially ordered by the relation of pointwise inequality $f\le g$.
 \end{lemma}
 \begin{proof}
  We will adapt the proof of \cite[Lemma A.53]{Bosco}. It suffices to prove (\ref{explicit}) on $S$-valued points, for all extremally disconnected set $S$. Let $M:=\bigoplus_{j\in J} \prod_{I_j}\Zz_p$. We note that, thanks to the flatness of the condensed abelian group $M$ (which follows from Lemma \ref{condtors}), the derived $p$-adic completion of $M$ agrees with its underived $p$-adic completion $\varprojlim_{n\in \Nn}M/p^n$. Then, we have
  $$M^\wedge_p(S)=\varprojlim_{n\in \Nn}\bigoplus_{j\in J}\Hom(S, \prod_{I_j}\Zz_p/p^n)$$ from which we deduce that
  $$M^\wedge_p(S)= \textstyle\{(g_j)_{j\in J} \text{ with } g_j\in \mathscr{C}^0(S,\prod_{I_j}\Zz_p): \forall \varepsilon>0,\; g_j(S)\subseteq \prod_{I_j}(\Zz_{p})_{\le \varepsilon} \text{ for all but finitely many } g_j\}$$
  which, in turn, identifies with
  $$\varinjlim_{f:J\to [0, 1],\, f\to 0} \left(\prod_{j\in J}\prod_{I_j}\mathscr{C}^0(S, (\Zz_{p})_{\le f(j)})\right)$$
  thus showing (\ref{explicit}). 
 \end{proof}

  The following lemma was used above and in the main body of the paper.
  
  \begin{lemma}[{\cite[Proposition 3.4]{CScomplex}}]\label{condtors}
   A condensed abelian group $M\in \CondAb$ is flat if and only if, for all extremally disconnected sets $S$, the abelian group $M(S)$ is torsion-free.
  \end{lemma}

  
  

  \subsection{$\infty$-category of nuclear complexes}\label{nuccomp}
   
   In this section, we collect some general properties and characterizations of the $\infty$-category of nuclear complexes attached to an analytic ring, which are due to Clausen--Scholze, focusing in particular on a special class of analytic rings relevant to the main body of the paper.  This section should be read in conjunction with \cite[\S A.6]{Bosco}.

   \begin{convnot}
   
  We recall that an analytic ring $(A, \cl M)$ (in the sense of \cite[Definition 7.4]{Scholzecond}) is \textit{commutative} if $A$ is commutative (and unital), and \textit{normalized} if the map $A\to \cl M[*]$ is an isomorphism. All analytic rings  will be assumed to be commutative and normalized.
  
   For an analytic ring $(A, \cl M)$, we denote by $D(A)$ the derived $\infty$-category $D(\Mod_A^{\cond})$, and we denote by $D(A, \cl M)$ the derived $\infty$-category of $(A, \cl M)$-complete complexes in $D(A)$, equipped with the symmetric monoidal tensor product $\otimes_{(A, \cl M)}^{\LL}$. For $M\in D(A, \cl M)$ we write $$M^{\vee}:=\underline{\Hom}_{D(A, \cl M)}(M, A)$$ for its \textit{dual}.
  \end{convnot}

  \begin{df}[{\cite[Definition 13.10]{Scholzeanalytic}}]\label{nuc2}
  Let $(A, \cl M)$ be an analytic ring.   A complex $N\in D(A, \cl M)$ is \textit{nuclear} if, for all $S\in \ExtrDisc$, the natural map
  \begin{equation}\label{aine2}
   (\cl M[S]^{\vee}\otimes_{(A, \cl M)}^{\LL}N)(*)\to N(S)
  \end{equation}
  in $D(\Ab)$ is an isomorphism. We denote by $\Nuc(A, \cl M)$ the full $\infty$-subcategory of $D(A, \cl M)$ spanned by the nuclear complexes.
 \end{df}

  \begin{rem}\label{genecol2}
   We note that the $\infty$-subcategory $\Nuc(A, \cl M)\subset D(A, \cl M)$ is stable under all colimits (as both the source and the target of (\ref{aine2}) commute with colimits in $N$), and under finite limits (as $\otimes_{( A, \cl M)}^{\LL}$ commutes with finite limits).
  \end{rem}

  In order to recall a useful characterization of nuclear complexes, we need the following definitions.
 
 \begin{df}[{\cite[Definition 8.1]{CScomplex}}]
   Let $(A, \cl M)$ be an analytic ring.
   A map $f:M\to N$ in $D(A, \cl M)$ is called \textit{trace-class} if it lies in the image of the natural map
   $$(M^{\vee}\otimes_{(A, \cl M)}^{\LL}N)(*)\to \Hom_{D(A, \cl M)}(M, N).$$
 \end{df}

 \begin{df}[{\cite[Definition 13.12]{Scholzeanalytic}}]
  Let $(A, \cl M)$ be an analytic ring. An object $M\in D(A, \cl M)$ is called \textit{basic nuclear} if it can be written as the colimit of a diagram
   $$P_0\overset{f_0}{\to}P_1\overset{f_1}{\to} P_2\overset{f_2}{\to}\cdots$$
   where $P_n\in D(A, \cl M)$ are compact objects and $f_n$ are trace-class maps.
 \end{df}
 
 \begin{prop}[{\cite[Proposition 13.13]{Scholzeanalytic}}]\label{bnn}
  Let $(A, \cl M)$ be an analytic ring. An object in $D(A, \cl M)$ is nuclear if and only if it can be written as a filtered colimit of basic nuclear objects.
 \end{prop}

 We deduce the following result.

 \begin{cor}\label{bcnuc}
  Let $f:(A, \cl M)\to (A, \cl N)$ be a morphism of analytic rings. The base change functor $$-\otimes_{(A, \cl M)}^{\LL}(B, \cl N):D(A, \cl M)\to D(B, \cl N)$$ preserves nuclear objects.
 \end{cor}
 \begin{proof}
  It suffices to apply Proposition \ref{bnn} observing that the base change functor preserves compact objects and trace-class maps.
 \end{proof}
 
  To further study the $\infty$-category of nuclear complexes, we will use the following the definition.
 
 \begin{df}[{\cite{CS}}]\label{tracedef2}
  Let $(A, \cl M)$ be an analytic ring. We define the \textit{trace-class functor}
   $$(-)^{\trace}:D(A, \cl M)\to D(A): M \mapsto M^{\trace}$$
 where  $M^{\trace}$ is defined on $S$-valued points,\footnote{Via the equivalence of $\infty$-categories $\Cond(D(\Ab))\cong D(\CondAb)$.} for $S\in \ExtrDisc$, as
  $$M^{\trace}(S)=(\cl M[S]^{\vee}\otimes_{(A, \cl M)}^{\LL}M)(*).$$
 \end{df}

 Now, we collect some basic properties of the trace-class functor.

 \begin{lemma}\label{incc2}   Let $(A, \cl M)$ be an analytic ring.
  \begin{enumerate}[(i)]
  \item\label{incc:02} The trace-class functor $(-)^{\trace}:D(A, \cl M)\to D(A)$ takes values in $D(A, \cl M)$. 
  \item\label{incc:12} A map $f:P\to M$ in $D(A, \cl M)$ with $M$ compact object is trace-class if and only if it factors through $M^{\trace}$.
  \item\label{incc:22} An object $M\in D(A, \cl M)$ is nuclear if and only if the natural map $M^{\trace}\to M$ is an isomorphism.
  \end{enumerate}
 \end{lemma}
 \begin{proof}
 For part \listref{incc:02}, one can adapt the argument of \cite[Lemma A.47(i)]{Bosco}. For part \listref{incc:12}, as in the proof \cite[Lemma A.47(ii)]{Bosco}, the case $P=\cl M[S]$, with $S$ extremally disconnected set, is clear; for a general $P$ compact object of $D(A, \cl M)$, we can reduce to the previous case writing $P$ as a retract of a finite complex whose terms are objects of the form $\cl M[S]$, with $S$ extremally disconnected set (\cite[Tag 094B]{Thestack}). Part (\ref{incc:22}) follows immediately from the definitions.
 \end{proof}

 \begin{prop}[\cite{CS}, \cite{AndrPhd}]\label{incu2}
  Let $(A, \cl M)$ be an analytic ring.  For any $M\in D(A, \cl M)$, the object $M^{\trace}$ (and in particular any nuclear object in $D(A, \cl M)$)\footnote{Here, we use Lemma \ref{incc2}\listref{incc:22}.} can be written as a colimit of shifts of objects of the form $\cl M[S]^{\vee}$ for $S$ extremally disconnected sets.
 \end{prop}
 \begin{proof}
  Recalling that $D(A, \cl M)$ is generated, under shifts and colimits, by $\cl M[T]$ for varying extremally disconnected sets $T$, we can apply the same argument of \cite[Proposition A.48]{Bosco}.
 \end{proof}

 Next, we focus on the categorical properties of the $\infty$-category of nuclear complexes for the following special class of analytic rings used in the main body of the paper.

 \begin{notation}\label{inpartnot}
  Let $F$ be a non-archimedean local field, and let $A$ be a solid $F$-algebra. 
  
  For the analytic ring $(A, \Zz)_{\solidif}=(A, \cl M_A)$ from \cite[Proposition A.29]{Bosco} (i.e. the analytic ring structure on $A$ induced from the analytic ring $\Zz_{\solidif}$), we denote by $\Nuc_A:=\Nuc((A, \Zz)_{\solidif}))$ the full $\infty$-subcategory of nuclear complexes of $\Solid_A:=D((A, \Zz)_{\solidif}))$, and we write $\solid_A$ for the symmetric monoidal tensor product $\otimes_{(A, \Zz)_{\solidif}}$.
 \end{notation}

  \begin{theorem}[\cite{CS}]\label{nuclearbanach2}\
  Let $F$ be a non-archimedean local field, and let $A$ be a nuclear solid $F$-algebra.\footnote{Given a solid $F$-algebra $A$ we say that it is \textit{nuclear} if the underlying solid $F$-module is nuclear in the sense of \cite[Definition A.40]{Bosco} with respect to the analytic ring $(F, \Zz)_{\solidif}$ (as we will see in the proof below, this is equivalent to requiring that the complex $A[0]\in \Solid_F$ is nuclear in the sense of Definition \ref{nuc2}). For example, any Fréchet $F$-algebra is a nuclear $F$-algebra by \cite[Proposition A.64]{Bosco}. } 
  \begin{enumerate}[(i)]
   \item \label{nuclearbanach:12} The subcategory $\Nuc_A\subset \Solid_A$ is a stable $\infty$-category, closed under the tensor product $\dsolid_A$, finite limits, countable products, and all colimits.
   \item \label{nuclearbanach:22} The $\infty$-category $\Nuc_A$ is generated, under shifts and colimits, by the objects $\underline{\Hom}_A(A[S], A)$, for varying $S$ profinite sets.
   \item \label{nuclearbanach:32} An object $M\in \Solid_A$ lies $\Nuc_A$ if and only if $H^i(M)[0]$ lies in $\Nuc_A$ for all $i$.
  \end{enumerate}
 \end{theorem}
 \begin{proof}
  First we note that, by \cite[Proposition 2.11]{Andr}, we have $\cl M_A[S]^\vee\cong \underline{\Hom}_A(A[S], A)$, for all profinite sets $S$. Moreover, as the latter objects are flat for the tensor product $\solid_A$ by \cite[Theorem A.43(ii)]{Bosco}, a complex $N[0]\in \Solid_A$ concentrated in degree 0 is nuclear, in the sense of Definition \ref{nuc2}, if and only if the solid $A$-module $N$ is nuclear, in the sense of \cite[Definition A.40]{Bosco}.
  
  Thanks to the above observations, part \listref{nuclearbanach:22} follows from Proposition \ref{incu2} and \cite[Theorem A.43(ii)]{Bosco}.
  For part \listref{nuclearbanach:12}, the closure of $\Nuc_A\subset \Solid_A$ under finite limits and all colimits was observed more generally in Remark \ref{genecol2}; for the closure under the tensor product $\dsolid_A$ and countable products, taking K-flat resolutions in $\Nuc_A$ (which exist thanks to part \listref{nuclearbanach:22}), we can reduce to the statement of \cite[Theorem A.43(i)]{Bosco}.
  
  For part \listref{nuclearbanach:32}, if $M\in \Solid_A$ lies $\Nuc_A$, then $H^i(M)[0]$ lies in $\Nuc_A$ for all $i$, using again that the objects $\cl M_A[S]^\vee$, for varying profinite sets $S$, are flat for the tensor product $\solid_A$. Conversely, passing to the Postnikov limit, $M\simeq \varprojlim_{n}\tau^{\ge -n}M$, by part \listref{nuclearbanach:12}, we can suppose that $M$ is cohomologically bounded below, and using $M\simeq \varinjlim_{n}\tau^{\le n}M$, we can even suppose that $M$ is bounded, and then concentrated in one degree, in which case the implication is clear.
 \end{proof}
 
 \begin{rem}\label{A=F_nuc}
  Let $A$ be a nuclear solid $F$-algebra and let $M\in \Solid_A$. We have that $M$ lies in $\Nuc_A$ if and only if $M$ (regarded in $\Solid_F$) lies in $\Nuc_F$. In fact, by \cite[Proposition 5.35]{Andr}, for all extremally disconnected sets $S$, we have a natural isomorphism $\cl M_A[S]^\vee\cong\cl M_F[S]^\vee\dsolid_F A$.
 \end{rem}

  \subsection{Quasi-coherent, nuclear, and perfect complexes on analytic adic spaces}\label{andr}
 
  In this section, we recall some results of Andreychev on quasi-coherent, nuclear, and perfect complexes on analytic adic spaces, that we need in the main body of the paper.
 
 \begin{notation} 
 Given a pair $(A, A^+)$ with $A$ a complete Huber ring and $A^+$ a subring of $A^\circ$, we denote by $(A, A^+)_{\solidif}$ the associated analytic ring, \cite[\S 3.3]{Andr}.
 \end{notation}

 \begin{theorem}[{\cite[Theorem 4.1]{Andr}}]\label{4.1}
 Let $Y$ an analytic adic space.
  The association taking any affinoid subspace $U=\Spa(A, A^+)\subset Y$ to the $\infty$-category $D((A, A^+)_{\solidif})$ defines a sheaf of $\infty$-categories on $Y$.  We define the $\infty$-category of \textit{quasi-coherent complexes on $Y$} as the global sections of such sheaf, and we denote it by $\QCoh(Y)$.
 \end{theorem}

 Next, we recall that also nuclear objects satisfy analytic descent on analytic adic spaces.

 \begin{theorem}[{\cite[Theorem 5.41]{Andr}}]\label{5.41}
  Let $Y$ an analytic adic space.
 The association taking any affinoid subspace $U=\Spa(A, A^+)\subset Y$ to the $\infty$-category $\Nuc((A, A^+)_{\solidif})$ defines a sheaf of $\infty$-categories on $Y$. We define the $\infty$-category of \textit{nuclear complexes on $Y$} as the global sections of such sheaf, and we denote it by $\Nuc(Y)$.
 \end{theorem}

  It turns out that the $\infty$-category of nuclear objects associated to an analytic complete Huber pair does not depend on the ring of definition. More precisely, we have the following result.
 
 \begin{theorem}[{\cite{AndrPhd}}]\label{unp}
 Let $(A, A^+)$ a pair with $A$ an analytic complete Huber ring and $A^+$ a subring of $A^\circ$. The $\infty$-category $\Nuc((A, A^+)_{\solidif})$ is generated, under shifts and colimits, by the objects $\underline{\Hom}_A(A[S], A)$ for varying $S$ profinite sets.
 \end{theorem}

 In what follows, given a condensed ring $R$, we denote by $\Perf_R\subset D(\Mod_R^{\cond})$ the $\infty$-subcategory of perfect complexes over $R$, \cite[Definition 5.1]{Andr}. \medskip
 
  Andreychev showed that, passing to dualizable objects, Theorem \ref{5.41} implies the following result.
 
 \begin{theorem}[{\cite[Theorem 5.3]{Andr}}]\label{5.3}
 Let $Y$ an analytic adic space.
 The association taking any affinoid subspace $U=\Spa(A, A^+)\subset Y$ to the $\infty$-category $\Perf_A$ defines a sheaf of $\infty$-categories on $Y$. We define the $\infty$-category of \textit{perfect complexes on $Y$} as the global sections of such sheaf, and we denote it by $\Perf(Y)$.
 \end{theorem}


\normalem

\bibliography{biblio}{}
\bibliographystyle{amsalpha}

\end{document}